\documentclass[twoside,11pt,leqno]{amsart}
\usepackage{amssymb}
\makeatletter
\theoremstyle{plain}
\newtheorem*{result}{Theorem}
\newtheorem{thm}{Theorem}[section]
\newtheorem{prop}[thm]{Proposition}
\newtheorem{cor}[thm]{Corollary}
\newtheorem{lemma}[thm]{Lemma}

\theoremstyle{definition}
\newtheorem{defn}[thm]{Definition}
\newtheorem{defns}[thm]{Definitions}

\theoremstyle{remark}
\newtheorem{rem}[thm]{Remark}
\newtheorem{rems}[thm]{Remarks}

\newtheorem{example}[thm]{Example}

\newcommand{\emphdef}{\textit}                  
\newcommand{\tcite}[1]{\textup{\cite{#1}}}      
\numberwithin{equation}{section}

\newcounter{numl}

\newcommand{\labelnuml}{\textup{(\roman{numl})}}
\newenvironment{numlist}{\begin{list}{\labelnuml}%
{\usecounter{numl}\setlength{\leftmargin}{0pt}%
\setlength{\itemindent}{2\parindent}%
\setlength{\itemsep}{\smallskipamount}
\def\makelabel ##1{\hss \llap {\upshape ##1}}}}{\end{list}}
\newenvironment{bulletlist}{\begin{list}{\labelitemi}%
{\setlength{\leftmargin}{\parindent}\def
\makelabel ##1{\hss \llap {\upshape ##1}}}}{\end{list}}
\newenvironment{bulletpars}{\begin{list}{\labelitemi}%
{\setlength{\leftmargin}{0pt}%
\setlength{\itemindent}{\parindent}%
\setlength{\itemsep}{\smallskipamount}\def
\makelabel ##1{\hss \llap {\upshape ##1}}}}{\end{list}}
\newcommand{\acknowledge}{\subsection*{Acknowledgements}}
\newcommand{\thismonth}{\ifcase\month\or
  January\or February\or March\or April\or May\or June\or
  July\or August\or September\or October\or November\or December\fi
  \space\number\year}
\newcommand{\low}{\@ifnextchar^{}{^{\vphantom x}}}
\newcommand{\high}{\@ifnextchar_{}{_{\vphantom I}}}
\DeclareSymbolFont{script}{U}{eus}{m}{n}
\DeclareSymbolFontAlphabet{\mathscr}{script}
\DeclareMathSymbol{\EuWedge}{0}{script}{"5E}
\DeclareMathAlphabet{\mathrmsl}{OT1}{cmr}{m}{sl}
\newcommand{\rssymb}[2]{\newcommand{#1}{{\mathrmsl{#2}}}}
\newcommand{\calsymb}[2]{\newcommand{#1}{{\mathcal{#2}}}}
\newcommand{\bbsymb}[2]{\newcommand{#1}{{\mathbb{#2}}}}
\newcommand{\lieoper}[2]{\newcommand{#1}{\mathop{\mathfrak{#2}\null}\nolimits}}
\newcommand{\oper}[3][n]{\newcommand{#2}{\mathop{\mathrm{#3}\null}\ifx
  n#1\nolimits\else\limits\fi}}
\newcommand{\rsoper}[3][n]{\newcommand{#2}{\mathop{\mathrmsl{#3}\null}\ifx
  n#1\nolimits\else\limits\fi}}
\bbsymb\C{C} \bbsymb\F{F} \bbsymb\HQ{H}\bbsymb\N{N} \bbsymb\Q{Q}
\bbsymb\R{R} \bbsymb\U{U} \bbsymb\V{V} \bbsymb\W{W} \bbsymb\Z{Z}
\calsymb\cA{A} \calsymb\cB{B} \calsymb\cC{C} \calsymb\cD{D} \calsymb\cE{E}
\calsymb\cF{F} \calsymb\cG{G} \calsymb\cH{H} \calsymb\cI{I} \calsymb\cJ{J}
\calsymb\cK{K} \calsymb\cL{L} \calsymb\cM{M} \calsymb\cN{N} \calsymb\cO{O}
\calsymb\cP{P} \calsymb\cQ{Q} \calsymb\cR{R} \calsymb\cS{S} \calsymb\cT{T}
\calsymb\cU{U} \calsymb\cV{V} \calsymb\cW{W} \calsymb\cX{X} \calsymb\cY{Y}
\calsymb\cZ{Z}
\newcommand{\eps}{\varepsilon}
\newcommand{\gam}{\gamma} \newcommand{\Gam}{\Gamma}
\newcommand{\lam}{\lambda}
\renewcommand{\geq}{\geqslant} \renewcommand{\leq}{\leqslant}
\rsoper\End{End} \rsoper\Hom{Hom}                
\rsoper\Sym{Sym} \rsoper\Skew{Skew}
\rsoper\Aut{Aut}                                 
\rsoper\GL{GL} \rsoper\SL{SL}\rsoper\Symp{Sp}
\rsoper\CO{CO} \rsoper\On{O} \rsoper\SO{SO} \rsoper\Spin{Spin}
\rsoper\CU{CU} \rsoper\Un{U} \rsoper\SU{SU}
\rsoper\Diff{Diff} \rsoper\SDiff{SDiff}
\rsoper\Stab{Stab}                               
\lieoper\der{der}                                
\lieoper\gl{gl} \lieoper\sgl{sl}\lieoper\symp{sp}
\lieoper\co{co} \lieoper\so{so} \lieoper\spin{spin}
\lieoper\cu{cu} \lieoper\un{u}  \lieoper\su{su}
\rsoper\Vect{Vect} \rsoper\Ham{Ham}
\lieoper\stab{stab}                              
\oper\real{Re}  
\oper\imag{Im}  
\newcommand{\ip}[1]{\langle#1\rangle}

\newcommand{\norm}[2][]{|\mkern-2mu|#2|\mkern-2mu|
  _{\lower1pt\hbox{${}_{#1}$}}}
\newcommand{\Norm}[2][]{\bigl|\mkern-3mu\bigr|#2\bigr|\mkern-3mu\bigr|
  _{\lower1pt\hbox{${}_{#1}$}}}
\newcommand{\liebrack}[1]{[#1]}

\newcommand{\abrack}[1]{[\mkern-3mu[#1]\mkern-3mu]}

\newcommand{\act}{\mathinner{\cdot}}
\newcommand{\lie}[1]{\mathfrak{#1}}

\newcommand{\restr}[1]{|_{#1}\low}
\newcommand{\Restr}[1]{\Big|_{#1}}
\newcommand{\setof}[1]{\lbrace#1\rbrace}

\newcommand{\cross}{\mathbin{{\times}\!}\low}
\newcommand{\dual}{^{*\!}}
\newcommand{\mult}{^{\scriptscriptstyle\times}}

\newcommand{\gr}{_{\mathrm{gr}}}
\newcommand{\dsum}{\oplus}                  
\newcommand{\vtens}{\mathinner\tens}        
\newcommand{\tens}{\otimes}                 
\newcommand{\Wedge}{\EuWedge}               
\newcommand{\skwend}{\mathinner{\scriptstyle\vartriangle}}
\newcommand{\interior}{\mathinner\lrcorner} 
\newcommand{\subnormal}{\ltimes}            
\newcommand{\intersect}{\mathinner\cap}     
\newcommand{\setdif}{\smallsetminus}
\newcommand{\from}{\colon}                  
\newcommand{\isom}{\cong}                   
\newcommand{\Laplace}{\Delta}               
\newcommand{\del}{\partial}                 
\newcommand{\dbar}{\overline\partial}       
\newcommand{\dbyd}[1]{\del/\del{#1}}        
\newcommand{\Proj}{\mathrmsl{P}}            
\newcommand{\RP}[1]{\R\Proj^{#1}}           
\newcommand{\CP}[1]{\C\Proj^{#1}}           
\newcommand{\HP}[1]{\HQ\Proj^{#1}}          
\newcommand{\half}{\tfrac12}                
\newcommand{\Cinf}{\mathrm{C}^\infty}       
\newcommand{\ie}{\textit{i.e.}}             
\rsoper\dimn{dim}                           
\rsoper\kernel{ker}\rsoper\image{im}        
\rsoper\alt{alt}   \rsoper\sym{sym}         
\rsoper\Ad{Ad}     \rsoper\ad{ad}           
\rsoper\CoAd{CoAd} \rsoper\coad{coad}       
\rsoper\trace{tr}  \rsoper\trfree{tf}       
\rsoper\detm{det}                           
\rsoper\Vol{Vol}                            
\rsoper\divg{div}                           
\rsoper\rank{rank}                          
\rsoper\degree{deg}                         
\rssymb\iden{id}                            
\rssymb\vol{vol}                            
\renewcommand{\d}{{\mathrmsl{d}}}           
\newcommand{\varthinspace}{\mskip3mu minus 3mu}
\renewcommand{\@secnumfont}{\relax}
\renewcommand{\thesubsection}{\thesection\Alph{subsection}}
\makeatother
\newcommand{\vspan}[1]{\mathopen<#1\mathclose>} 
\newcommand{\ltens}{\varthinspace}          
\newcommand{\lip}[1]{(#1)}                  
\newcommand{\Lip}[1]{\bigl(#1\bigr)}        
\newcommand{\mfd}{M}                        
\newcommand{\submfd}{{\mathchar"0106}}      
\newcommand{\conf}{\mathsf{c}}              
\newcommand{\cip}{\ip}                      
\rssymb{\Gr}{Gr}                            
\rssymb\ricci{ric}                          
\rssymb\scal{scal}                          
\rssymb\nr{r}                               
\rssymb\ns{s}                               
\rsoper\CCoda{CCoda}                        
\rsoper\QCoda{{\mathscr B}}                 
\rsoper\pr{pr}                              
\newcommand{\g}{\mathfrak{g}}               
\newcommand{\h}{\mathfrak{h}}               
\newcommand{\m}{\mathfrak{m}}               
\newcommand{\p}{\mathfrak{p}}               
\newcommand{\Mob}{\text{\slshape M\"ob}}    
\newcommand{\CV}{V}                         
\newcommand{\US}{U}                         
\newcommand{\LC}{\cL}                       
\newcommand{\PL}{\Proj(\LC)}                
\newcommand{\LH}{\partial}                  
\newcommand{\Ln}{{\mathchar"0103}}          
\newcommand{\Lnc}{{\hat\Ln}}                
\newcommand{\Lnps}{\Ln^{(1)}}               
\newcommand{\Mb}{\mathscr{V}}               
\newcommand{\Ms}{\mathscr{M}}               
\newcommand{\Mh}{\mathscr{H}}               
\newcommand{\Mq}{\mathscr{S}}               
\newcommand{\CD}{\mathfrak{D}}              
\newcommand{\CDD}{\CD^D}                    
\newcommand{\CDs}{\CD^\nabla}               
\newcommand{\CDVs}{\CD^{\nabla,\CV}}        
\newcommand{\dDVs}{\d^{\CD^{\nabla,\CV}}}   %
\newcommand{\RDVs}{R^{\CD^{\nabla,\CV}}}    %
\newcommand{\CDh}{\CD^{\h}}                 
\newcommand{\dDh}{\d^{\h}}                  %
\newcommand{\RDh}{R^{\h}}                   %
\newcommand{\CDhV}{\CD^{\h_\CV}}            
\newcommand{\dDhV}{\d^{\h_\CV}}             %
\newcommand{\RDhV}{R^{\h_\CV}}              %
\newcommand{\jDh}{j^{\h}}                   
\newcommand{\jDhV}{j^{\h_\CV}}              
\newcommand{\QC}{Q}
\newcommand{\CQ}{Q}
\newcommand{\CS}{{\mathcal N}}              
\newcommand{\CA}{A}                         
\newcommand{\II}{\mathrm{I\kern-1ptI}}      
\newcommand{\Sh}{\mathbb{S}}                
\newcommand{\wnv}{\xi}                      
\newcommand{\nmv}{U}                        
\newcommand{\con}{\nu}                      
\newcommand{\mco}{{\mathscr P}}             
\newcommand{\mni}{{\mathscr I}}             
\newcommand{\mff}{\chi}                     
\newcommand{\refl}[1]{\rho\low_{#1}}        
\newcommand{\modulo}{\mathbin{\mathrmsl{mod}}}
\newcommand{\tangent}{^{\mkern-2mu\top}}    
\newcommand{\normal}{^{\mkern-1mu\perp}}    
\newcommand{\iI}{{\boldsymbol i}}           
\newcommand{\Quabla}{\pmb{\square}}         
\newcommand{\restrnormal}[1]{|^{\,\perp}_{\! #1}}
\newcommand{\tsum}{\mathop{\textstyle\sum}\nolimits}
\newcommand{\cupp}{\mathbin{\scriptstyle\sqcup}}
\newcommand{\hc}[1]{\boldsymbol[#1\boldsymbol]}
\newcommand{\Hc}[1]{\pmb{\bigl[}#1\pmb{\bigr]}}
\newcommand{\mtrip}[3]{\Biggl[\begin{matrix} #1\\ #2\\ #3\end{matrix}\Biggr]}
\newcommand{\skwendwedge}{\mathinner
{\lower1pt\hbox{$\stackrel{\smash{\lower4pt\hbox{$\wedge$}}}
{\smash{\hbox{$\scriptscriptstyle\vartriangle$}}}$}}}
\def\mob/{M\"obius}
\def\GCR/{Gau\ss--Codazzi--Ricci}
\def\cso/{M\"obius differential}
\def\pcq/{p.c.q.\null}
\def\pcqs/{p.c.q.s}
\def\BGG/{Bernstein--Gelfand--Gelfand}
\usepackage{hyperref}
\begin{document}
\title{Conformal submanifold geometry I--III}
\author{Francis E. Burstall}
\email{feb@maths.bath.ac.uk}
\author{David M. J. Calderbank}
\email{D.M.J.Calderbank@bath.ac.uk}
\address{Mathematical Sciences\\ University of Bath\\
Bath BA2 7AY\\ UK.}
\date{\thismonth}
\begin{abstract}
In Part I, we develop the notions of a M\"obius structure and a conformal
Cartan geometry, establish an equivalence between them; we use them in Part II
to study submanifolds of conformal manifolds in arbitrary dimension and
codimension.  We obtain Gau\ss--Codazzi--Ricci equations and a conformal
Bonnet theorem characterizing immersed submanifolds of $S^n$. These methods
are applied in Part III to study constrained Willmore surfaces, isothermic
surfaces, Guichard surfaces and conformally-flat submanifolds with flat normal
bundle, and their spectral deformations, in arbitrary codimension. The high
point of these applications is a unified theory of M\"obius-flat submanifolds,
which include Guichard surfaces and conformally flat hypersurfaces.
\end{abstract}
\maketitle
\vspace{-0.7cm}
\setcounter{tocdepth}{1}
\tableofcontents
\vspace{-0.4cm}

\section*{Introduction}

Our purpose in this work is to present a comprehensive, uniform, invariant,
and self-contained treatment of the conformal geometry of submanifolds. This
paper begins the study in three parts. In the first two parts, we present the
general theory of conformal submanifolds: in Part I we develop an intrinsic
theory of conformal manifolds; then in Part II, we investigate the geometry
induced on an immersed submanifold of a conformal manifold. The central result
is an analogue in conformal geometry of the theorem of O.~Bonnet which
characterizes submanifolds of euclidean space in terms of geometric data
satisfying \GCR/ equations.  We demonstrate, in Part III, how this theory
gives a systematic description of three integrable systems arising in
conformal submanifold geometry, which describe three classes of submanifolds
and their spectral deformations: constrained Willmore surfaces, isothermic
surfaces, and `M\"obius-flat' submanifolds, these last being Guichard or
channel surfaces in dimension two (and codimension one), and conformally flat
submanifolds in higher dimension.

The conformal geometry of submanifolds has been of great interest to
differential geometers for more than a century, and our contribution is far
from being the first treatment of the conformal Bonnet theorem: the ideas go
back at least to E.~Cartan's beautiful 1923 paper~\cite{Car:ecc}.  Since then,
there have been many attempts to understand conformal submanifold geometry,
either by authors wishing to build on Cartan's work, or by those unaware of
it. In particular, the subject was taken up in the 1940's by
K.~Yano~\cite{Yan:ecc,Yan:cg1-4,Yan:ftc,Yan:gcs,Yan:fc12,Yan:fc34,YaMu:tfc},
who showed (together with Y.~Mut\^o) that $m$-dimensional conformal
submanifolds of $S^n$ could be characterized, for $m\geq 3$, by tensors
$g_{jk}$, $M_{jkP}$ and $L_{PQk}$ satisfying five equations, three of which
are analogues of the \GCR/ equations in euclidean geometry---these last three
are also discussed more recently by L.~Ornea and G.~Romani~\cite{OrRo:cgr}.
Another extensive account, from Cartan's point of view, was developed in the
1980's by C.~Schiemangk and R.~Sulanke~\cite{ScSu:mg1,Sul:mg2,Sul:mg3}, who
obtained a conformal Bonnet theorem at least for the generic case of
submanifolds with no umbilic points.

The issue of umbilic points is an significant one. Recall that these are
points where the tracefree part of the second fundamental form, which is a
conformal invariant, vanishes. However, following ideas of
T.~Thomas~\cite{Tho:dig}, A.~Fialkow~\cite{Fia:ctc,Fia:cdg} already noticed in
1944 that in the absence of umbilic points, there is a unique metric in the
conformal class on the submanifold, with respect to which the tracefree second
fundamental form has unit length. This observation allowed him to obtain a
treatment of generic conformal submanifold geometry in purely riemannian
terms. Around the same time J.~Haantjes~\cite{Haa:cg1,Haa:cg2,Haa:cg3} and
J.~Maeda~\cite{Mae:dmg} presented a conformal theory of curves and surfaces.
Subsequent authors have also obtained results in varying degrees of
generality: let us mention, for instance, G.~Laptev~\cite{Lap:dgim}, M.~Akivis
and V.~Goldberg~\cite{Aki:cdg, AkGo:cdg}, G.~Jensen~\cite{Jen:hoc},
C.~Wang~\cite{WaCP:smg,WaCP:mgs} and the book~\cite{BCGGG:eds} of R.~Bryant
\textit{et al.} on exterior differential systems.

In view of all this work, we would be bold to claim that the main results of
Parts I and II are new. However, we believe there is no complete discussion of
conformal submanifold geometry in the literature with all three of the
following features:
\begin{bulletlist}
\item manifest conformal invariance;
\item no restriction on umbilic points;
\item uniform applicability in arbitrary dimension and codimension.
\end{bulletlist}
Nevertheless, we have no quarrel with the Reader who would prefer to regard
Parts I--II of this work as a modern gloss on Cartan's seminal
paper~\cite{Car:ecc}: as we shall see, recent developments in conformal
geometry make such a gloss extremely worthwhile. Our approach is also greatly
inspired by Sharpe's significant book~\cite{Sha:dg}, which presents the
general theory of Cartan geometries to a modern audience, with many
applications. Unfortunately, his application to conformal submanifold geometry
contains a technical error, which leads him to restrict attention to the
generic case only (no umbilic points) when studying surfaces.\footnote{The
error is in the evaluation of the Ricci trace in the proof
of~\cite[Proposition \textbf{7}.4.9 (a)]{Sha:dg}, where Sharpe claims that a
component of the Cartan connection does not influence the curvature of
surface: in fact it is the case of curves, not surfaces, which is special in
this respect. Hence in~\cite[Theorem \textbf{7}.4.29]{Sha:dg}
and~\cite[Corollary \textbf{7}.4.30]{Sha:dg}, $n\neq2$ should be replaced by
$n\neq1$.}

Another motivation for the present work is the recent development of a deeper
understanding of the algebraic structures underlying conformal geometry, which
we wish to apply. There are five aspects to this which we now explain: \mob/
structures, parabolic geometries, tractor bundles, Lie algebra homology, and
Bernstein--Gelfand--Gelfand operators.

A central difficulty in conformal submanifold geometry lies in the description
of low dimensional submanifolds. The problem, in a nutshell, is that curves
and surfaces in a conformal manifold acquire more intrinsic geometry from the
ambient space than simply a conformal metric. This is closely related to the
fact that in dimension one or two there are more local conformal
transformations than just the \mob/ transformations: in one or two dimensions,
\mob/ transformations are real or complex projective transformations
respectively. In fact, we shall see that a curve in a conformal manifold
acquires a natural real projective structure, whereas a surface acquires a
(possibly) non-integrable version of a complex projective structure.  We call
these \emphdef{M\"obius structures}---the $2$-dimensional ones were introduced
in~\cite{Cal:mew} under this name.

Our paper might more properly be called `M\"obius submanifold geometry'
(cf.~\cite{ScSu:mg1,WaCP:smg}): M\"obius structures provide a notion of
conformal geometry which applies uniformly in all dimensions, modelled on the
$n$-sphere $S^n$ with its group of M\"obius transformations. It is well known
(as observed by Darboux~\cite{Dar:tgs}) that this group is isomorphic to the
semisimple group $O_+(n+1,1)$ of time-oriented orthogonal transformations of
an $(n+2)$-dimensional lorentzian vector space: indeed the Lorentz
transformations act transitively on the projective light-cone, which is an
$n$-sphere, and the stabilizer of a light-line is a parabolic subgroup of
$O_+(n+1,1)$. (Similarly, the model for conformal geometry in signature
$(p,q)$ is the projective light-cone in $\R^{p+1,q+1}$; for notational
convenience we restrict attention to euclidean signature, but the theory
applies more generally with minor modifications.)

Cartan geometries provide a systematic way to make precise the notion of a
curved manifold modelled on a homogeneous space~\cite{Sha:dg}.  The geometries
modelled on $G/P$, where $G$ is semisimple and $P$ is a parabolic subgroup,
are called \emphdef{parabolic geometries}~\cite{CaDi:di,CpSc:pgc,CpSl:pg,
CSS:bgg}.  Hence (M\"obius) conformal geometry is a parabolic geometry.

One way to provide a uniform description of conformal submanifold
geometry---and this is the approach taken by Sharpe---is to use Cartan
connections throughout. However this raises a subtle question of philosophy,
concerning the distinction between `intrinsic' and `extrinsic' geometry. The
problem is that there is a lot of `room' in a Cartan connection to hide
extrinsic data: even in higher dimensions a conformal Cartan connection can
contain much more information than just a conformal metric. However, it is
well-known---and due to Cartan~\cite{Car:ecc} of course---that there is a
preferred class of Cartan connections, the \emph{normal} Cartan connections
which, in any dimension $n\geq3$, correspond bijectively (up to isomorphism)
with conformal metrics. We extend this result to dimensions one and two, by
establishing, in a self-contained way, the existence and uniqueness of the
normal Cartan connection for (M\"obius) conformal geometry in all dimensions.

To do this, we use a linear description of conformal Cartan connections,
pioneered by Thomas~\cite{Tho:dig} (although it is also implicit
in~\cite{Car:ecc}), and rediscovered independently by
P.~Gauduchon~\cite{Gau:ccw} and by T.~Bailey \textit{et al.}~\cite{BEG:tcp},
then developed further by A.~\v Cap and
A.R.~Gover~\cite{CpGo:tbi,CpGo:tpg,CpGo:stc}.  A Cartan connection may be
viewed as a connection on a principal $G$-bundle satisfying an open condition
with respect to a reduction to $P$; however, it can be more convenient to work
with a connection on a vector bundle. Such vector bundles with connection,
induced by a Cartan connection, are now called \emphdef{tractor bundles} or
\emphdef{\textup(local\textup) twistor
bundles}~\cite{BEG:tcp,Bas:ahs,CpGo:tpg}.  In conformal Cartan geometry, when
$G=O_+(n+1,1)$, the conformal Cartan connection on the associated
$\R^{n+1,1}$-bundle is easy to characterize, and,
following~\cite{BEG:tcp,Gau:ccw}, this is what we do.

The existence and uniqueness of the normal Cartan connection is now known to
be a characteristic feature of parabolic geometries, governed by a beautiful
algebraic machine, \emphdef{Lie algebra
homology}~\cite{Bas:ahs,CpSc:pgc,CSS:ahs2,Kos:lac,Och:gfh,Tan:epl}: the
normality condition means that the curvature of the Cartan connection is a
$2$-cycle in a chain complex for this homology theory.

It turns out that Lie algebra homology also governs the extrinsic geometry of
conformal submanifolds. \textit{A priori}, there are many ways to split an
ambient Cartan connection, along a submanifold, into tangential and normal
parts. However, as observed by Sharpe~\cite{Sha:dg}, there is a unique choice
such that the associated second fundamental form is tracefree. This way of
normalizing the extrinsic geometry amounts to requiring that the `extrinsic'
part of the Cartan connection is a $1$-cycle in a Lie algebra homology chain
complex. (Sharpe does not then normalize the induced `intrinsic' Cartan
connection, which leads him to remark~\cite[page 266]{Sha:dg} that ``the
Willmore form, while appearing to be \emph{extrinsic} data from the Riemannian
perspective, is seen to be \emph{intrinsic} data from the perspective of the
\mob/ geometry induced on $M$.'' We do not agree with this viewpoint.)

The final ingredient of our approach, the \emphdef{Bernstein--Gelfand--Gelfand
operators}, reveals that the \GCR/ equations also have a natural homological
interpretation. This is a feature of parabolic geometry that has only been
fully elucidated in the last ten years~\cite{CSS:bgg,CaDi:di}. Since any
tractor bundle $W$ is equipped with a connection, there is an associated
twisted deRham sequence of $W$-valued differential forms; these bundles also
form the chain complex computing $W$-valued Lie algebra homology (where the
Lie algebra boundary operator acts in the opposite direction to the twisted
deRham differential). It turns out that the twisted deRham sequence descends
to give a sequence of (possibly) higher order operators between the Lie
algebra homology bundles.  This Bernstein--Gelfand--Gelfand sequence is a
curved version of the (generalized) Bernstein--Gelfand--Gelfand complex on a
(generalized) flag variety $G/P$ which is a resolution of a $G$-module
corresponding to $W$~\cite{BGG:dob,Lep:bgg}.  The curved analogues were
introduced by M.~Eastwood and J.~Rice~\cite{EaRi:cio} and
R.~Baston~\cite{Bas:ahs} in special cases, and by \v Cap, J.~Slovak and
V.~Sou\v cek~\cite{CSS:bgg} in general.  Further, it was shown by the second
author and T.~Diemer~\cite{CaDi:di} that wedge products of tractor-valued
differential forms descend to give bilinear differential operators between Lie
algebra homology bundles.

This is the most novel aspect of our conformal Bonnet theorem: we show that
the \GCR/ equations may be written in terms of the manifestly
M\"obius-invariant linear and bilinear differential
Bernstein--Gelfand--Gelfand operators defined on Lie algebra homology. This is
not just an aesthetic point, but a practical one: it reveals the minimal,
homological, data describing the geometry of the conformal submanifold, and
shows that of the five \GCR/ equations appearing in other treatments (such
as~\cite{YaMu:tfc}), only three are needed in each dimension---which three
depends on the Lie algebra homology in that dimension. This fact is tedious to
check directly, but follows easily from~\cite{CaDi:di}.

This treatment of conformal submanifold geometry suggests a broader context
for our work: \emph{parabolic subgeometries}. If we regard conformal
submanifolds as being modelled on the conformal $m$-sphere inside the
conformal $n$-sphere, then we might more generally consider geometries
modelled on arbitrary (homogeneous) inclusions between generalized flag
varieties. A particularly important case is the conformal $n$-sphere (or a
hyper-quadric of any signature) in $\R P^{n+1}$, which is the model for
hypersurfaces in projective space.

In a subsequent work~\cite{BuCa:psg} (see~\cite{BuCa:sgfv} for a foretaste,
and some preliminary results) we shall show that our methods do indeed
generalize to a wider setting, supporting the claim that our machinery is
natural.  In the present work, we demonstrate that it is also effective, by
exploring some applications. A number of classical results follow effortlessly
(once the machine is up and running), often in greater generality than was
previously known.

\subsection*{Contents and results}

Having set out our stall, let us explain in more detail the fruit on offer.
As we remarked at the start of this introduction, our development of conformal
submanifold geometry, begins with two parts, like those of
Cartan~\cite{Car:ecc}: intrinsic conformal geometry, and the geometry of
submanifolds.  In the first part of the paper, we present an unashamedly
modern treatment of (M\"obius) conformal geometry in all dimensions. We have
an auxiliary goal here, to show that it is perfectly possible---and indeed
desirable---to carry out computations in conformal geometry without
introducing a riemannian metric in the conformal class. Indeed, when a choice
must be made, we wish to show that it is actually simpler to introduce a Weyl
structure, \ie, a torsion-free conformal connection. This avoids the conformal
rescaling arguments and logarithmic derivatives that are still prevalent in a
surprisingly (to us) large portion of the literature. We believe that
conformal geometry is more about conformal invariance than conformal
covariance (in line with the modern coordinate free, rather than covariant,
view of differential geometry).

In fact we offer the Reader \emph{two} manifestly invariant approaches to
conformal geometry: a top-down, or holistic, one, the \emphdef{conformal
Cartan connection}, and a bottom-up, or reductionist, one, the
\emphdef{M\"obius structure}.  We develop these in parallel in the first part
of the paper: the `A' paragraphs
(\S\S\ref{par:sphere-conf-cart}--\ref{par:diff-lift}) concern conformal Cartan
connections, while the `B' paragraphs
(\S\S\ref{par:dens-conf-metric}--\ref{par:mob-str}) concern M\"obius
structures. The goal of our treatment is to prove that these two notions of
conformal structure are equivalent, which we reach in
section~\ref{sec:equivalence}.

Both approaches are motivated by the projective light-cone description of the
conformal sphere $S^n$, the `flat model' for conformal geometry. The top-down
approach, which consists of a (filtered) vector bundle $\CV$ equipped with a
`conformal Cartan connection', generalizes the bundle $S^n\times\R^{n+1,1}$
with its trivial connection. In section~\ref{sec:conf-man}, we show how
straightforward it is to prove that a manifold $M$ with a flat conformal
Cartan connection is locally isomorphic to $S^n$---the same idea underlies our
proof of the conformal Bonnet theorem in section~\ref{sec:conf-bonnet}.  We
also show how such a conformal Cartan connection induces the most basic
ingredient in the bottom-up approach: a conformal metric.

In section~\ref{sec:conf-alg}, we introduce bundles of conformal algebras:
first the bundle $\so(\CV)$ of filtered Lie algebras, then its associated
graded algebra, which we identify with $TM\dsum \co(TM)\dsum T\dual M$.  Here
we discuss the most crucial elements of Lie algebra homology theory for the
development. In section~\ref{sec:weyl-geom} we turn to Weyl derivatives. From
the top-down approach, these are equivalently complementary subspaces to the
filtration of $\CV$, whereas from the bottom-up approach, they correspond to
torsion-free conformal connections on $TM$.

The key idea in the first part of the paper is introduced in
section~\ref{sec:mob-geometry}. Here we reveal extra information hidden in a
conformal Cartan connection in the form of two second order differential
operators. These operators motivate our definition of M\"obius structure.  We
then define a distinguished class of `conformal' M\"obius structures, thus
paving the way for our proof in section~\ref{sec:equivalence} of the following
result.
\begin{result}
There is a one-to-one correspondence, up to natural isomorphism, between
M\"obius structures and conformal Cartan geometries, and the conformal Cartan
connection is normal if and only if the M\"obius structure is conformal.
\end{result}
For manifolds of dimension at least $3$, similar results have been obtained
in~\cite{Car:ecc,CpGo:tbi,Gau:ccw,Tho:dig}, and a $2$-dimensional version is
sketched in~\cite{Cal:mew}, but the above theorem covers a wider class of
Cartan connections than in these references, a generalization which is vital
for our applications in submanifold geometry.

An immediate consequence of this theorem is a canonical way to normalize a
non-normal conformal Cartan geometry, a procedure which will be very important
in the rest of the paper. We discuss some further ramifications in
sections~\ref{sec:mob-str-low}--\ref{sec:conf-geom-rev}: we provide a more
explicit analysis of M\"obius structures for curves and surfaces, relating
them to the schwarzian derivative; we discuss the gauge theory and moduli of
conformal Cartan connections; and we revisit the flat model to discuss
spaceform geometries and symmetry breaking.

\smallbreak

In the second part of the paper, we apply the formalism of Part I to
submanifolds of conformal manifolds. We emphasise a gauge-theoretic point of
view, but use vector bundles and connections rather than frames and matrices
of 1-forms. Although the latter are more elementary, the choice of frame can
obscure the geometry and lead to large and complicated matrices in
applications. Without frames, the theory has conceptual simplicity: some
investment is needed to master the abstraction, but the payoff is a uniform
and efficient calculus which adapts geometrically to almost any application.

As in Part I, we give two approaches, one emphasising bundles and connections,
the other, more primitive homological data, and we relate these points of
view.  In sections~\ref{sec:sub-mob}--\ref{sec:geom-mob-red}, we show how such
a submanifold inherits a conformal M\"obius structure from the ambient
geometry. We first study how M\"obius structures on the submanifold are
induced by a choice of `M\"obius reduction', then we show that Lie algebra
homology distinguishes a canonical M\"obius reduction (cf.~\cite{Sha:dg}) and
hence obtain a canonical induced conformal M\"obius structure on the
submanifold. This leads to canonical \GCR/ equations relating the curvature of
these induced M\"obius structures to the ambient space.

In the case of submanifolds of the conformal sphere $S^n$, a M\"obius
reduction is the same thing as a sphere congruence enveloped by the
submanifold, which are of great interest in their own right, although we
postpone a detailed studied of sphere congruences \textit{per se} to a sequel
to this paper~\cite{BuCa:45}. The canonical M\"obius reduction in this case is
the central sphere congruence of W.~Blaschke and G.~Thomsen~\cite{Bla:vud} or
the conformal Gau\ss\ map of Bryant~\cite{Bry:dtw,Bry:scg}.  We show that the
minimal data characterizing the submanifold (locally, up to M\"obius
transformation) are certain Lie algebra homology classes satisfying natural
homological \GCR/ equations. These equations encode the flatness of the
ambient conformal Cartan connection of $S^n$, meaning that the immersion can
be reconstructed (locally, up to M\"obius transformation) from the \GCR/
data. Thus we obtain, in section~\ref{sec:conf-bonnet}, an analogue of the
classical Bonnet theorem.

\begin{result} An $m$-manifold $\submfd$ can be locally immersed in $S^n$
with a given induced conformal \mob/ structure, connection on the weightless
normal bundle and tracefree second fundamental form \textup(or conformal
acceleration for $m=1$\textup) if and only if these data satisfy the
homological \GCR/ equations. Moreover, in this case, the immersion is unique
up to a M\"obius transformation of $S^n$.
\end{result}
In section~\ref{sec:weyl-conf-bonnet} we introduce ambient Weyl structures
along a submanifold and use them to give explicit formulae for the operators
entering into the homological \GCR/ equations. Then, in
section~\ref{sec:curv-surf}, we specialize this result first to curves, for
which we compute some M\"obius invariants, and then to surfaces, where we
relate our conformal Bonnet theorem to the explicit approach
of~\cite{BPP:sdfs}, developed by the first author and his coworkers at the
same time as our general theory. Finally we sketch the relation of our
approach to the quaternionic formalism for surfaces in $S^4\cong\HP1$
expounded in~\cite{BFLPP:cgs}.

\smallbreak
The third part of the paper concerns applications, with a view to
demonstrating the speed and effectiveness of our formalism once the machinery
is in place. In section~\ref{sec:env-sphere-cong}, we show how the most basic
aspects of conformal submanifold geometry, such as totally umbilic
submanifolds and channel submanifolds, have straightforward treatments in our
approach. We also discuss symmetry breaking and give an almost
computation-free proof of Dupin's Theorem on orthogonal coordinates.

In sections~\ref{sec:willmore}--\ref{sec:mflat}, we apply our methods to
constrained Willmore surfaces, isothermic surfaces, Guichard surfaces and
conformally flat submanifolds with flat normal bundle, which all find their
natural home in conformal M\"obius geometry and have in common an integrable
systems interpretation: the \GCR/ data come in one parameter families, leading
to spectral deformations and a family of flat connections, which we derive
straightforwardly in arbitrary codimension. For Willmore and constrained
Willmore surfaces, we show that the well-known relation with harmonicity of
the central sphere congruence has a homological interpretation.  We also
demonstrate easily that products of curves, constant mean curvature surfaces,
generalized H-surfaces, and quadrics are all isothermic, and we compute the
spectral deformation for products of curves.

In section~\ref{sec:mflat}, we turn to a class of submanifolds in in arbitrary
dimension and codimension, which we call \emphdef{M\"obius-flat}: if the
dimension is $3$ or more and the codimension is one, this is the well known
theory of conformally-flat hypersurfaces, and in higher codimension, we simply
add flatness of the normal bundle as an extra hypothesis. However, we also
present a new $2$-dimensional theory, which unifies conformally flat
submanifolds with channel surfaces and the surfaces of
Guichard~\cite{Gui:ssi,Cal:asg}. In this theory, we suppose that the surface
(with flat normal bundle) envelopes a sphere congruence for which the induced
normal Cartan connection (or equivalently, the induced conformal M\"obius
structure) is flat. When the sphere congruence is the central sphere
congruence, this means that the canonically induced M\"obius curvature of the
surface (which is an analogue of the Cotton--York curvature of conformal
$3$-manifolds) is zero, and we say the surface is \emphdef{strictly
M\"obius-flat}, whereas in general, this curvature is the exterior derivative
of a quadratic differential commuting with the shape operators, which we call
a \emphdef{commuting Cotton--York potential}. In codimension one, the strictly
M\"obius-flat surfaces are Dupin cyclides (orbits of a two dimensional abelian
subgroup of the M\"obius group), and we obtain a transparent
conformally-invariant derivation of their properties and classification.
M\"obius-flat surfaces in general have a much richer theory: in particular, as
we prove, since they include the Guichard surfaces, they also include the
surfaces of constant gaussian curvature in a spaceform, just as higher
dimensional submanifolds of constant gaussian (sectional) curvature in a
spaceform are conformally-flat. In a subsequent paper we shall show that the
classical transformation theory of Guichard surfaces extends to M\"obius-flat
submanifolds in arbitrary dimension and codimension.

\acknowledge The first author would like to thank Christoph Bohle, Franz
Pedit, Ulrich Pinkall, Aurea Quintino, Susana Santos and Chuu-Lian Terng for
helpful conversations. The second author is deeply indebted to Paul Gauduchon
and Tammo Diemer for discussions on Cartan connections and submanifold
geometry respectively. He also thanks Dalibor Smid for helpful comments and
computations and Daniel Clarke for his remarks. He is grateful to the
Leverhulme Trust, the William Gordon Seggie Brown Trust and the Engineering
and Physical Sciences Research Council for financial support during the long
period (for part of which he was based at the University of Edinburgh and the
University of York) in which this project came to fruition. Both authors would
like to thank Neil Donaldson, Nigel Hitchin and Udo Hertrich-Jeromin for
helpful comments and conversations, and the Centro di Georgi, Pisa, for its
hospitality during a key stage of this project.

\subsection*{A notational apology} Notation presents significant difficulties
in this work, since there are many different geometric objects which
can be constructed from each other in various ways. In this paper the
Reader will find the following constructions:
\begin{bulletlist}
\item a conformal Cartan connection $\CD^\CV$ from an enveloped sphere
congruence $\CV$;
\item a M\"obius structure $\Ms^\CD$ from a conformal Cartan connection $\CD$;
\item a normalized Ricci curvature $\nr^{D,\Ms}$ from a M\"obius structure
$\Ms$ and a Weyl derivative $D$.
\end{bulletlist}
Combining these constructions could lead to unreadable superscripts, so we
have adopted the policy of omitting intermediate steps whenever this causes no
confusion, leading to notations such as $\Ms^\CV$ and $\nr^{D,\CV}$.  We hope
that our preference for less decorated notation is more a help to the Reader
than a hindrance.

\part{Conformal M\"obius Geometry}

In the first part of our paper, we provide a complete theory of the intrinsic
conformal geometry that we shall use to study conformal submanifold geometry
in the later parts.  Although conformal geometry is well-established, our
theory has many novelties, in particular the notion of a M\"obius structure,
which allows us to describe, in intrinsic terms, the geometry of the large
class of conformal Cartan connections that arise in submanifold geometry.  We
also hope that it is a helpful treatment for several kinds of Reader, from the
representation theorist to the riemannian geometer, because of the
complementary approaches provided by the A and B sections. Let us describe
their contents in more detail.

In \S\ref{par:sphere-conf-cart}, we introduce conformal Cartan connections
$(\CV,\Ln,\CD)$ as curved versions of the model for conformal geometry, the
celestial sphere or projective lightcone in Minkowski space. The conformal
structure induced by such connections motivate the first ingredient of the
bottom-up approach, the conformal metric, which we discuss in
\S\ref{par:dens-conf-metric}, together with the disarmingly simple notion of
densities, which render unnecessary the use of riemannian metrics in the
conformal class.

Algebra in conformal geometry is more subtle than in riemannian geometry,
because the stabilizer of a point (a lightline) in the model is not reductive,
but a parabolic subgroup of the conformal (or M\"obius) group, which is
essentially the isometry group of Minkowski space. Fortunately there is an
algebraic machine, called \emphdef{Lie algebra homology}, which relates
representations of this parabolic group to its reductive part, the
\emphdef{Levi factor}, which is here the conformal linear group. We discuss
this for the filtered Lie algebra $\so(\CV)$ in \S\ref{par:filtered-lie-alg},
then for its associated graded algebra $TM\dsum\co(TM)\dsum T\dual M$ in
\S\ref{par:graded-lie-alg}.

We then introduce \emphdef{Weyl derivatives}, \ie, covariant derivatives on
density line bundles. In~\S\ref{par:weyl-str}, we show that they induce
splittings of the filtration of $\CV$ and hence a curvature decomposition and a
gauge theory.  In \S\ref{par:weyl-conn}, we see that they induce torsion-free
conformal connections and curvature, whose dependence on the Weyl derivative
we describe.

In \S\ref{par:diff-lift}, we introduce a natural differential lift from
$1$-densities to sections of $\CV$ and use it to define two second order
\emphdef{M\"obius operators}, one linear, one quadratic, associated to a
conformal Cartan connection. In \S\ref{par:mob-str}, we show that similar
operators exist naturally on any conformal manifold of dimension $n\geq3$, and
introduce \emphdef{M\"obius structures}, which cover such operators in
generality as well as allowing us to describe the natural ones.

We combine our two threads in section~\ref{sec:equivalence}, where we construct
a conformal Cartan connection from a M\"obius structure and prove that this
leads to an equivalence, so that the natural affine structure on the space of
M\"obius structures can be lifted to the space of conformal Cartan
connections, providing a way to `normalize' such a connection.

In section~\ref{sec:mob-str-low}, we discuss how to compute the M\"obius
structure from the Cartan connection, particularly in dimensions one and two,
where the M\"obius structure can be reinterpreted as a real or (perhaps
non-holomorphic) complex projective structure, and hence related to schwarzian
derivatives, cf.~\cite{BPP:sdfs,Cal:mew}.

Section~\ref{sec:conf-geom-rev} is concerned with further consequences of our
main result. In \S\ref{par:gauge-theory-moduli} we describe the moduli space
of conformal Cartan geometries inducing a given conformal metric. In
\S\ref{par:spaceform} we revisit the model geometry of $S^n$ to relate it to
the flat M\"obius structure of euclidean space and other spaceform geometries.
This sort of symmetry breaking often shows up in submanifold geometry, and in
\S\ref{par:sym-break} we end Part I with a general version of such phenomena.

\section{Conformal geometry}
\label{sec:conf-man}

\subsection{The sphere and conformal Cartan connections}
\label{par:sphere-conf-cart}

The flat conformal structure on the $n$-sphere arises most naturally by
viewing it not as the round sphere $\cS^n$ in the euclidean space $\R^{n+1}$,
but as the celestial sphere $S^n$ in the lorentzian spacetime
$\R^{n+1,1}$~\cite{Dar:tgs}.  Recall that $\R^{n+1,1}$ is real vector space of
dimension $n+2$ with a nondegenerate inner product $\lip{\cdot\,,\cdot}$ of
signature $(n+1,1)$. Inside $\R^{n+1,1}$, we distinguish the \emph{light-cone}
$\LC$:
\begin{equation*}
\LC=\setof{v\in\R^{n+1,1}\setdif\{0\}:\lip{v,v}=0},
\end{equation*}
which is a submanifold of $\R^{n+1,1}$.

Clearly, if $v\in\LC$ and $r\in\R\mult$ then $rv\in\LC$, so that $\R\mult$
acts freely on $\LC$ and we may take the quotient $\PL\subset
\Proj(\R^{n+1,1})$:
\begin{equation*}
\PL=\LC/\R\mult\cong\setof{U\subset\R^{n+1,1}:\text{$U$ is a
$1$-dimensional null subspace}},
\end{equation*}
which is a smooth $n$-sphere. The projection $q\from\LC\to\PL$ is an
$\R\mult$-bundle over $\PL$, which may be viewed as the nonzero vectors in a
tautological line bundle $\Ln$ over $\PL$: $\Ln$ is defined to be the
subbundle of the trivial bundle $\PL\cross\R^{n+1,1}$ whose fibre at
$U\in\PL$ is $U$ itself, viewed as a null line in $\R^{n+1,1}$. It is
convenient to orient $\Ln$, by choosing one of the components $\LC^+$ of
$\LC$: this choice is equivalently a `time-orientation' of $\R^{n+1,1}$.

Note that $T\LC\isom q^*\Ln^\perp\subset\LC\cross\R^{n+1,1}$ with vertical
bundle $q^*\Ln$. Hence the differential $\d\sigma$ of a section $\sigma$ of
$\LC^+\subset\Ln$ can be viewed $\Ln^\perp$-valued $1$-form on $\PL$ such that
$\d\sigma \modulo\Ln$ is an isomorphism $T\PL\to \Ln^\perp/\Ln$. This
isomorphism is clearly algebraic and linear in $\sigma$, so there is a
canonical isomorphism $T\PL\ltens\Ln\isom\Ln^\perp/\Ln$.
Here, and in the following, we omit tensor product signs when tensoring with
line bundles.

Since $\Ln^\perp/\Ln$ inherits a metric from $\R^{n+1,1}$, each positive
section $\sigma$ defines a metric on $\PL$ by
$X,Y\mapsto\lip{\d\low_X\sigma,\d\low_Y\sigma}$. Rescaling $\sigma$ evidently
gives a conformally equivalent metric, and this equips $\PL$ with a conformal
structure. In particular, by considering the conic sections of
$\Ln\subset\PL\cross\R^{n+1,1}$ given by $\lip{v,\sigma}=k$ for fixed
$v\in\R^{n+1,1}$ and $k\in\R\mult$, one sees that this conformal structure is
flat (cf.~\S\ref{par:spaceform}).

The beauty of this model is that it \emph{linearizes} conformal geometry.  Let
$\On_+(n+1,1)$ be the group of Lorentz transformations of $\R^{n+1,1}$
preserving the time orientation. Then the linear action of $\On_+(n+1,1)$ on
$\R^{n+1,1}$ preserves $\LC$ and so descends to an action on $\PL$. Since the
metric on $\Ln^\perp/\Ln$ is preserved, $\On_+(n+1,1)$ acts by conformal
diffeomorphisms on $\PL$ and this gives an isomorphism between $\On_+(n+1,1)$
and the group $\Mob(n)$ of \emphdef{\mob/ transformations} of $S^n=\PL$ (which
are the global conformal diffeomorphisms for $n\geq2$ and the projective
transformations of $S^1\cong \R P^1$ for $n=1$). $\Mob(n)$ acts transitively
on $\PL$ with stabilizer a parabolic subgroup $P\cong\CO(n)\subnormal\R^{n*}$,
which identifies $S^n$ with $\Mob(n)/P$, a generalized flag variety. Note that
signature $(k+1,1)$ linear subspaces of $\R^{n+1,1}$ give the conformal
$k$-spheres in $S^n$ by intersecting with
$\LC$---see~\S\ref{par:mob-sphere-cong} for more details.

A conformal manifold is a `curved version' of $S^n$: curved versions of
homogeneous spaces are usually defined as Cartan connections on principal
bundles; however, since we have described the flat model linearly, using flat
differentiation on $\PL\cross\R^{n+1,1}$, we shall also describe conformal
Cartan connections linearly, cf.~\cite{BEG:tcp,CpGo:tpg,Gau:ccw, Tho:dig}.
\begin{defn}
A \emphdef{conformal Cartan connection} on an $n$-manifold $\mfd$ is defined
by:
\begin{bulletlist}
\item the \emphdef{Cartan vector bundle} $\CV$, a rank $n+2$ vector bundle
with a signature $(n+1,1)$ lorentzian metric on each fibre;
\item the \emphdef{tautological line} $\Ln$, an oriented (so
trivializable) null line subbundle of $\CV$;
\item the \emphdef{Cartan connection} $\CD$, a metric connection on $\CV$
satisfying the following Cartan condition: the \emphdef{soldering form}
$\beta^\CD\colon T\mfd \to \Hom(\Ln,\Ln^\perp/\Ln)$, defined by
\begin{equation}
\beta^\CD_X\sigma=-\CD\low_X\sigma\modulo\Ln
\end{equation}
for any section $\sigma$ of $\Ln$ and any vector field $X$, is a bundle
isomorphism. (This makes sense because the right hand side is algebraic in
$\sigma\vtens X$, as it is for $\mfd=\PL$, $\CD=\d$.)
\end{bulletlist}
\end{defn}
Our data equip $\CV$ and $\mfd$ with structures intertwined by the soldering
form $\beta^\CD$. Our first observation is that the lorentzian metric on $\CV$
induces a conformal metric on $\mfd$, generalizing the case $\mfd=\PL$.  For
this, we denote by $\cL^+$ the positive ray subbundle of $\Ln$, and a section
$\sigma$ of $\cL^+$ (\ie, a positive section of $\Ln$) will be called a
\emphdef{gauge}.  We then obtain a conformal class of metrics on $\mfd$
defined by $X,Y\mapsto\lip{\CD\low_X\sigma,\CD\low_Y\sigma}$ for each gauge
$\sigma$.  In more invariant terms, we can write
$\lip{\CD\low_X\sigma,\CD\low_Y\sigma} =\cip{X,Y}\sigma^2$, where $\cip{X,Y}$
is a fibrewise inner product on $T\mfd$ with values in $(\Ln\dual)^2$ which is
independent of the choice of gauge.

\begin{prop}\tcite{Car:ecc}\label{p:flat-ccc}
The celestial $n$-sphere $S^n=\PL$ is naturally equipped with a flat conformal
Cartan connection. Conversely, if $\mfd$ is equipped with a flat conformal
Cartan connection, then $\mfd$ is locally isomorphic to $S^n$.
\end{prop}
\begin{proof} We take $\CV=\PL\cross\R^{n+1,1}$, with the constant lorentzian
metric, the tautological null line subbundle, and flat differentiation. We
have already seen that the Cartan condition holds, and $R^\CD=0$ by
definition.

Conversely, if $\CD$ is flat, then the inclusion $\Ln\to\CV$ defines a map
$\Phi$ from each simply connected open subset $\Omega$ of $\mfd$ to the space
of parallel null lines in $\CV\restr\Omega$, which is diffeomorphic to
$S^n$. By the Cartan condition, $\Phi$ is a local diffeomorphism. Furthermore
$\Phi$ clearly identifies $\Ln$ with the tautological line bundle over $S^n$,
and $\CD$ with the trivial connection on $S^n\cross\R^{n+1,1}$. In particular
$\Phi$ is conformal.
\end{proof}
The same argument will be used later to obtain an immersion into $S^n$ in
the proof of the Bonnet theorem for conformal submanifolds.

\begin{rem}\label{r:classical-cart}
A Cartan connection in the usual sense (see~\cite{CDS:rcd,Sha:dg}), modelled
on a homogeneous space $G/P$, is a principal $G$-bundle $\cC$ equipped with a
principal $G$-connection $\omega$ and a reduction $\cG\subset\cC$ to a
principal $P$-bundle such that $\omega\restr\cG \modulo\p\colon T\cG\to \g/\p$
is an isomorphism on each tangent space.  When $G=\Mob(n)$ and $G/P=S^n$, it
is easy to see that this induces a conformal Cartan connection on the bundle
associated to the standard representation $\R^{n+1,1}$ of
$G\cong\On_+(n+1,1)$, and this linear representation suffices to recover the
original Cartan connection.
\end{rem}

\subsection{Densities and conformal metrics}
\label{par:dens-conf-metric}

In \S\ref{par:sphere-conf-cart}, we obtained the conformal class of riemannian
metrics on $S^n$ from gauges, \ie, sections of $\cL^+\subset\Ln$. This is an
entirely general phenomenon: a conformal structure on a manifold $\mfd$ is an
oriented line subbundle of $S^2T\dual \mfd$ whose positive sections are
riemannian metrics (\ie, positive definite).  One small difficulty with such a
definition is that each conformal structure has its own line bundle.  However,
we can resolve this by relating any such line bundle to the density bundle of
$\mfd$.

If $E$ is a real $n$-dimensional vector space and $w$ any real number, then
the oriented one dimensional linear space $L^w=L^w_E$ carrying the
representation $A\mapsto|\detm A|^{w/n}$ of $\GL(E)$ is called the space of
\emphdef{densities of weight $w$} or \emphdef{$w$-densities}.  This space
can be constructed canonically from $E$ as the space of maps
\begin{equation*}
\lam\from(\Wedge^n E)\setdif\{0\}\to\R\;\text{ such that }\;
\lam(r A)=|r|^{-w/n}\lam(A)\;\text{ for all }\;
r\in\R\mult,\; A\in(\Wedge^nE)\setdif\{0\}.
\end{equation*}
Note that the modulus of $\omega\in \Wedge^n E\dual$ is an element $|\omega|$
of $L^{-n}$ defined by $|\omega|(A) = |\omega(A)|$.

The \emphdef{density line bundle} $L^w=L^w_{T\mfd}$ on an $n$-dimensional
manifold $\mfd$ is the bundle whose fibre at $x\in \mfd$ is
$L^w_{\smash{T_x\mfd}}$.  This is the associated bundle
$\GL(\mfd)\cross_{\smash{\GL(n)}} L^w(n)$ to the frame bundle $\GL(\mfd)$,
where $L^w(n)$ is the space of $w$-densities of $\R^n$.  The density bundles
are oriented real line bundles, but there is no preferred trivialization
unless $w=0$.

\begin{prop}\label{p:same-L}
Let $\Ln\otimes\Ln$ be a line subbundle of $S^2T\dual\mfd$, with $\Ln$
oriented, whose positive sections are nondegenerate. Then there is a canonical
oriented isomorphism $\Ln\to L^{-1}$.
\end{prop}
\begin{proof} The inclusion of $\Ln^2:=\Ln\otimes\Ln$ into $S^2T\dual \mfd$ is
a section of $S^2T\dual \mfd\ltens(\Ln^{2})^*=S^2(T\mfd\ltens\Ln)\dual$, and
this defines a nondegenerate metric on $T\mfd\ltens\Ln$ by assumption. The
modulus of the volume form of this metric is a positive section of
$L^{-n}\ltens(\Ln^{*})^n$, and the positive $n$th root of this, a section of
$L^{-1}\ltens \Ln^{*}\cong\Hom(\Ln,L^{-1})$, is the isomorphism we seek.
\end{proof}

This proposition tells us that we may view the conformal structure as a
section of $S^2T\dual\mfd\ltens L^2$. The sections we obtain in this way are
not arbitrary, but \emphdef{normalized} by the condition that the modulus of
their volume form is the canonical section $1$ of $L^{-n}\ltens
L^{n}=\mfd\cross\R$.

\begin{defn}\tcite{Tho:dig}
A \emphdef{conformal metric} on $\mfd$ is a normalized metric $\conf$ on the
\emphdef{weightless tangent bundle} $T\mfd\ltens L^{-1}$. It induces an inner
product $\cip{\cdot\,,\cdot}$ on $T\mfd$ with values in $L^2$.
\end{defn}

The use of densities allows a simple geometric dimensional analysis for
tensors~\cite{Die:PhD,Wey:stm}. Sections of $L=L^1$ may be thought of as
scalar fields with dimensions of length: a positive section $\ell$ defines a
\emphdef{length scale}. The tensor bundle
$L^w\tens(T\mfd)^j\tens(T\dual\mfd)^k$ (and any subbundle, quotient bundle,
element or section) will be said to have \emphdef{weight} $w+j-k$.

A conformal metric $\conf$ is weightless, and we shall use it freely to
`raise and lower indices' with the proviso that the weight of tensors is
preserved.  For example, a $1$-form $\gam$ is identified in this way with
a vector field of weight $-1$, that is, a section of $T\mfd \ltens
L^{-2}$. Where necessary for clarity, we denote this isomorphism and its
inverse by $\sharp$ and $\flat$ in the usual way.

\section{Conformal algebra}
\label{sec:conf-alg}

\subsection{The filtered Lie algebra bundle}
\label{par:filtered-lie-alg}

Let $(\CV,\Ln,\CD)$ be a conformal Cartan connection on $\mfd$.  The
filtration $0\subset\Ln\subset\Ln^\perp\subset\CV$ induces a filtration of the
Lie algebra bundle $\so(\CV)$:
\begin{equation*}
0\subset \so(\CV)_{-1}\subset\so(\CV)_0\subset\so(\CV)_1=\so(\CV),
\end{equation*}
where
\begin{align*}
\so(\CV)_{-1} &= \{S\in\so(\CV)\colon S\restr\Ln=0,\,
S(\Ln^\perp)\subseteq\Ln\},\\
\so(\CV)_0 &= \{S\in\so(\CV)\colon S(\Ln)\subseteq\Ln\}.
\end{align*}
Hence $\so(\CV)_0$ is the stabilizer $\stab(\Ln)$ of $\Ln$, which is a bundle
of parabolic subalgebras, whereas $\so(\CV)_{-1}$ is the Killing annihilator
$\stab(\Ln)^\perp$, which is the bundle of (abelian) nilradicals of
$\stab(\Ln)$. By definition, for $S\in\stab(\Ln)^\perp$, $S\restr{\Ln^\perp}$
has image and kernel $\Ln$, and it is easy to see that this restriction
defines an isomorphism $\stab(\Ln)^\perp\cong \Hom(\Ln^\perp/\Ln,\Ln)$.

These filtrations give rise to a bundle
\begin{equation*}
\so(\CV)\gr= \so(\CV)_{-1}\dsum
\so(\CV)_0/\so(\CV)_{-1}\dsum \so(\CV)_1/\so(\CV)_0
\end{equation*}
of graded Lie algebras acting on the bundle
$\CV\gr=\Ln\dsum\Ln^\perp/\Ln\dsum\CV/\Ln^\perp$ of graded lorentzian vector
spaces. Only $\so(\CV)_0/\so(\CV)_{-1}$ preserves the grading, and this
defines an isomorphism $\so(\CV)_0/\so(\CV)_{-1}\cong
\so(\Ln\dsum\CV/\Ln^\perp)\dsum \so(\Ln^\perp/\Ln)$. On the other hand,
restriction to $\Ln$ defines an isomorphism from $\so(\CV)/\so(\CV)_0$ to
$\Hom(\Ln,\Ln^\perp/\Ln)$.

The key point now is that the soldering form allows us to identify each graded
component of $\CV\gr$, and hence of $\so(\CV\gr)$, with a natural vector
bundle on $\mfd$.  For this, begin by noting that $\beta^\CD\colon
T\mfd\to\Hom(\Ln,\Ln^\perp/\Ln)$ naturally induces an isomorphism
$T\mfd\ltens\Ln\cong\Ln^\perp/\Ln$.  As in \S\ref{par:dens-conf-metric}
(cf.~Proposition~\ref{p:same-L}) the conformal metric induces an
identification $\Ln\cong L^{-1}$, which we shall use freely. We thus have a
projection $\pi\colon\Ln^\perp\to T\mfd\ltens L^{-1}$, with kernel $\Ln$,
given by
\begin{equation}\label{eq:pi-V}
\pi\CD_X\sigma=-X\tens\sigma.
\end{equation}
We next use minus the metric on $\CV$ to identify $\CV/\Ln^\perp$ with
$\Ln\dual\cong L$ and so obtain a projection $p\colon\CV\to L$, with kernel
$\Ln^\perp$, such that for $v\in\CV$, $\sigma\in\Ln$, we have
\begin{equation}\label{eq:p-V}
\ip{p(v),\sigma}=-\lip{v,\sigma}.
\end{equation}
(We have introduced minus signs in the definitions of $\beta^\CD$ and $p$ so
that derivatives appear positively in components of jets later on:
see~\S\ref{par:jet-deriv-mob}.)

Dualizing the soldering form gives an isomorphism of $T\dual\mfd$ with
$\Hom(\Ln^\perp/\Ln,\Ln)$, and hence with $\stab(\Ln)^\perp$.  We have
therefore realized $T\dual\mfd$ as an abelian subalgebra of $\so(\CV)$ with
action on $\CV$ determined by
\begin{equation}\label{eq:gam-act}
\gam\act\sigma=0, \qquad  \gam\act\CD\low_X\sigma=-\gam(X)\sigma,
\end{equation}
for $\gam\in T\dual\mfd$, $X\in T\mfd$, and $\sigma\in\Ln$.

The soldering form also identifies $\so(\CV)/\so(\CV)_0$ with $TM$ (restrict
to $\Ln\cong L^{-1}$) and $\so(\CV)_0/\so(\CV)_{-1}$ with
$\co(TM):=\vspan{\iden}\dsum\so(TM)$ (use $TM\cong (TM\ltens L^{-1})\tens L
\cong \Ln^\perp/\Ln \tens \CV/\Ln^\perp$). The projections $p$ and $\pi$ onto
graded pieces, which we introduced for $\CV$, have analogues for $\so(\CV)$
making these identifications explicit: $p\colon \so(\CV)\to TM$, with kernel
$\stab(\Ln)$ is defined by $(pS)\vtens\sigma= \pi(S\sigma)$ for
$S\in\so(\CV)$; $\pi\colon \stab(\Ln)\to \co(TM)$, with kernel
$\stab(\Ln)^\perp$, is defined $(\pi S)\circ p=p\circ S$ and $(\pi S)\circ\pi
= \pi\circ S\restr{\Ln^\perp}$ for $S\in\stab(\Ln)$. In particular:
\begin{equation}
\pi(S(\CD_X\sigma))=-(\pi S)(X\vtens\sigma).
\end{equation}

These constructions lead to several other useful formulae. First we have
$\lip{\CD_X\sigma,\CD_Y\sigma}=-(\sigma,\CD_X\CD_Y\sigma)$, yielding
\begin{equation} \label{eq:pD2}
p(\CD_X\CD_Y\sigma)=\cip{X,Y}\sigma.
\end{equation}
It will often be convenient to use the metric on $T\mfd\ltens L^{-1}$ to view
$\pi$ as a map $\Ln^\perp\to T\dual\mfd\ltens L$. With $v$ any section of
$\Ln^\perp$, \eqref{eq:pD2} can then be rewritten as:
\begin{equation} \label{eq:pDv}
p(\CD v)=-\pi v.
\end{equation}
We also have, for any $v\in\CV$, $\lip{\pi(\gam\act v),X\tens\sigma}
=\lip{\gam\act v,\beta^\CD_X\sigma}=-\lip{v,\gam\act\beta^\CD_X\sigma}=
-\gam(X)\lip{v,\sigma}$ so that
\begin{equation}\label{eq:pi-gam-act}
\pi(\gam\act v)=\gam\vtens p(v).
\end{equation}

We will have frequent recourse hereafter to the Lie algebra homology of the
bundle of abelian Lie algebras $T\dual\mfd$.  For this, let $W$ be any vector
bundle carrying a fibrewise representation of $T\dual\mfd$. Define, for each
$k$, $\LH\colon \Wedge^{k+1}T\dual M\tens W\to \Wedge^{k}T\dual\mfd\tens W$ by
\begin{equation}
\LH\alpha=\tsum_i \eps_i\act (e_i\interior\alpha),\qquad\text{\ie,}\qquad
\LH\alpha\low_{X_1,\ldots X_k} =
\tsum_i \eps_i\act\alpha\low_{e_i,X_1,\ldots X_k}.
\end{equation}
Here $\eps_i,e_i$ are dual local frames of $T\dual \mfd$ and $T\mfd$.  Now if
$\alpha=\LH\beta$, then
\begin{equation*}
\LH\alpha\low_{X_1,\ldots X_k} =\sum_{i,j}
\eps_i\act\eps_j\act\beta\low_{e_j,e_i,X_1,\ldots X_k}
=\sum_{i<j} \liebrack{\eps_i,\eps_j}\act\beta\low_{e_j,e_i,X_1,\ldots X_k}=0
\end{equation*}
since $T\dual \mfd$ is abelian. Hence $(\Wedge^{\bullet}T\dual\mfd\tens
W,\LH)$ is a chain complex, and we denote its cycles by $Z_\bullet(T\dual
\mfd,W)$ and its homology, called \emphdef{Lie algebra homology}, by
$H_\bullet(T\dual\mfd,W)$.

We shall only be interested in the cases $W=\CV$ and $W=\so(\CV)$. In
particular let us compute $Z_1(T\dual\mfd,\CV)$: if $\alpha\in T\dual\mfd\tens
\CV$ with $\LH \alpha=0$ then, in particular, from~\eqref{eq:pi-gam-act},
\begin{equation*}
0=\pi\LH \alpha=\tsum_i\eps_i\tens p(\alpha\low_{e_i})
\end{equation*}
whence $\alpha$ takes values in $\Ln^\perp$.  Now we deduce
from~\eqref{eq:gam-act} that
\begin{equation}\label{eq:H1V}
Z_1(T\dual\mfd,\CV)=\{\alpha\in T\dual\mfd\tens
\Ln^\perp:\tsum_i\cip{\eps_i,\pi\alpha\low_{e_i}}=0\}.
\end{equation}

\subsection{The graded Lie algebra bundle}
\label{par:graded-lie-alg}

On any manifold $\mfd$, the Lie algebra bundle of endomorphisms $\gl(T\mfd)$
acts fibrewise on the density bundle $L$ via $A\act\ell=(\tfrac1n\trace
A)\ell$ and hence also on $S^2 T\dual M\ltens L^2$. The stabilizer of a
conformal metric $\conf$ is the bundle of conformal linear Lie algebras
$\co(T\mfd)$. Thus $A\in\co(T\mfd)$ if and only if
\begin{equation*}
\cip{AX,Y}+\cip{X,AY}=\tfrac2n(\trace A)\cip{X,Y}.
\end{equation*}
The Lie bracket on $\co(T\mfd)$ can be extended to one on $T\dual
M\dsum\co(T\mfd)\dsum T\mfd$: declare $T\mfd$ and $T\dual\mfd$ to be abelian
subalgebras and, for $(\gam,A,X)\in T\dual\mfd\dsum\co(T\mfd)\dsum T\mfd$, set
\begin{equation}
\abrack{A,X}=AX, \qquad\qquad \abrack{\gam,A}=\gam\circ A;
\end{equation}
finally define $\abrack{X,\gam}=-\abrack{\gam,X}\in\gl(TM)$ by
\begin{equation}
\abrack{X,\gam}\act Y= -\abrack{\gam,X}\act
Y=\gam(X)Y+\gam(Y)X-\cip{X,Y}\gam^\sharp.
\end{equation}
Observe that $\abrack{X,\gam}\in \co(T\mfd)$ with
$\trace\abrack{X,\gam}=n\gam(X)$ and $\abrack{X,\gam}\act Y
=\abrack{Y,\gam}\act X$.

Define an inner product of signature $(n+1,1)$ on $L^{-1}\dsum T\dual\mfd
\ltens L \dsum L$ by setting
\begin{equation*}
\lip{v,v}=\cip{\theta,\theta}-2\sigma\ell,
\end{equation*}
for $v=(\sigma,\theta,\ell)$. An action of $(\gam,A,X)\in T\dual
M\dsum\co(T\mfd)\dsum T\mfd$ on $(\sigma,\theta,\ell)\in L^{-1}\dsum
T\dual\mfd\ltens L\dsum L$ is given as follows: $A\in\co(T\mfd)$ acts in the
natural way on each component, and we set
\begin{align}\label{eq:act-on-V}\notag
X\act\sigma&=X\vtens\sigma\in T\mfd\ltens L^{-1}\cong T\dual\mfd \ltens L&
\gam\act\sigma&=0\\
X\act\theta&=\theta(X)\in L&
\gam\act\theta&=\cip{\gam,\theta}\in L^{-1}\\
X\act\ell&=0&
\gam\act\ell&=\gam\vtens\ell\in T\dual\mfd\ltens L.
\notag
\end{align}
This action is skew with respect to the inner product, giving a Lie algebra
isomorphism $T\dual\mfd\dsum\co(T\mfd)\dsum T\mfd\cong\so(L^{-1}\dsum
T\dual\mfd\ltens L\dsum L)$: the Lie bracket of $S,T\in T\dual
M\dsum\co(T\mfd)\dsum T\mfd$ can be computed via $\abrack{S,T}\act
v=S\act(T\act v)-T\act(S\act v)$ for any $v\in L^{-1}\dsum T\dual\mfd\ltens
L\dsum L$.

Of course, we have not plucked this Lie algebra action out of thin air: if the
conformal metric $\conf$ is induced by a conformal Cartan connection
$(\CV,\Ln,\CD)$, then the identifications of the previous paragraph induce
isomorphisms of $\CV\gr$ with $L^{-1}\dsum T\dual\mfd\ltens L\dsum L$ and of
$\so(\CV)\gr$ with $T\dual\mfd\dsum\co(T\mfd)\dsum T\mfd$ which intertwine the
natural action of $\so(\CV)\gr$ on $\CV\gr$ with the above action of
$T\dual\mfd\dsum\co(T\mfd)\dsum T\mfd$ on $L^{-1}\dsum T\dual\mfd\ltens L\dsum
L$. It follows that the above graded Lie algebra structure on
$T\dual\mfd\dsum\co(T\mfd)\dsum T\mfd$ is precisely that induced by the
filtered Lie algebra structure on $\so(\CV)$. Hence for all
$S_{-1}\in\so(\CV)_{-1}=\stab(\Ln)^\perp$, $S_0,T_0\in\so(\CV)_0=\stab(\Ln)$
and $T\in\so(\CV)$, we have $p\liebrack{S_0,T}=\abrack{\pi S_0,p T}$,
$\pi\liebrack{S_0,T_0}=\abrack{\pi S_0,\pi T_0}$, $p\liebrack{S_{-1},
T}=\abrack{S_{-1},p T}$ and $\pi\liebrack{S_{-1}, T_0}=\abrack{S_{-1},\pi
T_0}$. We shall use these formulae freely.

As in \S\ref{par:filtered-lie-alg}, we may define Lie algebra homology, but
now the operator $\LH$ is also graded and may be restricted to graded
pieces. To analyse this, we fix a length scale $\ell$ and define an involution
$v\mapsto v^*=(-\lam\ell^{-2},\theta,-\sigma\ell^2)$, where $v=
(\sigma,\theta,\lam)\in L^{-1}\dsum T\dual\mfd \ltens L \dsum L$, so that
$\ip{v,v}:=\lip{v^*,v}=\ip{\theta,\theta}+\lam^2\ell^{-2}+\sigma^2\ell^2$ is
positive definite\footnote{For indefinite signature conformal structures, we
further need a choice of maximal positive definite subbundle of $T\mfd$ to
define such an involution.} on $\CV$. We further define $S\mapsto
S^*=-(X^{\flat}\ell^{-2} ,A^T,\gam^{\sharp}\ell^2)$, for $S=(\gam,A,X)\in
T\dual\mfd\dsum\co(T\mfd)\dsum T\mfd$, where the musical isomorphisms are
induced by $g$. We compute that $(S\act v)^*=S^*\act v^*$ and hence
$\abrack{S,T}^*=\abrack{S^*,T^*}$. Further, $\ip{S,S}:=\lip{S^*,S}$ is
positive definite, where $\lip{S,S}$ is the invariant inner product of
signature $(\half n(n+1),n+1)$ on $T\dual\mfd\dsum\co(T\mfd)\dsum T\mfd$ given
by
\begin{equation*}
\lip{S,S}=\cip{A_0,A_0}-\mu^2-2\gam(X)
\end{equation*}
for $S=(\gam,A,X)$ and $A=\mu\,\iden+A_0$ with $A_0\in\so(T\mfd)$.  It follows
that $S\mapsto S^*$ is a Cartan involution of the graded Lie algebra, as in
Kostant's celebrated paper~\cite{Kos:lac}. Following this paper, we now
compute (minus) the adjoint of $\LH$ on $T\dual\mfd\dsum\co(T\mfd)\dsum T\mfd$
with respect to the metrics $g$ and $\ip{,}$:
\begin{equation*}
\ip{\LH\alpha,\beta}=\lip{(\LH\alpha)^*,\beta} = \Lip{\tsum_i
\abrack{\eps_i^*,(e_i\interior\alpha)^*},\beta} =-\Lip{\alpha^*,\tsum_i
e_i^*\wedge\abrack{\eps_i^*,\beta}} =-\ip{\alpha,\abrack{\iden\wedge\beta}},
\end{equation*}
where $\abrack{\iden\wedge\beta}=\sum_i \eps_i\wedge\abrack{e_i,\beta}$.  This
is a special case of the operator $w\mapsto\iden\mathinner{\wedge\cdot}
w:=\sum_i \eps_i\wedge(e_i\act w)$, defined on any representation $W$ of
$T\dual\mfd\dsum\co(T\mfd)\dsum T\mfd$, independently of any choice of length
scale. Note however, that it is \emph{not} naturally defined on
representations of the filtered Lie algebra $\so(\CV)$.

We shall be particularly interested in the \emphdef{Ricci contraction}
$\LH\colon\Wedge^2T\dual\mfd\tens \co(T\mfd)\to T\dual \mfd\tens T\dual\mfd$
given by $\LH\colon R\mapsto \ricci$, with
\begin{equation*}
\ricci\low_X(Y)=\tsum_i\abrack{\eps_i,R_{e_i,X}}(Y)=\tsum_i\eps_i(R_{e_i,X}Y),
\end{equation*}
and the \emphdef{Ricci map} $\nr\mapsto \abrack{\iden\wedge\nr}\colon
T\dual\mfd\tens T\dual\mfd \to\Wedge^2T\dual\mfd\tens\co(T\mfd)$, where
$\abrack{\iden\wedge\nr}\low_{X,Y}=
\abrack{X,\nr\low_Y}-\abrack{Y,\nr\low_X}$. As a special case of the above
theory, the Ricci contraction is minus the adjoint of the Ricci map (using any
fixed length scale $\ell$).  In particular, the kernel of the Ricci
contraction and the image of the Ricci map (and vice versa) are orthogonal
complements with respect to $\ip{,}$. Moreover, $\LH\abrack{\iden\wedge\nr}$
is readily computed to be
\begin{equation}\label{eq:rc-rm}
(n-2)\sym_0\nr +2(n-1)\bigl(\tfrac1n\trace_\conf\nr\bigr)\conf +\tfrac
 n2\alt\nr
\end{equation}
from which we see that the Ricci map is injective when $n\geq3$, has kernel
$S^2_0T\dual\mfd$ when $n=2$, and is zero for $n=1$. (Here $(\alt\nr)\low_XY
=\nr\low_X Y-\nr\low_YX$.)

\section{Weyl geometry}
\label{sec:weyl-geom}

A popular approach in conformal geometry is to work with a riemannian metric
in the conformal class. This amounts to fixing a length scale $\ell$, or
equivalently, in the context of conformal Cartan connections, a gauge
$\sigma$. However, choosing a different metric leads to complicated
transformation formulae, so we find it convenient to use a more general
notion. For this, note that a length scale is parallel with respect to a
unique connection on $L$.
\begin{defn}
A \emphdef{Weyl derivative} on $\mfd$ is a connection on $L$.  It induces a
covariant derivative $D$ on $L^w$ for each $w\in\R$, whose curvature is a real
$2$-form $w F^D$, the \emphdef{Faraday curvature}.  If $F^D=0$ then $D$ is
said to be \emphdef{closed}: then there are local length scales $\ell$ with
$D\ell=0$; if such an $\ell$ exists globally then $D$ is said to be
\emphdef{exact}. Any two Weyl derivatives differ by a $1$-form so that Weyl
derivatives are an affine space modelled on $\Omega^1(\mfd,\R)$.
\end{defn}
In this section we relate Weyl derivatives to conformal Cartan connections and
conformal metrics,
cf.~\cite{CaDi:di,CDS:rcd,CaPe:ewg,CpSl:wpg,Die:PhD,Gau:ccw,Wey:stm}.

\subsection{Weyl structures}
\label{par:weyl-str}

\begin{defn} Let $(\CV,\Ln,\CD)$ be a conformal Cartan connection. Then
a \emphdef{Weyl structure} is a null line subbundle $\Lnc$ of
$\CV$ complementary to $\Ln^\perp$ (\ie, distinct from $\Ln$).
\end{defn}
Observe that $\Ln^\perp\intersect\Lnc^\perp$ is a complementary subspace to
$\Ln$ in $\Ln^\perp$. Conversely given such a complement $\US$, there is a
unique null line subbundle $\Lnc$ orthogonal to $\US$ and complementary to
$\Ln^\perp$. Thus a Weyl structure is given equivalently by a splitting
$\pr_{\Lnc}$ of the inclusion $\Ln\to \Ln^\perp$. It follows that Weyl
structures form an affine space modelled on the space of sections of
$\Hom(\Ln^\perp/\Ln,\Ln)\cong\stab(\Ln)^\perp$, with $\pr_{\Lnc+\gam}(v)
=\pr_{\Lnc}(v)+\gam\act v$. This affine structure has a gauge-theoretic
interpretation that we wish to emphasise:
$\smash{\Lnc}+\gam=\exp(-\gam)\smash{\Lnc}$.

A Weyl structure defines a covariant derivative $D=\pr_{\Lnc}\circ\CD\restr\Ln$
on $\Ln$, and so a Weyl derivative via the usual identification $\Ln\cong
L^{-1}$ provided by the conformal structure.  Conversely, given a Weyl
derivative and so a covariant derivative $D$ on $\Ln$, we take $U$ to be the
image of $\CDD i\low_\Ln\colon T\mfd\tens\Ln\to\Ln^\perp$, where
$i\low_\Ln\colon\Ln\to\CV$ is the inclusion and $\CDD$ is the induced
connection on $\Hom(\Ln,\CV)$: this is a complement to $\Ln$ since
$\pi\circ\CDD i\low_\Ln=-\iden_{T\mfd\ltens\Ln}$. Thus we have a bijection
$D\mapsto \smash{\Lnc}^D$ between Weyl derivatives and Weyl structures which,
in view of the soldering identification
$T\dual\mfd\cong\Hom(\Ln^\perp/\Ln,\Ln)$ enunciated in \eqref{eq:gam-act}, is
easily seen to be affine: $\smash{\Lnc}^{D+\gam}=\smash{\Lnc}^D+\gam$.

\begin{rem} There is yet another (equivalent) definition of Weyl structure
that really gets to the Lie-theoretic heart of the matter. The decomposition
$\CV=\Ln\dsum\US\dsum\Lnc$ is equivalently given by an element
$\eps\in\so(\CV)$ which acts by $-1$ on $\Ln$, $0$ on $\US$ and $1$ on
$\Lnc$. It is easy to see that $\eps\in\stab(\Ln)$ is a lift of the identity
in $\co(T\mfd)\cong \stab(\Ln)/\stab(\Ln)^\perp$ and conversely any such lift
splits the filtration of $\CV$. See~\cite{CDS:rcd} for more on this approach.
\end{rem}

The decomposition $\CV=\Ln\dsum U\dsum\Lnc$ induced by a Weyl structure
identifies $\CV$ with $\CV\gr$ so that we have an isomorphism $\CV\cong
L^{-1}\dsum T\dual\mfd\ltens L\dsum L$. Hence we also obtain an isomorphism
$\so(\CV)\cong T\dual\mfd\dsum\co(T\mfd)\dsum T\mfd$ between the Lie algebra
bundles of section~\ref{sec:conf-alg}.  (Note however, that different Weyl
structures give rise to different isomorphisms!)

These decompositions (a reduction of the structure group of $\CV$ to
$\CO(n)$) induce a decomposition of the connection $\CD$ into a
$T\dual\mfd$-valued $1$-form $\nr^{D,\CD}$, a $\CO(n)$-connection $D^\CD$
(restricting to the Weyl derivative $D$ on $L$ and $L^{-1}$ and a metric
connection on $T\dual\mfd\ltens L$), and a $T\mfd$-valued $1$-form.  This last
is minus the soldering form which, with our identifications, is minus the
identity.  To summarize, if we write $(\sigma,\theta,\ell)$ for the components
of $v$, we have
\begin{equation}\label{eq:conn-comp}
\CD\low_Xv=\nr^{D,\CD}_X\act v+D^\CD_X v-X\act v
=\mtrip{D\low_X\sigma+\nr^{D,\CD}_X(\theta)}
{D^\CD_X\theta+\nr^{D,\CD}_X \,\ell-\sigma X}{D\low_X\ell-\theta(X)}.
\end{equation}
It is natural to ask how these components transform under a change
$\Lnc\mapsto\exp(-\gam)\Lnc$ of Weyl structure. We first examine how $\CD$
changes under gauge transformation by $\exp(\gam)$: $(\exp\gam\act \CD)_Xv=
\exp(\gam)\CD_X(\exp(-\gam)v)$ differs from $\CD_Xv$ by the right logarithmic
derivative of the exponential map at $\gam$, in the direction $\CD_X\gam$,
applied to $v$. Using the standard formula
\begin{equation*}\notag
\d(R_{\exp(-\gam)})_{\exp\gam}\circ \d\exp_\gam(\chi)
= F(\ad\gam)(\chi),\qquad F(t)=\tfrac1t(e^t-1) = 1+\half t +\cdots
\end{equation*}
for this logarithmic derivative (see~\cite{KMS:nodg}), we then obtain
\begin{equation}\label{eq:log-deriv}
\exp\gam\act\CD = \CD-\CD\gam -\half \abrack{\gam,\pi\CD\gam},
\end{equation}
since $\CD\gam$ is in $\stab(\Ln)$, on which $\ad\gam$ is $2$-step nilpotent,
and hence $F(\ad\gam)(\CD_X\gam)=\CD_X\gam+\frac12 \liebrack{\gam,\CD_X\gam}$.
We can expand $\abrack{\gam,\pi\CD_X\gam}= -2\gam(X)\gam +
\cip{\gam,\gam}X^\flat$ explicitly as needed.

This computation provides the required transformation formulae by switching
from the active to passive view of gauge transformations: in the former, the
connection is transformed but the `gauge' (here a Weyl structure
$\Lnc\subset\CV$) is fixed; in the latter, the opposite process is
performed. Since $\exp(\gam)$ acts trivially on $\so(\CV)\gr$, we deduce that
the components $\nr^{D,\CD}$ and $D^\CD$ transform as follows:
\begin{align}\label{eq:gauge-Weyl}
(D+\gam)^\CD&= D^{\exp\gam\act\CD} = D^{\CD}-\abrack{\gam,\cdot}\\
\label{eq:gauge-nric}
\nr^{D+\gam,\CD}&=\nr^{D,\exp\gam\act\CD} =
\nr^{D,\CD} - D\gam + \gam\vtens\gam - \half\cip{\gam,\gam}\conf.
\end{align}

\subsubsection*{Curvature}

We now compute the curvature of a conformal Cartan connection $\CD$ with
respect to a Weyl structure $\Lnc\subset\CV$. Using~\eqref{eq:conn-comp}, \ie,
$\CD=\nr^{D,\CD}+D^\CD-\iden$, we have
\begin{equation}\label{eq:curv-comp}
R^\CD=\d^{\smash{D^\CD}}\nr^{D,\CD}
+\;\;R^{D^\CD}\!-\abrack{\iden\wedge\nr^{D,\CD}}\;\;-\d^{D^\CD}\iden.
\end{equation}
The last component, $-\d^{D^\CD}\iden$, \ie, minus the torsion of $D^\CD$, is
equal to $p(R^\CD)$, hence independent of the Weyl structure, so we call it
the \emphdef{torsion} of $\CD$. We refer to $W^\CD:=R^{D^\CD}
-\abrack{\iden\wedge\nr^{D,\CD}}$ as the \emphdef{Weyl curvature} of $\CD$
(with respect to $D$ or $\Lnc$); if the torsion of $\CD$ is zero, $W^\CD=
\pi(R^\CD)$, independently of $\Lnc$. Finally
$C^{\CD,D}:=\d^{D^\CD}\nr^{D,\CD}$ is called the \emphdef{Cotton--York
curvature} of $\CD$ (with respect to $\Lnc$); it is the full curvature
(independent of $\Lnc$) if both the torsion and Weyl curvature vanish.

We shall be interested in the following natural conditions on the curvature of
$\CD$.
\begin{defn}
Let $(\CV,\Ln,\CD)$ be a conformal Cartan connection.  Then
\begin{bulletlist}
\item $\CD$ is \emphdef{torsion-free} if $p(R^\CD)=0$, \ie,
$R^\CD\in\Omega^2(\mfd,\stab(\Ln))$, or equivalently $R^\CD\Ln\subseteq\Ln$;
\item $\CD$ is \emphdef{strongly torsion-free} if $R^\CD\restr\Ln=0$;
\item $\CD$ is \emphdef{normal} if $\LH R^\CD=0$.
\end{bulletlist}
\end{defn}
In fact, a normal conformal Cartan connection is strongly torsion-free.  For
this, note that $\Gam\mapsto\abrack{\iden\wedge\Gam}$ defines an isomorphism
$T\dual\mfd\tens\so(T\mfd)\to\Wedge^2T\dual\mfd\tens T\mfd$ (this amounts to
the familiar algebra that determines a metric connection from its torsion).
Thus $\abrack{\iden\wedge(\cdot)}$ surjects so that $\LH\colon
\Wedge^2T\dual\mfd\tens T\mfd\to T\dual\mfd\tens\co(T\mfd)$ injects.
Furthermore, it is easy to see that also
$\LH\colon\Wedge^2T\dual\mfd\tens\co(T\submfd)\to T\dual\mfd\tens T\dual\mfd$
injects on $\Wedge^2T\dual\mfd\tens\vspan{\iden_{T\mfd}}$.

If $\CD$ is torsion-free, $\pi(R^\CD)=W^\CD$.  Thus $\CD$ is normal iff it is
torsion-free and $\LH W^\CD=0$.

We can now single out the conformal Cartan connections that will be
of interest to us.
\begin{defn}
A \emphdef{conformal Cartan geometry} is a conformal Cartan connection
$(\CV,\Ln,\CD)$ with $\CD$ strongly torsion-free. It is said to be
\emphdef{normal} if $\CD$ is normal.
\end{defn}

\subsection{Weyl connections}
\label{par:weyl-conn}

\begin{defn} Let $\conf$ be a conformal metric on $\mfd$. Then
a \emphdef{Weyl connection} is a torsion-free connection $D$ on $T\mfd$ which
is \emphdef{conformal} in the sense that $D\conf=0$.
\end{defn}
If $\ell$ is a length scale, then the Levi-Civita connection $D^g$ of the
metric $g=\ell^{-2}\conf$ is a Weyl connection: since $g$ is parallel and
$\ell$ is parallel (with respect to the induced connection on $L$), so is
$\conf=\ell^2 g$. Any Weyl connection induces a Weyl derivative and
Levi-Civita connections induce exact Weyl derivatives.  The existence
and uniqueness of the Levi-Civita connection associated to a length scale
admits the following generalization~\cite{Die:PhD,Gau:ccw,Wey:stm}.

\begin{prop}\label{p:ftcg}
The affine map sending a connection on $T\mfd$ to the induced covariant
derivative on $L$ induces a bijection between Weyl connections and Weyl
derivatives.
\end{prop}

Indeed, if $D$ is a given Weyl connection (for example, the Levi-Civita of a
compatible metric), then any other such is of the form $D+\Gam$ where $\Gam$
is a $\co(T\mfd)$-valued $1$-form with $\abrack{\iden\wedge\Gam}=0$.  A
standard argument shows that all such $\Gamma$ are of the form
$\Gam=\abrack{\iden,\gam}$ for a section $\gam$ of $T\dual\mfd$ and then the
corresponding Weyl derivative on $L$ is $D+\tfrac1n\trace\Gam=D+\gam$.

One of the reasons why Weyl derivatives can be more convenient than length
scales is that they form an affine space, modelled on $\Omega^1(\mfd,\R)$.
When we construct a geometric object using a Weyl derivative $D$, we often
want to know how it depends upon this choice. We view such an object as a
function $F(D)$ and say that it is conformally-invariant if and only if this
function is constant (\ie, independent of $D$). For this, it is often helpful
to use the fundamental theorem of calculus: $F(D)$ is constant iff its
derivative with respect to $D$ is zero.  This amounts to checking that
$\del_\gam F(D)=0$ for all Weyl derivatives $D$ and all $1$-forms $\gam$, \ie,
that for any one parameter family $D(t)$ with
\begin{align*}\notag
D(0)=D\quad\text{and}\quad \frac \d{\d t} D(t)\Restr{t=0}&=\gam\\
\tag*{we have}
\del_\gam F(D):=\frac \d{\d t} F(D(t))\Restr{t=0}&=0.
\end{align*}
More generally, Taylor's theorem implies that if
$\del^{k+1}_{\gam,\ldots\gam}F(D)=0$ for all $D$ and $\gam$ then
\begin{equation*}
F(D+\gam)=F(D)+\del_\gam F(D)+\half\del^2_{\gam,\gam}F(D)
+\cdots+\tfrac1{k!}\del^k_{\gam,\ldots\gam}F(D).
\end{equation*}
This approach often simplifies the calculation of nonlinear terms.

For example, for any section $\ell$ of $L$, $\del_\gam
D\low_X\ell=\gam(X)\ell$ and for any vector field $Y$, $\del_\gam D\low_XY
=\abrack{X,\gam}\act Y$. Similarly, if $s$ is a section of a bundle
associated to the conformal frame bundle, then $\del_\gam D\low_X
s=\abrack{X,\gam}\act s$ where the dot denotes the natural action of
$\lie{co}(T\mfd)$. This expresses the obvious fact that all first covariant
derivatives depend affinely on $D$.

For the second derivative, we obtain
\begin{align*}
\del_\gam D^2_{X,Y}s&=\abrack{X,\gam}\act D\low_Ys
-D\low_{\abrack{X,\gam}\act Y}s+D\low_X\bigl(\abrack{Y,\gam}\act s\bigr)
-\abrack{D\low_XY,\gam}\act s\\
&=\abrack{Y,D\low_X\gam}\act s
+\abrack{X,\gam}\act D\low_Ys+\abrack{Y,\gam}\act D\low_Xs
-D\low_{\abrack{X,\gam}\act Y}s
\end{align*}
In particular, for any section $\ell$ of $L$ we have
\begin{equation}\label{eq:lin-ch-second-L}
\del_\gam D^2_{X,Y}\ell=(D\low_X\gam)(Y)\ell+\cip{X,Y}\cip{\gam,D\ell}.
\end{equation}

\subsubsection*{Curvature}

We now apply the transformation formula for the second derivative to the
curvature $R^D$ of the Weyl connection $D$, which is the $\co(T\mfd)$-valued
$2$-form such that $D^2_{X,Y}s-D^2_{Y,X}s=R^D_{X,Y}\act s$ for any section $s$
of any bundle associated to the conformal frame bundle.  We deduce that
\begin{equation}\label{eq:lin-ch-curv}
\del_\gam R^D_{X,Y}=-\abrack{X,D\low_Y\gam}+\abrack{Y,D\low_X\gam}=
-\abrack{\iden\wedge D\gam}\low_{X,Y},
\end{equation}
\ie, $\del_\gam R^D=-\abrack{\iden\wedge D\gam}$. Hence the part of $R^D$
orthogonal to the image of the Ricci map is conformally-invariant: this is the
\emphdef{Weyl curvature} $W$ of the conformal metric, and it has vanishing
Ricci contraction: $W$ itself vanishes for $n<4$ since the Ricci map is then
surjective. On the other hand, for $n\geq 3$, the Ricci map is injective, and
so we can write
\begin{equation}\label{eq:curv-decomp}
R^D=W+\abrack{\iden\wedge\nr^D},
\end{equation}
with $\nr^D$ uniquely determined by $D$ and, in view of
\eqref{eq:lin-ch-curv},
\begin{equation}\label{eq:lin-ch-nric}
\del_\gam\nr^D=-D\gam.
\end{equation}
Let $\ricci^D$ be the Ricci contraction $\LH R^D$ of $R^D$, so that,
by~\eqref{eq:rc-rm}, we have
\begin{equation}\label{eq:ric-nric}
\ricci^D=(n-2)\sym_0\nr^D +2(n-1)\bigl(\tfrac1n\trace_\conf\nr^D)\conf +\tfrac
n2\alt\nr^D.
\end{equation}
A simple application of the Bianchi identity $\abrack{\iden\wedge R^D}=0$ gives
$\alt\ricci^D = -\trace R^D = -n F^D$ and hence we deduce $\nr^D$ is the
\emphdef{normalized Ricci curvature}, defined by
\begin{equation*}
\nr^D = \nr^D_0+\tfrac1n \ns^D\conf-\half F^D,
\end{equation*}
where we set
\begin{equation*}
\nr^D_0=\frac1{n-2}\sym_0\ricci^D,\qquad
\ns^D=\frac1{2(n-1)}\trace_\conf\ricci^D.
\end{equation*}
When $n=3$, $\d^D\nr^D$ is independent of $D$, and known as the
\emphdef{Cotton--York curvature} of $\conf$.

When $n=2$, $\ns^D$ is still defined and the Ricci map is still injective on
$L^{-2}$ so that
\begin{equation}\label{eq:lin-ch-nscal}
\del_\gam\ns^D=-\trace_\conf D\gam.
\end{equation}
We wish to bring the cases $n=1$ and $n=2$ in to line with the higher
dimensional case. One way to do this uses a pair of differential operators,
which we call a \emphdef{\mob/ structure}.

\section{M\"obius manifolds}
\label{sec:mob-geometry}

\subsection{The differential lift and M\"obius operators}
\label{par:diff-lift}

A conformal Cartan connection $(\CV,\Ln,\CD)$ induces more structure on $\mfd$
than just a conformal metric. This additional structure arises from the
following differential operator, cf.~\cite{CaDi:di,CSS:bgg}.

\begin{prop}\label{p:diff-lift} There is a unique linear map
$j^\CD\colon\Cinf(\mfd,L)\to\Cinf(\mfd,\CV)$ such that
\begin{bulletlist}
\item $p(j^\CD\ell)=\ell$, \ie, $j^\CD\ell$ is a lift of $\ell$ with respect
to the projection $p\colon\CV\to L$\textup;
\item $\LH\CD(j^\CD\ell):=\sum_i \eps_i\act\CD_{e_i} j^\CD\ell =0$, \ie, $\CD
j^\CD\ell$ is a section of $Z_1(T\dual\mfd,\CV)$.
\end{bulletlist}
The components of $j^\CD\ell$, relative to a Weyl structure with Weyl
derivative $D$, are
\begin{equation}\label{eq:diff-lift-comp}
(\tfrac1n\trace_\conf\psi,D\ell,\ell) \quad \text{where}\quad
\psi=D^\CD D\ell + \nr^{D,\CD}\ell.
\end{equation}
In particular $j^\CD$ is a second order linear differential operator.
Furthermore
$\pi\CD j^\CD\ell=\psi-\frac1n(\trace_\conf\psi)\conf\in T\dual\mfd\tens
T\dual\mfd\ltens L$.
\end{prop}
\begin{proof} With respect to any Weyl structure, a lift of $\ell$
has components $(\sigma,\theta,\ell)$ and, using \eqref{eq:act-on-V}
and \eqref{eq:conn-comp}, we readily compute
\begin{equation*}\notag
\sum_i \eps_i\act\CD_{e_i}\mtrip{\sigma}{\theta}{\ell}=
\mtrip{-n\sigma+\trace_\conf(D^\CD\theta+\nr^{D,\CD}\ell)}
{D\ell-\theta}{0}
\end{equation*}
which vanishes precisely when $\theta=D\ell$ and
$\sigma=\tfrac1n\trace_\conf\psi$. The last part follows
from~\eqref{eq:conn-comp}.
\end{proof}

We refer to $j^\CD\ell$ as the \emphdef{differential lift} of $\ell$. Using
this, we now define two second order differential operators on $L$, one
linear, one quadratic.

\begin{defn} Let $(\CV,\Ln,\CD)$ be a conformal Cartan connection. Define
$\Mh^\CD\colon \Cinf(\mfd,L)\to \Cinf(\mfd,S^2_0T\dual\mfd\ltens L)$ and
$\Mq^\CD\from\Cinf(\mfd,L)\to\Cinf(\mfd,\R)$ by
\begin{equation*}\notag
\Mh^\CD\ell = \sym\pi\CD j^\CD\ell,\qquad
\Mq^\CD\ell = \lip{j^\CD\ell,j^\CD\ell}.
\end{equation*}
\end{defn}
Using~\eqref{eq:diff-lift-comp} and~\eqref{eq:conn-comp}, we obtain the
explicit formulae
\begin{equation}\label{eq:m-ops-weyl}\begin{split}
\Mh^\CD\ell &= \sym_0 \bigl(D^\CD D\ell + \nr^{D,\CD}\ell\bigr)\\
\Mq^\CD\ell &= \ip{D\ell,D\ell}-\tfrac2n\ell\trace_\conf
\bigl(D^\CD D\ell + \nr^{D,\CD}\ell\bigr)
\end{split}\end{equation}
relative to a Weyl structure with Weyl derivative $D$.  When $\CD$ is
torsion-free, $D^\CD$ is the Weyl connection on $TM$ induced by $D$. (Also in
this case, the skew part of $\pi\CD j^\CD\ell$ is
$F^D+\alt\nr^{D,\CD}=\tfrac1n\trace W^\CD$ so that $\Mh^\CD\ell=\pi\CD
j^\CD\ell$ precisely when $\CD$ is strongly torsion-free.)

Our goal, achieved in the next section, is to recover the conformal Cartan
connection (up to isomorphism) from the underlying conformal metric and these
two operators.

\subsubsection*{The differential lift in general}

For later use we remark that the above is a special case of a more general
construction~\cite{CaDi:di,CSS:bgg}, which, for any representation $W$ of
$\so(\CV)$, provides a differential lift of any section of a Lie algebra
homology bundle $H_k(T\dual\mfd, W)$ to give a representative section of
$Z_k(T\dual\mfd, W)$.  The key ingredient in this construction is the first
order quabla operator $\Quabla_\CD:=\LH\circ \d^\CD+\d^\CD\circ\LH$ on
$C_k(T\dual\mfd,W)$ and the following observation.

\begin{prop}\tcite{CaDi:di} $\Quabla_\CD$ is invertible on the image of $\LH$.
\end{prop}
This result is not difficult: using a Weyl structure, one shows that
$\Quabla_\CD$ differs from an invertible algebraic operator (Kostant's
$\Quabla$~\cite{Kos:lac}) by a nilpotent first order operator; the inverse is
then given by a geometric series, and is a differential operator of finite
order.

We then define $\Pi=\iden-\Quabla_\CD^{-1}\circ\LH\circ
\d^\CD-\d^\CD\circ\Quabla_\CD^{-1}\circ\LH$, and note the following
properties:
\begin{bulletlist}
\item $\Pi$ vanishes on the image of $\LH$;
\item $\Pi$ maps into the kernel of $\LH$;
\item $\Pi\restr{\kernel\LH}\cong\iden \modulo \image\LH$.
\end{bulletlist}
It follows that there is a canonical differential lift $j^\CD\colon
\hc{\alpha}\mapsto\Pi\alpha$ from Lie algebra homology to
cycles~\cite{CSS:bgg}.  According to~\cite{CaDi:di}, $j^{\CD}\hc{\alpha}$ is
characterized by the following properties: $\LH j^{\CD}\hc{\alpha}=0$,
$\Hc{j^{\CD}\hc{\alpha}}=\hc{\alpha}$ and $\LH \d^\CD j^{\CD}\hc{\alpha}=0$.
This generalizes the characterization of the differential lift in
Proposition~\ref{p:diff-lift}.

The differential operators $\Mh^\CD,\Mq^\CD$ also have a homological
interpretation and generalization. For $n\geq 2$, $\Mh^\CD$ is the \BGG/
operator $\hc{\alpha}\mapsto \hc{\CD\Pi\alpha}$ from
$\Cinf(\mfd,H_0(T\dual\mfd,V))$ to $\Cinf(\mfd,H_1(T\dual\mfd,V))$, which is
part of the general theory of~\v Cap--Slov\'ak--Sou\v cek~\cite{CSS:bgg},
whereas (for all $n$) $\Mq^\CD(\ell)=\ell\sqcup\ell$,
$\sqcup\colon\Cinf(\mfd,H_0(T\dual\mfd,V))\times\Cinf(\mfd,H_0(T\dual\mfd,V))
\to\Cinf(\mfd,\R)$ being the bilinear operator of
Calderbank--Diemer~\cite{CaDi:di} induced by the metric pairing
$\CV\times\CV\to\R$. We shall discuss this further in
section~\ref{sec:conf-bonnet}.

\subsection{M\"obius structures}
\label{par:mob-str}

We are now going to define a class of linear and quadratic differential
operators on \emph{any} manifold with a conformal metric. First, for $\ell$ a
section of $L$ and $D$ a Weyl connection, define:
\begin{equation}\begin{split}
\Mh^D\ell&=\sym_0D^2\ell\qquad\qquad\quad
\in C^\infty(\mfd,S^2_0T\dual\mfd\ltens L),\\
\Mq^D\ell&=\cip{D\ell,D\ell} -\tfrac2n\ell\,\Laplace^{\!D}\ell
\quad\in C^\infty(\mfd,\R),
\end{split}\end{equation}
where $\Laplace^{\!D}\ell=\trace_\conf D^2\ell$, a section of $L^{-1}$.  A
fundamental feature of $\Mh^D$ and $\Mq^D$ is that replacing $D$ by $D+\gam$
only alters them by zero-order terms.  Indeed from~\eqref{eq:lin-ch-second-L},
we have
\begin{equation}\label{eq:lin-ch-weyl-op}\begin{split}
\del_\gam\Mh^D(\ell)&=(\sym_0D\gam)\ell,\\
\del_\gam\Mq^D(\ell)&=2\cip{\gam\ell,D\ell}-\tfrac2n n\ell\cip{\gam,D\ell}
-\tfrac2n \ell(\trace_\conf D\gam)\ell =-\tfrac2n(\trace_\conf D\gam)\ell^2.
\end{split}\end{equation}
This observation will allow us to make a key definition. Before this, though,
we compare these variations with~\eqref{eq:lin-ch-nric}
and~\eqref{eq:lin-ch-nscal} to deduce that the formulae
\begin{align*}
&&&&&&\Mh^\conf\ell&=\Mh^D\ell+\nr^D_0\ell&& (n\geq3)&&&&\\
&&&&&&\Mq^\conf\ell&=\Mq^D\ell-\tfrac2n\ns^D\ell^2&& (n\geq2)&&&&
\end{align*}
define conformally-invariant second order differential operators, one linear,
one quadratic.
\begin{rem}\label{rm:can-mos}
These operators, when they exist, may be defined in a manifestly invariant way
on positive sections of $L$ by associating to such a length scale $\ell$ the
normalized Ricci curvature $\nr^D$ of the unique Weyl connection $D$ with
$D\ell=0$ (this is the Levi-Civita connection of the metric
$g=\ell^{-2}\conf$, and $\nr^D$ is the normalized Ricci curvature of
$g$). Then, by definition, since $D\ell=0$ we have $\Mh^\conf\ell=\nr^D_0\ell$
and $\Mq^\conf\ell=-\frac2n \ns^D\ell^2$, cf.~\cite{Gau:ccw}. Hence positive
solutions of $\Mh^\conf$ are sometimes called \emphdef{Einstein gauges} and
$\Mq^\conf$ is sometimes called the \emphdef{scalar curvature
metric}. $\Mh^\conf$ itself is often called the \emphdef{conformal
\textup(tracefree\textup) Hessian}.
\end{rem}
\begin{rem} We also note that differentiating~\eqref{eq:lin-ch-weyl-op} again,
using $\del_\gam D\gam=-2\gam\vtens\low_0\gam
+\tfrac{n-2}{n}\cip{\gam,\gam}\conf$, gives $\del_{\gam,\gam}^2\Mh^D(\ell)
=-2\gam\vtens\low_0\gam\,\ell$ and $\del_{\gam,\gam}^2\Mq^D(\ell)
=-\tfrac2n(n-2)\cip{\gam,\gam}\ell^2$, which are both independent of $D$ so
that $\Mh^D$ and $\Mq^D$ depend quadratically on $D$:
\begin{align*}
\Mh^{D+\gam}\ell&=\Mh^D\ell+(\sym_0D\gam-\gam\vtens\low_0\gam)\ell,\\
\Mq^{D+\gam}\ell&=\Mq^D\ell
-\tfrac1n\bigl(2\trace_\conf D\gam+(n-2)\cip{\gam,\gam}\bigr)\ell^2.
\end{align*}
\end{rem}
We are now motivated to make our definition, cf.~\cite{Cal:mew,OsSt:sd}.
\begin{defn}
A \emphdef{\mob/ structure} on a conformal manifold $(\mfd,\conf)$ is a pair
$\Ms=(\Mh,\Mq)$ of differential operators such that
\begin{bulletlist}
\item $\Mh\from\Cinf(\mfd,L)\to\Cinf(\mfd,S^2_0T\dual\mfd\ltens L)$ is a
linear differential operator with $\mathop{\Mh\strut}\ell-\Mh^D\ell$ zero
order in $\ell$ for some (hence any) Weyl connection $D$.
\item $\Mq\from\Cinf(\mfd,L)\to\Cinf(\mfd,\R)$ is a quadratic differential
operator with $\mathop{\Mq\strut}\ell-\Mq^D\ell$ zero order in $\ell$ for
some (hence any) Weyl connection $D$.
\end{bulletlist}
\end{defn}
Let us summarize two basic motivations leading to this definition.
\begin{prop} First, if $\conf$ is a conformal metric on a manifold $\mfd$ of
dimension $n\geq3$, then there is a canonical \mob/ structure
$\Ms^\conf:=(\Mh^\conf,\Mq^\conf)$. on $(\mfd,\conf)$.

Second, if $(\CV,\Ln,\CD)$ is a torsion-free conformal Cartan connection over
a manifold $\mfd$ of any dimension, then $\Ms^\CD:=(\Mh^\CD,\Mq^\CD)$ is a
\mob/ structure on $\mfd$.
\end{prop}
\mob/ structures form an affine space modelled on $\Cinf(\mfd,S^2T\dual\mfd)$.
Indeed, for $\QC$ a section of $S^2T\dual\mfd$, let $\QC_0$ be its tracefree
part, $\trace_\conf\QC$ its trace, and define, for $\Ms=(\Mh,\Mq)$,
\begin{equation}\label{eq:mob-aff}
\Ms+\QC:=\bigl(\ell\mapsto\mathop{\Mh\strut}\ell+\QC_0\ell,
\ell\mapsto\mathop{\Mq\strut}\ell-\tfrac2n(\trace_\conf\QC)\ell^2\bigr).
\end{equation}
In particular if we define $\Ms^D:=(\Mh^D,\Mq^D)$, then for any \mob/
structure $\Ms$, $\Ms-\Ms^D$ satisfies (for fixed $\Ms$):
\begin{equation}\label{eq:dgamQ}
\del_\gam(\Ms-\Ms^D)=-\sym D\gam,
\end{equation}
which is the origin of the choice of coefficient of $\trace_\conf\QC$ in our
definition of $\Ms+\QC$. By adding $-\half F^D$ to this, we obtain a
generalized notion of normalized Ricci curvature and an associated
decomposition of $R^D$.
\begin{defn}
Let $\conf,\Ms$ be a \mob/ structure.  Then the \emphdef{\mob/ Ricci
curvature} of a Weyl derivative $D$ with respect to $\Ms$ is defined by
$\nr^{D,\Ms}:=\Ms-\Ms^D-\smash{\half} F^D$.  The \emphdef{Weyl curvature}
$W^\Ms$ of $\Ms$ is given by $W^{\Ms}:=R^D-\abrack{\iden\wedge\nr^{D,\Ms}}$
for any Weyl derivative $D$. The \emphdef{Cotton--York curvature} $C^{\Ms,D}$
of $\Ms$ with respect to $D$ is given by $C^{\Ms,D}:=\d^D\nr^{D,\Ms}$.
\end{defn}
Note that $W^\Ms$ is independent of the choice of $D$ and thus is an invariant
of the \mob/ structure: indeed, using~\eqref{eq:dgamQ} and $\del_\gam
F^D=\d\gam$, we have $\del_\gam\nr^{D,\Ms}=-D\gam$, which cancels
with~\eqref{eq:lin-ch-curv}. (We remark that we can differentiate once more to
find
\begin{equation}\label{eq:ch-nric}
\nr^{D+\gam,\Ms}=\nr^{D,\Ms}-D\gam+\gam\vtens\gam-\half\cip{\gam,\gam}\conf,
\end{equation}
so that $\nr^{D,\Ms}$ has the same transformation law~\eqref{eq:gauge-nric} as
$\nr^{D,\CD}$.)

We now fix a distinguished class of \mob/ structures by requiring $W^\Ms$ to
coincide with the Weyl curvature $W$ of the underlying conformal metric.  In
general, we have $R^D=W^{\Ms}+\abrack{\iden\wedge\nr^{D,\Ms}}$, with the
second term in the image of the Ricci map, so $W$ is the component of $W^\Ms$
orthogonal to this image, \ie, in $\kernel\LH$. Hence
\begin{equation}\label{eq:weyl}
W^\Ms=W\quad\Leftrightarrow\quad\text{$W^\Ms$ is in the kernel of the Ricci
contraction, \ie, $\LH W^\Ms=0$.}
\end{equation}
When $n\geq3$, this amounts to requiring that the \mob/ Ricci curvature with
respect to $\Ms$ of each Weyl connection $D$ coincides with the normalized
Ricci curvature: $\nr^{D,\Ms}=\nr^D$.  More generally, by~\eqref{eq:rc-rm}
and~\eqref{eq:ric-nric}, we see that this is equivalent to requiring that
$\sym_0\nr^{D,\Ms}=\sym_0\nr^D$ for $n\geq 3$ and that
$\frac1n\trace_\conf\nr^{D,\Ms}=\ns^D$ for $n\geq 2$ (note that
$\alt\nr^{D,\Ms}=\alt\nr^D$ by definition). In other words, when $n\geq3$, a
conformal metric alone is sufficient to fix a canonical choice of \mob/
structure; when $n=2$, we still have a canonical choice $\Mq^\conf$ of
quadratic operator but lack the linear operator; when $n=1$, the linear
operator necessarily vanishes but we lack the quadratic operator.  To obtain a
uniform approach, we impose the missing data as extra structure as follows.
\begin{defn}
A \emphdef{conformal M{\"o}bius structure} on a $1$-manifold $\mfd$ is a second
order quadratic differential operator $\Mq^\conf$ from $L$ to $\R$ such that
$\Mq^\conf-\Mq^D$ is zero order (for any $D$); $\conf$ is the
canonical section of $S^2T\dual\mfd\ltens L^2$ and $\Mh^\conf=0$.

A \emphdef{conformal M{\"o}bius structure} on a $2$-manifold $\mfd$ is a
conformal metric, together with a second order linear differential operator
$\Mh^\conf$ from $L$ to $S^2_0T\dual\mfd\ltens L$ such that $\Mh^\conf-\Mh^D$
is zero order (for any $D$); $\Mq^\conf$ is the canonical \mob/ operator
of $\conf$.

In higher dimensions a \emphdef{conformal M{\"o}bius structure} is just a
conformal metric; $\Mh^\conf,\Mq^\conf$ are the canonical \mob/
operators of $\conf$.

A \emphdef{\mob/ manifold} is a manifold equipped with a conformal \mob/
structure $(\conf,\Mh^\conf,\Mq^\conf)$.
\end{defn}
We now summarize the situation.
\begin{prop}\label{p:conf-iff-no-ric}
$(\conf,\Ms)$ is a conformal \mob/ structure if and only if $\LH W^\Ms=0$.
\end{prop}

\begin{rem}
In the presence of a conformal \mob/ structure, we define the normalized Ricci
curvature of a Weyl derivative $D$ to be $\nr^{D,\Ms^\conf}$.  When $n=2$,
this provides a hitherto unavailable tracefree part of the normalized Ricci
curvature.  This idea was exploited in~\cite{Cal:mew} to do Einstein--Weyl
geometry on \mob/ $2$-manifolds.
\end{rem}

\renewcommand{\thesubsection}{\thesection.\arabic{subsection}}
\section{The conformal equivalence problem}
\label{sec:equivalence}

It is our contention that conformal Cartan geometries are the same as \mob/
structures and that normal conformal Cartan geometries are the same as
conformal \mob/ structures. For this, a common setting is provided by the
bundle $J^2L$ of $2$-jets of sections of $L$.

A \mob/ structure $\Ms=(\Mh,\Mq)$ is defined by a linear bundle map $\Mh\from
J^2L\to S^2_0T\dual\mfd\ltens L$ and a quadratic form $\Mq\from
J^2L\to\R$. The kernel of $\Mh$ is a rank $n+2$ subbundle $\Mb$ of $J^2L$, on
which we have a quadratic form by restricting $\Mq$. We easily see that this
polarizes to give a nondegenerate inner product of signature $(n+1,1)$ on
$\Mb$ and that the line $L^{-1}$ of pure trace elements of
$S^2T\dual\mfd\ltens L$ is a null line subbundle. We shall exploit the jet
structure of $J^2L$ to find a canonical choice of strongly torsion-free
conformal Cartan connection on $\Mb$.

Conversely, a conformal Cartan geometry $(\CV,\Ln,\CD)$ gives rise to a \mob/
structure $\Ms^\CD=(\Mh^\CD,\Mq^\CD)$ and so we obtain a subbundle $\Mb$ of
$J^2L$ with a conformal Cartan connection. We shall see that the
differential lift, viewed as a bundle map $J^2L\to\CV$, restricts to give an
isomorphism from $\Mb$ to $\CV$ intertwining the conformal Cartan connections.

\subsection{The jet derivative and M\"obius bundles}
\label{par:jet-deriv-mob}

Recall (see~\cite{KMS:nodg} for further details) that for any vector bundle
$E$ over $\mfd$, the \emphdef{$k$-jet bundle} $J^kE$ is the bundle whose fibre
at $x$ is $J^k_xE:=\Cinf(\mfd,E)/\cI^k_x$ where $\cI^k_x$ is the subspace of
sections which vanish to order $k$ at $x$. Thus $J^k_xE$ may be regarded as
the space of $k$th order Taylor series, at $x$, of sections of $E$, and there
is a $k$th order differential operator $j^k\from
\Cinf(\mfd,E)\to\Cinf(\mfd,J^kE)$ sending a section $s$ to its equivalence
class $j^k_xs$ (the \emphdef{$k$-jet} or $k$th order Taylor series) at each
$x\in \mfd$.

Since $\cI^k\subset\cI^{k-1}$, there is, for each $k$, a natural surjective
bundle map $\pi_k=\pi_k^E\from J^kE\to J^{k-1}E$, whose kernel is the bundle
of $k$th derivatives $S^kT\dual\mfd\tens E$. Note that $\pi_k\circ
j^k=j^{k-1}$ and $j^0$ identifies $E$ with $J^0E$.

\begin{defn}
The (\emphdef{semiholonomic}) \emphdef{jet derivative} on $J^kE$ is the first
order differential operator
$\d\low_k\from\Cinf(\mfd,J^kE)\to\Omega^1(\mfd,J^{k-1}E)$ defined by
$\d\low_ks=j^1(\pi_{k}s)-s$.
\end{defn}

This requires a brief explanation: $j^1(\pi_{k}s)$ is a section of
$J^1(J^{k-1}E)$. On the other hand, $J^kE$ is a subbundle of $J^1(J^{k-1}E)$:
the inclusion sends $j^k_xs$ to $j^1_xj^{k-1}s$.  The difference
$j^1(\pi_{k}s)-s$ lies in the kernel of $\pi_1^{J^{k-1}E}$ and hence defines
a section of $T\dual\mfd\tens J^{k-1}E$.

The jet derivative is thus given by the difference between formal and actual
differentiation: it measures the failure of a $k$th order Taylor series to be
the first derivative of the $(k-1)$st order Taylor series obtained by
truncation; note $\d_k\circ j^k=0$. Our main use for it is to distinguish a
preferred class of connections on subbundles of jet bundles.

\def\adapted/{symmetric}\def\Adapted/{Symmetric}\def\an/{a}
\begin{defn} Let $F$ be a subbundle of $J^kE$, let $\d$ and $\pi$ be the
restrictions to $F$ of the jet derivative and the projection $J^kE\to
J^{k-1}E$.  Following~\cite{Gau:ccw}, a connection $\nabla$ on $F$ is said to
be \emphdef{\adapted/} if $\pi\circ\nabla=\d$.  The space of \adapted/
connections, if it is nonempty, is an affine space modelled on
$\Omega^1(\mfd,\Hom(F,\kernel\pi))$.
\end{defn}
We shall only need this formalism on $J^2L$, where the jet derivative is a
differential operator $\Cinf(\mfd,J^2L)\to\Omega^1(\mfd,J^1L)$. We next define
the subbundles we shall use.

\begin{defn}
A \emphdef{\mob/ bundle} on a manifold $\mfd$ is a subbundle $\Mb$ of $J^2L$
with a nondegenerate metric such that $\kernel\pi\intersect\CV\subset
S^2T\dual\mfd\tens L$ is a null line subbundle.
\end{defn}
The following result explains the canonical nature of our forthcoming
constructions.

\begin{prop}\tcite{Gau:ccw}
Let $\Mb$ be a \mob/ bundle on $\mfd$. Then there is at most one \adapted/
metric connection on $\Mb$.
\end{prop}
\begin{proof} We shall show that $\Hom(\Mb,\kernel\pi)\intersect\so(\Mb)=0$.
Indeed a skew map $\Mb\to \kernel\pi$ must factor through the projection
$\Mb\to\Mb/\kernel\pi^\perp$. However, $\kernel\pi$ and $\Mb/\kernel\pi^\perp$
are one dimensional, and the maps
$\Mb\to\Mb/\kernel\pi^\perp\to\kernel\pi\to\Mb$ are symmetric, not skew.
\end{proof}

\subsection{M\"obius structures induce conformal Cartan geometries}
\label{par:mos-ccg}

Let $\Ms=(\Mh,\Mq)$ be a \mob/ structure and define $\Mb=\kernel\Mh\subset
J^2L$. We wish to show that $\Mb$ is a \mob/ bundle and that there is
\an/ \adapted/ metric connection on $\Mb$.

In order to do this, we use a Weyl derivative $D$ to decompose $J^2L$
compatibly with $\Ms$: given $\Ms,D$, identify $J^2L$ with $S^2T\dual
M\ltens L\dsum T\dual\mfd\ltens L\dsum L$ by sending $j^2\ell$ to
$(\psi,\theta,\ell)$ where
\begin{equation*}
\psi=D^2\ell+\nr^{D,\Ms}\ell,\qquad\theta=D\ell.
\end{equation*}
This induces an identification of $J^1L$ with $T\dual\mfd\ltens L\dsum L$
sending $j^1\ell$ to $\lip{D\ell,\ell}$.  Observe that $\psi$ \emph{is}
symmetric, since the skew part of $\nr^{D,\Ms}$ is $-\half F^D$ so that
$\psi=\sym(D^2\ell+\nr^{D,\Ms}\ell)$.

We have $\Mh(\psi,\theta,\ell)=\psi_0$ (the tracefree part of $\psi$) so that
$\Mb$ is identified with
\begin{equation*}
\{(\psi,\theta,\ell):\text{$\psi=\sigma\conf$ for
$\sigma\in L^{-1}$}\}.
\end{equation*}
We thus have an isomorphism $L^{-1}\dsum T\dual\mfd\ltens L\dsum L\cong \Mb$
sending $(\sigma,\theta,\ell)$ to $(\sigma\conf,\theta,\ell)\in J^2L$.

Now $\Mq(\psi,\theta,\ell)=\cip{\theta,\theta}-\frac2n(\trace_\conf\psi)\ell$
so that on $\Mb$ we have
\begin{equation*}
\Mq(\sigma,\theta,\ell)=\cip{\theta,\theta}-2\sigma\ell
\end{equation*}
which is a metric of signature $(n+1,1)$ for which $L^{-1}$ is null.  Hence
$\Mb$ is a \mob/ bundle. Furthermore, we have identified $\Mb$ with
$L^{-1}\dsum T\dual\mfd\ltens L\dsum L$ and so we are in the algebraic setting
of \S\ref{par:graded-lie-alg}.  In particular, the Lie bracket on
$\so(\Mb)$ is identified, under this decomposition, with the bracket
$\abrack{\cdot\,,\cdot}$ defined there.

The jet derivative on $J^2L$ reads in our components
\begin{equation}\label{eq:mob-jet-deriv}
\d\low_2(\psi,\theta,\ell)=
(D\theta+\nr^{D,\Ms}\ell-\psi,D\ell-\theta)
\end{equation}
so that on $\Mb$ we have
\begin{equation*}
\d(\sigma,\theta,\ell)=
(D\theta+\nr^{D,\Ms}\ell-\sigma\conf,D\ell-\theta).
\end{equation*}
Now define a connection $\nabla$ on $\Mb$ by
\begin{equation}\label{eq:mob-conn}
\nabla\low_X\mtrip{\sigma}{\theta}{\ell}=
(\nr^{D,\Ms}_X+D\low_X-X)\mtrip{\sigma}{\theta}{\ell}=
\mtrip{D\low_X\sigma+\nr^{D,\Ms}_X(\theta)}
{D\low_X\theta+\nr^{D,\Ms}_X\,\ell-\sigma X}{D\low_X\ell-\theta(X)}
\end{equation}
where $\nr^{D,\Ms}_X\in T\dual\mfd$, $X\in T\mfd$ are viewed as elements of
$\so(L^{-1}\dsum T\dual\mfd\dsum L)$ via \eqref{eq:act-on-V}.  It is clear
that $\nabla$ is \adapted/.  Moreover, $\nabla$ is metric since $D$ is metric
and differs from $\nabla$ by an $\so(\Mb)$-valued $1$-form.  Further, with
$\Ln=L^{-1}$, $\nabla$ is a conformal Cartan connection whose soldering form
is precisely the `musical' isomorphism $T\mfd\ltens L^{-1}\to T\dual\mfd\ltens
L\cong\Ln^\perp/\Ln$ (which is the origin of our choice of sign in the
definition of the soldering form).

We have $\nabla=\nr^{D,\Ms}+D-\iden$ so that
\begin{equation*}
R^\nabla=\d^D\nr^{D,\Ms}+\;\;R^D\!-\abrack{\iden\wedge\nr^{D,\Ms}}
\;\;-\d^D\iden = C^{\Ms,D} + W^\Ms
\end{equation*}
since $D$ is torsion-free. Furthermore,
\begin{equation*}
\trace\abrack{\iden\wedge\nr^{D,\Ms}}=nF^D=\trace R^D
\end{equation*}
so that $W^\Ms$ has vanishing trace, from which we conclude that $R^\nabla
\restr{L^{-1}}=0$, \ie, $\nabla$ is strongly torsion-free. We therefore have
the following result, cf.~\cite{Car:ecc,Gau:ccw,Tho:dig} when $n\geq3$.
\begin{prop}\label{p:canon-cart}
Let $\Ms=(\Mh,\Mq)$ be a \mob/ structure.  Then the \mob/ bundle
$\Mb=\kernel\Mh$ has a unique \adapted/ metric connection $\nabla$ and
$(\Mb,L^{-1},\nabla)$ is a conformal Cartan geometry with curvature given
\textup(in components\textup) by
\begin{equation*}
R^\nabla=C^{\Ms,D}+W^\Ms.
\end{equation*}
\end{prop}
When is $\nabla$ normal? The answer is immediate: since $T\dual\mfd$ is
abelian, $\LH R^\nabla=\LH W^\Ms$ and so, in view of Proposition
\ref{p:conf-iff-no-ric}, we have the following conclusion.
\begin{prop}\label{p:norm-iff-conf}
$(\Mb,L^{-1},\nabla)$ is a normal conformal Cartan geometry if and only if
$\Ms$ is a conformal \mob/ structure.
\end{prop}

\subsection{Conformal Cartan geometries are M\"obius structures}
\label{par:ccg-mos}

We have seen already that any conformal Cartan geometry $(\CV,\Ln,\CD)$ on
$\mfd$ determines a \mob/ structure in a canonical way: we saw already in
\S\ref{par:sphere-conf-cart} that a conformal Cartan connection determines a
conformal metric $\conf$ on $\mfd$; then, in \S\ref{par:diff-lift} we defined
differential operators $\Mh^\CD\ell=\sym\pi\CD j^\CD\ell$ and $\Mq^\CD\ell =
\lip{j^\CD\ell,j^\CD\ell}$ on $L$, using the differential lift $\ell\mapsto
j^\CD\ell$, which define a \mob/ structure $\Ms^\CD=(\Mh^\CD,\Mq^\CD)$ if
$\CD$ is torsion-free. It remains to show that these two constructions are
mutually inverse up to natural isomorphism.

In one direction, this is straightforward: the \mob/ structure associated to
the conformal Cartan connection $(\Mb,L^{-1},\nabla)$ in this way is
$(\conf,\Ms)$ itself. This is clear for the conformal structure. For the rest,
we recall that for any conformal Cartan connection $(\CV,\Ln,\CD)$, a Weyl
structure determines an isomorphism between $\CV$ and $L^{-1}\dsum
T\dual\mfd\ltens L\dsum L$, with respect to which the Cartan connection may be
written:
\begin{equation}\label{eq:ccc-conn}
\CD\low_X\mtrip{\sigma}{\theta}{\ell}=
(\nr^{D,\CD}_X+D\low_X-X)\mtrip{\sigma}{\theta}{\ell}=
\mtrip{D\low_X\sigma+\nr^{D,\CD}_X(\theta)}
{D\low_X\theta+\nr^{D,\CD}_X\,\ell-\sigma X}{D\low_X\ell-\theta(X)}.
\end{equation}
In the case that $(\CV,\Ln,\CD)=(\Mb,L^{-1},\nabla)$, this decomposition
coincides with that of \S\ref{par:mos-ccg}, and we have
$\nr^{D,\CD}=\nr^{D,\Ms}$, so that $\Ms^\CD=\Ms$.

It remains to prove that the constructions are inverse the other way around,
\ie, that we recover $(\CV,\Ln,\CD)$ as the canonical Cartan connection of
$(\conf,\Ms^\CD)$. We can only expect to do this up to isomorphism, since
$\CV$ is an abstract bundle. Furthermore, since the canonical Cartan
connection of $(\conf,\Ms^\CD)$ is \emph{strongly} torsion-free, we must
suppose that $\CD$ is strongly torsion-free also.  In this case
$\alt\nr^{D,\CD}=-F^D$, so that $\nr^{D,\CD}$ is the \mob/ Ricci curvature
$\nr^{D,\Ms^\CD}$ of $D$ with respect to $\Ms^\CD$.

\begin{prop}\label{p:cart-mob-cart}
Let $(\CV,\Ln,\CD)$ be a conformal Cartan geometry and let $\Mb\subset J^2L$
be the \mob/ bundle of the induced \mob/ structure
$(\conf,\Ms^\CD)$. Then the differential lift $j^\CD$, as a bundle map
$J^2L\to \CV$, restricts to an isomorphism from $(\Mb,L^{-1},\nabla)\to
(\CV,\Ln,\CD)$.
\end{prop}
\begin{proof} We choose a Weyl derivative $D$ and compute in components.
Since $\nr^{D,\CD}=\nr^{D,\Ms^\CD}$, the bundle map $j^\CD\colon J^2L\to \CV$
is then given by
$(\psi,\theta,\ell)\mapsto(\frac1n\trace_\conf\psi,\theta,\ell)$.  Hence it is
an isometry when restricted to $\Mb$ and maps $L^{-1}$ to $\Ln$.
Comparing~\eqref{eq:mob-conn} with~\eqref{eq:ccc-conn} we see that this
isomorphism identifies the connections $\nabla$ and $\CD$.
\end{proof}

Let us combine Propositions~\ref{p:canon-cart},~\ref{p:norm-iff-conf}
and~\ref{p:cart-mob-cart} to summarize our results so far.

\begin{thm}\label{th:mos=ccg}
Let $\conf$ and $\Ms=(\Mh,\Mq)$ define a \mob/ structure on a manifold
$\mfd$. Then there is a unique \adapted/ metric connection on the \mob/ bundle
$\Mb=\kernel\Mh\subset J^2L$ and this is a conformal Cartan geometry
compatible with the same \mob/ structure. Any conformal Cartan geometry
arises in this way up to natural isomorphism, and is normal if and only if the
\mob/ structure is conformal.
\end{thm}

\begin{rem}\label{rem:gen-ccc}
It is straightforward to extend this equivalence to general torsion-free
conformal Cartan connections by extending the notion of a \mob/ structure to
include a $2$-form $\cF$. In order to make room for this $2$-form in the jet
bundle picture, it is necessary to work with semiholonomic jets---or
equivalently, to extend the jet derivative to
$J^2L\dsum(\Wedge^2T\dual\mfd\tens L)$. By this device it is even possible to
consider arbitrary conformal Cartan connections, but the notion of a \mob/
structure must be modified further in the presence of torsion. Since we do not
see any advantage in this generality, and it would have clouded the exposition
considerably, we have restricted attention to the strongly torsion-free case.
\end{rem}

\subsection{Normalization of conformal Cartan geometries}
\label{par:norm-conf-cart}

\mob/ structures compatible with a fixed conformal metric on a manifold $\mfd$
form an affine space modelled on sections of $S^2T\dual\mfd$.  We now exploit
Theorem~\ref{th:mos=ccg} to lift this affine structure to conformal Cartan
geometries. Let $(\CV,\Ln,\CD)$ be a conformal Cartan geometry and
$\QC\in\Cinf(\mfd,S^2T\dual\mfd)$.  We view $\QC$ as an $\so(\CV)$-valued
$1$-form via~\eqref{eq:gam-act} so that
\begin{equation}
\label{eq:Q-act}
\CQ_X\sigma=0,\qquad \CQ_X\CD_Y\sigma=-\QC_X(Y)\sigma, \qquad
\CQ_X(\CV)\subseteq\Ln^\perp.
\end{equation}
In particular, $\CD+\CQ$ is a conformal Cartan connection with the same
soldering form as $\CD$ and so induces the same conformal metric on $\mfd$.
In fact, more can be said.
\begin{lemma}\label{l:affinity} If $\CD$ is strongly torsion-free, so
is $\CD+\CQ$ and $\Ms^{\CD+\CQ}=\Ms^\CD+\QC$.
\end{lemma}
\begin{proof} Since $T\dual\mfd\subset\so(\CV)$ is abelian, we have
\begin{equation*}
R^{\CD+\CQ}=R^\CD+\d^\CD\CQ+\half\liebrack{\CQ\wedge\CQ}=R^\CD+\d^\CD\CQ,
\end{equation*}
so $\CD+\CQ$ is strongly torsion-free iff $\d^\CD\CQ\restr\Ln=0$.  Now,
using~\eqref{eq:Q-act}, we have
\begin{equation*}
(\d^\CD\CQ)_{X,Y}\sigma=(\CD_X\CQ)_Y\sigma-(\CD_Y\CQ)_X\sigma=
-\CQ_Y\CD_X\sigma+\CQ_X\CD_Y\sigma=(\QC_Y X{-}\QC_XY)\sigma=0.
\end{equation*}
For the second assertion, we must compute $j^{\CD+\CQ}\ell$: using
\eqref{eq:pi-gam-act} and \eqref{eq:pDv}, we see that this differs from
$j^\CD\ell$ by a section of $\Ln$ and then the definition of the
differential lift, along with \eqref{eq:H1V}, shows that
$j^{\CD+\CQ}\ell=j^\CD\ell+\frac1n(\trace_\conf\QC)\ell$.  Thus
$\lip{j^{\CD+\CQ}\ell,j^{\CD+\CQ}\ell}=\lip{j^{\CD}\ell,j^{\CD}\ell}
-\frac2n(\trace_\conf\QC)\ell^2$.  Finally, for any section $\sigma$ 
of $\Ln$,
\begin{equation*}
\pi(\CD+\CQ)_X(v+\sigma)=\pi\CD_X v+\QC_X\vtens p(v)-X\tens\sigma
\end{equation*}
from which we deduce that $\pi(\CD+\CQ)j^{\CD+\CQ}\ell =\pi\CD
j^\CD\ell+(\QC-\frac1n(\trace_\conf\QC)\conf)\ell$ and the
result immediately follows.
\end{proof}

This procedure allows us to `normalize' a conformal Cartan geometry, which
will be important later. For this we use the fact,
cf.~Proposition~\ref{p:norm-iff-conf} and Proposition~\ref{p:conf-iff-no-ric},
that a conformal Cartan geometry is normal iff the Weyl curvature of the
corresponding \mob/ structure has zero Ricci contraction. Now for any \mob/
structure $\Ms$, we have $W^{\Ms+\QC}=W^\Ms-\abrack{\iden\wedge\QC}$ and so
\begin{equation}\label{eq:W-change}
\LH W^{\Ms+\QC}=\LH W^{\Ms}-(n-2)\QC_0-2(n-1)(\tfrac1n\trace_\conf\QC)\conf.
\end{equation}
Since $\LH W^{\Ms}$ is symmetric (and is tracelike for $n=2$ and vanishes for
$n=1$) we have:
\begin{prop}\label{p:make-normal}
Let $(\CV,\Ln,\CD)$ be a conformal Cartan geometry. Then there is a unique
section $\QC$ of $S^2 T\dual\mfd$ such that
\begin{numlist}
\item $\CD-\CQ$ is normal\textup;
\item $\QC_0=0$ if $n=2$ and $\QC=0$ if $n=1$, \ie, $\QC$ is in the image of
the Ricci contraction $\LH$.
\end{numlist}
\end{prop}

\section{\mob/ structures in low dimensions}
\label{sec:mob-str-low}

The low dimensional cases $\dimn\mfd=1$ or $\dimn\mfd=2$ merit special
attention for two related reasons.  First, any such $\mfd$ admits conformal
coordinates so that any conformal metric is flat.  On the other hand, our
theory introduces a new ingredient in addition to the conformal metric for, in
this case, $\mfd$ supports many conformal \mob/ structures compatible with a
given conformal metric. We now discuss the geometrical meaning of these data
and explain how to compute them from a normal conformal Cartan connection.

In all this we shall make use of a further feature of low dimensional
geometry. As we noted in Remark~\ref{r:classical-cart}, our conformal Cartan
connections are a linear representation of the usual notion. We chose the
standard representation for simplicity, but other choices are possible. If we
take $G=\Spin_0(n+1,1)$ (which is a double cover of the identity component of
$\Mob(n)$) we can use the spin representation. Such a choice is particularly
effective for $n\leq 4$, when special isomorphisms of $G$ with more familiar
Lie groups give concrete realizations of spinors.

\subsection{Computing the M\"obius structure}

Let us consider in general how to compute the \mob/ structure $\Ms^\CD$
associated to a conformal Cartan geometry $(\CV,\Ln,\CD)$ via the differential
lift: \S\ref{par:diff-lift} provides formulae for this, but they can be
difficult to use since one must decompose $\CD$ with respect to a Weyl
structure. Instead, we shall write $\Ms^\CD=\Ms^D+\sym\nr^{D,\CD}$ relative to
a Weyl derivative $D$, and compute $\sym\nr^{D,\CD}$ directly.

Recall that if $i\low_\Ln\colon\Ln\to\CV$ is the inclusion, we can define
$\CDD i\low_\Ln\colon T\mfd\tens\Ln\to\Ln^\perp\subset \CV$. Differentiating
again using~\eqref{eq:conn-comp}, we find that
\begin{equation} \label{eq:tautological-mobius}
\bigl((\CDD)^2i\low_\Ln\bigr)_{X,Y}\sigma+\nr^{D,\CD}_{X,Y}\sigma=
X\act Y\act\sigma,
\end{equation}
for vector fields $X,Y$ and any section $\sigma$ of $\Ln$. In other words,
$(\CDD)^2i\low_\Ln+\nr^{D,\CD}i\low_\Ln=\conf\tens i_{\Lnc}$, where
$i_{\Lnc}\colon\smash{\Lnc}\to \CV$ is the inclusion of the Weyl structure
corresponding to $D$. The tracefree part of this formula then yields:
\begin{equation}\label{eq:taut-mob-tracefree}
\sym_0\bigl((\CDD)^2i\low_\Ln+\nr^{D,\CD}i\low_\Ln\bigr)=0.
\end{equation}
The trace part is still difficult to use because of the appearance of
$i_{\Lnc}$. To get around this, we note that if $e_i\vtens\sigma$ is an
orthonormal frame of $T\mfd\ltens\Ln$, and
$\hat\sigma\in\Cinf(\mfd,\smash{\Lnc})$ with $\lip{\sigma,\hat\sigma}=-1$,
then $\sigma$, $(\CDD_{e_i} i\low_\Ln)\sigma$ and $\hat\sigma$ form a frame of
$\CV$ relative to which we have, for any $v\in\CV$,
\begin{equation*}
\lip{v,v}=\sum_{i=1}^n \lip{v,(\CDD_{e_i}i_\Ln)\sigma}
-2\lip{v,\sigma}\lip{v,\hat\sigma}.
\end{equation*}
Combining this with the trace of~\eqref{eq:tautological-mobius}, using $X\act
Y\act\sigma=\cip{X,Y}\hat\sigma$, we then have
\begin{equation}\label{eq:taut-mob-trace}
\sum_{i=1}^n\Bigl\{\Lip{v,(\CDD_{e_i} i\low_\Ln)\sigma}^2
-\tfrac2n\lip{v,\sigma}
\Lip{v,((\CDD)^2_{e_i,e_i}i\low_\Ln)\sigma}\Bigr\}-\tfrac2n
\ns^{D,\CD}\lip{v,i\low_\Ln}^2=\lip{v,v}.
\end{equation}
If we take $v=\hat\sigma$, this formula simplifies, because $v$ is then null
and orthogonal to $(\CDD_{e_i} i\low_\Ln)\sigma$. However, the computational
advantage of~\eqref{eq:taut-mob-trace} is the explicitness of the dependence
on $D$.
\begin{rem}
The formulae~\eqref{eq:taut-mob-tracefree}--\eqref{eq:taut-mob-trace} have a
conceptual explanation: viewing $i\low_\Ln$ as a section of $L\tens\CV$, the
$\CD$-coupled differential lift of $i\low_\Ln$ is minus the metric in
$\CV\tens \CV$.  This last is parallel, so that the $\CD$-coupled \mob/
operator $\Mh^{\CD,\CD}$ gives zero when applied to $i\low_\Ln$. On the other
hand $\Mq^{\CD,\CD}(i\low_\Ln)$ is a partial contraction of this differential
lift with itself, which reproduces the metric.
\end{rem}
Now specialize to the case where $D$ is exact with gauge $\sigma$ (\ie,
$D\sigma=0$), so that $F^D=0$. Assume also that $v$ is null.
Then~\eqref{eq:taut-mob-tracefree}--\eqref{eq:taut-mob-trace} read:
\begin{gather}\label{eq:comp-mob-trfree}
\sym_0\CDD \CD\sigma+\nr^{D,\CD}_0\sigma=0;\\
\tfrac1n\ns^{D,\CD}\lip{v,\sigma}^2 = \tfrac12 |\lip{v,\CD\sigma}|^2-
\lip{v,\sigma}\Lip{v,\tfrac1n\trace_\conf(\CDD\CD\sigma)}.
\label{eq:comp-mob-trace}
\end{gather}

\subsection{M\"obius 1-manifolds and projective structures}

A conformal Cartan connection $(\CV,\Ln,\CD)$ on a $1$-manifold $\mfd$ is
automatically flat, hence a normal conformal Cartan geometry. Further, by
Proposition~\ref{p:flat-ccc}, there is a distinguished family of local
immersions of $\mfd$ into $\PL$ (here a circle), which cover $\mfd$ and are
related on overlaps by elements of $\Mob(1)=O_+(2,1)$.

Now $\Spin_0(2,1)\cong \SL(2,\R)$, corresponding to the realization of $S^1$
as the real projective line $\R P^1$.  Thus $\R^{2,1}=S^2\R^2$, the set of
symmetric matrices on $\R^2$, where the quadratic form is minus the
determinant.  More precisely, fix a symplectic form
$\omega_\R\in\Wedge^2\R^{2*}$ and consider $\R^2\tens\R^2$ with the metric
\begin{equation*}
\lip{v_1 \vtens v_2,w_1 \vtens w_2}= -\omega_\R(v_1,w_2)\omega_\R(v_2,w_1)
\end{equation*}
The symmetric tensors form a subspace on which the metric has signature
$(2,1)$. The quadratic map $x\mapsto x\vtens x\colon\R^2\to\R^{2,1}$ then
induces a diffeomorphism $\RP{1}\to\PL$.

The trivial $\R^2$ bundle and tautological line subbundle $\cO(-1)_\R\subset\R
P^1\times\R^2$ may be abstracted as follows: a `spin Cartan connection' is a
real rank $2$ symplectic vector bundle $(\CV_\R,\omega_\R)\to \mfd$ with
symplectic connection $\nabla$ and a line subbundle $\Ln_\R$ such that
$\nabla\restr{\Ln_\R} \modulo \Ln_\R$ defines an isomorphism $T\mfd\tens
\Ln_\R\to \CV_\R/\Ln_\R$. This is a one dimensional projective structure and
is always flat. Given such a spin Cartan connection, we obtain a conformal
Cartan connection on $\CV=S^2\CV_\R$ with $\Ln =\Ln_\R^2$: the induced
connection preserves the bundle metric defined pointwise as above.

It follows that a conformal Cartan connection carries the information of a
projective structure on $\mfd$, and it is natural to enquire how this arises
in explicit geometric terms. We do this now by showing how projective
coordinates, Hill operators and schwarzian derivatives arise from $\CD$.

First, let $x$ be a coordinate on $\mfd$ and $D^x$ the unique connection on
$T\mfd$ with $D^x\del_x=0$, where $\del_x$ is the vector field with $\d
x(\del_x)=1$. (Any connection on a $1$-manifold is a Levi-Civita connection,
hence a Weyl connection: $D^x$ is the Levi-Civita connection of the metric $\d
x^2$ induced by $x$.) If $\sigma_x=|\d x|$ and $v$ is a null section of $\CV$
with $\lip{v,\sigma_x}\neq 0$, then~\eqref{eq:comp-mob-trace} yields a formula
for the only nontrivial datum of the \mob/ structure:
\begin{equation*}
\ns^{D^x,\CD}|\del_x|^2 = \frac
{\lip{v,\CD_{\del_x}\sigma_x}^2}{2\lip{v,\sigma_x}^2}-
\frac{\Lip{v,\CD_{\del_x}\CD_{\del_x}\sigma_x}}{\lip{v,\sigma_x}}
\end{equation*}
Suppose now that $w$ is another coordinate with $\d w=w'\d x$ and assume
for simplicity that $w'>0$ so that $\sigma_w:=|\d w| = w'\sigma_x$.
Then $\CD_{\del_w}\sigma_w=\CD_{\del_x}\sigma_x+w''/w'$ and hence
\begin{equation}\label{eq:clean-schwarzian}
\ns^{D^w,\CD}|\del_w|^2 =\ns^{D^x,\CD}|\del_x|^2-\Bigl(\frac{w''}{w'}\Bigr)'
+\frac12 \Bigl(\frac{w''}{w'}\Bigr)^2.
\end{equation}
We define a projective coordinate to be a coordinate $x$ with
$\ns^{D^x,\CD}=0$, so that $\Mq^\CD=\Mq^{D^x}$. Such coordinates can be found
by solving a nonlinear ODE.
\begin{prop} Let $(\mfd,\Ms^\CD)$ be a \mob/ $1$-manifold. Then there are
\textup(local\textup) projective coordinates on $\mfd$, and if $x$ is such a
coordinate, then another coordinate $w(x)$ is projective if and only if the
schwarzian $S_x(w):=(w''/w')'-\half(w''/w')^2$ of $w$ with respect to $x$ is
zero.
\end{prop}
Another approach to one dimensional projective structures is via Hill (or
Hill's) operators $\Laplace$, \ie, second order differential operators on
$\Cinf(\mfd,L^{1/2})$ which differ from a second derivative $D^2$ by a zero
order term. For this we view $\CD$ as a spin Cartan connection on $\CV_\R$
(with $S^2\CV_\R=\CV$) preserving a symplectic form $\omega_\R$ and such that
$X\vtens\lam\mapsto-\CD_X\lam\modulo\Ln_\R$ defines an isomorphism from
$T\mfd\tens\Ln_\R$ to $\CV_\R/\Ln_\R\cong_{\omega_\R} \Ln_\R\dual=L^{1/2}$.

As in Proposition~\ref{p:diff-lift} any section $f$ of $L^{1/2}$ has a
differential lift $j^\CD f$ which is the unique section of $\CV_\R$ that
projects onto $f$ and satisfies $\CD j^\CD f\in
\kernel\LH\cap\Omega^1(\mfd,\CV_\R)= \Omega^1(\mfd,\Ln_\R)$. If we choose a
Weyl connection $D$, then the derivative $\CD^D i_{\Ln_\R}$ of the inclusion
$i_{\Ln_\R}\colon\Ln_\R\to \CV_\R$ determines a splitting $\CV_\R=\Ln_\R\dsum
\Ln_\R\dual$, and using this decomposition we find
\begin{equation*}
\CD(\lam,f)=(D\lam+\tfrac12 \ns^{D,\CD} f,Df-\lam).
\end{equation*}
It follows that $j^\CD f$ is essentially the $1$-jet of $f$ and $f\mapsto
\Laplace^\CD f:=\CD j^\CD f$ is a Hill operator: with respect to any Weyl
derivative $D$, we have $j^\CD f =(Df, f)$, and $\Laplace^\CD f
=D^2f+\frac12\ns^{D,\CD}f$. This provides another way of computing the \mob/
structure, and it is well known that the zero order term of a Hill operator
transforms by schwarzian derivative under changes of coordinate.

The two approaches are related in a simple way.
\begin{prop} The quadratic operator $\Mq^\CD(\ell)$ is related to the Hill
operator $\Laplace^\CD f$, when $\ell=f^2$, by
$\Mq^\CD(\ell)=-4f^3\Laplace^\CD f$.
\end{prop}
\begin{proof} This is straightforward to prove using a Weyl derivative.
Alternatively, observe that $j^\CD\ell=(j^\CD f)\tens (j^\CD f) +\sigma$,
where $\sigma=\LH\CD ((j^\CD f)\tens (j^\CD f))$. (This follows because $\LH
\CD\sigma=-\sigma$.) We then compute, using dual frames $e,\eps$ (\ie,
$\eps(e)=1$), that $\Mq^\CD(\ell)=2\lip{j^\CD f\otimes j^\CD
f,\sigma}=4\omega_\R(j^\CD f,\CD_e j^\CD f)\omega_\R(j^\CD f,\eps\cdot j^\CD
f)=-4 f^3 \CD j^\CD f$.
\end{proof}
We have already noted that $\CV_\R$ can be identified the $1$-jet bundle of
$L^{1/2}$; its parallel sections are `affine with respect to $\CD$', \ie,
$1$-jets of solutions of $\Laplace^\CD f=0$. Similarly, $\CV=S^2\Ln_\R$ is the
$2$-jet bundle of $L$, and $L$ is the tangent bundle if $\mfd$ is oriented; the
parallel sections are $2$-jets of projective vector fields, \ie, vector fields
$X$ with $\cL_X\Laplace^\CD=0$.

\begin{rem}
The interpretation of $\ns^{D}$ as a one dimensional analogue of scalar
curvature suggests the following Yamabe-like problem: given a projective
structure $\Laplace$ on $\mfd=\R/\Z$, is there a gauge of constant scalar
curvature? This amounts to solving $\Laplace\mu = c \mu^{-3}$ for a constant
$c$ and a positive section $\mu$ of $L^{1/2}$ which we normalize via
$\int_\mfd\mu^{-2}=1$: if $D\mu=0$, we then have $\ns^D=c\mu^{-2}$. Solutions
are critical points of the Yamabe-like functional
$Y(\mu)=(\int_\mfd\mu^{-2})(\int_\mfd \mu\,\Laplace\mu)$. Given a solution
$\mu$, we have $\Laplace (f\mu) = (f''+c f)\mu^{-3}$ for any $f\in
C^\infty(M,\R)$, from which it is straightforward to determine the developing
map $\delta$ from $\tilde\mfd=\R$ to $\R P^1\cong\R\cup\{\infty\}$:
$\delta(x)=\frac{1}{\sqrt{c}}\tan(\sqrt{c}x)$ for $c>0$, $\delta(x)=x$ for
$c=0$ and $\delta(x)=\frac{1}{\sqrt{-c}}\tanh(\sqrt{-c}x)$ for $c<0$. The
constant $c$ determines the conjugacy class of the holonomy and, for $c>0$
(when the developing map surjects), the winding number of the fundamental
domain $[0,1)$: the fundamental domain injects for $c\leq 4\pi^2$. Any such
data arise for some (unique) $c\in \R$, and so any projective structure admits
a constant scalar curvature gauge. Without the geometric interpretation, the
existence of critical points for $Y$ is not obvious, since $Y$ is unbounded:
if $\mu_\eps(x)=(1+\eps+\cos(2\pi x))\d x^{-1/2}$, where $x\colon
M\stackrel{\cong}{\rightarrow}\R/\Z$, then $\int_M\mu_\eps^{-2}$ can be
arbitrarily large for $\eps>0$.
\end{rem}

\subsection{M\"obius 2-manifolds and schwarzian derivatives}

When $n=2$, $\Mob(2)=O_+(3,1)$, and $\Spin_0(3,1)\cong \SL(2,\C)$
corresponding to the realization of $S^2$ as the complex projective line $\C
P^1$. Thus $\R^{3,1}= (\C^2\tens\bar\C^2)_\R$, the set of hermitian matrices
on $\C^2$, where the quadratic form is minus the determinant.  More precisely,
fix a symplectic form $\omega_\C\in\Wedge^2\C^{2*}$ and consider
$\C^2\tens_\C\bar\C^2$ with the metric
\begin{equation*}
(v_1\vtens\bar v_2,w_1\vtens\bar w_2)=
-\omega_\C(v_1,w_1)\overline{\omega_\C(v_2,w_2)}.
\end{equation*}
The conjugation $x\vtens \bar{y}\mapsto y\vtens\bar{x}$ fixes a real form
$\R^{3,1}$ (the hermitian matrices) on which the metric is real of signature
$(3,1)$. The quadratic map $x\mapsto x\vtens\bar{x}\colon\C^2\to\R^{3,1}$ then
induces a conformal diffeomorphism $\CP{1}\to\PL$.

The trivial $\C^2$-bundle and the tautological complex line subbundle
$\cO(-1)_\C\subset\C P^1\times\C^2$ may be abstracted as follows: a `spin
Cartan connection' as a complex rank $2$ symplectic vector bundle
$(\CV_\C,\omega_\C)\to \mfd$ with complex symplectic connection $\nabla$ and a
complex line subbundle $\Ln_\C$ such that the linear bundle map
$\nabla\restr{\Ln_\C} \modulo \Ln_\C\colon T\mfd\tens \Ln_\C\to \CV_\C/\Ln_\C$
is an isomorphism. A spin Cartan connection induces a conformal Cartan
connection via $\CV=(\CV_\C\tens\bar\CV_\C)_\R$ and $\Ln =(\Ln_\C\tens\bar
\Ln_\C)_\R$ (the fixed points of the conjugation involution): the induced
connection preserves the bundle metric defined pointwise as above.

The conformal \mob/ structures compatible with a given conformal metric form
an affine space modelled on $C^\infty(\mfd,S^2_0T\dual\mfd)$.  To compute this
extra datum, we compare a given \mob/ structure with that provided by a
holomorphic coordinate $z$: given this, $\d z \,\d\bar{z}$ is a flat metric
compatible with $\conf$ and so arises from a length scale $\ell$ which is
parallel for the (flat) Levi-Civita connection $D^z$ of this metric.
Contemplate the \mob/ structure $\Ms^z:=\Ms^{D^z}$: since $F^{D^z}$ vanishes,
$\nr^{D^z,\Ms^z}=0$ so that, first, $C^{D^z,\Ms^z}=0$ and, second,
$W^{\Ms^z}=R^{D^z}=0$, whence the corresponding Cartan connection $\CD^z$ is
flat by Proposition~\ref{p:canon-cart}.

Since $S^2_0T\dual\mfd\tens\C\cong \Omega^{2,0}\mfd\dsum \Omega^{0,2}\mfd$,
where $\Omega^{2,0}\mfd=(T^{1,0}\mfd)^{-2}$, $\Omega^{0,2}\mfd=
(T^{0,1}\mfd)^{-2}$, and $T\mfd\tens\C=T^{1,0}\mfd\dsum T^{0,1}\mfd$ is the
$\pm\iI$-eigenspace decomposition of the complex structure $J\colon T\mfd\to
T\mfd$, we can identify $C^\infty(\mfd,S^2_0T\dual\mfd)$ with the space of
smooth quadratic differentials and write $\nr^{D^z,\CD}_0=q\,\d z^2+\bar{q}\,\d
\bar{z}^2$. We now take $D=D^z$ in equation~\eqref{eq:comp-mob-trfree} and note
that $D^z_{\del_z}\del_z=0$ to conclude:
\begin{equation*}\notag
\CD_{\del_z}\CD_{\del_z}\sigma+q\sigma=0.
\end{equation*}
This gives an effective method to compute $\nr^{D^z,\CD}_0$.
\begin{prop}\label{p:Q-calc}
Let $(\CV,\Ln,\CD)$ a conformal Cartan geometry on a $2$-manifold $\mfd$ with
a holomorphic coordinate $z$ for the induced conformal metric. Let
$\Ms=\Ms^z+\nr^{D^z,\CD}$ be the induced \mob/ structure and write
$\nr^{D^z,\CD}_0=q\,\d z^2+\bar{q}\,\d \bar{z}^2$.  Then
\begin{equation}
\label{eq:comp-sch}
\CD_{\del_z}\CD_{\del_z}\sigma+q\sigma=0
\end{equation}
where $\sigma$ is the \textup(unique\textup) positive section of $\Ln$ with
$\lip{\CD\sigma,\CD\sigma}=\d z\,\d\bar{z}$.
\end{prop}

It is instructive to apply this analysis to the \mob/ structure $\Ms^w$ where
$w$ is another holomorphic coordinate.  Let $\CD^w$ be the corresponding
Cartan connection and $\sigma_w$ the section of $\Ln$ with
$\lip{\CD^w\sigma_w,\CD^w\sigma_w}=\d w\,\d\bar{w}=|w'|^2\d z\,\d \bar{z}$,
where $'$ denotes differentiation with respect to $z$. Then $\sigma
=\sigma_w/|w'|$ is $D^z$-parallel, and we have from Proposition~\ref{p:Q-calc}
that $0=\CD^w_{\del_w}\CD^w_{\del_w}\sigma_w=(w')^{-1}\CD^w_{\del_z}\bigl(
(w')^{-1}\CD^w_{\del_z}(|w'|\sigma)\bigr)$, from which we compute that
\begin{gather*}\notag
\CD^w_{\del_z}\CD^w_{\del_z}\sigma+\half S_z(w)\sigma=0\\
\tag*{where}
S_z(w)=\Bigl(\frac{w''}{w'}\Bigr)'-\frac{1}{2}\Bigl(\frac{w''}{w'}\Bigr)^2
\end{gather*}
is the classical schwarzian derivative of $w$. We thus have the following
transformation law.
\begin{prop}\label{th:schwarzian}
$\Ms^w=\Ms^z+\real(S_z(w)\d z^2)$.
\end{prop}
This prompts us, with \cite{Cal:mew}, to define the \emph{schwarzian
derivative $S_z(\Ms)$ of a \mob/ structure $\Ms$ with respect to $z$} by
\begin{equation*}\notag
(\Ms-\Ms^z)^{2,0}=\half S_z(\Ms)\d z^2.
\end{equation*}
For coordinate \mob/ structures, we have seen that $S_z(\Ms^w)=S_z(w)$, and in
all cases, it has the classical transformation law:
$\Ms-\Ms^w=(\Ms-\Ms^z)+(\Ms^z-\Ms^w)$ giving
\begin{equation*}\notag
S_w(\Ms)=S_z(\Ms)(\d z/\d w)^2+S_w(z).
\end{equation*}
Note that the holomorphicity of $S_z(\Ms)$ is independent of the coordinate
$z$ and amounts to the flatness of $\Ms$: if $\Ms$ is flat then, since $D^z$
is flat and $\Ms$ is canonical, $\nr^{D^z,\Ms}=\QC^D_0$; hence
$\d^{D^z}\QC^D_0=0$, \ie, $S_z(\Ms)$ is holomorphic. Conversely, if $S_z(\Ms)$
is holomorphic, solving a holomorphic ODE yields a coordinate $w$ with
$S_w(\Ms)=0$, \ie, $\Ms=\Ms^w$, which is flat. In particular, any flat $\Ms$
is locally of the form $\Ms^w$ for some holomorphic coordinate $w$.

(Alternatively, note that since $D^z$ is flat, a canonical $\Ms$ has
$\nr^{D^z,\Ms}=Q^D_0$ while, since $\dimn M=2$, $W^\Ms=0$ whence $\Ms$ is flat
if and only if $\d^{D^z}Q^D_0=0$ which last amounts to the holomorphicity of
$S_z(\Ms)$.)

As in the one dimensional case, we relate the schwarzian derivative appearing
here to Hill operators and projective structures using the spin representation
$\Ln_\C\subset\CV_\C$ of $\CD$ above.

A Weyl derivative $D$ provides a decomposition $\CV_\C=\Ln_\C\dsum\Ln_\C\dual$
with respect to which the differential lift of $f\in\Ln_\C\dual$ may be
written $j^\CD f=(\del^D f, f)$, where $\del^D f$ is the complex linear part
of $Df$, \ie, $\del^D f= Df-\dbar f$. (For $D=D^z$, this was denoted by a
prime above.) It follows that $\hc{\CD j^\CD f}= ((\del^D)^2 f+q^{D,\CD}
f,\dbar f)$ and the first component is a complex Hill operator, where
$q^{D,\CD}=(\nr^{D,\CD}_0)^{2,0}$: if $D=D^z$, this may be written $f''+ q f$,
explaining the appearance of the complex schwarzian derivative above.

It is easy to relate the two pictures: the `real' \mob/ operator
$\Mh\ell=\sym_0 D^2\ell+\nr^{D,\CD}_0\ell$ is obtained from the complex Hill
operator $(\del^D)^2 f+q^D f$, when $\ell=f\vtens\bar f$, by coupling the
latter to the antiholomorphic structure $\del$ on $\bar\Ln_\C$ and taking the
real part.

\subsection{Quaternionic geometry of M\"obius 3- and 4-manifolds}
\label{par:quaternionic}

The spin formalism can also be applied to conformal $3$- and $4$-manifolds.
When $n=4$, $\Mob(4)=\Spin_0(5,1)$, and $\Spin_0(5,1)\cong \SL(2,\HQ)$
corresponding to the realization of $S^4$ as the quaternionic projective line
$\HQ P^1$. Thus $\R^{5,1}$ is the set of quaternionic hermitian matrices on
$\HQ^2$.

The trivial $\HQ^2$-bundle and the tautological quaternionic line subbundle
$\cO(-1)_\HQ\subset\HQ P^1\times\HQ^2$ may be abstracted as follows: a `spin
Cartan connection' as a quaternionic rank $2$ vector bundle $\CV_\HQ\to \mfd$
with unimodular quaternionic connection $\nabla$ and a quaternionic line
subbundle $\Ln_\HQ$ such that the linear bundle map $\nabla\restr{\Ln_\HQ}
\modulo \Ln_\HQ\colon T\mfd\tens \Ln_\HQ\to \CV_\HQ/\Ln_\HQ$ defines an
isomorphism $TM\cong\Hom_\HQ(\Ln_\HQ,\CV_\HQ/\Ln_\HQ)$: then
$\End_\HQ(\Ln_\HQ)$ and $\End_{\HQ}(\CV_\HQ/\Ln_\HQ)$ are bundles of
quaternions acting on $TM$ by dilations and $\pm$selfdual rotations with
respect to the induced conformal metric. The quaternionic realization of $S^4$
has proved to be very effective in the conformal geometry of surfaces in $S^3$
and $S^4$~\cite{BFLPP:cgs,Her:imdg}, and we shall discuss this briefly later.
The conformal Cartan connection may be recovered using the real structure
$j\wedge j$ on $\Wedge^2_\C \CV_\HQ$.

The case $n=3$ is similar: $\Spin_0(4,1)\cong \Symp(1,1)$ corresponds to the
realization of $S^3$ as a real quadric in $\HQ P^1$. The `spin Cartan
connections' are therefore similar to the four dimensional case, except that
$\CV_\HQ$ is equipped with a signature $(1,1)$ quaternion-hermitian metric
with respect to which $\Ln_\HQ$ is null. We leave the details to the
interested Reader.

In three and four dimensions, a normal Cartan connection is determined by the
conformal metric, so there is no additional data to compute.  Nevertheless, we
have a differential lift and quaternionic \mob/ operators on sections of
$\CV_\HQ/\Ln_\HQ$, which we momentarily discuss in the four dimensional case,
purely for general interest. Using a Weyl structure $D$,
$\CV_\HQ\cong\Ln_\HQ\dsum \CV_\HQ/\Ln_\HQ$ and $\CD_X(\lam,f)=
(D_X\lam+\nr^{D,\CD}_X(f),D_X f-X(\lam))$, where we identify $TM$ with
$\Hom(\Ln_\HQ,\CV_\HQ/\Ln_\HQ)$. In these terms $j^\CD f=(\del^D_\HQ f,f)$,
where the first term is a Dirac operator, and $\hc{\CD j^\CD f}= Df-\del^D_\HQ
f$ is the twistor operator.

\section{Conformal geometry revisited}
\label{sec:conf-geom-rev}

\subsection{Gauge theory and moduli of conformal Cartan geometries}
\label{par:gauge-theory-moduli}

Consider an $n$-manifold $\mfd$ equipped with a Cartan vector bundle $\CV$
with null line subbundle $\Ln$, and let $\cA$ denote the set of conformal
Cartan connections on $(\CV,\Ln)$. The \emphdef{gauge group}
$\cG:=\{g\in\Cinf(\mfd,\On(\CV))\colon g\Ln^+\subset\Ln^+\}$ is the space of
sections of a bundle of parabolic subgroups of $\On(\CV)$ whose Lie algebra
bundle is $\stab(\Ln)\subset\so(\CV)$. It acts on $\cA$ is the usual way: for
$g\in\cG$ we have
\begin{equation}
g\act\CD:=g\circ\CD\circ g^{-1}=\CD-(\CD g)g^{-1}.
\end{equation}
Since $g$ acts orthogonally on $\CV$ and preserves $\Ln^+$, $g\act\CD$ and
$\CD$ induce the same conformal metric on $\mfd$. Hence if
$\cA^\conf\subset\cA$ denotes the subset of connections inducing a fixed
conformal metric $\conf$, $\cA^\conf$ is preserved by $\cG$.  We next ask when
gauge equivalent conformal Cartan connections induce the same soldering form.

\begin{prop}\label{p:gauge-traf} Let $\CD$ be a conformal Cartan connection
and $g$ a gauge transformation. Then $\CD$ and $g\act \CD$ have the same
soldering form iff $g=\exp\gam$ with $\gam\in T\dual\mfd\subset \stab(\Ln)$.
\end{prop}
\begin{proof} $g\act\CD$ has the same soldering form as $\CD$ if and only
if $g$ commutes with $\beta^\CD$, \ie, $g\beta^\CD \sigma=\beta^\CD g\sigma$.
However, $g\Ln^+\subset\Ln^+$ so that $g\sigma=e^u\sigma$ for some function
$u$ and now $g\beta^\CD \sigma=e^u\beta^\CD \sigma$ implies, by the Cartan
condition, that
\begin{equation*}\notag
g\restr{\Ln^\perp}=e^u\iden+\gam\circ\pi
\end{equation*}
with $\gam\colon\Ln^\perp/\Ln\to\Ln$.  Now $g$ acts orthogonally so, for
$v\in\Ln^\perp$,
\begin{equation*}\notag
\lip{v,v}=\lip{gv,gv}=e^{2u}\lip{v,v}
\end{equation*}
whence $e^u=1$.  In particular,
$\smash{g\restr{\Ln^\perp}}=\iden+\gam\circ\pi$.  However, it is not
difficult to see that any orthogonal $g$ is determined
by its values on $\Ln^\perp$ and so we conclude that $g=\exp\gam$
under the usual 
soldering identification of $\Hom(\Ln^\perp/\Ln,\Ln)$ with $T\dual\mfd$.
Certainly any $g$ of this form preserves $\beta^\CD$.
\end{proof}

We now combine this gauge theory with the normalization of conformal Cartan
connections discussed in~\S\ref{par:norm-conf-cart}.  Let $\CD$ be a fixed
strongly torsion-free connection inducing the conformal metric $\conf$. Then
Lemma~\ref{l:affinity} assures us that any other strongly torsion-free
connection induces the same \mob/ structure as one of the form $\CD+\CQ$, and
is therefore gauge equivalent to it by Proposition~\ref{p:cart-mob-cart}. In
fact the connections of the form $\CD+\CQ$ provide a slice to the gauge
action: it just remains to show that $\CD$ and $\CD+\CQ$ are only gauge
equivalent if $\QC=0$. However, since they have the same soldering form,
Proposition~\ref{p:gauge-traf} and equation~\eqref{eq:log-deriv} show that
$\QC = -\CD\gam - \half\abrack{\gam,\pi\CD\gam}$. Therefore
$0=-\CQ_X\sigma=(\CD_X\gam)\sigma=-\gam\act\CD_X\sigma=\gam(X)\sigma$ for any
vector field $X$. Thus $\gam=0$ and $\QC=0$.

\begin{thm} Let $\cA^\conf_{\mathrm{stf}}$ be the set of strongly
torsion-free conformal Cartan connections inducing the conformal metric
$\conf$. Then $\cG$ acts freely on $\cA^\conf_{\mathrm{stf}}$ and, if $\CD$ is
a fixed element of $\cA^\conf_{\mathrm{stf}}$, there is a unique connection of
the form $\CD+\CQ$ in each $\cG$-orbit.
\end{thm}

\subsection{Constant vectors, spaceform geometries and stereoprojection}
\label{par:spaceform}

If $(\CV,\Ln,\CD)$ is a conformal Cartan geometry on $\mfd$, then we have seen
that the compatible riemannian metrics correspond to sections $\sigma$ of
$\LC^+\subset\Ln$ via $g_\sigma(X,Y)=\lip{\CD_X\sigma,\CD_Y\sigma}$. This
induces a Weyl structure $\Lnc\subset\CV$ spanned by $\hat\sigma$ with
$\lip{\hat\sigma,\hat\sigma}=0$ and $(\hat\sigma,\CD_X\sigma)=0$.

When $\CD$ is flat, a particularly important class of sections of $\LC^+$ are
the \emph{conic sections}: these sections are given by the intersection of
$\LC^+$ with an affine hyperplane, \ie, $\{\sigma\in
\LC^+:\lip{v_\infty,\sigma}=-1\}$, where $v_\infty$ is a parallel section of
$\CV$ and we assume that $\Ln\not\subseteq v_\infty^\perp$ on $\mfd$.  Since
$\lip{\CD_X\sigma,v_\infty}=0$ and $\lip{\CD_X\sigma,\sigma}=0$, the
corresponding Weyl structure is spanned by
\begin{equation}
\hat\sigma=v_\infty+\half\lip{v_\infty,v_\infty}\sigma.
\end{equation}
The length scale $\ell$ dual to $\sigma$ (with $\ip{\sigma,\ell}=1$) is the
homology class $v_\infty\modulo\Ln^\perp$, and $v_\infty$ is the differential
lift $j^{\CD}\ell$ of Proposition~\ref{p:diff-lift}. In this case the
corresponding metric $g$ is an Einstein metric (since $\Mh\,\ell=0$) of
constant curvature $\frac2n\ns^g\ell^2= - \Mq\,\ell=-\lip{v_\infty,v_\infty}$,
and is the metric of a spherical, euclidean or hyperbolic spaceform.

To see this explicitly, suppose that $\mfd$ is simply connected, so we may
assume $\CV=\mfd\cross\R^{n+1,1}$ with $\CD=\d$ and $v_\infty$ constant. Then
$x\mapsto \Ln_x$ defines a local diffeomorphism from $\mfd$ to $\PL\setdif
\Proj(v_\infty^\perp\intersect \LC)$---which is $\PL$,
$\PL\setdif\{\vspan{v_\infty}\}$ or $\PL$ with a hypersphere removed,
according to whether $v_\infty$ is timelike, null, or spacelike,
cf.~\S\ref{par:sphere-conf-cart}.

\begin{bulletpars}
\item If $v_\infty$ is non-null, then define
$v=\sigma+v_\infty/\lip{v_\infty,v_\infty}$. Then $\lip{v,v_\infty}=0$ and
$\lip{v,v}=-1/\lip{v_\infty,v_\infty}$, so that $x\mapsto v(x)$ is a local
diffeomorphism from $\mfd$ to $\{v\in v_\infty^\perp:\lip{v,v}=\pm r^2\}$,
where $1/r^2=|\lip{v_\infty,v_\infty}|$, which is a hyperboloid or a sphere of
radius $r$ according to whether $\lip{v_\infty,v_\infty}$ is positive or
negative. Since $\d_X v=\d_X\sigma$, this local diffeomorphism is an isometry.

\item If $v_\infty$ is null, let $v_0$ be another constant null vector with
$\lip{v_0,v_\infty}=-1$ and identify $\{v_0,v_\infty\}^\perp$ with euclidean
$\R^n$. Then the \emphdef{stereographic projection} or
\emphdef{stereoprojection}
\begin{equation}
\sigma\mapsto v=\pr_{\R^n}\sigma=\sigma-v_0+\lip{\sigma,v_0}v_\infty
\end{equation}
is a local diffeomorphism from $\mfd$ to $\R^n$ with inverse
\begin{equation}
v\mapsto \sigma = \exp(v\wedge v_\infty)\act v_0
= v+v_0+\half\lip{v,v}v_\infty.
\end{equation}
Since $\lip{\d_X\sigma,v_\infty}=0$, we have $\lip{\d_X v,\d_Y v}=
\lip{\d_X\sigma,\d_Y\sigma}=g(X,Y)$ and so stereoprojection is also an isometry
for the metric $g$ induced by $\sigma$. This proves that flat conformal Cartan
geometries are conformally flat in the usual sense, the charts given by
stereoprojection being conformal. Furthermore, these charts are not just
conformal, but \mob/ (which is only an issue for curves and surfaces):
since $\Mh^\conf\ell=0=\Mq^\conf\ell$ (\ie, $\nr^g=0$), the flat \mob/
structure given by the flat metric on $\R^n$ pulls back to the flat \mob/
structure on $S^n$.
\end{bulletpars}

\subsection{Symmetry breaking}
\label{par:sym-break}

The introduction of a constant vector $v_\infty\in\R^{n+1,1}$ breaks symmetry
from the \mob/ group of $S^n$ to the isometry group $\Stab(v_\infty)$ of a
spaceform (or the homothety group in the flat case). Such constant vectors
sometimes arise naturally in conformal submanifold geometry. Let us study,
more generally, a constant $(k+1)$-plane $W$ (or equivalently a constant
decomposable $(k+1)$-vector $\omega$ up to scale) for $0\leq k\leq n$.

\begin{bulletpars}
\item If the induced metric on $W$ is nondegenerate, then replacing $W$ by
$W^\perp$ and $k$ by $n-k$ if necessary, we may assume that $W$ has signature
$(k,1)$ and $W^\perp$ has signature $(n-k+1,0)$. Now any null vector $\sigma$
not in $W$ can be scaled so that $\sigma=w+x$ where $w$ is a positively
oriented timelike unit vector in $W$ and $x$ is spacelike unit vector in
$W^\perp$, both uniquely determined. This defines an isometry from
$(\cH^k+W^\perp)\cap\LC$ to $\cH^k\times \cS^{n-k}$, where $\cH^k$ is (the
positive sheet of) the hyperboloid in $W$, and $\cH^k\times \cS^{n-k}$ is
equipped with a product of constant curvature metrics. Hence we obtain a
conformal diffeomorphism from $\PL\setdif \Proj(W\cap\LC)$, \ie, $S^n\setdif
S^{k-1}$, to $\cH^k\times \cS^{n-k}$.

\item If the induced metric on $W$ is degenerate, \ie, $W$ has signature
$(k,0)$ and $W^\perp$ has signature $(n-k,0)$, then $W\cap
W^\perp=W\cap\LC=W^\perp\cap\LC$ is a null line spanned by a null vector $\hat
w$. If $w_0$ is another null vector with $\lip{\hat w,w_0}=-1$ then any null
vector $\sigma$ not in the span of $W$ is a constant multiple of $w_0+\lam
\hat w + x+y$ for unique $x\in W\cap\hat w_0^\perp$ and $y\in W^\perp\cap\hat
w_0^\perp$, identifying $(w_0+W+W^\perp)\cap \LC$ isometrically with a product
$\R^k\times\R^{n-k}$ of euclidean spaces.  Up to translation and scale this
isometry is independent of the choice of $w_0$ and $\hat w$, we have a
conformal diffeomorphism from $S^n\setdif\{\mathit{pt}\}$ to
$\R^k\times\R^{n-k}$.
\end{bulletpars}

As in the case of vectors, the choice $W$ induces a Weyl structure $\Lnc$,
hence a compatible metric (up to homothety), on the open subset of $S^n$ where
$\Ln\not\subseteq W, W^\perp$ such that $\Ln\dsum\Lnc$ is the $(1,1)$-plane
$(\submfd\times W\dsum\Ln)\cap (\submfd\times W^\perp\dsum\Ln)$ on this open
subset: in the nondegenerate case, $\Lnc=\refl{W}\Ln$ where
$\refl{W}=\iden_W-\iden_{W^\perp}$; in the degenerate case $\Lnc=W\cap
W^\perp=\vspan{\hat w}$ is constant. Since $\submfd\times
W^\perp\cap(\Ln\dsum\Lnc)$ is a constant line inside $\submfd\times
W\dsum\Lnc$ (and similarly for $\submfd\times W\cap(\Ln\dsum\Lnc)$ inside
$\submfd\times W^\perp\dsum\Lnc$), $\Lnc$ is the Weyl structure of the
product of constant curvature metrics.

If we view the representative $\omega\in \Wedge^{k+1}W$ as a constant
$(k+1)$-vector on $S^n$, then its Lie algebra homology class $\hc{\omega}$
defines a section of $H_0(T\dual S^n,S^n\times\Wedge^{k+1}\R^{n+1,1})$, which
turns out to be the bundle $\Wedge^k(TS^n\ltens \Ln)\ltens L\cong
\Wedge^kT\dual S^n\ltens L^{k+1}$ and $\hc{\omega}$ is a decomposable section
tangent to one of the factors of the product structure. The equation
$\d\omega=0$, viewed as an equation on $\hc{\omega}$ implies that the latter
is a \emphdef{conformal Killing form} in the sense of~\cite{Sem:ckf}.

This symmetry breaking is particularly straightforward when $k=1$: a
decomposable $2$-form $\omega$ identifies $S^n\setdif S^{n-2}$,
$S^n\setdif\{\mathrm{pt}\}$ or $S^n\setdif S^0$ with $\cS^1\times \cH^{n-1}$,
$\R\times \R^{n-1}$ or $\cH^1\times \cS^{n-1}$ respectively, according to
whether $\omega$ is spacelike, null, or timelike, respectively (\ie, whether
the corresponding $2$-plane is euclidean, degenerate, or lorentzian
respectively). Note that $\omega$ may also be viewed as an element of
$\so(n+1,1)$, \ie, an infinitesimal \mob/ transformation of $S^n$, its Lie
algebra homology class $\hc{\omega}$ being the vector field generating this
action. Since $\omega$ is a wedge product of two vectors in $W$, it follows
that $\hc{\omega}$ generates of the obvious $S^1$ or $\R$ action on
$\cS^1\times \cH^{n-1}$, $\R\times \R^{n-1}$ or $\cH^1\times \cS^{n-1}$.

\part{Submanifolds and The Conformal Bonnet Theorem}

Now we come to the heart of the matter: an immersed submanifold of a \mob/
manifold inherits a conformal \mob/ structure in a canonical way.  In
particular, a curve or surface immersed in a conformal $n$-manifold with
$n\geq3$ acquires more intrinsic geometry from the ambient space than merely a
conformal metric.

The key ingredient of our theory is the notion of a M\"obius reduction,
introduced in~\S\ref{par:mob-red} as a reduction of structure group along a
submanifold $\submfd$ immersed in a M\"obius manifold
$\mfd$. In~\S\ref{par:mob-sphere-cong}, we show that this generalizes the
classical notion of an enveloped sphere congruence. In~\S\ref{par:mr-split} we
compute the data induced on $\submfd$ by a M\"obius reduction, and then
in~\S\ref{par:gauge-mob}, using a gauge theoretic interpretation of M\"obius
reductions, we show how these data depend on the choice of M\"obius reduction.

One of these data is an analogue of the second fundamental form (or shape
operator) in riemannian geometry (we make the analogy precise in
\S\ref{par:mr-ambient-weyl-riem}) and this datum only depends on the M\"obius
reduction through its trace, a generalized mean curvature conormal
vector. This allows us to achieve our first goal in~\S\ref{par:norm-prim},
where we show that there is a unique M\"obius reduction whose second
fundamental form is tracefree. By normalizing the induced Cartan connection,
we thus obtain a canonical conformal Cartan geometry on $\submfd$.

For submanifolds of the conformal $n$-sphere, the canonical M\"obius reduction
is the central sphere congruence~\cite{Bla:vud}. In this case, we show
in~\S\ref{par:recover} that the primitive data induced on $\submfd$ by the
reduction suffice to recover the ambient flat conformal Cartan connection.  As
we show in~\S\ref{par:conf-bonnet-sphere}, this leads at once to a conformal
Bonnet theorem, and we present two approaches to the \GCR/ equations
characterizing conformal immersions. The first, introduced in
\S\ref{par:abs-hom-conf-bonnet}, is an unashamedly abstract treatment using
Lie algebra homology and the \BGG/ machinery of~\cite{CaDi:di}. The second,
presented in~\S\ref{par:conc-hom-conf-bonnet}, uses the decomposition with
respect to a Weyl structure (described and studied in
in~\S\S\ref{par:decomp-weyl}--\ref{par:homological-reduction}) to give a more
explicit formulation.

In~\S\ref{par:curves}, we specialize our theory to curves and study their
M\"obius invariants. We turn to surfaces in~\S\ref{par:surfaces}, where, with
an eye on the applications in Part III, we study quadratic differentials and
the relation between the differential lift and the Hodge star operator.  We
also relate our theory to the work of~\cite{BFLPP:cgs,BPP:sdfs}.

\section{Submanifolds and M\"obius reductions}
\label{sec:sub-mob}

\subsection{Submanifolds of M\"obius manifolds}
\label{par:sub-mob}

Let $\phi\from\submfd\to\mfd$ be an immersion of an $m$-manifold $\submfd$
into a \mob/ $n$-manifold $\mfd$, with $n=m+k$ and $m,k>0$.  Thus, for
$n\geq3$, our assumption is no more than that $\mfd$ be a conformal manifold:
it is only for the case of curves immersed in surfaces that we are imposing
extra structure on the ambient space.  In any case, $\mfd$ is equipped with a
normal conformal Cartan geometry $(\CV_\mfd,\Ln,\CD^\mfd)$.

The conformal metric on $\mfd$ restricts to give a conformal metric on
$\submfd$.  At first we obtain a section of $S^2T\dual
\submfd\ltens(L_\mfd\restr\submfd)^2$, but then Proposition~\ref{p:same-L}
identifies $L_\mfd\restr\submfd$ with the density bundle $L_\submfd$ of
$\submfd$.  (To generalize our theory to indefinite signature metrics we need
to assume that the induced conformal metric on $\submfd$ is nondegenerate.)
Henceforth, we write $L$ and $\Ln=L^{-1}$ for both the intrinsic and ambient
line bundles, and often omit restriction maps (pullbacks to $\submfd$). Thus
$(\CV_\mfd,\CD^\mfd)$ will be viewed as a vector bundle with connection over
$\submfd$. Along $\submfd$, we have an orthogonal decomposition
\begin{equation*}\notag
T\mfd=T\submfd\dsum N\submfd,
\end{equation*}
where $N\submfd$ is the normal bundle to $\submfd$ in $\mfd$.  Further, the
operators $\pi$ and $p$, defined by the soldering isomorphism of $\mfd$,
restrict to operators along $\submfd$, so that in particular, we have
$\pi\colon \Ln^\perp\to T\mfd\ltens\Ln(\cong T\dual\mfd\ltens L)$ over
$\submfd$. We thus obtain two subbundles $\pi^{-1}(N\submfd\ltens\Ln)$ and
$\pi^{-1}(T\submfd\ltens\Ln)$ of $\CV_\mfd$ with sum $\Ln^\perp$ and
intersection $\Ln$. We denote the latter subbundle by $\Lnps$: it is the
smallest subbundle of $\Ln^\perp$ containing $\Ln$ and its tangential first
derivatives, \ie, $\CD^{\mfd}_{X}\sigma\in\Lnps$ for all sections $\sigma$ of
$\Ln$ and $X\in T\submfd$.

Our approach to \mob/ submanifold geometry is to forget about $\mfd$ and study
only $(\CV_\mfd,\Ln,\CD^\mfd)$ along $\submfd$. This will be particularly
effective in the case $\mfd=S^n$, since then $\CV_\mfd$ (restricted to
$\submfd$) is $\submfd\times\R^{n+1,1}$ with $\CD^\mfd$ just the flat
derivative $\d$, and the immersion of $\submfd$ into $S^n=\PL$ can be
recovered from the subbundle $\Ln \subset \submfd\times\R^{n+1,1}$ (whose
fibres are null lines in $\R^{n+1,1}$, hence points in $\PL$).

\subsection{M\"obius reductions}
\label{par:mob-red}

We regard $(\CV_\mfd,\Ln,\CD^\mfd)$ as a generalized conformal Cartan geometry
on $\submfd$.  It is not actually a conformal Cartan geometry, since
$\CV_\mfd$ has rank $n+2>m+2$ so a reduction of structure group is required.
This prompts the following definition.

\begin{defn}
A signature $(m+1,1)$ subbundle $\CV\subset\CV_\mfd$ (over $\submfd$) is
called a \emphdef{\mob/ reduction} iff $\CV$ contains the rank $m+1$ bundle
$\Lnps$ (so that $\Lnps=\Ln^\perp\intersect\CV$).
\end{defn}
Otherwise said, $\CV^\perp$ is a complement to $\Ln$ in
$\pi^{-1}(N\submfd\ltens\Ln)\subset\Ln^\perp$, or equivalently, an orthogonal
complement to $\Lnps$ in $\Ln^\perp$. From this, it is clear that \mob/
reductions exist and that $\CV^\perp$ is a rank $k$ bundle on which the metric
is definite. Also $\pi\from\Ln^\perp\to T\mfd\ltens\Ln$ restricts to a metric
isomorphism from $\CV^\perp$ onto the weightless normal bundle
$N\submfd\ltens\Ln$.

A \mob/ reduction $\CV$ determines and is determined by a bundle map
$\mni_\CV\colon N\submfd\ltens\Ln\to\pi^{-1}(N\submfd\ltens\Ln)
\subset\Ln^\perp$ with $\pi\circ \mni_\CV=\iden_{N\submfd\ltens\Ln}$ via
$\CV^\perp=\image\mni_\CV$.  It follows that \mob/ reductions form an affine
space modelled on the sections of $\Hom(N\submfd\ltens\Ln,\Ln)\cong
N\dual\submfd$.  We write $\CV\mapsto \CV+\con$ for this affine structure,
where $\con\in\Cinf(\submfd,N\dual\submfd)$: explicitly, $\mni_{\CV+\con}=
\mni_\CV-\con\tens\iden_\Ln$. (The sign here is for consistency with our
sign convention for Weyl structures.)

\subsection{Sphere congruences}
\label{par:mob-sphere-cong}

When $\mfd=S^n$, \mob/ reductions admit a classical interpretation.  In this
case, $\CV_\mfd\to \submfd$ is just the trivial $\R^{n+1,1}$ bundle,
$\CD^\mfd$ is flat differentiation $\d$, and the line bundle $\Ln\to\submfd$
is the pullback of the tautological bundle by the immersion
$\phi\colon\submfd\to S^n$. Positive sections of $\Ln$ are precisely the
\emphdef{lifts} $f\from\submfd\to\LC^+$ of $\phi$ with $q\circ f=\phi$.

Recall from \S\ref{par:sphere-conf-cart} that $\ell$-spheres in the conformal
sphere $S^n=\PL$ correspond bijectively to signature $(\ell+1,1)$ subspaces of
$\R^{n+1,1}$ (or equivalently to the orthogonal spacelike $(n-\ell)$-planes)
via $W\mapsto\Proj(W\cap\LC)$. Hence the space of $\ell$-spheres in $S^n$ is
identified with the grassmannian of signature $(\ell+1,1)$ subspaces of
$\R^{n+1,1}$, and a signature $(\ell+1,1)$ subbundle of
$\submfd\times\R^{n+1,1}$ is the same as a map of $\submfd$ into this space.

\begin{defn} A \emphdef{$\ell$-sphere congruence} parameterized by an
$m$-manifold $\submfd$ is a signature $(\ell+1,1)$ subbundle $\CV$ of
$\submfd\times\R^{n+1,1}$. It is said to be \emphdef{enveloped} by an
immersion $\phi$, \ie, a null line subbundle $\Ln$ satisfying the Cartan
condition, iff $\Lnps\subset\CV$, \ie, $\CV$ contains the image of every lift
$f$ and its first derivatives $\d_X f$, $X\in T\submfd$.
\end{defn}
Geometrically, the enveloping condition means that each $\ell$-sphere
$\Proj(\CV_x\cap\LC)$ has first order contact with $\phi(\submfd)$ at
$\phi(x)$.  When $\ell=m$, this is exactly the condition that $\CV$ be a \mob/
reduction for $\Ln$. To summarize:
\begin{quote}
\emph{\mob/ reductions are the same as $m$-sphere congruences enveloped by
$\submfd$}.
\end{quote}

\section{Geometry of M\"obius reductions}
\label{sec:geom-mob-red}

\subsection{Splitting of the ambient connection}
\label{par:mr-split}

Let $\CV$ a \mob/ reduction of $\CV_\mfd$ along $\submfd$. The orthogonal
decomposition $\CV_\mfd=\CV\dsum\CV^\perp$ induces a decomposition of
$\CD^\mfd$: we have metric connections $\CD^\CV$, $\nabla^\CV$ on $\CV$,
$\CV^\perp$, respectively, given by composing the restriction of $\CD^\mfd$
with orthoprojection.  The remaining, off-diagonal, part of $\CD^\mfd$ is the
following.

\begin{defn}
The \emphdef{\cso/} $\CS^\CV\from T\submfd\to
\Hom(\CV,\CV^\perp)\dsum\Hom(\CV^\perp,\CV)$ of a \mob/ reduction $\CV$ is
defined by
\begin{equation*}\notag
\CS^{\CV}_X\wnv=(\CD^\mfd_X\wnv)\tangent,\qquad\qquad
\CS^{\CV}_Xv=(\CD^\mfd_Xv)\normal,
\end{equation*}
for $v\in \CV$, $\wnv\in \CV^\perp$, where ${}\tangent,{}\normal$ denote the
orthoprojections onto $\CV,\CV^\perp$ respectively.
\end{defn}
Note that $\CS^\CV$ is zero order in $v,\wnv$, and $\lip{\CS^\CV_X\wnv,v}
=-\lip{\wnv,\CS^\CV_Xv}$ since $\CD^\mfd$ is a metric connection. In
Lie-theoretic terms, we have a fibrewise symmetric decomposition
\begin{gather}
\label{eq:decomp-g}
\g:=\so(\CV_\mfd)=\h_\CV\dsum\m_\CV,\\
\tag*{where}
\h_\CV=\so(\CV)\dsum\so(\CV^\perp),\qquad
\m_\CV=\g\intersect
\bigl(\Hom(\CV,\CV^\perp)\dsum
\Hom(\CV^\perp,\CV)\bigr).
\end{gather}
Accordingly $\CD^\mfd$ decomposes as
\begin{equation}\label{eq:tn}
\CD^\mfd=\CDVs + \CS^\CV,
\end{equation}
where $\CDVs=\CD^\CV+\nabla^\CV$ is the direct sum connection and $\CS^\CV$ is
a $\m_\CV$-valued $1$-form.

The fundamental tools in submanifold geometry are the equations which give the
curvature of $\CD^\mfd$ in terms of this decomposition:
\begin{subequations}\label{eq:gcr}
\begin{align}\label{eq:gr}
(R^{\CD^\mfd})_\h &= \RDVs+\half\liebrack{\CS^\CV\wedge\CS^\CV};\\
(R^{\CD^\mfd})_\m &= \dDVs\CS^\CV.
\label{eq:c}
\end{align}
\end{subequations}
We shall see that the $\so(\CV)$ and $\so(\CV^\perp)$ components
of~\eqref{eq:gr} are generalized versions of the Gau\ss\ and Ricci equations
respectively, while~\eqref{eq:c} is a Codazzi equation. As they stand, these
equations are conceptually simple and general: $\CV_\mfd$ could be any vector
bundle with $\g$ connection and $\CV$ any nondegenerate subbundle.  However,
they depend on the choice of \mob/ reduction and we need to understand this
dependence. First, though, there is a lot of information hidden in the
interaction with the tautological line bundle $\Ln$, which we now begin to
unravel.

Let $\CV$ a \mob/ reduction of $(\CV_\mfd,\Ln,\CD^\mfd)$ along $\submfd$.
Then $\Lnps\subset \CV$, so that $\CD^\mfd_X\sigma\in\CV$ and thus
\begin{equation*}
\CD^\CV_X\sigma=\CD^\mfd_X\sigma,\qquad\CS^\CV_X\sigma=0,
\end{equation*}
for all sections $\sigma$ of $\Ln$ and $X\in T\submfd$. This has several
consequences for the components $\CD^\CV$, $\nabla^\CV$ and $\CS^\CV$ of
$\CD^\mfd$. First, $\CD^\CV$ is a conformal Cartan connection for the
conformal structure on $\submfd$ inherited from $\mfd$: the soldering form of
$\CD^\CV$ is the pullback of that of $\CD^\mfd$.  In particular,
$\pi\restr\Lnps$ is the operator $\pi_\submfd\colon\Lnps\to T\submfd\ltens\Ln$
defined by the soldering form of $\CD^\CV$, and we shall drop the subscript
henceforth.

Moreover, $\CD^\CV_X\sigma=\CD^\mfd_X\sigma$ implies that
$R^{\CD^\CV}_{X,Y}\sigma$ is the projection onto $\CV$ of
$R^{\smash{\CD^\mfd}}_{X,Y}\sigma$ which vanishes since $\CD^\mfd$ is strongly
torsion-free. We summarize this as follows.
\begin{prop}\label{p:mr-conf}
$(\CV,\Ln,\CD^\CV)$ is a conformal Cartan geometry on $\submfd$ whose induced
conformal metric is the restriction of the conformal metric on $\mfd$.
\end{prop}

Thus in addition to a conformal metric, $\submfd$ acquires a \mob/ structure
$\Ms^\CV:=\Ms^{\CD^\CV}$.  When $\dimn\submfd\leq 2$, this packs a punch as we
shall see.

Now contemplate $\CS^\CV$: since $\CS^\CV\restr\Ln$ vanishes, $\CS^\CV$ takes
values in $\stab(\Ln)\cap\m_\CV$ so that $\pi\CS^\CV$ is a tensorial object: a
$1$-form with values in $\co(T\mfd)$.
\begin{defn}
The \emphdef{second fundamental form} and \emphdef{shape operator} of a \mob/
reduction $\CV$ are the bundle maps $\II^\CV\from T\submfd\tens T\submfd\to
N\submfd$ and $\Sh^\CV\colon T\submfd\tens N\submfd\to T\submfd$ defined in
terms of the \cso/ $\CS^\CV$ by the following formulae:
\begin{equation*}
\II^\CV_XY\,\sigma=-\pi\bigl(\CS^\CV_X(\CD^\mfd_Y\sigma)\bigr),\qquad
\cip{\Sh^\CV_X \nmv, Y}\sigma^2 =\lip{\CS^\CV_X\mni_\CV(\nmv\vtens\sigma),
\CD^\mfd_Y\sigma};
\end{equation*}
thus $\Sh^\CV\in \Omega^1(\submfd,\Hom(N\submfd,T\submfd))$ is the transpose
of $\II^\CV\in \Omega^1(\submfd,\Hom(T\submfd,N\submfd))$ and
$\pi\CS^\CV=\II^\CV-\Sh^\CV$.  The \emphdef{mean curvature covector}
$H^\CV\in\Cinf(\submfd,N\dual\submfd)$ of $\CV$ is given by
$H^\CV=\tfrac1m\trace\Sh^\CV=\tfrac1m\cip{\trace_\conf\II^\CV,\cdot}$.
\end{defn}

\begin{prop} Let $\CV$ be a \mob/ reduction of $\CV_\mfd$ on $\submfd$. Then
$\II^\CV$ is symmetric, \ie, is a section of $S^2T\dual\submfd\tens N\submfd$.
\textup(Equivalently $\Sh^\CV$ is a section of
$N\dual\submfd\tens\Sym_\conf(T\submfd)$.\textup)
\end{prop}
\begin{proof} We have
\begin{equation*}\notag
(\II^\CV_Y X -\II^\CV_X Y )\sigma=
\pi(\CS^\CV_X(\CD^\mfd_Y\sigma)-\CS^\CV_Y(\CD^\mfd_X\sigma))=
\pi((R^{\CD^\mfd}_{X,Y}\sigma)\normal)=0
\end{equation*}
since $\CD^\mfd$ is torsion-free.
\end{proof}

Finally $\nabla^\CV$ induces a metric connection $\nabla=\pi\circ
\nabla^\CV\circ\mni_\CV$ on $N\submfd\ltens\Ln$.

We now see how all these data depend on the choice of reduction $\CV$
which will require a closer look at the space of \mob/ reductions.

\subsection{Gauge theory of M\"obius reductions}
\label{par:gauge-mob}

The affine structure on the space of \mob/ reductions may be usefully
understood in terms of gauge transformations of $\CV_\mfd$: we have $\CV+\con
= \exp(-\con)\CV$. Indeed, for $v\in \CV_\mfd$, $\nmv\tens\sigma\in
N\submfd\ltens\Ln$, we compute
\begin{align*}
\exp(-\con)\mni_\CV(\nmv\vtens\sigma) &=
\mni_\CV(\nmv\vtens\sigma)-\con\act \mni_\CV(\nmv\vtens\sigma)=
\mni_\CV(\nmv\vtens\sigma)-\ip{\con,\nmv}\sigma =
\mni_{\CV+\con}(\nmv\vtens\sigma)
\end{align*}
so that $\exp(-\con)\circ\mni_\CV=\mni_{\CV+\con}$. Thus if
$v\in\CV+\con$, $\exp(\con)v\in\CV$.

\begin{prop}\label{p:ch-red} Let $\CV$ be a \mob/ reduction of $\CV_\mfd$
on $\submfd$.  Then, for $\con\in\Cinf(\submfd,N\dual\submfd)$,
\begin{equation}\label{eq:ch-comp}
\begin{split}
\exp(\con)\act\CD^{\CV+\con}&=
\CD^\CV + \abrack{\con,\pi\CS^\CV} - \half\abrack{\con,\pi\CDVs\con},\\
\exp(\con)\act\CS^{\CV+\con} &= \CS^\CV-\CDVs\con,\\
\exp(\con)\act\nabla^{\CV+\con}&=\nabla^\CV,
\end{split}
\end{equation}
and $\liebrack{\CS^\CV,\con} - \half\liebrack{\con,\CDVs\con}$, $\CDVs\con$
take values in $\stab(\Ln)^\perp\cap \h_\CV$, $\stab(\Ln)\cap \m_\CV$
respectively.
\end{prop}
\begin{proof}  From~\eqref{eq:log-deriv}, we have $\exp(\con)\act\CD^\mfd=
\CD^\mfd-\CD^\mfd\con-\half\abrack{\con,\pi\CD^\mfd\con}$.  Now
$\CD^\mfd\con=\CDVs\con+\liebrack{\CS^\CV,\con}
=\CDVs\con-\abrack{\con,\pi\CS^\CV}$, since $\CS^\CV$ takes values in
$\stab(\Ln)$. In particular, $\pi\liebrack{\CS^\CV,\con}=0$.

Since $\CV+\con=\exp(-\con)\CV$, $\exp(\nu)\act\CDVs$ is the reduction of
$\exp(\nu)\act\CD^\mfd$ to $\h_\CV$, whence
\begin{align*}
\exp(\con)\act\CD^{\nabla,\CV+\con}&=
\bigl(\exp(\con)\act\CD^{\mfd}\bigr)_{\h_\CV}=
\CDVs+\abrack{\con,\pi\CS^\CV}-\half\abrack{\con,\pi\CDVs\con},\\
\exp(\con)\act\CS^{\CV+\con}&= \bigl(\exp(\con)\act\CD^{\mfd}\bigr)_{\m_\CV}=
\CS^\CV-\CDVs\con.
\end{align*}
\eqref{eq:ch-comp} follows after restricting the first of these formulae to
$\CV$ and $\CV^\perp$.
\end{proof}

With this in hand, we can determine the dependence of our data on the \mob/
reduction.

\begin{prop}\label{p:ch-mobII} The conformal structure $\conf$ on $\submfd$
and connection $\nabla$ on $N\submfd\ltens\Ln$ are independent of $\CV$,
whereas
\begin{gather}\label{eq:ch-mobV}
\Ms^{\CV+\con} = \Ms^\CV + \con(\II^\CV) - \half\ip{\con,\con}\conf,\\
\label{eq:ch-II^\CV}
\II^{\CV+\con} =\II^\CV - \conf\vtens\con,\qquad  
\Sh^{\CV+\con}=\Sh^\CV-\con\tens\iden.
\end{gather}
\end{prop}
\begin{proof}
The independence of the conformal structure is immediate from
Proposition~\ref{p:mr-conf}.  Moreover, since $\pi\circ\exp(\con)=\pi$, we
have $\pi\circ\nabla^{\CV+\con}\circ\mni_{\CV+\con}=
\pi\circ\exp(\con)\act\nabla^{\CV+\con}\circ\mni_{\CV}$ which is
$\pi\circ\nabla^\CV\circ\mni_{\CV}$ by Proposition~\ref{p:ch-red}.  Thus
$\nabla$ is independent of $\CV$.

Similarly, $\pi\CS^{\CV+\con}= \pi(\exp(\con)\act\CS^{\CV+\con})
=\pi\CS^\CV-\pi\CDVs\con$ by Proposition~\ref{p:ch-red}.  Moreover
\begin{equation}\label{eq:pi-D-nu}
\pi\CDVs\con=-\abrack{\iden,\con}=\iden\tens\con-\conf\vtens\con
\end{equation}
and so \eqref{eq:ch-II^\CV} follows since $\II^\CV-\Sh^\CV=\pi\CS^\CV$.

Finally, for~\eqref{eq:ch-mobV}, note that $\CD^{\CV+\con}$ is
gauge-equivalent to $\CD^\CV+\abrack{\con,\pi\CS^\CV} -
\half\abrack{\con,\pi\CDVs\con}$ by Proposition~\ref{p:ch-red}, so that these
connections induce the same \mob/ structures, namely
$\Ms^\CV+\abrack{\con,\pi\CS^\CV} - \half\abrack{\con,\pi\CDhV\con}$.
Moreover,
\begin{align*}
\abrack{\con,\pi\CS^\CV}
&=\abrack{\con,\II^\CV-\Sh^\CV} = \con(\II^\CV),\\
\abrack{\con,\pi\CDVs\con}&=
-\abrack{\con,\abrack{\iden,\con}}=\cip{\con,\con}\conf,
\end{align*}
(using \eqref{eq:pi-D-nu}), whence \eqref{eq:ch-mobV} follows.
\end{proof}

\subsection{Normalization and the primitive data}
\label{par:norm-prim}

The submanifold geometry developed thus far is not a pure theory of conformal
submanifolds as it depends on the choice of a \mob/ reduction. Furthermore,
there is no reason for the induced conformal Cartan geometry to be normal.
The key to rectifying this is equation~\eqref{eq:ch-II^V}, which shows that
the tracefree part $\II^0$ of the second fundamental form $\II^\CV$ (or
equivalently the tracefree part $\Sh^0$ of the shape operator $\Sh^\CV$) is
independent of the choice of $\CV$, whereas the mean curvature depends on
$\CV$ via $H^{\CV+\con} = H^\CV-\con$. Hence there is a unique \mob/ reduction
with tracefree second fundamental form---\ie, vanishing mean curvature---given
by $\CV+H^\CV=\exp(-H^\CV)\CV$ for any \mob/ reduction $\CV$.

\begin{defn} The unique \mob/ reduction $\CV\low_\submfd$ with
$H^{\CV_\submfd}=0$ will be called the \emphdef{canonical} \mob/ reduction.
We denote the \cso/ of $\CV\low_\submfd$ by $\CS$ and identify
$\CV\low_\submfd$ with $N\submfd\ltens\Ln$ so that
$\nabla^{\CV_\submfd}=\nabla$. We also write $\h$ for $\h_{\CV_\submfd}$ and
$\m$ for $\m_{\CV_\submfd}$ so that $\g=\h\dsum \m$.
\end{defn}

The canonical \mob/ reduction is picked out by a homological condition: for
any reduction $\CV$, $\CS^\CV$ is a $1$-chain in the complex defining the Lie
algebra homology $H_\bullet(T\dual \submfd,\m_\CV)$ (where the standard action
on $\CV$ and the trivial action on $\CV^\perp$ are used to define the
$\so(\CV)$-module structure on $\m_\CV$). Applying the boundary operator to
this $1$-chain gives
\begin{equation*}
\LH\CS^\CV=\tsum_i\liebrack{\eps_i,\CS^\CV_{e_i}}=
\tsum_i\abrack{\eps_i,\pi\CS^\CV_{e_i}}=
\tsum_i\abrack{\eps_i,\II^\CV_{e_i}-\Sh^\CV_{e_i}}=-\trace\Sh^\CV.
\end{equation*}
Thus $\CV\low_\submfd$ is the unique reduction whose \cso/ is a $1$-cycle.
When $m=1$, $\CS$ is a section of $T\dual\submfd\tens N\dual\submfd$, which we
shall call the {\it conformal acceleration} $A$ of the curve.

We now have a canonical choice $(\CV\low_\submfd,\Ln,\CD^{\CV_\submfd})$ of
conformal Cartan geometry on $\submfd$. The remaining problem is that the
associated \mob/ structure need not be a conformal \mob/ structure, that is,
$\CD^{\CV_\submfd}$ need not be normal (when $m\geq2$).  But now
Proposition~\ref{p:make-normal} comes to the rescue and assures us of a unique
section $\QC$ of $S^2 T\dual \submfd$ such that
\begin{bulletlist}
\item $\CD^{\CV_\submfd}-\CQ$ is a normal Cartan connection;
\item $\QC_0=0$ when $m=2$ and $\QC=0$ when $m=1$ (\ie, $\QC$ is in the image
of $\LH$).
\end{bulletlist}
We denote the normal Cartan connection so defined by $\CD^\submfd=
\CD^{\CV_\submfd}-\CQ$ and the induced conformal \mob/ structure by
$\Ms^\submfd=\Ms^{\CV_\submfd}-\QC$.

\begin{thm}\label{t:data}
Let $\phi\from\submfd\to\mfd$ be an immersion of an $m$-manifold $\submfd$
into a \mob/ $n$-manifold $\mfd$. Then $\phi$ equips $\submfd$ canonically
with the following primitive data\textup:
\begin{bulletlist}
\item a conformal \mob/ structure $(\conf,\Ms^\submfd)$\textup;
\item a rank $n-m$ euclidean vector bundle $N\submfd\ltens\Ln$ with a metric
connection $\nabla$\textup;
\item a section $\II^0$ of $S^2_0 T\dual\submfd\tens N\submfd$ for
$m\geq2$\textup;
\item a section $\CA$ of $T\dual\submfd\tens N\dual\submfd$ for $m=1$.
\end{bulletlist}
In addition there is an affine bijection $\CV\mapsto -H^\CV$ from \mob/
reductions to sections of $N\dual\submfd$ such that the inverse image of zero
is the canonical \mob/ reduction.
\end{thm}
It is not immediately apparent that these data suffice to recover the
connection $\CD^\mfd= \CD^\submfd+\nabla+\CQ+\CS$ on $\CV_\mfd
=\CV\low_\submfd\dsum\CV_\submfd^\perp$, but we shall see that they do when
$\CD^{\mfd}$ is flat.

The primitive data can be computed from an arbitrary \mob/ reduction $\CV$ by
writing $\CV=\CV\low_\submfd-H^\CV$ and applying Proposition~\ref{p:ch-mobII}
to see that $\CV$ gives the same conformal metric and normal connection as
$\CV\low_\submfd$, whereas the remaining data are related by
\begin{align}
\II^0 &= \II^{\CV} - \conf\vtens H^\CV, \qquad
\Sh^0 = \Sh^{\CV} - \iden\vtens H^\CV & &(m\geq2)\\
\CA &= \CS^{\CV} - \CDVs H^\CV & & (m=1)\\
\label{eq:MS-formula}
\Ms^\submfd&=\Ms^{\CV_\submfd}-\QC = \Ms^\CV - \QC^\CV,&&\text{where}\\
\QC^\CV :\!&= \QC - H^\CV(\II^\CV) + \half|H^\CV|^2\conf
=\QC - H^\CV(\II^0) - \half|H^\CV|^2\conf.&&
\label{eq:Q-rel}
\end{align}
We can also decompose $\CD^\mfd$ on $\CV\dsum\CV^\perp$ as $\CD^\mfd=
\CD^{\submfd,\CV}+\nabla^\CV+\CQ^\CV+\CS^\CV$, where
$\CD^{\submfd,\CV}=\CD^\CV-\CQ^\CV$ is the normal Cartan connection
$\exp(H^\CV)\cdot\CD^\submfd$ on $\CV$.

To summarize, given a submanifold $\phi\colon\submfd\to\mfd$ and a \mob/
reduction $\CV$, we have a decomposition of the ambient Cartan connection
\begin{equation}\label{eq:gen-decomp}
\CD^\mfd=\CDhV+ \CQ^\CV + \CS^\CV
\end{equation}
where $\CDhV:=\CD^{\submfd,\CV}+\nabla^\CV$ for a normal conformal Cartan
connection $\CD^{\submfd,\CV}$. By the formulae above, this decomposition
induces the primitive data of Theorem~\ref{t:data}. The curvature of
$\CD^\mfd$ with respect to this decomposition is given by:
\begin{subequations}\label{eq:ngcr}
\begin{align}\label{eq:ngr}
(R^{\CD^\mfd})_{\h_\CV}
&= \RDhV+\dDhV \CQ^\CV + \half\liebrack{\CS^\CV\wedge\CS^\CV};\\
(R^{\CD^\mfd})_{\m_\CV}
&= \dDhV\CS^\CV+\liebrack{\CQ^\CV\wedge\CS^\CV}.
\label{eq:nc}
\end{align}
\end{subequations}
Here $\dDhV$ is the exterior derivative coupled to $\CDhV$, and $\RDhV$ is the
curvature of $\CDhV$.

\section{Lie algebra homology and the conformal Bonnet theorem}
\label{sec:conf-bonnet}

We now specialize to the case that the ambient manifold $\mfd$ is the
conformal sphere $S^n$ so that $\CV_\mfd$ is the trivial $\R^{n+1,1}$-bundle
over $\submfd$ and $\CD^\mfd=\d$. In this setting, we will see that the
immersion $\Ln$ of $\submfd$ is completely determined by the primitive data of
Theorem~\ref{t:data}.  It follows that the equations of~\eqref{eq:ngcr}, with
$\CV=\CV_\submfd$, are equations on these data, the \emphdef{\GCR/ equations}.
Furthermore, we shall see that one can recover (locally) an immersion of
$\submfd$ into $S^n$ from an arbitrary solution of these equations.  In all
this, we make substantial use of the machinery of Lie algebra homology and the
associated theory of \BGG/ operators.  The Reader with no taste for such
abstraction is directed to section \ref{sec:weyl-conf-bonnet} where Weyl
structures are used to provide a more explicit formulation.

Before all this, however, we show that the canonical \mob/ reduction is an
object well-known to conformal submanifold geometers.

\subsection{The central sphere congruence}
\label{par:csc}

For $\mfd=S^n$, we identified \mob/ reductions with the classical notion of
enveloped sphere congruences in \S\ref{par:mob-sphere-cong}. The canonical
\mob/ reduction is also known classically: it is \emphdef{central sphere
congruence} or \emphdef{conformal Gau\ss\ map}~\cite{Bry:dtw,Tho:ukg} which
may be defined in terms of a lift $f\colon\submfd\to \LC^+$ by
\begin{equation*}
\CV_\mathrm{cent}=\vspan{f,\d f,\trace D\d f}.
\end{equation*}
Here $D$ is the Levi-Civita connection of $\lip{\d f,\d f}$, which is also
used to compute the trace. Equivalently, if $e_i\vtens f$ is an orthonormal
basis of $T\submfd\ltens\Ln$, $\CV_\mathrm{cent}=\vspan{f,\d f,
\sum_i\d_{e_i}\d_{e_i}f}$.  Now we compute the $\CV_\submfd^\perp$ component
of $\sum_i\d_{e_i}\d_{e_i}f$ to be $\sum_i\CS_{e_i}(\d_{e_i}f)$, $\pi$ of
which is $-\sum_i\II^0_{e_i}e_i\,f=0$.  Thus
$\CV_\mathrm{cent}=\CV\low_\submfd$.

More geometrically, $\submfd_{(x)}:=\Proj((\CV_\submfd)_x\cap\LC)$ is the
\emphdef{mean curvature sphere}, \ie, the unique sphere tangent to $\submfd$
at $x$, which has the same mean curvature covector as $\submfd$ with respect
to some (hence any) enveloped sphere congruences $\CV$ of $\submfd$ and
$\CV_{(x)}$ of $\submfd_{(x)}$ which agree at $x$. Indeed, the central sphere
congruence of $\submfd_{(x)}$ is clearly constant, equal to $(\CV_\submfd)_x$
(the \mob/ differential is zero, which is certainly a cycle). Hence $\submfd$
and $\submfd_{(x)}$ have central sphere congruences which agree at $x$.  As we
shall in~\S\ref{par:mr-ambient-weyl-riem}, it follows that $\submfd$ and
$\submfd_{(x)}$ have the same mean curvature covector at $x$ with respect to
some (hence any) ambient metric.

\subsection{Lifting the primitive data}
\label{par:recover}

Consider an immersion $\phi\from\submfd\to S^n$ with central sphere congruence
(\ie, canonical \mob/ reduction) $\CV\low_\submfd$. Then, using the
decomposition $\g=\h\dsum\m$ induced by
$\submfd\times\R^{n+1,1}=\CV\low_\submfd\dsum\CV_\submfd^\perp$,
\eqref{eq:gen-decomp} specializes to give
\begin{equation}\label{eq:flat-conn}
\d = \CDh + \CQ + \CS,
\end{equation}
where $\CDh$ is the $\h$-connection $\CD^\submfd+\nabla$,
$\CS\in\Omega^1(\submfd,\m)$, and $\CQ$ is the unique section of
$S^2T\dual\submfd\cap\image\LH$ for which $\CD^\submfd$ is normal.

Write $\dDh$ for $\d^{\CDh}$ and $\RDh$ for the curvature of $\CDh$.  Then
\eqref{eq:ngcr} specializes as follows:
\begin{subequations}
\label{eq:fullGCR}
\begin{align}\label{eq:fullGR}
0 &= \RDh + \dDh\CQ+\half\liebrack{\CS\wedge\CS};\\
\label{eq:fullC}
0 &= \dDh\CS+\liebrack{\CQ\wedge\CS}.
\end{align}
\end{subequations}

Recall that our normalization conditions on the \mob/ reduction
$\CV\low_\submfd$ and the \mob/ structure $\Ms^\submfd$ amount to the
requirement that $\RDh$ and $\CS$ are cycles in $Z_2(\submfd,\h)$ and
$Z_1(\submfd,\m)$ respectively (the former condition is equivalent to
$R^{\CD^\submfd}$ being a cycle).  Moreover, we see that the homology class of
$\CS$ yields primitive data:
\begin{align*}
\hc{\CS}&=
\begin{cases}
\II^0-\Sh^0 &\text{for $m\geq2$;}\\
A &\text{for $m=1$.}
\end{cases}\\
H_1(T\dual\submfd,\m)&=
\begin{cases}
\bigl(S^2_0 T\dual\submfd\tens N\submfd\dsum
\Sym_c(T\submfd)\tens N\dual\submfd\bigr)
\cap\Omega^1(\submfd,\co(T\mfd)) &\text{for $m\geq2$;}\\
T\dual\submfd\tens N\dual\submfd\subset \Omega^1(\submfd,T\dual\mfd)
&\text{for $m=1$.}
\end{cases}
\end{align*}
This is analogous to the primitive data given by the homology class
$\hc{R^{\CD^\submfd}}$, which is the Weyl curvature of $(\submfd,\conf)$ when
$m\geq4$ and the Cotton--York curvature of $(\submfd,\conf)$ or
$(\submfd,\conf,\Ms^\submfd)$ when $m=3$ or $m=2$ respectively.

We now show that both $\CS$ and $\CQ$ can be recovered from these homology
classes, from which it will follow that the \GCR/ equations \eqref{eq:fullGCR}
are equations on these data. In~\S\ref{par:abs-hom-conf-bonnet}, we shall
realize these equations in terms of \BGG/ (BGG) differential operators on Lie
algebra homology bundles. A key ingredient is the differential
lift~\cite{CaDi:di,CSS:bgg} associated to $\CD^\submfd$ that we introduced
in~\S\ref{par:diff-lift}.  We let $\Quabla_\h$ denote $\Quabla_{\CD^\submfd}$
coupled to the connection $\nabla$ on $\CV_\submfd^\perp$, to obtain a
generalized differential lift $j^\h$ which can be applied to homology classes,
such as $[\CS]$, with a normal component.

\begin{prop}\label{p:recover}
Let $(\CV\low_\submfd,\Ln,\CD^\submfd)$ be a normal conformal Cartan geometry
on $\submfd$, let $(\CV_\submfd^\perp,\nabla)$ be a euclidean vector bundle
with metric connection, let $\CS$ be a $\m$-valued $1$-form in $\kernel\LH$,
and let $\CQ$ be a $\h$-valued $1$-form in $\image\LH$.

Suppose that these data satisfy the full \GCR/ equations~\eqref{eq:fullGCR}.
Then $\CQ=-\half\Quabla_{\smash{\h}}^{-1}\LH\abrack{\hc{\CS}\wedge\hc{\CS}}$
and $\CS=\jDh\hc{\CS}$.  Furthermore,
$\CS\in\Omega^1(\submfd,\m\cap\stab(\Ln))$ and $\pi\CS$ is symmetric and
tracefree, whereas $\CQ\in\Omega^1(\submfd,T\dual\submfd)$ and is symmetric.
\end{prop}
\begin{proof}
Since $\LH\RDh=0$, \eqref{eq:fullGR} implies that $\Quabla_\h\CQ=\LH
\dDh\CQ=-\half\LH\liebrack{\CS\wedge\CS}$, so that
$\CQ=-\half\Quabla_{\smash{\h}}^{-1}\LH\liebrack{\CS\wedge\CS}$ since it is in
the image of $\LH$.  Now since $\LH\CS =0$, $\CS$ takes values in
$\stab(\Ln)\cap\m$ and $\pi\CS$ is tracefree. Thus $\CQ$ takes values in
$\stab(\Ln)^\perp\cap\h=T\dual\submfd$, and hence $\liebrack{\CQ\wedge\CS}$ is
a $2$-form with values in $\stab(\Ln)^\perp\cap\m
\subset\Hom(\CV\low_\submfd/\Ln^\perp, \CV_\submfd^\perp)\dsum
\Hom(\CV_\submfd^\perp,\Ln)$. Since $\LH$ vanishes on $\Ln$,
$\LH\liebrack{\CQ\wedge\CS}=0$, so that~\eqref{eq:fullC} implies $\LH
\dDh\CS=0$ and hence $\CS=\jDh\hc{\CS}$. We also deduce
$\LH\abrack{\iden\wedge\pi\CS}=0$, and then $\abrack{\iden\wedge\pi\CS}=0$,
\ie, $\pi\CS$ is symmetric.

It remains to establish the symmetry of $\CQ$. Since
$\Quabla_\h\CQ=\LH\abrack{\iden\wedge\CQ}$, $\Quabla_\h$ is an algebraic
isomorphism on $\image\LH$, so it suffices to show that
$\LH\liebrack{\CS\wedge\CS}$ is symmetric. This equals $\LH R$ where
$R=\abrack{\pi\CS\wedge\pi\CS}$, which is in
$\Omega^2(\submfd,\so(T\submfd))$. By the Jacobi identity and the symmetry of
$\pi\CS$, $R$ satisfies the algebraic Bianchi identity $\abrack{\iden\wedge
R}=0$. It follows that its associated Ricci contraction $\LH R$ is symmetric.
\end{proof}

We end by noting that since $\QC$ and $\CS$ are determined by the primitive
data, then, given $H^\CV$, so are $\QC^\CV=\QC - H^\CV(\II^0) -
\half|H^\CV|^2\conf$ and $\CS^\CV=\exp(H^\CV)\act\CS+\CDhV H^\CV$.  This
generality turns out to be convenient in applications, so we digress from the
homological theory to give a general formulation of the conformal Bonnet
theorem.

\subsection{The conformal Bonnet theorem for enveloped sphere congruences}
\label{par:conf-bonnet-sphere}

Let $\Ln$ be an immersion of $\submfd$, and $\CV$ an enveloped sphere
congruence with mean curvature covector $H^\CV$. Then~\eqref{eq:gen-decomp}
gives a decomposition of the flat connection on
$\submfd\times\R^{n+1,1}=\CV\dsum\CV^\perp$ as
\begin{equation}\label{eq:flat-conn-V}
\CDhV+ \CQ^\CV + \CS^\CV,
\end{equation}
where $\CQ^\CV\in \Omega^1(\submfd,T\dual\submfd)$ is the unique section of
$S^2T\dual\submfd$ such that $\CDhV$ restricts to the normal conformal Cartan
connection on $\CV$ inducing the conformal \mob/ structure $\Ms^\submfd$ of
the central sphere congruence $\CV_\submfd=\CV+H^\CV$.  It follows
from~\S\ref{par:recover} that the primitive data of Theorem~\ref{t:data},
together with $H^\CV$, determine $\QC^\CV$ and $\CS^\CV$. These data then
satisfy \GCR/ equations, which are \emph{equivalent} to the flatness of the
connection~\eqref{eq:flat-conn-V}:
\begin{subequations}\label{eq:flatGCR}
\begin{align}\label{eq:flatgr}
0&= \RDhV+\dDhV \CQ^\CV + \half\liebrack{\CS^\CV\wedge\CS^\CV};\\
0&= \dDhV\CS^\CV+\liebrack{\CQ^\CV\wedge\CS^\CV}.
\label{eq:flatc}
\end{align}
\end{subequations}
A converse is therefore available: the conformal immersion and enveloped
sphere congruence can be recovered from the data they induce (\ie, the
primitive data of Theorem~\ref{t:data} and the mean curvature covector
$H^\CV$). Our goal now is to establish such a converse.

To this end, let $\submfd$ be an $m$-manifold with conformal \mob/ structure
$(\conf,\Ms^\submfd)$. Suppose that $N\submfd\to\submfd$ is a rank $n-m$
vector bundle with a metric and metric connection $\nabla$ on $N\submfd \ltens
L^{-1}$, and that $\II^0$ is a section of $S^2_0 T\dual\submfd\tens N\submfd$
for $m\geq2$ (with transpose $\Sh^0$) and $\CA$ is a section of
$T\dual\submfd\tens N\dual\submfd$ for $m=1$. Finally let $H^\CV$ be a section
of $N\dual\submfd$.

Denote by $(\CV, \Ln,\CD^{\submfd,\CV})$ the induced normal conformal Cartan
geometry (which identifies $\Ln$ with $L^{-1}$) and set $\CV^\perp= N\submfd
\ltens \Ln$, $\nabla^\CV=\nabla$ and $\CDhV=\CD^{\submfd,\CV}+\nabla^\CV$.
Let $\CS^\CV=\jDhV(\II^0-\Sh^0)+\CDhV H^V$ for $m\geq 2$ and $\CS^\CV=A+\CDhV
H^V$ for $m=1$, where $\jDhV$ is the differential lift operator defined by
$\CD^{\submfd,\CV}$ and $\nabla^\CV$. Last of all, let $\QC^\CV$ be given
by~\eqref{eq:Q-rel}, where $\QC=-\half\Quabla_{\smash{\h}}^{-1}\LH
\abrack{(\II^0-\Sh^0)\wedge (\II^0-\Sh^0)}$.

Then the connection~\eqref{eq:flat-conn-V} given by these data is flat
iff~\eqref{eq:flatGCR} hold.

\begin{thm}\label{th:conf-bonnet-sphere}
Let $\submfd$ be an $m$-manifold with conformal \mob/ structure
$(\conf,\Ms^\submfd)$, let $(N\submfd \ltens L^{-1},\nabla)$ be a rank $n-m$
euclidean vector bundle on $\submfd$ with metric connection, and let
$\CS^\CV$, $\CQ^\CV$ be determined as above by these data together with a
section $H^\CV$ of $N\dual\submfd$ and a section $\II^0$ of $S^2_0
T\dual\submfd\tens N\submfd$ \textup(for $m\geq2$\textup) or $\CA$ of
$T\dual\submfd\tens N\dual\submfd$ \textup(for $m=1$\textup).

Then there is locally an immersion of $\submfd$ into $S^n$ with enveloped
sphere congruence $\CV$ inducing these data if and only if the equations
of~\eqref{eq:flatGCR} hold.  If so, then the immersion and enveloped sphere
congruence are unique up to \mob/ transformations of $S^n$.
\end{thm}
(Obviously, when $m=1$, the equations of this theorem are vacuous, so the
theorem then shows that a projective curve can be embedded into $S^n$ with
arbitrary conformal acceleration $\CA$, and that $H^\CV$ parameterizes the
possible enveloped circle congruences.)
\begin{proof} It remains to recover the immersion, assuming~\eqref{eq:flatGCR}
holds.  The bundle $\CV\dsum\CV^\perp$ has a metric of signature
$(n+1,1)$---given by the sum of the metrics on the summands---for which $\Ln$
is null, and the connection $\CDhV+\CQ^\CV+\CS^\CV$ on $\CV\dsum\CV^\perp$ is
a metric connection, which is flat by~\eqref{eq:flatGCR}.  Thus, locally, we
have a parallel metric isomorphism $\CV\dsum\CV^\perp\cong
\submfd\times\R^{n+1,1}$, unique up to constant gauge transformations, and the
inclusion $\Ln\to\CV\low_\submfd\dsum\CV_\submfd^\perp$ induces a map
$\phi\colon\submfd\to\PL$; $x\mapsto\Ln_x\subset\R^{n+1,1}$. The Cartan
condition on $\CD^\submfd$ ensures that $\phi$ is an immersion and
$\CS^\CV\restr\Ln=0$ ensures that (the image in $\submfd\times \R^{n+1,1}$ of)
$\CV$ is an enveloped sphere congruence, with mean curvature $\frac1m\trace
\II^\CV=H^\CV$.
\end{proof}

The above theorem is very general, but depends on a choice of enveloped sphere
congruence and so it is not truly conformal: as we shall see in
\S\ref{par:tangent-cong}, it includes, as a special case, the Bonnet theorem
for submanifolds of spaceforms. However, if we set $H^\CV=0$ so that
$\CV=\CV\low_\submfd$ is the central sphere congruence,
then~\eqref{eq:flatGCR} reduces to~\eqref{eq:fullGCR} and this specialization
of Theorem~\ref{th:conf-bonnet-sphere} is purely conformal. In the next
paragraph, though, we see that such a formulation of the conformal Bonnet
Theorem is not optimal.

\subsection{The abstract homological conformal Bonnet theorem}
\label{par:abs-hom-conf-bonnet}

The normalization conditions used to fix the \mob/ reduction and \mob/
structure on $\submfd$ have the homological interpretation that $\CS$ and
$R^{\CD^\submfd}$ are cycles, in $Z_1(\submfd,\m)$ and
$Z_2(\submfd,\so(\CV\low_\submfd))$ respectively. We now show that the \GCR/
equations also have a homological description, in terms of
Bernstein--Gelfand--Gelfand (BGG) operators associated to the representations
$\h=\so(\CV\low_\submfd)\dsum\so(\CV_\submfd^\perp)$ and $\m$ of
$\SO(\CV\low_\submfd)$, where $\CV_\submfd^\perp$ carries the trivial
representation.

BGG sequences of invariant linear differential operators were first introduced
on curved geometries by Eastwood--Rice~\cite{EaRi:cio}, for $4$-dimensional
conformal geometry, and by Baston~\cite{Bas:ahs} and \v Cap--Slovak--Sou\v
cek~\cite{CSS:bgg} in more general contexts. These constructions were
simplified and extended to invariant multilinear differential operators
in~\cite{CaDi:di}. We need this generality here, and in particular, we need
the operators $\Quabla_\h=\LH\circ\dDh+\dDh\circ\LH$ and
$\Pi=\iden-\Quabla_{\smash{\h}}^{-1}\circ\LH\circ
\dDh-\dDh\circ\Quabla_{\smash{\h}}^{-1}\circ\LH$ of \S\ref{par:diff-lift}, but
here coupled to the normal connection $\nabla$ (if required).

Using this, we can define linear and bilinear BGG operators between sections
of Lie algebra homology bundles, closely related to the differential lift
$j^\h$. We shall only need them on $H_1(T\dual\submfd,\m)$, where, for
sections $\hc{\alpha}$ and $\hc{\beta}$, we have
\begin{equation*}\notag
\d_{BGG}\hc{\alpha} = \hc{\Pi \dDh \Pi\alpha},\qquad\qquad\qquad
\hc{\alpha}\cupp\hc{\beta}= \hc{\Pi(\liebrack{\Pi\alpha\wedge\Pi\beta})},
\end{equation*}
which are sections of $H_2(T\dual\submfd,\m)$ and $H_2(T\dual\submfd,\h)$
respectively.

We now have all the ingredients to establish our homological formulation of
the \GCR/ equations. These equations are simpler than the full equations since
they only involve the homological objects $\hc{\CS}$ and
$\hc{\RDh}=\hc{R^{\CD^\submfd}}+R^\nabla$. The Bernstein--Gelfand--Gelfand
operators provides a conceptually elegant formulation of these equations, and
the associated homological machinery yields a remarkably efficient proof.

\begin{thm}\label{th:hom-gcr}
Let $(\CV\low_\submfd,\Ln,\CD^\submfd)$ be a normal conformal Cartan geometry
on $\submfd$, let $(\CV_\submfd^\perp,\nabla)$ be a euclidean vector bundle
with metric connection, let $\CS$ be a $\m$-valued $1$-form in $\kernel\LH$,
and let $\QC$ be a $T\dual\submfd$-valued $1$-form in $\image\LH$.

Then the data $(\CDh,\CS,\QC)$ satisfy the full \GCR/
equations~\eqref{eq:fullGCR} if and only if $\CS=\jDh\hc{\CS}$,
$\CQ=-\half\Quabla_{\smash{\h}}^{-1}\LH\liebrack{\CS\wedge\CS}$, and the
following homological \GCR/ equations hold\textup:
\begin{subequations}\label{eq:homGCR}
\begin{align}\label{eq:homGR}
0&=\hc{\RDh}+\half\hc{\CS}\cupp\hc{\CS}\\
0&=\d_{BGG}\hc{\CS}.\label{eq:homC}
\end{align}
\end{subequations}
\end{thm}
\begin{proof}
We first observe that $\LH \RDh=0$ and $\dDh \RDh=0$, so that
$\RDh=\jDh\hc{\RDh}$. By Proposition~\ref{p:recover}, we also know
that~\eqref{eq:fullGCR} imply that $\CS=\jDh\hc{\CS}$ and
$\CQ=-\half\Quabla_\h^{-1}\LH\liebrack{\CS\wedge\CS}$.

It remains to show that under these conditions, the
equations~\eqref{eq:fullGCR} are equivalent to~\eqref{eq:homGCR}. By the
uniqueness of the canonical differential representative, the latter are
equivalent to
\begin{align*}
0&= \RDh+\half\Pi^2\liebrack{\CS\wedge\CS}= \RDh+\half\Pi
\liebrack{\CS\wedge\CS}\\
&=\RDh+ \half\liebrack{\CS\wedge\CS} +\dDh\CQ
-\half\Quabla_\h^{-1}\LH \dDh\liebrack{\CS\wedge\CS}\\
0&=\Pi \dDh\CS = \dDh\CS -\Quabla_\h^{-1}\LH\liebrack{\RDh\wedge\CS}.
\end{align*}
Here we have expanded the definition of $\Pi$: for the first equation we have
used the fact that $\Pi \dDh\CQ=\dDh\CQ-\dDh \Quabla_\h^{-1}\LH
\dDh\CQ-\Quabla_\h^{-1}\LH (\dDh)^2\CQ= -\Quabla_\h^{-1}\LH
\liebrack{\RDh\wedge\CQ}=0$, since $\liebrack{\RDh\wedge\CQ}$ has values in
$T\dual\submfd$, hence is in the kernel of $\LH$.  Now the second equation
implies that $\dDh\CS$ has values in $\CV_\submfd^\perp\tens\Ln$ and hence
$\Quabla_\h^{-1}\LH \dDh\liebrack{\CS\wedge\CS}=0$.  By substitution of the
first equation into the second, the equations~\eqref{eq:homGCR} are therefore
equivalent to
\begin{align*}
0&= \RDh+\dDh\CQ+\half\liebrack{\CS\wedge\CS}\\
0&=\dDh\CS + \Quabla_\h^{-1}\LH\liebrack{\dDh\CQ \wedge\CS},
\end{align*}
since $\liebrack{\CS\wedge\liebrack{\CS\wedge\CS}}=0$ by the Jacobi identity.
Now note that $\liebrack{\CQ\wedge\dDh\CS}=0$ since the Lie bracket of
$T\dual\submfd$ with $\CV_\submfd^\perp\tens\Ln$ is trivial. Hence the last
term is equal to $\Quabla_\h^{-1}\LH \dDh\liebrack{\CQ \wedge\CS}$.

We noted in the proof of Proposition~\ref{p:recover} that $\liebrack{\CQ
\wedge\CS}$ is in the kernel of $\LH$. The equivalence of the full and
homological \GCR/ equations now rests on the fact that it is actually in the
image of $\LH$. For $m\geq3$ this follows easily from the fact that
$\liebrack{\CQ \wedge\CS}$ is a $2$-form with values in
$\CV_\submfd^\perp\ltens\Ln$, since
$\LH\from\Wedge^3T\dual\submfd\tens\Ln^\perp \to\Wedge^2T\dual\submfd\tens\Ln$
is then surjective. The result is vacuous for $m=1$, so it remains to prove
that for $m=2$, $\liebrack{\CQ \wedge\CS}=0$. For this, it suffices to show
that $(\CQ\wedge\CS)\low_{X,Y}\con=0$ for any normal $\con$ and vector fields
$X,Y$.  Now $\CQ$ vanishes on $\Ln$, so this reduces immediately to
$\QC_X(\Sh^0_Y\con)-\QC_Y(\Sh^0_X\con)$, which vanishes because (for $m=2$)
$\QC\in\image\LH$ is tracelike and $\Sh^0$ is symmetric.
\end{proof}

\begin{rem} The subtle point in the above proof was the fact that
$\liebrack{\CQ \wedge\CS}\in\image\LH$. If this had not worked out, the
homology class would be $\ip{\hc{\CS},\hc{\CS},\hc{\CS}}$, where
$\ip{\cdot,\cdot,\cdot}$ is one of the trilinear differential operators
defined in~\cite{CaDi:di}. It is thus a homological fluke that the \GCR/
equations are quadratic in $\II^0$ and $\Sh^0$. In more general circumstances,
one should anticipate the appearance of multilinear differential operators of
higher degree.
\end{rem}

Although elegant, this theorem is quite abstract at the present, and we have
already suggested that the Reader who prefers direct calculations to abstract
machinery should turn to the next section, where we make the above equations
and computations more explicit. For this reason, we postpone the statement of
the homological conformal Bonnet theorem until we have understood the
homological \GCR/ equations more explicitly. However, for the Reader with no
taste for the nitty-gritty, we note that the homological conformal Bonnet
theorem which we will state as Theorem~\ref{th:hom-conf-bonnet} is essentially
the obvious corollary of Theorem~\ref{th:conf-bonnet-sphere}, with
$\CV=\CV\low_\submfd$, and Theorem~\ref{th:hom-gcr} above.

\section{Weyl structures and the conformal Bonnet theorem}
\label{sec:weyl-conf-bonnet}

We have seen that Weyl derivatives and the associated apparatus of Weyl
structures and connections provide an efficient computational tool in
conformal geometry. The same is true in conformal submanifold geometry.

\subsection{Decomposition with respect to a Weyl structure}
\label{par:decomp-weyl}

Let $\submfd$ be an immersed submanifold of $\mfd$. A M\"obius reduction $\CV$
provides a decomposition $\so(\CV_\mfd)=\h_\CV\dsum\m_\CV$ over
$\submfd$~\eqref{eq:decomp-g}. On the other hand, a choice of Weyl structure
on $\mfd$ may be restricted (pulled back) to $\submfd$ to give a decomposition
$\CV_\mfd\cong\Ln\dsum U_\mfd\dsum\hat\Ln$, and hence $\so(\CV_\mfd)\cong
T\dual\mfd\dsum\co(T\mfd)\dsum T\mfd$ (\S\ref{par:weyl-str}). (Here, as
before, we omit pullbacks to $\submfd$.)

If these data are compatible (\ie, $\hat\Ln\subset\CV$) then $\CV\cong\Ln\dsum
U\dsum\hat\Ln$ with $U_\mfd=U\dsum \CV^\perp$, so that
\begin{equation}\label{eq:master-decomp}
\so(\CV_\mfd)\;=\;
\begin{matrix}\h_\CV\\ \dsum\\ \m_\CV
\end{matrix}\;\cong\;
\begin{matrix}T\dual\submfd\dsum
\h_0\dsum T\submfd\\ \dsum\\ N\dual\submfd\dsum\m_0\dsum N\submfd
\end{matrix}
\;\cong\;T\dual\mfd\dsum\co(T\mfd)\dsum T\mfd,
\end{equation}
where
$\h_0=\so(T\submfd)\dsum\so(N\submfd)\dsum\R\,\iden_{T\mfd}\subset\co(TM)$
and $\m_0=\co(TM)\cap
\bigl(\Hom(T\submfd,N\submfd)\dsum\Hom(N\submfd,T\submfd)\bigr)$. In
the top line, $\h_\CV=\so(\CV)\dsum\so(\CV^\perp)$, and the
isomorphism identifies $\h_0$ with $\co(T\submfd)\dsum\so(\CV^\perp)$,
and $\so(\CV)$ with $T\dual\submfd\dsum\co(T\submfd)\dsum
T\submfd$. In the bottom line, $N\dual\submfd\;(\cong
V^\perp\ltens\Ln)$, $\m_0$ and $N\submfd$ are identified with the
intersections of $\m_\CV$ with
$\Hom(\Ln,\CV^\perp)\dsum\Hom(\CV^\perp,\hat\Ln)$,
$\Hom(U,\CV^\perp)\dsum\Hom(\CV^\perp,U)$ and
$\Hom(\hat\Ln,\CV^\perp)\dsum\Hom(\CV^\perp,\Ln)$ respectively.

For the M\"obius reduction $\CV$,~\eqref{eq:gen-decomp} gives
\begin{equation*}
\CD^\mfd= \CDhV+\CQ^\CV+\CS^\CV,
\end{equation*}
where $\CDhV=\CD^{\submfd,\CV}+\nabla^\CV$ and $\CD^{\submfd,\CV}$ is a normal
conformal Cartan connection on $\CV$. We now apply~\eqref{eq:master-decomp} in
two ways.

First, $\hat\Ln\subset \CV_\mfd$ is the restriction (pullback) to $\submfd$ of
a Weyl structure on $\mfd$, and so $\CD^\mfd=\nr^{D^\mfd}+D^\mfd-\iden$, where
the normalized Ricci tensor $\nr^{D^\mfd}$ and the identity map $\iden$ are
viewed, by restriction, as $1$-forms on $\submfd$ with values in $T\dual\mfd$
and $T\mfd$ respectively.

Second, $\hat\Ln\subset \CV$ is a Weyl structure on $\submfd$, so
$\CD^{\submfd,\CV}=\nr^{D,\submfd}+D^{\CV}-\iden$, where $D$ is the
corresponding Weyl derivative, $\nr^{D,\submfd}$ is the normalized Ricci
curvature of the induced conformal M\"obius structure, $D^\CV$ is the induced
conformal connection, and $\iden=\iden_{T\submfd}$.

Putting these together using~\eqref{eq:master-decomp}, we have
\begin{equation}\label{eq:full-decomp}
\left\lbrace\quad\begin{matrix}
\CD^\mfd&=&\nr^{D^\mfd}&+&D^\mfd&-&\iden&=&(\CDhV+\CQ^\CV)+\CS^\CV\\
\CDhV+\CQ^\CV&=& \nr^{D,\submfd}+\CQ^\CV&+&  D^{\CV}+\nabla^\CV &-&\iden\\
\CS^\CV&=& \CA^{D,\CV} &+& \II^{\CV}-\Sh^{\CV},
\end{matrix}\right.
\end{equation}
where we note that the $\m_0$-component of $\CS^\CV$ is given by the second
fundamental form and shape operator of $\CV$. Comparing coefficients, we
deduce that $\nr^{D^\mfd}=(\nr^{D,\submfd}+\CQ^\CV)+A^{D,\CV}$ and $D^\mfd=
(D^{\CV}+\nabla^\CV) + (\II^{\CV}-\Sh^{\CV})$. The second equation is the
decomposition of $D^\mfd$ (along $\submfd$) into a direct sum connection on
$U\dsum \CV^\perp\cong (T\submfd\dsum N\submfd)\ltens \Ln$ and $1$-forms
valued in $\Hom(T\submfd,N\submfd)$ and $\Hom(N\submfd,T\submfd)$.

\subsection{The Gau\ss--Codazzi--Ricci equations}
\label{par:weyl-gcr}
 
The decomposition \eqref{eq:master-decomp}--\eqref{eq:full-decomp} may be used
to expand \eqref{eq:ngr} as
\begin{subequations}\label{eq:gcr-system}
\begin{align}
\label{eq:cotangent-part-h}
(R^{\CD^\mfd})_{T\dual\submfd}
&=\d^D\nr^{D,\submfd}+\d^D\QC^\CV+\abrack{(\II^\CV-\Sh^\CV)\wedge\CA^{D,\CV}}\\
\label{eq:co-part-h}
(R^{\CD^\mfd})_{\h_0}&=R^D+R^{\nabla}
-\abrack{\iden\wedge(\nr^{D,\submfd}+\QC^\CV)}
+\half\abrack{(\II^\CV-\Sh^\CV)\wedge(\II^\CV-\Sh^\CV)}\\
\label{eq:tangent-part-h}
0&=\d^D\iden,
\end{align}
while~\eqref{eq:nc} reads
\begin{align}
\label{eq:cotangent-part-m}
(R^{\CD^\mfd})_{N\dual\submfd}
&=\d^{\nabla,D}\!\CA^{D,\CV}
+\abrack{(\II^\CV-\Sh^\CV)\wedge(\nr^{D,\submfd}+\QC^\CV)}\\
\label{eq:co-part-m}
(R^{\CD^\mfd})_{\m_0}
&=\d^{\nabla, D}\II^\CV-\d^{\nabla, D}\Sh^\CV-\abrack{\iden\wedge\CA^{D,\CV}}\\
\label{eq:tangent-part-m}
0&=\abrack{\iden\wedge(\II^\CV-\Sh^\CV)},
\end{align}
\end{subequations}
where $\abrack{\cdot}$ denotes the algebraic bracket on $T\dual\mfd\dsum
\co(T\mfd)\dsum T\mfd$ and $D^\CV$ is denoted by $D$ for simplicity.

Since $\CD^\mfd$ is the pullback of a normal Cartan connection,
$R^{\CD^\mfd}=C^{\mfd,D^\mfd}+W^\mfd$ with $C^{\mfd,D^\mfd}$ the pullback of
the Cotton--York curvature of $\mfd$ (defined using any extension of $D^\mfd$
to $\mfd$) and $W^\mfd$ the Weyl curvature of $\mfd$. On the other hand,
$R^D-\abrack{\iden\wedge\nr^{D,\submfd}}=W^\submfd$, the Weyl curvature of
$\Ms^\submfd$, and $\d^D\nr^{D,\submfd}=C^{\submfd,D}:=C^{\Ms^\submfd,D}$, the
Cotton--York curvature of $\Ms^\submfd$ with respect to $D$.

Using this, and expanding the algebraic brackets, we can
rewrite~\eqref{eq:gcr-system} as follows. First, in~\eqref{eq:cotangent-part-h}
and \eqref{eq:cotangent-part-m}, the algebraic brackets are simply
contractions. Second, in~\eqref{eq:co-part-h}, the algebraic bracket is just
the commutator, while in~\eqref{eq:co-part-m} we have
\begin{equation*}
\abrack{\iden\wedge\CA^{D,\CV}}\act Z =
(\CA^{D,\CV})^\sharp\wedge\conf(Z,\cdot),\qquad
\abrack{\iden\wedge\CA^{D,\CV}}\act \nmv =
\iden\wedge\CA^{D,\CV}(\nmv)
\end{equation*}
for tangent vectors $Z$ and normal vectors $\nmv$. Finally,
\eqref{eq:tangent-part-h} and \eqref{eq:tangent-part-m} are identities since
$D$ is torsion-free and $\II^\CV$ is symmetric.

To summarize, after restricting $\h_0$ components to $T\submfd$ and
$N\submfd$, the system~\eqref{eq:gcr-system}, which is equivalent
to~\eqref{eq:ngcr}, yields the following \GCR/ equations:

\subsubsection*{The Gau\ss\ equations.}
\begin{align}
\label{eq:g1}
(W^\mfd\restr{T\submfd})\tangent&=W^\submfd-\abrack{\iden\wedge\QC^\CV}
- \Sh^\CV\wedge \II^\CV;\\
(C^{\mfd,D^\mfd})\tangent&=C^{\submfd,D}+\d^D\QC^\CV+\CA^{D,\CV}\wedge\II^\CV;
\label{eq:g2}
\end{align}

\subsubsection*{The Codazzi equations.}
\begin{align}
\label{eq:c1}
\begin{split}
(W^\mfd\restr{T\submfd})\normal&=\d^{\nabla,D}\II^\CV
-(\CA^{D,\CV})^\sharp\wedge\conf,\\
(W^\mfd\restr{N\submfd})\tangent&=-\d^{\nabla,D}\Sh^\CV
-\iden\wedge\CA^{D,\CV};
\end{split}\\
(C^{\mfd,D^\mfd})\normal&=\d^{\nabla,D}\CA^{D,\CV}-(\nr^{D,\submfd}+\QC^\CV)
\wedge\Sh^\CV;
\label{eq:c2}
\end{align}

\subsubsection*{The Ricci equation.}
\begin{equation}\label{eq:r}
(W^\mfd\restr{N\submfd})\normal= R^\nabla-\II^\CV\wedge \Sh^\CV.
\end{equation}
Note that the second equation of~\eqref{eq:c1} is minus the transpose of the
first.

\subsection{M\"obius reductions, ambient Weyl structures and riemannian
metrics} \label{par:mr-ambient-weyl-riem}

We now discuss the relationship between \mob/ reductions and Weyl structures
on $\mfd$, and hence describe the conformal geometry of submanifolds in the
more familiar context of submanifolds in riemannian geometry.

For this, recall from section~\ref{sec:weyl-geom} that a Weyl derivative on
$\mfd$ is a covariant derivative on $L$ or $\Ln=L^{-1}$; using $\CD^\mfd$ and
the induced conformal metric, this datum is equivalently a torsion-free
conformal connection on $T\mfd$ or a null line subbundle $\Lnc\subset\CV_\mfd$
complementary to $\Ln^\perp$. This last definition makes sense along a
submanifold $\submfd$, \ie, we define an \emphdef{ambient Weyl structure along
$\submfd$} to be such a complement $\Lnc$ in $\CV_\mfd\restr\submfd$.  Thus
any ambient Weyl structure is the restriction to $\submfd$ of some Weyl
structure on $\mfd$.

An ambient Weyl structure $\Lnc$ along $\submfd$ determines a \mob/ reduction
$\CV=\Lnps\dsum\Lnc$, together with a Weyl structure $\Lnc \subset\CV$ for the
conformal Cartan geometry $(\CV,\Ln,\CD^\CV)$, or equivalently, a Weyl
derivative $D$ on $\submfd$. The map sending $\Lnc$ to the pair $(\CV,D)$ is
an affine bijection. Indeed, since ambient Weyl structures are equivalently
complements $(\Ln\dsum\Lnc)^\perp\cap\CV$ to $\Ln$ in $\Ln^\perp$ (along
$\submfd$), they form an affine space modelled on
$\Cinf(\submfd,\Hom(\Ln^\perp/\Ln,\Ln))\cong\Cinf(\submfd,T\dual\mfd)$:
$\Lnc+\gamma=\exp(-\gamma)\act\Lnc$ and one readily checks that this affine
structure is induced by the natural affine structure $(\CV,D)+\gam=
(\CV+\gam\normal,D+\gam\tangent)$.

\begin{rem}\label{rem:mob-via-weyl}
An ambient Weyl structure can be regarded as an operator
\begin{equation*}
V_\mfd\to T\dual\mfd\ltens L=T\dual\submfd\ltens L\dsum N\dual \submfd\ltens L
\end{equation*}
vanishing on $\Ln$, whose restriction to $\Ln^\perp/\Ln$ is the soldering
isomorphism. Identifying $\CV_\mfd\modulo\Ln$ with the restriction to
$\submfd$ of $J^1L$ on $\mfd$, such operators arise as restrictions to
$\submfd$ of the jet bundle map induced by a Weyl derivative on $\mfd$. The
$T\dual\submfd\ltens L$ component of this operator gives the jet bundle map
$J^1L\to T\dual\submfd\ltens L$ induced by the induced Weyl derivative on
$\submfd$, while the $N\dual\submfd\ltens L$ component is the operator
$\mco_\CV\colon \CV_\mfd\to N\dual\submfd \ltens L$ with
$\mco_\CV(v)=\pi(v\normal)$ induced by the \mob/ reduction $\CV$. This
provides another way to see the affine bijection between ambient Weyl
structures and pairs $(D,\CV)$---furthermore, since $D^\CV$ is the pullback to
$\submfd$ of a Weyl connection on $\mfd$, it transforms as a Weyl connection
should: $(D+\gam\tangent)^{\CV+\gam\normal}=D^\CV+\abrack{\cdot,\gam}$.

As a special case of this, the canonical \mob/ reduction may be described by a
natural differential operator, just like the canonical \mob/ structure
$(\Mh^\conf,\Mq^\conf)$, cf.~Remark~\ref{rm:can-mos}.  Indeed, a positive
section $\ell$ of $L_\mfd$ on a neighbourhood of $\submfd$ in $\mfd$ provides
a metric with respect to which we can compute the mean curvature covector
$H^\ell$ of $\submfd$ in $\mfd$. Now if $D$ is the Weyl derivative and
Levi-Civita connection of the metric $\ell^2\conf$, then
$\mco_{\CV_\submfd}=\mco_D+H^D p$, where $\mco_D$ defines the \mob/ reduction
associated to $D$.  Since $(\mco_D\modulo\Ln)(j^1\ell)
=(D\ell)\restrnormal\submfd=0$ and $H^D=H^\ell$, we have
$(\mco_{\CV_\submfd}\modulo\Ln)(j^1\ell)=H^\ell\ell$.
\end{rem}

If $D^\mfd$ is a Weyl connection on (a neighbourhood of $\submfd$ in) $\mfd$,
it induces an ambient Weyl structure along $\submfd$, hence a \mob/ reduction
$\CV$ and a Weyl structure on $\submfd$, and any \mob/ reduction and Weyl
structure arise in this way.  Further, the second fundamental form is given by
the familiar expression $\II^\CV_X Y= (D^\mfd_XY)\normal\in N\submfd$.

In particular, let $g$ be a compatible riemannian metric on (a neighbourhood
of $\submfd$) in $\mfd$. Then, using the Levi-Civita connection of $g$, we
obtain:
\begin{bulletlist}
\item a \mob/ reduction $\CV_g$ along $\submfd$ with $H^{\CV_g}=H^g$, \ie,
$\II^{\CV_g}=\II^g$, where $\II^g$ and $H^g$ are the usual riemannian second
fundamental form and mean curvature covector respectively;
\item a Weyl structure $\Lnc_g\subset \CV_g$ on $\submfd$ corresponding to
the Levi-Civita connection $D$ of the induced metric on $\submfd$.
\end{bulletlist}
The other quantities associated with this reduction and Weyl structure are:
\begin{equation}\label{eq:metric-quant}
\nr^{D,\CV_g} = (\nr^g)\tangent, \qquad \CA^{D,\CV_g}=(\nr^g)\normal,
\end{equation}
\ie, the components of the ambient normalized Ricci curvature $\nr^g$ after
the latter is pulled back to give a $T\dual\mfd$-valued $1$-form on $\submfd$.
(See~\S\ref{par:decomp-weyl}.)

According to \S\ref{par:norm-prim}, we can then compare $\CV_g$ with the
canonical \mob/ reduction $\CV\low_\submfd$ by writing
$\CV\low_\submfd=\CV_g+H^g$, and hence relate the riemannian and conformal
quantities. We have
\begin{equation} \label{eq:mean-curv}
\II^0=\II^g-(H^g)^\sharp\tens g,
\qquad \Sh^0 = \Sh^g -\iden\tens H^g
\end{equation}
and the pullback $D^g$ of the Levi-Civita connection of $\mfd$ to $\submfd$ is
related to the intrinsic Levi-Civita connection $D$ by
\begin{equation} \label{eq:levciv}
D^g=D+(\II^g-\Sh^g)+\nabla^D=D^{\CV_\submfd}+\abrack{\cdot,H^g},
\end{equation}
where $\nabla^D$ is the induced connection on $N\submfd$ and
$D^{\CV_\submfd}=D+(\II^0-\Sh^0)+\nabla^D$ is the connection on
$T\mfd\restr\submfd$ induced by $D$ and the canonical \mob/ reduction
$\CV\low_\submfd$. We also have
\begin{align}\label{eq:Rnr}
\nr^{D,\submfd}&=(\nr^g)\tangent-\QC^g\\
\CA^D:=\CA^{D,\CV_\submfd}&=(\nr^g)\normal-\nabla^D H^g,
\label{eq:Rca}
\end{align}
where $\QC^g=\QC- H^g(\II^0)-\half\cip{H^g,H^g}\conf$. Substituting
into~\eqref{eq:g1}--\eqref{eq:r}, with $\CV=\CV\low_\submfd$, the conformally
invariant equations \eqref{eq:g1}, \eqref{eq:c1} and~\eqref{eq:r} can be
viewed as trace-free parts of the riemannian \GCR/ equations.

When $g$ is an Einstein metric, $(\nr^g)\tangent$ is a multiple of $\conf$ and
$(\nr^g)\normal=0$. As we have seen in \S\ref{par:spaceform}, such metrics
arise when $\mfd=S^n$, and we shall discuss this further in
\S\ref{par:tangent-cong} and \S\ref{par:sym-break-sub}.

We emphasise that \mob/ reductions (or Weyl structures) do \emph{not} all
arise from from ambient metrics.

\subsection{Reduction to homological data}
\label{par:homological-reduction}

We next apply $\LH$ to the Gau\ss\ and Codazzi equations to get formulae for
$\CA^{D,\CV}$ and $\QC^\CV$.

\begin{prop}\label{p:QA-formulae}
Let $\phi\from\submfd\to\mfd$ be a conformal immersion. Then
\begin{gather}
\label{eq:CA-via-II}
(m-1)\CA^{D,\CV}_X=-(\divg^{\nabla,D}\II^\CV)(X)+m\nabla^D_X H^\CV
+(\tsum_i W^\mfd_{e_i,X}\eps_i)\normal\\
\label{eq:QV-via-II}
\begin{split}
(m-2)\QC^\CV_0(X,Y)+2(m-1)(&\tfrac1m\trace_\conf\QC^\CV)\cip{X,Y}\\
&=\cip{\II^\CV_X,\II^\CV_Y}-m H^\CV(\II^\CV_XY)
-\tsum_i\eps_i(W^\mfd_{e_i,X}Y).
\end{split}
\end{gather}
\end{prop}
\begin{proof} Using an orthonormal frame $e_i$ with dual frame $\eps_i$,
we compute that
\begin{equation*}\notag
\LH(\iden\wedge\CA^{D,\CV}(\wnv))_X=\tsum_i
\eps_i\act(\CA^{D,\CV}_X(\wnv)e_i-\CA^{D,\CV}_{e_i}(\wnv)X)
=(m-1)\CA^{D,\CV}_X(\wnv),
\end{equation*}
while, since $D$ is torsion-free, we have
\begin{align*}
\LH(\d^{\nabla, D}\Sh^\CV)_X\wnv&=
\tsum_i\cip{e_i,(\nabla{\tens}D^\CV\,\Sh^\CV)_{e_i,X}\wnv-(\nabla{\tens}D^\CV\,
\Sh^\CV)_{X,e_i}\wnv}\\
&=\tsum_i\cip{-(\nabla{\tens}D^\CV\,\II^\CV)_{e_i,X}e_i
+(\nabla{\tens}D^\CV\, \II^\CV)_{X,e_i}e_i,\wnv}\\
&=-\cip{(\divg^{\nabla,D}\II^\CV)(X),\wnv}+m\ip{\nabla^D_X H^\CV,\wnv}.
\end{align*}
Applying $\LH$ to~\eqref{eq:c1} therefore yields \eqref{eq:CA-via-II}.

For the other equation, note that
\begin{equation*}\notag
\LH\abrack{\iden\wedge\QC^\CV}
=(m-2)\QC^\CV_0+2(m-1)(\tfrac1m\trace_\conf\QC^\CV)\conf,
\end{equation*}
where we use~\eqref{eq:rc-rm}, cf.~\eqref{eq:W-change}.  Next, using the
symmetry properties of $\II^\CV$, we have
\begin{equation*}\notag
\LH(\Sh^\CV\wedge\II^\CV)_XY=
\tsum_i\bigl(\cip{e_i,\Sh^\CV_{e_i}\II^\CV_XY}
-\cip{e_i,\Sh^\CV_X\II^\CV_{e_i}Y}\bigr)
=m H^\CV(\II^\CV_XY)-\cip{\II^\CV_X,\II^\CV_Y}.
\end{equation*}
Substituting these into the Ricci contraction of~\eqref{eq:g1} yields
\eqref{eq:QV-via-II}.
\end{proof}

In particular, observe that~\eqref{eq:CA-via-II} with $\CV=\CV\low_\submfd$
and $\CA^D:=\CA^{D,\CV_\submfd}$, \ie,
\begin{equation}\label{eq:N-normal}
(m-1)\CA^{D}_X=-(\divg^{\nabla,D}\II^0)(X)
+(\tsum_i W^\mfd_{e_i,X}\eps_i)\normal
\end{equation}
determines $\CA^D$ and hence $\CS$ entirely from the conformal metric, $\II^0$
and the normal connection $\nabla$, provided that $(\tsum_i
W^\mfd_{e_i,X}\eps_i)\normal=0$.

\begin{rem}
With a little more work, one can deduce a manifestly conformally invariant
formula for $\CS$: first note that $\divg^D$ is independent of the Weyl
derivative $D$ when applied to symmetric trace-free $2$-tensors of weight
$-m$.  If $\ell$ is a length scale, $\ell^{1-m}\II^0$ is such a tensor with
values in $N\submfd\ltens\Ln$ so that $\divg^\nabla(\ell^{1-m}\II^0)$ is
invariantly defined.  Now a calculation using \eqref{eq:act-on-V} yields
\begin{equation}\label{eq:12}
(m-1)\CS\low_X
(j^{\CD^\submfd}\ell)=-\ell^m\divg^{\nabla}\big(\ell^{1-m}\II^0\bigr)(X)
+\tsum_i(W^\mfd_{e_i,X}\eps_i)\normal\tens\ell.
\end{equation}
As remarked in the introduction, this corrects an error in
\cite[Proposition \textbf{7}.4.9 (a)]{Sha:dg}.

For submanifolds of $S^n$, equation \eqref{eq:12} amounts to the assertion
that $\CS=\jDh\hc{\CS}$, which we have already seen in
Theorem~\ref{th:hom-gcr}.
\end{rem}

Similarly~\eqref{eq:QV-via-II}, with $\CV=\CV\low_\submfd$, determines $\QC$
completely from $\II^0$ and the conformal metric (recall that $\QC_0$ vanishes
when $m=2$ and $\QC=0$ when $m=1$) provided $\tsum_i\eps_i(W^\mfd_{e_i,X}Y)=0$
(which obviously holds when $\mfd$ is conformally flat, in which
case~\eqref{eq:QV-via-II} is an explicit form of $\Quabla_{\smash{\h}}
Q=-\frac12 \LH[\CS\wedge\CS]$ as in Theorem~\ref{th:hom-gcr}):
\begin{equation}\label{eq:Qformula}
\QC_X(Y) = \begin{cases} 0 &\text{for }m=1;\\
\tfrac14 |\II^0|^2 \cip{X,Y}& \text{for }m=2;\\
\tfrac1{m-2}\bigl(\ip{\II^0_X,\II^0_Y}-\tfrac1{2(m-1)} |\II^0|^2\ip{X,Y}\bigr)
& \text{for }m\geq3.
\end{cases}
\end{equation}

When $m=2$ and $\tsum_i\eps_i(W^\mfd_{e_i,X}Y)=0$, $\QC=
\frac14|\II^0|^2\conf$ is essentially the Willmore
integrand~\cite{Wil:tcr} and from~\eqref{eq:MS-formula}, we obtain
\begin{equation}\label{eq:MS-formula2}
\Ms^\submfd=\Ms^\CV - \QC^\CV\qquad\text{where} \qquad \QC^\CV
=-H^\CV(\II^0)-\half K^\CV\conf,
\end{equation}
and $K^\CV=|H^\CV|^2-\half|\II^0|^2=\ip{\detm\II^\CV}$, the determinant being
evaluated using the metric on $\CV^\perp$.  We refer to $K^\CV$ as the
\emphdef{gaussian curvature} of the surface with respect to $\CV$.  With
respect to a Weyl derivative $D$, we then have the following formula:
\begin{equation}\label{eq:nice}
\nr^{D,\submfd} =\nr^{D,\CV}+H^\CV(\II^0)+\tfrac12 K^\CV\conf.
\end{equation}

\subsection{The concrete homological conformal Bonnet theorem}
\label{par:conc-hom-conf-bonnet}

We now specialize once more to immersions of $\submfd$ into the conformal
$n$-sphere $\mfd=S^n$, when $\CV_\mfd=\submfd\times\R^{n+1,1}$ and $\CD^\mfd$
is flat differentiation $\d$. Our goal is to illuminate the abstract aspect of
the homological conformal Bonnet theorem by describing the homological \GCR/
equations in a more explicit way. We shall see that there really are fewer
homological than full \GCR/ equations: although the Ricci equation is
unaffected, the Gau\ss\ and Codazzi equations simplify. We also indicate how
to show this by direct computation.

To this end, we let $\CV=\CV\low_\submfd$ be the central sphere congruence and
introduce an arbitrary compatible Weyl structure $\CV\low_\submfd=\Ln\dsum \Ln
\ltens T\submfd\dsum\Lnc$ (with induced Weyl derivative $D$)

The \GCR/ equations~\eqref{eq:flatGCR}, with $\CD^\mfd=\d$ and
$\CV=\CV\low_\submfd$ (hence $H^\CV=0$) specialize to the following:
\begin{subequations}
\begin{align}
\label{eq:23}
0&=W^\submfd-\abrack{\iden\wedge\QC}-\Sh^0\wedge\II^0\\
\label{eq:24}
0&=C^{\submfd,D}+\d^D\QC+\CA^D\wedge\II^0\\
\label{eq:25}
0&=\d^{\nabla,D}\II^0-(\CA^{D})^\sharp\wedge\conf,\qquad
0=\d^{\nabla,D}\Sh^0 +\iden\wedge\CA^{D};\\
\label{eq:26}
0&=\d^{\nabla^D}\CA^D-(\nr^{D,\submfd}+\QC)\wedge\Sh^0\\
\label{eq:27}
0&=R^\nabla-\II^0\wedge\Sh^0.
\end{align}
\end{subequations}
Now the homological formulation shows that of the first four equations, only
two are relevant in each dimension.
\begin{bulletlist}
\item If $m=2$, there are no Weyl tensors or Codazzi tensors on
$\submfd$:~\eqref{eq:23} and~\eqref{eq:25} are trivial, while~\eqref{eq:24}
and~\eqref{eq:26} are independent of $D$.

\item If $m=3$, there are still no Weyl tensors on $\submfd$:~\eqref{eq:23} is
trivial, while~\eqref{eq:26} is an automatic consequence of the others, which
are independent of $D$,

\item If $m\geq4$,~\eqref{eq:24} and~\eqref{eq:26} follow from the others,
which are independent of $D$.
\end{bulletlist}
It is possible to check these facts directly, without recourse to the general
theory of~\cite{CaDi:di} used above: it is straightforward that~\eqref{eq:23}
and~\eqref{eq:25} are trivial in low dimensions, but rather more difficult to
check (by differentiating) that~\eqref{eq:24} and~\eqref{eq:26} are
consequences of these in higher dimensions. In fact these relations follow
easily from the theory of the differential lift~\cite{CaDi:di}: one must check
that the right hand sides of the above equations satisfy
$\d^D\eqref{eq:23}=\abrack{\iden\wedge\eqref{eq:24}}$ and
$\d^{D,\nabla}\eqref{eq:25}=\abrack{\iden\wedge\eqref{eq:26}}$.

Instead of doing this, we shall instead concentrate on the fact that the
relevant equations in each dimension are \mob/-invariant, \ie, independent of
$D$.  This is clear for the Ricci equation~\eqref{eq:27} and for~\eqref{eq:23}
(which is only relevant when $m\geq 4$). We now look at the other equations,
and describe explicitly the \mob/-invariant operators involved.

We first note that~\eqref{eq:25} (which is only relevant when $m\geq3$)
may be rewritten
\begin{equation*}\notag
\CCoda^\nabla \Sh^0 =0,
\end{equation*}
where $\CCoda$ is the conformal Codazzi operator
\begin{equation}
\CCoda S = \d^D S +\frac1{m-1}\iden\wedge(\divg^D S)^\flat
\end{equation}
on symmetric tracefree endomorphisms $S$ of weight $-1$. It is easy to check
that $\CCoda$ is a conformally-invariant first order differential operator
(see~\cite{Feg:cio} for the general theory), which is identically zero for
$m=2$. Coupling to the connection $\nabla$ on $L\ltens N\dual\submfd$ gives
$\CCoda^\nabla$.

We next consider~\eqref{eq:26} when $m=2$. Applying the Hodge star operator,
we have
\begin{equation*}\notag
\divg^{D,\nabla}\divg^{D,\nabla}J\II^0+\ip{J\II^0, \nr^{D,\submfd}_0}=0,
\end{equation*}
where we have used the fact that only the symmetric tracefree part
$\nr^{D,\submfd}_0$ of $\nr^{D,\submfd}+\QC$ can contribute to the second
term, and the definition $\cip{(J\II^0)_X,\wnv}=*\cip{\II^0_X,\wnv}$ of the
complex structure on $S^2_0T\dual \submfd\tens N\submfd$. We can rewrite this
as
\begin{equation*}\notag
(\Mh^\nabla)^*(J\II^0)=0,
\end{equation*}
where $(\Mh^\nabla)^*$ is the formal adjoint of the \mob/ structure coupled to
$\nabla$, \ie, $\Mh^\nabla\con=\sym_0(D^\nabla)^2\con+
\nr^{D,\submfd}_0\con$. This is a \mob/-invariant second order differential
operator.

Finally we must consider~\eqref{eq:24} when $m=2$ or $3$. In both cases
$\QC_X(Y)=\cip{\II^0_X,\II^0_Y}-\frac14|\II^0|^2\cip{X,Y}$ (which equals
$\frac14|\II^0|^2\cip{X,Y}$ for $m=2$) and $\CCoda^\nabla\II^0=0$ so that
$\CA^D\wedge\II^0=\ip{\II^0,\d^{\nabla,D}\II^0}$. Hence~\eqref{eq:24} is
equivalent to the equation
\begin{equation*}\notag
C^\submfd+\ip{\QCoda^\nabla(\II^0)}=0,
\end{equation*}
where
\begin{equation}
\QCoda (S) = \d^D(S^2-\tfrac14 |S|^2\iden)+S(\d^D S)
\end{equation}
is a quadratic first order differential operator on symmetric tracefree
endomorphisms $S$ of weight $-1$ (viewed also as $L^{-1}\ltens
T\submfd$-valued $1$-forms), which is applied to $\II^0$ by coupling to the
connection $\nabla$ and contracting by the metric $\ip{\cdot\,,\cdot}$ on the
weightless normal bundle. Thus $\ip{\QCoda^\nabla(\II^0)}=
\d^DQ+\ip{\II^0,\d^{\nabla,D}\II^0}$. Again one can check (by differentiating
with respect to $D$) that $\QCoda$ is conformally-invariant for $m=2,3$.  In
the case $m=3$, the tensor $Q$ was introduced by Cartan~\cite{Car:dhc} in his
study of conformally-flat hypersurfaces in $S^4$. We shall see its use in the
study of conformally-flat submanifolds in Theorem~\ref{th:mob-flat-explicit}.

Apart from $(\Mh^\nabla)^*$, the operators we have introduced are at most
first order, and so depend only on $\conf$, not on the \mob/ structure
$\Ms$. The dependence of $(\Mh^\nabla)^*(J\II^0)$ is straightforward: if $q$
is a quadratic differential then
\begin{equation}\label{eq:mh-dual-q}
((\Mh+q)^\nabla)^*(J\II^0)=(\Mh^\nabla)^*(J\II^0)+\ip{q,J\II^0}.
\end{equation}

We summarize this discussion by stating our homological conformal Bonnet
theorem.

\begin{thm}\label{th:hom-conf-bonnet}
$\submfd$ can be locally immersed in $S^n$ with induced conformal \mob/
structure $(\conf,\Ms^\submfd)$, weightless normal bundle
$(N\submfd\ltens\Ln,\nabla)$ and tracefree second fundamental form $\II^0$
\textup(or conformal acceleration $\CA$\textup) if and only if \textup(with
$\Sh^0$ the transpose of $\II^0$ and $Q$ as in~\eqref{eq:Qformula}\textup)
\begin{align}
\label{eq:hG1}
0&=W^\submfd-\abrack{\iden\wedge\QC}-\Sh^0\wedge\II^0&&m\geq4\\
\label{eq:hG2}
0&=C^\submfd+\ip{\QCoda^\nabla(\II^0)}&&m=2,3\\
\label{eq:hC1}
0&=\CCoda^\nabla\II^0&& m\geq3\\
\label{eq:hC2}
0&=(\Mh^\nabla)^*(J\II^0) && m=2\\
\label{eq:hR}
0&=R^\nabla-\II^0\wedge\Sh^0.&&
\end{align}
Moreover, in this case, the immersion is unique up to \mob/ transformations of
$S^n$.
\end{thm}

\section{Submanifold geometry in the conformal sphere}
\label{sec:curv-surf}

\subsection{Projective geometry of curves}
\label{par:curves}

In this section we briefly review the geometry of curves in $M=S^n$. In this
case the primitive data on the curve is a projective structure $\Laplace\colon
J^2L^{1/2}\to L^{-3/2}$ (equivalently given by the bilinear operator
$\Ms^\submfd$), metric connection $\nabla$ on a metric vector bundle
$N\submfd\Ln$ of rank $n-1$, and the conformal acceleration $\CA$, a
$N\dual\submfd$-valued $1$-form. Any such data give rise to a local conformal
immersion of $\submfd$, unique up to \mob/ transformation. The M\"obius
invariants of $\submfd$ can be computed in terms of an ambient Weyl connection
$D^M$. This induces an intrinsic Weyl structure $D$ with respect to which
$\Laplace = D^2 +\half \ns^{D,\submfd}$, where
$\ns^{D,\submfd}=\nr^{D^M}(T,T)+\half |H|^2$ for a weightless unit tangent
vector $T$ and (mean) curvature vector $H=D^M_T T$. The conformal acceleration
$\CA$ is given by $\CA =\nr^{D^M}(T)^\perp - \nabla^D_T H =\nr^{D^M}(T) -
\nr^{D^M}(T,T)T - D^M_T H - |H|^2 T$.

By specializing to the case $\nr^{D^M}=0$, these invariants can be related to
the standard euclidean theory of curves: here $H=\kappa N$ where $N$ is the
principal unit normal and $\kappa$ is the curvature, and $D^M_T N = \tau B$
where $B$ is the unit binormal, and $\tau$ is the torsion. Thus
$\ns^{D,\submfd}= \half\kappa^2$ and $\CA= -\dot\kappa N - \kappa \tau B$.

{\it Vertices} are points where $\CA = 0$ and we can identify $\int_\submfd
|\CA|^{1/2}$ with the conformal arclength of Musso~\cite{Mus:caf}: this is
well defined because $\CA$ is a section of $L^{-3}N\submfd\cong
L^{-2}\CV^\perp$, hence $|\CA|^{1/2}$ is a section of $L^{-1}$ which can be
integrated on a compact curve.  Note if we now define $\Laplace^D \mu = D^2\mu
+ w \ns^{D,\submfd}\mu$ on sections $\mu$ of $L^w$, then
$\del_\gam\Laplace^D=(2w-1) \cip{\gam,D\mu}$, which is algebraic in
$\gam$. This allows for the construction of further invariants such as
$4\ip{\Laplace^{D,\nabla}\CA ,\CA} - 5\ip{D\CA,D\CA}$.


\subsection{Surfaces, the conformal Bonnet theorem, and quaternions}
\label{par:surfaces}

Let $\submfd$ be a surface in $S^n$. Then the Hodge star operator on
$\Omega^1(\submfd,T\dual\submfd)$ restricts to give a complex structure on
$S^2_0T\dual\submfd$. As we have noted already, this identifies
$S^2_0T\dual\submfd$ with the bundle $\Omega^{2,0}\submfd$ of quadratic
differentials. More precisely the $\pm \iI$ eigenspace decomposition of this
complex structure may be written $S^2_0T\dual\submfd\tens\C\cong
\Omega^{2,0}\submfd\dsum \Omega^{0,2}\submfd$, where
$\Omega^{2,0}\submfd=(T^{1,0}\submfd)^{-2}$,
$\Omega^{0,2}\submfd=(T^{0,1}\submfd)^{-2}$, and
$T\submfd\tens\C=T^{1,0}\submfd\dsum T^{0,1}\submfd$ is the $\pm\iI$
eigenspace decomposition of the complex structure on the tangent bundle into
the two null lines of the complexified conformal metric. The isomorphism
$S^2_0T\dual\mfd\to \Omega^{2,0}\submfd$ is the projection $q\mapsto q^{2,0}$
with respect to this decomposition.

We now note that the exterior derivative $\d^D\colon
\Omega^1(\submfd,T\dual\submfd)\to \Omega^2(\submfd,T\dual\submfd)$ coupled to
a Weyl connection $D$ restricts to give the holomorphic structure on
$S^2_0T\dual\submfd$: $(\d^D q)^{2,1}=\dbar(q^{2,0})\in 
\Omega^{1,1}(\submfd,(T^{1,0}\submfd)^{*})$.  Since this is independent of the
choice of Weyl structure, we denote $\d^Dq$ by $\d q$. Similarly, $(\d
q)^{1,2} = \del(q^{0,2})$. We therefore have
\begin{align*}
\d q &= \dbar(q^{2,0}) + \del(q^{0,2})\\
\divg q & = {*(\dbar(q^{2,0}) - \del(q^{0,2}))},
\end{align*}
the second line being obtained from the first via $\d {*q}=*{\divg q}$, showing
that the divergence is also a conformally invariant operator on quadratic
differentials. Thus $\d q=0$ iff $\divg q=0$ iff $\dbar(q^{2,0})=0$ iff
$\del(q^{0,2})=0$, in which case we say $q$ is \emphdef{holomorphic}
quadratic differential.

Applying similar ideas to the \cso/ $\CS$ yields the following fact.
\begin{prop}\label{p:N-lift} Let $\Ln$ be an immersion of a surface $\submfd$
in $S^n$ whose central sphere congruence $\CV_\submfd$ has \cso/ $\CS$. Then
$\LH{*\CS}=0$ and ${*\CS}$ is the differential representative of its homology
class $J\II^0-J\Sh^0$.
\end{prop}
\begin{proof}
The fact that $\LH{*\CS}=0$ follows easily from the symmetry of the second
fundamental form (or shape operator). To show $*\CS$ is the differential
representative, we need to show that $\LH\dDh{*\CS}=0$, \ie,
$\dDh{*\CS}|_{\CV_\submfd^\perp}\in
\Omega^2(\submfd,\Hom(\CV_\submfd^\perp,\Ln))$. For this we write
$\Lnps\tens\C=\Ln^{(1,0)}+\Ln^{(0,1)}$, where $\Ln^{(1,0)}\cap
\Ln^{(0,1)}=\Ln\tens\C$, $\Ln^{(1,0)}/(\Ln\tens\C)\cong \Ln\ltens
T^{1,0}\submfd$ and $\Ln^{(0,1)}/(\Ln\tens\C)\cong \Ln\ltens T^{0,1}\submfd$.
Now $\CS=\CS^{1,0}+\CS^{0,1}$ and ${*\CS}=\iI(\CS^{1,0}-\CS^{0,1})$, where
\begin{equation*}
\CS^{1,0}\restr{\CV_\submfd^\perp}\in
\Omega^{1,0}(\submfd,\Hom(\CV_\submfd^\perp, \Ln^{(1,0)})),\qquad
\CS^{0,1}\restr{\CV_\submfd^\perp}\in
\Omega^{0,1}(\submfd,\Hom(\CV_\submfd^\perp,\Ln^{(0,1)})).
\end{equation*}
Since the $(1,0)$ and $(0,1)$ directions are null,
$\CD^{1,0}\Ln^{(1,0)}\subset \Ln^{(1,0)}$ and $\CD^{0,1}\Ln^{(0,1)}\subset
\Ln^{(0,1)}$. It follows that
\begin{equation*}
\dDh(\CS^{1,0})\restr{\CV_\submfd^\perp}
\in\Omega^{1,1}(\submfd,\Hom(\CV_\submfd^\perp,\Ln^{(1,0)})),
\quad
\dDh(\CS^{0,1})\restr{\CV_\submfd^\perp}
\in\Omega^{1,1}(\submfd,\Hom(\CV_\submfd^\perp,\Ln^{(0,1)})).
\end{equation*}
Now $\dDh{\CS}|_{\CV_\submfd^\perp}$ takes values in
$\Hom(\CV_\submfd^\perp,\Ln)$, hence so do
$\dDh(\CS^{1,0})\restr{\CV_\submfd^\perp}$,
$\dDh(\CS^{0,1})\restr{\CV_\submfd^\perp}$ and
$\dDh{*\CS}|_{\CV_\submfd^\perp}$.
\end{proof}
There is a conceptual explanation for the above result: in complexified
conformal geometry (or in signature $(1,1)$), $\II^0$ does not lie in an
irreducible homology bundle and its differential lift is the sum of the
differential lifts of the two irreducible components.

There is another fact we shall need: note that if $\eta\in
\Omega^1(\submfd,T\dual\mfd)$ then $\LH\eta=0$ and $\hc{\eta}$ is a quadratic
differential $q$ given by the trace-free symmetric tangential part of $\eta$.

\begin{prop} \label{p:eta-and-q} Suppose $\hc{\eta}=q$ and
and $\LH\dDh\eta=0$.  Then $\eta=q$ \textup(\ie, $\eta$ is trace-free,
symmetric and tangential\textup), $\dDh\eta=0$ if and only if $q$ is
holomorphic \textup(\ie, $\d q=0$\textup), and $\liebrack{\CS\wedge\eta}=0$ if
and only if $q$ commutes with the shape operators \textup(\ie,
$q\wedge\Sh^0=0$\textup).
\end{prop}
\begin{proof} If $0=\LH\dDh\eta=\LH\abrack{\iden\wedge\eta}$ then by
adjointness $0=\abrack{\iden\wedge\eta}$, from which it easily follows that
$\eta$ is tangential and trace-free symmetric. Then $\dDh\eta=0$ if and only
if $\d^D\eta=0$ for some (hence any) Weyl derivative $\d$, \ie, $q$ is
holomorphic. The last part is immediate.
\end{proof}
In this statement we are, as usual, identifying (the pullback to $\submfd$ of)
$T\dual\mfd$ with the nilradical bundle $\stab(\Ln)^\perp$ in
$\stab(\Ln)\subset \submfd\times\so(n+1,1)$. If we regard $\eta$ in this
way, then the statement $\eta=q$ means, more precisely:
\begin{equation}
\eta_X\act\sigma=0, \qquad \eta_X\act\CD_Y\sigma =
-q(X,Y)\sigma,\qquad\eta_X(\CV\low_\submfd) \subset \Lnps,\qquad
\eta_X\restr{\CV_\submfd^\perp}=0.
\end{equation}

\subsubsection*{The conformal Bonnet theorem in conformal coordinates}

Surfaces are conformally flat, \ie, locally, we can introduce a holomorphic
coordinate $z=x+iy$, so that $(x,y)$ are conformal (aka.~isothermal)
coordinates. Hence we can refer everything to the flat \mob/ structure $\Ms^z$
determined by $z$ and, in this setting, the \GCR/ equations have a very
down-to-earth flavour. Recall that the flat connection $\d=\CD^{S^n}$ on
$\submfd\times\R^{n+1,1}$ may be written $\d=\CD^\submfd+\CQ+\nabla+\CS$ with
$\QC$ pure trace so that, for $\sigma$ a section of $\Ln$,
\begin{equation*}
\del_z^2\sigma =
(\CD^\submfd+\CQ)\low_{\del_z}(\CD^\submfd+\CQ)\low_{\del_z}\sigma+
\CS\low_{\del_z}(\CD^\submfd_{\del_z}\sigma)
=\CD^\submfd_{\del_z}\CD^\submfd_{\del_z}\sigma+
\CS\low_{\del_z}(\CD^\submfd_{\del_z}\sigma).
\end{equation*}
Fix $\sigma$ by demanding that $(\d \sigma,\d\sigma)=\d z\,\d\bar z$ so that
$\sigma$ is parallel for the (flat) Weyl derivative $D^z$.  Then Proposition
\ref{p:Q-calc} gives
\begin{equation*}\notag
\del_z^2\sigma +q\sigma=-\II^0(\del_z,\del_z)\sigma,
\end{equation*}
where $\nr^{D^z}=q\d z^2+\bar q\d\bar{z}^2$ (so that $2q$ is the schwarzian
derivative of $\Ms^\conf$ with respect to $\Ms^z$).  Define $\kappa$, a
section of $\CV_\submfd^\perp$ by $\kappa= -\II^0(\del_z,\del_z)\sigma$ so that
\begin{equation*}\notag
\del_z^2\sigma +q\sigma=\kappa.
\end{equation*}
The remaining ingredients of the \GCR/ equations are now readily
expressed in terms of $q$, $\kappa$, $\nabla$:
\begin{equation*}
\QC(\del_z,\del_z)=\lip{\kappa,\bar\kappa}\qquad\text{and}\qquad
\CA^{D^z}=2(\nabla_{\del_z}\kappa)\sigma.
\end{equation*}
An easy computation then gives the following form of the \GCR/ equations:
\begin{gather*}
\del_{\bar z} q=3\lip{\nabla_{\del_z}\bar\kappa,\kappa}+
\lip{\bar\kappa,\nabla_{\del_z}\kappa}\\
\imag(\nabla_{\del_{\bar z}}\nabla_{\del_{\bar z}}\kappa+
\bar q\kappa)=0\\
R^\nabla_{\del_{\bar z},\del_z}\wnv=
2\lip{\kappa,\wnv}\bar\kappa-2\lip{\bar\kappa,\wnv}\kappa.
\end{gather*}
This formulation of the \GCR/ equations and the Bonnet theorem was developed
in \cite{BPP:sdfs} using a bare-hands approach.

\subsubsection*{Surfaces in the 4-sphere and quaternions}

For surfaces in $S^4$, \S\ref{par:quaternionic} reveals an alternative
approach to submanifold geometry, using the spin representation of
$\Spin(5,1)\cong \SL(2,\HQ)$ and quaternions. This is the setting for the
book~\cite{BFLPP:cgs}, in which quaternionic holomorphic structures are used
to study the global theory of conformal immersions. We shall content ourselves
here with a consideration of the local invariants from a quaternionic point of
view.

Let $\Ln_\HQ\subset H:=\submfd\times\HQ^2$ be an immersion of a surface
$\submfd$ into $S^4\cong\HP1$. The differential of this immersion is
$\beta\colon T\submfd\to \Hom_{\HQ}(\Ln_\HQ,H/\Ln_\HQ)\cong
TS^4\restr\submfd$. It follows that there are unique quaternion linear complex
structures on $\Ln_\HQ$ and $H/\Ln_\HQ$ up to sign which preserve, and agree
on, the image of $\beta$ (geometrically, they act by selfdual and antiselfdual
rotations on $TS^4$). We denote these complex structures by $J$: in symbols
${*\beta}=J\circ\beta=\beta\circ J$, where $*$ is the Hodge star operator of
the induced conformal metric on $\submfd$. The normal bundle of $\submfd$ then
consists of the elements of $\Hom_{\HQ}(\Ln_\HQ,H/\Ln_\HQ)$ which anticommute
with $J$, and we equip it with the complex structure given by precomposition
with $J$.

An enveloped sphere congruence in this language is a complex structure $J_\CV$
on $H$ which preserves $\Ln_\HQ$ and induces $J$ on $\Ln_\HQ$ and $H/\Ln_\HQ$.
Flat differentiation then splits as $\d=\CD^\CV+\CS^\CV$, where $\CD^\CV
J_\CV=0$ and $\CS^\CV$ anticommutes with $J_\CV$. The enveloping condition
means equivalently that $\CS^\CV$ preserves $\Ln_{\HQ}$. The second
fundamental form $\II^\CV$ arises from the induced endomorphisms of
$\Ln_{\HQ}$ and $H/\Ln_{\HQ}$ by restricting pre- and post-composition to the
image of $\beta$, yielding $1$-forms with values in $\Hom(T\submfd,N\submfd)$,
one $J$-linear, the other $J$-antilinear. The central sphere congruence
condition that $\II^\CV$ is tracefree may then be interpreted as saying that
$(\II^\CV)^{1,0}$ is $J$-linear and $(\II^\CV)^{0,1}$ is $J$-antilinear (or
vice versa). These are the objects are referred to as $A$ and $Q$
in~\cite{BFLPP:cgs}.

\section*{Preview of Part IV}

In this second part of our work, sphere congruences have been used as a tool
to define homological invariants which characterize conformal immersions into
$S^n$ up to M\"obius transformation. Part III is devoted to some applications
of this homological machinery to integrable conformal submanifold geometries.

Sphere congruences in conformal geometry (i.e., maps from an $m$-manifold
$\submfd$ into the grassmannian of $k$-spheres in $S^n$) may also be studied
in their own right, and were of considerable interest
classically~\cite{Bla:vud,Dar:tgs}. There is a rich interplay between sphere
congruences and submanifolds: for example, sphere congruences may have
enveloping submanifolds (which we have studied here for $k=m$) or orthogonal
submanifolds (which are of particular interest for $k=n-m$).

In Part IV of our work, we develop an approach to sphere congruences using the
bundle formalism, in which a sphere congruence is studied as a signature
$(k+1,1)$ subbundle $\CV$ of the trivial $\R^{n+1,1}$ bundle over
$\submfd$. This theory is essentially self-contained, since it depends only on
the straightforward idea to write $\d=\CD^\CV+\CS^\CV$, where $\CD^\CV$ is the
induced direct sum connection on $\submfd\cross\R^{n+1,1}=\CV\dsum\CV^\perp$,
and $\CS^\CV$ is the remaining, off-diagonal, part of $\d$. We have already
seen, for example, that an enveloping submanifold is a null line subbundle
$\Ln$ of $\CV$ on which $\CS^\CV$ vanishes (i.e., $\Ln^{(1)}:=\d \Ln \subseteq
\CV$). Similarly, an orthogonal submanifold is $\CD^\CV$-parallel null
$\Ln\subseteq \CV$ (so $\Ln^{(1)}\subseteq \Ln\dsum\CV^\perp$).

We focus in particular on {\it Ribaucour} sphere congruences (with $k=m$) and
{\it spherical systems} (with $k=n-m$---known as {\it cyclic systems} when
$n-m=1$). For Ribaucour sphere congruences, we establish the well-known
Bianchi permutability of Ribaucour transformations. Although Lie sphere
geometry provides a more general setting for these~\cite{BuHJ:rls}, extra
information is available in the conformal approach, leading to additional
applications.

For spherical systems, we devote our attention to flat spherical systems (for
which $\CD^\CV$ is a flat connection---for a general spherical system, it is
only assumed to be flat on $\CV^\perp$). These are examples of ``curved
flats'' arising in conformal geometry~\cite{FePe:cfss,TeUh:btlg}. Darboux
pairs of isothermic surfaces provide another example, in which a
$2$-dimensional $0$-sphere congruence is a curved flat~\cite{BHPP:cfis}.

It is well-known that curved flats admit B\"acklund transformations which can
be derived from a loop-group formalism, but we develop instead a direct
approach. This has (at least) two advantages: first, we do not need to
introduce frames; second, the \emph{definition} of the B\"acklund
transformations does not involve dressing (\ie, Birkhoff factorization)---this
is only needed (and then only implicitly and in a very simple form) for the
\emph{permutability} of B\"acklund transformations. We also introduce a theory
of polynomial conserved quantities for such curved flats, which has several
applications.

\part{Applications and Examples}

We now explore some applications of our conformal submanifold geometry
theory. We begin in \S\ref{par:shape-curv-channel} with the simplest examples,
the totally umbilic and channel submanifolds, and we discuss the curvature
spheres which can be used to describe them. In
\S\S\ref{par:tangent-cong}--\ref{par:sym-break-sub} we study the interaction
of submanifolds with the symmetry breaking induced by a constant vector or
$(k+1)$-plane (cf.~\S\S\ref{par:spaceform}--\ref{par:sym-break}) and find
homological characterizations of minimal submanifolds in a spaceform and
submanifolds splitting across a product decomposition. Then, in
\S\ref{par:dupin}, we give a fast analysis of orthogonal systems and Dupin's
Theorem.

We next turn to Willmore and constrained Willmore surfaces.  On any compact
surface $\submfd$ in $S^n$, the square norm (with respect to the conformal
metric and the metric on the weightless normal bundle), $|\II^0|^2$, of the
tracefree second fundamental form is a section of $L^{-2}$ (called the
Willmore integrand), and so may be invariantly integrated. The integral
$\int_\submfd|\II^0|^2$ is called the \emphdef{Willmore functional} $\cW$.

A compact Willmore surface is a critical point for $\cW$ on the space of all
immersions of a fixed compact surface $\submfd$, while a compact constrained
Willmore surface is a critical point restricted to immersions inducing a fixed
conformal structure on $\submfd$. It is usual to extend the definition of
Willmore and constrained Willmore surfaces to arbitrary surfaces by requiring
that the Euler--Lagrange equations hold, with a Lagrange multiplier (which is
a holomorphic quadratic differential) in the constrained case. We take this as
our starting point in~\S\ref{par:willmore}, where we derive the classical
equation in codimension one, and obtain a spectral deformation in arbitrary
codimension. Then, in \S\ref{par:harmonic}, we derive the (constrained)
Willmore equation from the functional, using the relation with the harmonic
map equation for the central sphere congruence. Although this is well-known,
our machinery does not get in the way the key idea, and reveals the
homological nature of this theory.

We next turn to isothermic surfaces, which were of great classical interest
partly because they are the only surfaces which admit a deformation which does
not alter the induced conformal metric, normal connection and tracefree second
fundamental form.  This is of particular interest in our theory, because in
this situation the induced \mob/ structure $\Ms^\submfd$ is the key invariant.
We give a manifestly conformally-invariant definition in
\S\ref{par:isothermic}, then describe the deformations and associated family
of flat connections. In \S\ref{par:iso-examples} we give examples: products of
curves in spaceforms, CMC and generalized $H$-surfaces, and quadrics
(previously considered as `mysterious' examples). We end with an intrinsic
equation for isothermic surfaces in codimension
one~\S\ref{par:when-isothermic}.

In any type of submanifold geometry, it is natural to ask when the induced
intrinsic geometry is flat. In conformal submanifold geometry this has an
unambiguous meaning for $m\geq3$: the induced conformal metric should be
(conformally) flat. For $m=1$, flatness is automatic, so it remains to
consider the case $m=2$.  Using only conformal structures, flatness would also
be automatic, but the theory of \mob/ structures provides an obvious
nontrivial condition: flatness of the induced conformal \mob/ structure or
equivalently of the normal Cartan connection (on the central sphere
congruence).  However, a more general condition turns out to be more natural:
that there is \emph{some} enveloped sphere congruence $\CV$ inducing a flat
conformal \mob/ structure on $\submfd$. Imposing also flatness of the
weightless normal bundle (which is automatic in codimension one), we thus
develop, in section~\ref{sec:mflat}, a new unified theory of ``M\"obius-flat
submanifolds''. In dimension $m\geq 3$ these are the conformally-flat
submanifolds, while in dimension $m=2$ and codimension one, they turn out to
be the classical Guichard surfaces~\cite{Cal:asg,Gui:ssi}. Using this theory
we show that channel submanifolds and constant Gaussian curvature surfaces are
M\"obius-flat, as are certain extrinsic products. We rederive results of
Cartan and Hertrich-Jeromin\cite{Car:dhc,Her:cfg} on conformally-flat
hypersurfaces and Guichard nets~\cite{Gui:sto}, and end this part by placing
Dupin cyclides in this context.

\section{Sphere congruences and symmetry breaking}
\label{sec:env-sphere-cong}

\subsection{Shape operators, curvature spheres, and channel submanifolds}
\label{par:shape-curv-channel}

\begin{defn} A \emphdef{curvature sphere} is an enveloped $m$-sphere
congruence $\CV$ such that $\II^\CV$ is degenerate (as an $N\submfd$-valued
bilinear form on $T\submfd$).  Let $T_\CV$ denote the subbundle of $T\submfd$
on which $\II^\CV$ degenerates: we assume for simplicity that its rank, called
the \emphdef{multiplicity} of $\CV$, is constant on $\submfd$.
\end{defn}

Curvature spheres are a convenient way of describing the eigenspaces and
eigenvalues of the shape operator. Indeed if $\CV$ is any enveloped sphere
congruence and $E$ is a simultaneous eigenspace for the normal components of
the shape operator $\Sh^\CV$ such that $\Sh^\CV(\nmv)$ has eigenvalue
$\con(\nmv)$ for all normal vectors $\nmv$, then $\CV+\con$ is a curvature
sphere with $T_{\CV+\con}=E$.

The existence of simultaneous eigenspaces is facilitated by the following
well-known fact.
\begin{prop} The normal components of $\Sh^\CV$ commute with each
other if and only if $\submfd$ has flat \textup(weightless\textup) normal
bundle.
\end{prop}
\begin{proof} By the Ricci equation~\eqref{eq:gr} and the symmetry of the
shape operator, we have
\begin{align*}
\cip{R^\nabla_{X,Y}\nmv_1,\nmv_2} =
\cip{(\II^\CV_X\Sh^\CV_Y-\II^\CV_Y\Sh^\CV_X)\nmv_1,\nmv_2}
&=\cip{\Sh^\CV_Y\nmv_1,\Sh^\CV_X\nmv_2}-\cip{\Sh^\CV_X\nmv_1,\Sh^\CV_Y\nmv_2}\\
&=\cip{\Sh^\CV_{\Sh^\CV_X\nmv_2}\nmv_1,Y} 
-\cip{\Sh^\CV_{\Sh^\CV_X\nmv_1}\nmv_2,Y}.
\end{align*}
The result follows.
\end{proof}
\begin{rem}\label{rem:commutingQ} Since $\QC^\CV$ (viewed as an
$\Ln^2$-valued endomorphism) is a linear combination of $\iden$ with normal
components of $\Sh^0$ and $\II^0\circ\Sh^0$, we have that
$\QC^{\CV}\wedge\Sh^0=0$ if the normal bundle is flat.
\end{rem}
We deduce from the proposition that if $\submfd$ has flat normal bundle, then
the normal components of $\Sh^\CV$ are simultaneously diagonalizable, since
they are symmetric. Thus if $\CV_1,\ldots \CV_\ell$ are the curvature spheres,
then $T\submfd$ is the orthogonal direct sum of the $T_{\CV_i}$. The
subbundles $T_{\CV_i}$ are the (simultaneous) eigenspaces of the shape
operator, and the conormal eigenvalues $\con_i$ give the curvature spheres via
$\CV_i = \CV+\con_i$.

\begin{prop}\label{p:dupin} If a curvature sphere $\CV$ has multiplicity
greater than $1$ then $T_\CV$ is an integrable distribution and $\CV$ is
constant \textup(\ie, a parallel subbundle of
$\submfd\times\R^{n+1,1}$\textup) along the leaves of the corresponding
foliation.
\end{prop}
\begin{proof} The Codazzi equation for $\CV$ implies that
\begin{equation}\label{eq:dupin}\begin{split}
0 &= (\nabla^D\II^\CV)_{X,Y} - (\nabla^D\II^\CV)_{Y,X} +\CA^{D,\CV}_X Y 
-\CA^{D,\CV}_Y X\\
&= \II^\CV_{[X,Y]} + \CA^{D,\CV}_X Y -\CA^{D,\CV}_Y X
\end{split}\end{equation}
for $X,Y$ in $T_\CV$. Thus $\II^\CV_{[X,Y]}$ is in $T_\CV$, but it is also in
$T_\CV^\perp$ by the symmetry of $\II^\CV$, and therefore $\II^\CV_{[X,Y]} =
0$ and $[X,Y]$ is in $T_\CV$.  Now contracting~\eqref{eq:dupin} with $Y$, we
get $\CA^{D,\CV}_X \ip{Y,Y} -\CA^{D,\CV}_Y \ip{X,Y}=0$, and taking $\ip{X,Y} =
0$, $\ip{Y,Y} \neq 0$, we deduce that $\CA^{D,\CV}_X = 0$ for all $X$ in
$T_\CV$.  Hence $\CS^\CV_X = 0$ for all $X$ in $T_\CV$, \ie, $\CV$ is parallel
in the $T_\CV$ directions.
\end{proof}
The first conclusion of this proposition is automatic for curvature spheres of
multiplicity one, but the second is not.
\begin{defns} Let $\submfd$ be a submanifold of $S^n$ of dimension $m\geq 2$
with flat normal bundle. Then $\submfd$ is said to be:
\begin{numlist}
\item a \emphdef{Dupin submanifold} iff its curvature spheres are all constant
along their foliations;
\item a \emphdef{totally umbilic submanifold} iff $\II^0=0$, \ie, it has
only one curvature sphere;
\item a \emphdef{channel submanifold} iff it is the envelope of a
$1$-parameter family of $m$-spheres, \ie, it admits an enveloped sphere
congruence $\CV$ whose derivative $\CS^\CV\colon T\submfd\to \m_\CV$ has rank
one (so that $\CV$ is, in particular, a curvature sphere of multiplicity
$m-1$).
\end{numlist}
\end{defns}
The following well-known observations are immediate consequences of
Proposition~\ref{p:dupin}.
\begin{cor} A totally umbilic $m$-submanifold $(m\geq2)$ of $S^n$ is an open
subset of its curvature sphere \textup(hence is a Dupin submanifold\textup).
\end{cor}
\begin{cor} Suppose $\submfd$ has flat normal bundle and exactly two
curvature spheres. Then $\submfd$ is either a surface, a channel submanifold,
or a Dupin submanifold.
\end{cor}

\subsection{Constant vectors and tangent congruences in spaceform geometries}
\label{par:tangent-cong}

In \S\ref{par:mr-ambient-weyl-riem} we showed how to break the conformal
invariance of our theory by introducing a compatible riemannian metric on
$\mfd$.  When $\mfd$ is an open subset of $S^n=\PL$, such a metric is given
by a section of the positive light-cone $\LC^+\subset\Ln$ over $\mfd$. As we
have discussed in \S\ref{par:spaceform}, a particularly important
class of such sections are the conic sections $\{\sigma\in
\LC^+:\lip{v_\infty,\sigma}=-1\}$ associated to nonzero vectors
$v_\infty\in\R^{n+1,1}$. We have seen that such a section induces a constant
curvature metric $g$ on $\mfd$, with Weyl structure
$\Lnc_g=\vspan{v_\infty+\frac12(v_\infty,v_\infty)\sigma}$ in
$\mfd\times\R^{n+1,1}$ (where $\lip{v_\infty,\sigma}=-1$).

Along a submanifold $\submfd$ of $\mfd$, $\Lnc_g$ defines an ambient Weyl
structure, inducing the sphere congruence $\CV_g=\Lnps\dsum\Lnc_g=
\Lnps\dsum\vspan{v_\infty}$.  We refer to a sphere congruence containing a
constant vector $v_\infty$ as the \emphdef{tangent congruence} to $\submfd$ in
the spaceform given by $v_\infty$. When $v_\infty$ is null and the geometry on
$\mfd\subseteq S^n\setdif\vspan{v_\infty}$ is euclidean, the tangent spheres
pass through the point at infinity, and so they stereoproject to tangent
planes in the usual sense.

\begin{prop} For a sphere congruence $\CV$ and Weyl structure $\Lnc\subset\CV$
\textup(with associated Weyl derivative $D$\textup), the following are
equivalent\textup:
\begin{numlist}
\item $\CV=\CV_g$ and $\Lnc=\Lnc_g$ for a conic section of $\LC^+$\textup;
\item $\CV$ is a tangent congruence and is enveloped by $\Lnc$ \textup(\ie,
$\CS|_{\Lnc}=0$\textup)\textup;
\item $\nr^{D,\CV}=\frac1n \ns\conf$ for a section $\ns$ of
$\Ln^2\cong\Hom(\Lnc,\Ln)$, and $A^{D,\CV}=0$.
\end{numlist}
\end{prop}
\begin{proof} By~\eqref{eq:metric-quant}, $\nr^{D,\CV_g}=(\nr^g)\tangent
=\smash{\frac1 n} \ns^g\conf$ and $\CA^{D,\CV_g} =(\nr^g)\normal=0$. Thus
(i)$\Leftrightarrow$(ii) and these imply (iii).  For the converse implication,
if $\nr^{D,\CV} =\frac1n \ns\conf$ and $\CA^{D,\CV}=0$, the Gau\ss\ equation
implies that $\ns$ is $D$-parallel. Let $\hat\sigma$ be a $D$-parallel section
of $\Lnc$ and $\ns(\hat\sigma)$ the induced section of $\Ln$; then $v_\infty=
\hat\sigma+\frac1n \ns(\hat\sigma)$ is a constant vector.
\end{proof}
If $v_\infty$ is null, $\Lnc_g$ is constant, so the submanifold it defines
is a point, but otherwise, it is immersed.  The \GCR/ equations in this case
are equivalent to the riemannian \GCR/ equations in the corresponding
spaceform geometry and the Bonnet theorem reduces to the riemannian one.

We now consider the relation between $v_\infty$ and $\CV\low_\submfd$. Since
$\CV_g=\exp(H^g)\CV\low_\submfd$ contains $v_\infty$, it follows that
$\CV_\submfd$ component of $v_\infty$ is $H^g\ell$, where $\ell$ is dual to
$\sigma$. In particular
\begin{quote}
\emph{$\submfd$ is a minimal submanifold in the spaceform iff $v_\infty\in
\CV_\submfd$ iff $\CV_\submfd=\CV_g$,}
\end{quote}
\ie, the tangent congruence and the central sphere congruence agree. We remark
that in general the homology class of $v_\infty$ is the pair $(\ell,H^g\ell)$
in $L\dsum L\ltens N\dual\submfd$, and so the minimal submanifolds are
precisely those for which the homology class of $v_\infty$ is tangential.

Finally, we ask when the Weyl structure of $g$ provides a second envelope of
the central sphere congruence, \ie, $A^D=0$ for the induced Weyl derivative
$D$: equation~\eqref{eq:Rca} here reads $\CA^D=-\nabla^D H^g$, \ie, the
submanifold has parallel mean curvature in the given spaceform geometry.

\subsection{Symmetry breaking for submanifolds}
\label{par:sym-break-sub}

We study again the metrics considered in \S\ref{par:sym-break}, associated to
a $(k+1)$-plane $W$ in $\R^{n+1,1}$, with orthogonal $(n-k+1)$-plane
$W^\perp$, where $0\leq k\leq n$, which identify an open subset of $S^n$ with
a product of spaceforms of dimension $k$ and $n-k$. For an $m$-dimensional
submanifold $\Ln\to\submfd$ of this open subset, let $\Lnc$ be the ambient
Weyl structure along $\submfd$ with $\Ln\dsum\Lnc=(\submfd\times
W\dsum\Ln)\cap(\submfd\times W^\perp\dsum\Ln)$, and $\CV$ the corresponding
enveloped sphere congruence.

It is natural to ask when $\submfd$ is a local product of immersions of a
$p$-manifold and an $(m-p)$-manifold into the spaceforms associated to $W$ and
$W^\perp$ (where $0\leq p\leq k$, $0\leq m-p\leq n-k$). It is clear that this
implies that the lines $\Wedge^{k+1} W$ and $\Wedge^{n-k+1}W^\perp$ lie in
$\Wedge^{p+1}\CV\tens\Wedge^{k-p}\CV^\perp$ and
$\Wedge^{m-p+1}\CV\tens\Wedge^{n-k-(m-p)}\CV^\perp$ respectively.  Hence a
unit vector $\omega$ in $\Wedge^{k+1} W$ may be written $\omega=\theta\wedge
v\wedge\chi$ with $\theta\in \Wedge^p\Lnps$, $v\in\CV$ and
$\chi\in\Wedge^{k-p}\CV^\perp$ (all decomposable). Since we require that
$\Ln\not\subseteq W\cup W^\perp$, we must have
$\theta\notin\Ln\wedge\Wedge^{p-1}\Lnps$ and $v\notin\Ln^\perp$. On the other
hand, $\Ln\dsum\Lnc$ has nontrivial intersection with $W$, so without loss, we
can take $v\in \Ln\dsum\Lnc$ and $\theta\in \Wedge^p\US$, where
$\US:=(\Ln\dsum\Lnc)^\perp\cap\CV$

Similarly $\Wedge^{n-k+1}W^\perp$ contains a unit vector of the form
$\omega^\perp=\theta^\perp\wedge v^\perp\wedge\chi^\perp$ with
$\theta^\perp\in\Wedge^{m-p}\US$, $v^\perp\in\Ln\dsum\Lnc$ and $\chi^\perp\in
\Wedge^{n-k-(m-p)}\CV^\perp$ (all decomposable).

Observe that the central sphere congruence may be written
$\CV\low_\submfd=\Ln\dsum\US\dsum\Lnc\low_\submfd$, where
$\Lnc\low_\submfd=\exp(-H^\CV)\Lnc$ is the null complement to $\Ln$ in
$\US^\perp\cap\CV\low_\submfd$. It follows that for nontrivial splittings
($0<p<m$), the conditions on $W,W^\perp$ that we have obtained have a
homological consequence: $\hc{\omega}_{S^n}\restr\submfd=\hc{\omega}_\submfd=
\theta\vtens(v\modulo\Ln^\perp)\vtens\chi$ and is a decomposable section of
\begin{multline*}
H_0(T\dual\submfd,\Wedge^{p+1}\CV\low_\submfd)\tens
\Wedge^{k-p}\CV_\submfd^\perp=\Wedge^p( T\submfd\ltens\Ln)\ltens L\tens
\Wedge^{k-p}\CV_\submfd^\perp\\ \subset \Wedge^{k}(TS^n\ltens\Ln)\ltens
L\restr\submfd=H_0(T\dual S^n,S^n\times\Wedge^{k+1}\R^{n+1,1})\restr\submfd,
\end{multline*}
and similarly for $\hc{\omega^\perp}_{S^n}\restr\submfd$. This turns out to be
a characterization.

\begin{thm} Let $W\subset\R^{n+1,1}$ be a $(k+1)$-dimensional subspace
with orthogonal space $W^\perp$. Then an immersed submanifold of $S^n\setdif
(P(W)\cup P(W^\perp))$ splits locally across the induced product structure as
a product of a $p$-submanifold and an $(m-p)$-submanifold $(0<p<m)$ if and
only if the homology class $\hc{\omega}_{S^n}$ of some \textup(hence
any\textup) nonzero $\omega\in\Wedge^{k+1}W$ is a section of
$\Wedge^p(T\submfd\ltens\Ln)\ltens L\tens \Wedge^{k-p}\CV_\submfd^\perp$ along
$\submfd$.
\end{thm}
\begin{proof} It remains to show that the homological condition implies the
splitting. For this recall from \S\ref{par:sym-break} that $\hc{\omega}_{S^n}$
and $\hc{\omega^\perp}_{S^n}$ are sections of the top exterior powers of the
two distributions tangent to the spaceform factors (up to a line bundle), so
the homological condition means that one of these distributions meets
$T\submfd\dsum N\submfd$ (along $\submfd$) in the sum of a tangential
$p$-plane and a normal $(k-p)$-plane. The other distribution, being
orthogonal, then meets $T\submfd\dsum N\submfd$ in a tangential $(m-p)$-plane
and a normal $(n-k-(m-p))$-plane. Now the tangential planes are integrable,
and this locally splits $\submfd$.
\end{proof}

In the case $k=1$, the above theorem has an intuitively clear meaning.
\begin{cor} The vector field $K$ on $S^n$ associated to a $2$-dimensional
subspace $W$ of $\R^{n+1,1}$ is tangent to $\submfd$ if and only if $\submfd$
is an open subset of a revolute, a cylinder or a cone, over a submanifold of
the same codimension in $\cH^{n-1}$, $\R^{n-1}$ or $\cS^{n-1}$
respectively. In this case the homology class of $\omega_\infty$ along
$\submfd$ is the corresponding tangent vector field.
\end{cor}
The properties of these submanifolds can easily be read off from the data on
$\submfd$ by imposing the condition that $\omega$, here equal to $\theta\wedge
v$, with $\theta\in\US$ and $v\in\Ln\dsum\Lnc$, is parallel.  We obtain
immediately that $\CS^{\CV}\theta=0=\CS^{\CV}v$, so that $\CV$ is a curvature
sphere, constant along the curvature lines in the direction of
$K=\theta\tens(v\modulo\Ln^\perp)$, $\Lnc$ is another envelope, and
$\CS^{\CV}$ (hence $\II^{\CV}$) preserves the decomposition
$T\submfd=\vspan{K}\dsum K\normal\cap T\submfd$ induced by the product metric
on $\cS^1\times \cH^{n-1}$, $\R\times \R^{n-1}$ or $\cH^1\times \cS^{n-1}$.
\begin{example}
Recall that a (Dupin) cyclide in $S^3$ is an orbit of a two dimensional
abelian subgroup of the \mob/ group $\Mob(3)$.  Dupin's classification of the
cyclides can be read off easily from the above. In the span of a two
dimensional abelian subalgebra of $\so(4,1)$ we can always find a basis of
decomposable elements (just consider the Jordan normal form). These
decomposables define two orthogonal $2$-planes in $\R^{4,1}$, reducing the
geometry to $\cH^{2}\times \cS^1$, $\R^{2}\times\R$, or $\cS^{2}\times \cH^1$,
and the cyclide as a surface of revolution, a cylinder, or a cone, according
to whether the $2$-plane is spacelike, degenerate, or timelike. At most one of
the planes can be degenerate or timelike. Therefore a cyclide, if not totally
umbilic, is a channel surface in two ways, and is either a circular torus of
revolution, a cylinder of revolution, or a cone of revolution.
\end{example}
If the splitting is trivial (\ie, $p=0$ or $p=m$), then $\hc{v_\infty}$ is no
longer a pure tangential class in general. Indeed, as we have seen in the case
$k=0$ (or $k=n$) of the spaceform geometries of the previous paragraph
(\S\ref{par:tangent-cong}), $\hc{v_\infty}$ is tangent to $\submfd$ if and
only if the immersion (into the spaceform factor) is minimal.

\subsection{Orthogonal systems and Dupin's Theorem}
\label{par:dupin}

An extreme form of symmetry breaking is given by $n$-tuply orthogonal
systems. Suppose that we have a local diffeomorphism
$\Ln\subset\submfd\times\R^{n+1,1}$ of a product of $1$-manifolds
$\submfd=\prod_{i=1}^n\submfd_i$ with $S^n=\PL$, such that the factors are
orthogonal for the induced conformal metric. Let $U_i=\d_{X_i}(\Ln)$, where
$X_i$ is a nonvanishing tangent vector field to $\submfd_i$ pulled back to the
product.  Thus the $U_i$ are mutually orthogonal rank $2$ bundles with sum
$\Ln^\perp$ and intersection $\Ln$. For each $i$ we have a one parameter
family of hypersurfaces $\submfd_{x_i}$ in $S^n$ by fixing the $i$th
coordinate $x_i\in\submfd_i$ as a parameter.  Dupin's Theorem is the following
one.
\begin{thm} For $i\neq j$, the families $\submfd_{x_i}$ and $\submfd_{x_j}$ 
meet each other in curvature lines.
\end{thm}
\begin{proof} It suffices to show that $X_j$ is tangent to a curvature line
on $\submfd_{x_i}$, \ie, $\II^{0(i)}_{X_j}X_k=0$ for $k\notin\{i,j\}$.  This
is a multiple of $\lip{\d_{X_j}\d_{X_k}\sigma_k,u_i}$ where $u_i$ is a unit
section of $U_i$ and $\sigma_k$ is a section of $\Ln$ pulled back from
$\submfd_k$.  Now we simply observe that
$\d_{X_j}\d_{X_k}\sigma_k=\d_{X_k}\d_{X_j}\sigma_k$ since $[X_j,X_k]=0$, which
is a section of $U_k$, since $\d_{X_j}\sigma_k$ is a section of $\Ln$ (for
$j\neq k$). As $U_k$ is orthogonal to $U_i$ (for $i\neq k$), we are done.
\end{proof}

The proofs in the literature take a euclidean point of view of this
result. Ours is quite different, and thus avoids computing the Levi-Civita
connection of the induced metric.

\section{Constrained Willmore surfaces}
\label{sec:willmore}

\subsection{Definition and spectral deformation}
\label{par:willmore}

\begin{defn}\label{CWdef} An immersion $\Ln$ of a surface $\submfd$ in $S^n$
is said to be \emphdef{constrained Willmore} if there is a holomorphic
quadratic differential $q\in C^\infty(\submfd,S^2T\dual_0\submfd)$ (\ie, $\d
q=0$) such that $(\Mh^\nabla)^*(\II^0)=\ip{q,\II^0}$.  We refer to such a
holomorphic quadratic differential $q$ as a \emphdef{Lagrange multiplier} for
the constrained Willmore immersion.  If $(\Mh^\nabla)^*(\II^0)=0$ the surface
is said to be \emphdef{Willmore} or \emphdef{conformally minimal}.
\end{defn}

Let $g$ be an ambient metric of constant curvature, and let $D$ be the induced
exact Weyl connection on $\submfd$. Then, by~\eqref{eq:N-normal}
and~\eqref{eq:Rnr}--\eqref{eq:Rca}, we have
\begin{equation*}
(\Mh^\nabla)^*(\II^0)
=-\divg^{\nabla,D}\CA^D+\cip{\nr^{D,\submfd}_0,\II^0}
=\Laplace^{\nabla,D}H^g+\cip{H^g(\II^0),\II^0}
\end{equation*}
so that the constrained Willmore equation in a spaceform may be written
\begin{equation*}
\Laplace^{\nabla,D}H^g+\cip{H^g(\II^0),\II^0} = \cip{q,\II^0}.
\end{equation*}
In codimension one, $\cip{H^g(\II^0),\II^0}=|\II^0|^2 H^g=2(|H^g|^2-K^g) H^g$,
where $K^g$ is the gaussian curvature $\detm \II^g$. This yields the usual
form of the Willmore equation when $q=0$ and $g$ is flat~\cite{Wil:tcr}.

An immediate consequence of this formulation is the well known fact that
surfaces with parallel mean curvature in a spaceform are constrained Willmore
(including CMC surfaces in codimension one), while minimal surfaces are
Willmore. For this, we take $q=H^g(\II^0)$ and observe that $q$ is holomorphic
by virtue of the Codazzi equation~\eqref{eq:25}, equation~\eqref{eq:Rca}, and
the fact that $\nabla^DH^g=0$.

The following proposition provides a more general class.

\begin{prop} Let $\submfd$ be a surface in a spaceform with ambient metric
$g$. Suppose that $H^g=H_++H_-$ where $H_{\pm}$ are complex conjugate sections
of $N\dual\submfd\tens\C$ such that $(\nabla^DH_-)^{0,1}=0$ and $\cip{\nabla^D
H_-,H_-}=0$. Then $\submfd$ is constrained Willmore.
\end{prop}
\begin{proof}
Since $(\nabla^DH_-)^{0,1}=0=(\nabla^DH_+)^{1,0}$, we have
\begin{align*}
\Laplace^{\nabla,D}H^g&=\divg^{\nabla,D}\bigl((\nabla^DH_+)^{1,0}
+(\nabla^DH_-)^{1,0}+(\nabla^DH_+)^{0,1} +(\nabla^DH_-)^{0,1}\bigr)\\
&=\divg^{\nabla,D}\bigl(-(\nabla^DH_+)^{1,0}
+(\nabla^DH_-)^{1,0}+(\nabla^DH_+)^{0,1} -(\nabla^DH_-)^{0,1}\bigr)\\
&=\iI\divg^{\nabla,D}{*\nabla^D(H_+-H_-)}\\
&=\iI{*R^\nabla}\cdot(H_+-H_-).
\end{align*}
Hence, using the Ricci equation~\eqref{eq:hR}, we have
\begin{equation*}
\Laplace^{\nabla,D}H^g+\cip{H^g(\II^0),\II^0} = \cip{q,\II^0}
\end{equation*}
where $q=2H_-(\II_0)^{2,0}+2H_+(\II_0)^{0,2}$. It remains to see that $q$ (or
equivalently $q^{2,0}$) is holomorphic. Since $(\nabla^DH_-)^{0,1}=0$, it
suffices to show that the contraction of $\dbar^\nabla(\II^0)^{2,0}$ with
$H_-$ is zero. But $\dbar^\nabla(\II^0)^{2,0}=(\d^{\nabla,D}\II^0)^{2,1}$,
which is a multiple of $(\nabla^DH^g)^{1,0}\wedge\conf$ by~\eqref{eq:25}
and~\eqref{eq:Rca}. Since $(\nabla^DH^g)^{1,0}=\nabla^DH_-$, the contraction
with $H_-$ vanishes by the second equation on $H_-$.
\end{proof}
\begin{cor} Let $\submfd$ be cooriented codimension two surface in a spaceform
with ambient metric $g$, and suppose $H^g$ is holomorphic with respect to the
natural complex structure on $N^*\submfd$ induced by its conformal structure
and orientation, and the holomorphic structure induced by $\nabla^D$. Then
$\submfd$ is constrained Willmore.
\end{cor}
\begin{proof} Let $H_+=(H^g)^{0,1}$ and $H_-=(H^g)^{1,0}$ be the projections
of $H^g$ onto the eigenspaces in $N^*\submfd\tens\C$ of the complex structure.
Since $H^g$ is holomorphic, $(\nabla^D H_-)^{0,1}=0$. On the other hand, these
eigenspaces are null, so $0=\d\cip{H_-,H_-}=2\cip{\nabla^DH_-,H_-}$. Hence we
can apply the proposition.
\end{proof}
In the case that $g$ is flat, this result is due to Bohle~\cite{Boh:cws}.

One reason to be interested in constrained Willmore surfaces is that they
admit a spectral deformation (in arbitrary codimension) and hence an
integrable systems interpretation. 

\begin{thm}\label{th:Willmore-def}
Let $\Ln\subset\submfd\times\R^{n+1,1}$ be an immersed surface in $S^n$ with
\GCR/ data $(\conf,\Ms^\submfd,\nabla,\II^0)$, and quadratic differential
$q\in C^\infty(\submfd,S^2_0T\dual\submfd)$. Then the following conditions are
equivalent.
\begin{numlist}
\item $\submfd$ is constrained Willmore with Lagrange multiplier $q$.

\item $(\conf,\Ms^\submfd+\frac12(e^{2tJ}-1)q,\nabla,e^{tJ}\II^0)$ satisfies
the \GCR/ equations for some $t\notin \pi\Z$.

\item $(\conf,\Ms^\submfd+\frac12(e^{2tJ}-1)q,\nabla,e^{tJ}\II^0)$ satisfies
the \GCR/ equations and defines a constrained Willmore immersion of $\submfd$
in $S^n$ for all $t\in\R$.
\end{numlist}
\end{thm}
\begin{proof} The \GCR/ equations hold for the data in (ii)--(iii) iff
\begin{align*}
0&=C^\submfd+\half(e^{2tJ}-1)\d q+\ip{\QCoda^\nabla(\II^0)}\\
0&=(\Mh^\nabla+\half(e^{2tJ}-1)q)^*(e^{tJ}J\II^0)\\
&= \cos t (\Mh^\nabla)^*(J\II^0)-\sin t (\Mh^\nabla)^*(\II^0)
+\half\cip{(e^{2tJ}-1)q,e^{tJ}J\II^0}\\
0&=R^\nabla-\II^0\wedge\Sh^0
\end{align*}
(using~\eqref{eq:mh-dual-q}), and we observe that
$\half\cip{(e^{2tJ}-1)q,e^{tJ}J\II^0}=\sin t \cip{q,\II^0}$.  Hence the \GCR/
equations hold for all $t$ iff they hold for some $t\notin \pi\Z$ iff $\d q=0$
and $(\Mh^\nabla)^*(\II^0)=\cip{q,\II^0}$.
\end{proof}

One way to realize this integrable system is to lift the spectral deformation
to a family of flat connections. This yields an alternative definition of
constrained Willmore surfaces.

\begin{thm}\label{th:Willmore-equiv}
Let $\Ln\subset\submfd\times\R^{n+1,1}$ and
$(\conf,\Ms^\submfd,\nabla,\II^0,q)$ be as in the previous Theorem.  Then the
following conditions are equivalent.
\begin{numlist}
\item $\submfd$ is constrained Willmore with Lagrange multiplier $q$.

\item There is a $T\dual\mfd$-valued $1$-form $\xi\in\Omega^1(\submfd,
\stab(\Ln)^\perp)\subset \Omega^1(\submfd,\so(\CV_\mfd))$ with $\hc{\xi}=q$
such that $\dDh\xi=0$ and $\d^{\CDh-\xi,\nabla}{*\CS}=0$.

\item There is a $T\dual\submfd$-valued $1$-form $\xi\in\Omega^1(\submfd,
\stab(\Ln)^\perp\cap\so(\CV_\submfd))\subset
\Omega^1(\submfd,\so(\CV_\submfd))$ with $\hc{\xi}=q$ such that
$\d+(e^{tJ}-1)\CS+\half(e^{2tJ}-1)\xi$ is flat for all $t\in\R$.
\end{numlist}
\end{thm}
\begin{proof} The constrained Willmore equation may be written
$0=(\Mh^\nabla-q)^*\II^0 =-(\Mh^\nabla-q)^* J^2\II^0$. Now by
Proposition~\ref{p:N-lift}, $*\CS$ is the differential lift of $J\II^0$ with
respect to $\CD^\submfd$, hence also with respect to $\CD^\submfd-\xi$
provided $\xi$ is a $T\dual M$-valued $1$-form.  It follows that the
constrained Willmore equation is equivalent to $\d^{\CDh-\xi,\nabla}{*\CS}=0$,
where $\xi$ is the differential lift of $q$ (with respect to $\CD^\submfd$),
the equation $\dDh\xi=0$ then being equivalent to $\d q=0$. The rest follows
by computing the curvature of the connections in (iii), and applying
Proposition~\ref{p:eta-and-q}.
\end{proof}

\subsection{Harmonicity, the functional and duality}
\label{par:harmonic}

Theorem~\ref{th:Willmore-equiv}, with $q=\xi=0$, has two more-or-less
immediate corollaries.  It also provides a simple way to obtain the
constrained Willmore equation from the functional. These results are all well
known~\cite{BFLPP:cgs,Bry:dtw}, but they are derived with particular ease in
our formalism.

\begin{cor} A immersed surface in $S^n$ is Willmore if and only if its
central sphere congruence is harmonic.
\end{cor}
\begin{proof} The central sphere congruence is harmonic if and only if
$\d^{\CDs}{*\CS}=0$. However, since $Q$ is tracelike and $\II^0$ is
trace-free, $\d^{\CDs}{*\CS} =\dDh{*\CS}$.
\end{proof}

\begin{cor} If the central sphere congruence of a Willmore surface has
a second envelope \textup(which is generic in codimension one\textup) then
this envelope is also a Willmore surface.
\end{cor}
\begin{proof} Since $\dDh\CS=0$ and $\dDh{*\CS}=0$ we have that
$\nr^{D,\CV_\submfd}\wedge\Sh^0=0$ and $\nr^{D,\CV_\submfd}\wedge
J\Sh^0=0$. Away from umbilics (\ie, on a dense open set because $\submfd$ is
not totally umbilic, since it has a second envelope, and $\II^0$ is analytic)
it follows that the symmetric trace-free part of $\nr^{D,\CV_\submfd}$ is
zero. Hence the second envelope has the same conformal structure and the same
central sphere congruence. Since this is harmonic, the second envelope is also
Willmore.
\end{proof}

\begin{thm} A compact surface is \textup(constrained\textup) Willmore, as
in Definition~\textup{\ref{CWdef}}, if and only if it is a critical point of
$\cW$ for immersions of $\submfd$ \textup(inducing the same conformal metric
on $\submfd$ in the constrained case\textup).
\end{thm}
\begin{proof} Since $\CS$ vanishes on $\Ln$, $\cW=\int_\submfd |\CS|^2
=\int_\submfd |\d S|^2$, where $S$ is the central sphere congruence (viewed as
a map from $\submfd$ into the grassmannian of $(3,1)$-planes in $S^n$). The
first variation of this functional is $\d\cW(\dot S)=\int_\submfd
\ip{\d^{\CDs}{*\CS},\dot S}$, where $\dot S\colon\submfd\to\m$ denotes the
variation of $S$. Now $\d^{\CDs}{*\CS} =\dDh{*\CS}$ and, by
Proposition~\ref{p:N-lift}, $\LH\dDh{*\CS}=0$, so that
$\CS\in\Omega^2(\submfd,\m\cap\stab(\Ln)^\perp)$ and hence $\d\cW(\dot S)$
vanishes automatically when $\dot S$ takes values in
$C^\infty(\submfd,\m\cap\stab(\Ln))=\image\LH$ (this is really just a special
case of Poincar\'e duality for Lie algebra homology). Hence $\d\cW(\dot
S)=\int_\submfd \ip{\hc{\dDh{*\CS}},\hc{\dot S}}=\int_\submfd
\ip{\d_{BGG}(J\II^0),\hc{\dot S}}$, with $\hc{\dot S}\in
C^\infty(\submfd,N\submfd)$ and $\d_{BGG}(J\II^0)$ is
$*(\Mh^\nabla)^*(\II^0)$, up to a sign.

For variations of $S$ coming from variations of the immersion, $\hc{\dot S}$
is essentially the normal variation vector field. For general variations,
$\hc{\dot S}$ can therefore be arbitrary, whereas for conformal variations are
characterized by $\Sh^0\hc{\dot S}\in C^\infty(\submfd,\sym_0(T\submfd))$
being in the image of the conformal Killing operator $X\mapsto \cL_X\conf$.
This last is adjoint to $\d$ on quadratic differentials. The result now
follows.
\end{proof}

\section{Isothermic surfaces}
\label{sec:isothermic}

\subsection{Definition and spectral deformation}
\label{par:isothermic}

The classical definition of an isothermic surface is a surface admitting a
conformal curvature line coordinate, \ie, a conformal coordinate $z=x+\iI y$
with respect to which the second fundamental form is diagonal. This is a
conformally-invariant condition, and there is a manifestly invariant way to
state it, which we take as our fundamental definition.

\begin{defn} An immersion $\Ln$ of a surface $\submfd$ in $S^n$ is said to be
\emphdef{isothermic} if there is a nonzero quadratic differential $q\in
C^\infty(\submfd,S^2T\dual_0\submfd)\subset\Omega^1(\submfd,T\dual\submfd)$
which is holomorphic (\ie, $\d q=0$) and commutes with shape operators in the
sense that $q\wedge\Sh^0=0$.
\end{defn}

\begin{rems} Thus $q^{(2,0)}=dz^2$ for some conformal curvature line
coordinates $z=x+\iI y$.

This definition has the merit of using just the \GCR/ data, which only
determine $\submfd$ locally up to \mob/ transformation. Hence, in this
approach, one can develop a theory of isothermic surfaces modulo \mob/
transformation.

Also note that because the normal components of $\Sh^0$ commute with $q$, they
commute with each other, and so the weightless normal bundle is flat by the
Ricci equation.
\end{rems}

We now show that the new \GCR/ data $(\conf,\Ms^\submfd+rq,\nabla,\II^0)$ also
define an isothermic surface for any $r$, with the same quadratic differential
$q$ (up to a constant multiple). The new surfaces are sometimes called the
\emphdef{T-transforms} of $\submfd$: note that only the \mob/ structure has
changed here.

\begin{thm}\label{th:isothermic-def}
Let $\Ln\subset\submfd\times\R^{n+1,1}$ be an immersed surface in $S^n$ with
\GCR/ data $(\conf,\Ms^\submfd,\nabla,\II^0)$, and quadratic differential
$q\in C^\infty(\submfd,S^2_0T\dual\submfd)$. Then the following conditions are
equivalent.
\begin{numlist}
\item $\Ln$ is isothermic with holomorphic commuting quadratic differential
$q$.

\item $(\conf,\Ms^\submfd+rq,\nabla,\II^0)$ satisfies the \GCR/ equations for
some $r\neq 0$.

\item $(\conf,\Ms^\submfd+rq,\nabla,\II^0)$ satisfies the \GCR/ equations and
defines a isothermic immersion of $\submfd$ in $S^n$ for all $r\in\R$.
\end{numlist}
\end{thm}
\begin{proof} The \GCR/ equations hold for the data in (ii)--(iii) iff
\begin{align*}
0&=C^\submfd+r\d q+\ip{\QCoda^\nabla(\II^0)}\\
0&=(\Mh^\nabla)^*(J\II^0)+r\ip{q,J\II^0}\\
0&=R^\nabla-\II^0\wedge\Sh^0
\end{align*}
(using~\eqref{eq:mh-dual-q}). Clearly these equations hold for all $r$ iff
they hold for $r=0$ and some nonzero $r$, since then $\divg q=0$ and $q$ is
orthogonal to $J\II^0$: since $S^2_0T\dual\submfd$ has rank $2$ this last
condition means that every component of $\II^0$ commutes with $q$.
\end{proof}

The T-transforms give rise an integrable systems interpretation of isothermic
surfaces: the \GCR/ data $(\conf,\Ms^\submfd+rq,\nabla,\II^0)$ defines a
pencil of flat connections on $\CV\low_\submfd\dsum\CV_\submfd^\perp$. It also
yields an alternative definition of an isothermic surface.

\begin{thm}\label{th:isothermic-equiv}
Let $\Ln\subset\submfd\times\R^{n+1,1}$ and
$(\conf,\Ms^\submfd,\nabla,\II^0,q)$, be as in the previous Theorem. Then the
following conditions are equivalent.
\begin{numlist}
\item $\Ln$ is isothermic with holomorphic commuting quadratic differential
$q$.

\item There is a $T\dual\mfd$-valued $1$-form $\eta\in\Omega^1(\submfd,
\stab(\Ln)^\perp)\subset \Omega^1(\submfd,\so(\CV_\mfd))$ with $\hc{\eta}=q$
such that $\d\eta=0$.

\item There is a $T\dual\submfd$-valued $1$-form $\eta\in\Omega^1(\submfd,
\stab(\Ln)^\perp\cap\so(\CV_\submfd))\subset
\Omega^1(\submfd,\so(\CV_\submfd))$ with $\hc{\eta}=q$ such that $\d+r\eta$ is
flat for all $r\in\R$.
\end{numlist}
\end{thm}
\begin{proof} Since $\d$ is flat and $T\dual\mfd$ is abelian, $\d+r\eta$ is
flat if and only if $\d\eta = 0$. This condition certainly implies that
$\LH\dDh\eta=0$ so that $\eta=j^\h q=q$ by Proposition~\ref{p:eta-and-q}. Now
$\d\eta=0$ if and only if $\d^{\CD^\submfd}\eta=0$ and
$\liebrack{\CS\wedge\eta}=0$, \ie, if and only if $\d q=0$ and
$q\wedge\Sh^0=0$.
\end{proof}
In subsequent work, we shall focus on (ii) above as our main definition of an
isothermic surface (this definition also generalizes~\cite{BDPP:issr}). The
classical definition of isothermic surfaces is essentially a homological
version of this `integrable systems' definition.  In the formulation given by
(iii), the T-transform of $\Ln$, with parameter $r$, is $\Ln$ itself, but
viewed as a (local) immersion in the space of $(\d+r\eta)$-parallel null lines
in $\submfd\times\R^{n+1,1}$.  The corresponding map into the space of
$\d$-parallel null lines is obtained from this via a gauge transformation.

\subsection{Examples and symmetry breaking}
\label{par:iso-examples}

The examples of isothermic surfaces which are easy to describe fall into three
classes, all of which exhibit some sort of symmetry breaking in the sense of
\S\S\ref{par:tangent-cong}--\ref{par:sym-break-sub}.

\subsubsection*{Cones, cylinders, revolutes and their generalizations}

The simplest examples of isothermic surfaces in $S^3$ are cones over a curve
in $\cS^2$, cylinders over a curve in $\R^2$, and revolutes (surfaces of
revolution) over a curve on $\cH^2$. These are the $2$-dimensional case of the
submanifolds described in \S\ref{par:sym-break-sub}. In higher codimension, we
obtain a larger class than just the cones, cylinders and revolutes in this
way. We focus on the nondegenerate case (generalizing cones and revolutes),
but the degenerate case is similar.

Suppose that
\begin{equation*}
\R^{n+1,1}= \R^{k,1} \dsum \R^{n-k+1}
\end{equation*}
and so $S^n\setdif S^{k-1}= \cH^k\times \cS^{n-k}$. We claim that any
Cartesian product of curves $\gamma_1(s)$ in $\cH^k$ and $\gam_2(t)$ in
$\cS^{n-k}$ is isothermic.  The product surface
$\submfd=\submfd_1\times\submfd_2$ has a natural lift
$\sigma(s,t):=\gam_1(s)+\gam_2(t)$ into the light-cone.  Clearly $s$ and $t$
are curvature line coordinates, and if $\gam_1$ and $\gam_2$ are parameterized
by arc-length, then, differentiating $\sigma(s,t)$, we see that they are also
conformal.

We observe that we can take $\eta=\sigma\skwend(d\gam_1-d\gam_2)$, where
$\skwend$ denotes the pairing $\R^{n+1,1}\times\R^{n+1,1}\to\so(n+1,1)$ with
$(u\skwend w) (v)=\lip{u,v}w-\lip{u,w}v$. Indeed $\eta$ takes values in the
nilradical of $\Ln$, so it remains to show that it is closed. Well:
\begin{equation*}
d\eta=(d\gamma_1+d\gamma_2)\skwendwedge (d\gamma_1-d\gamma_2)
\end{equation*}
(wedge product, using $\skwend$ to multiply coefficients). Now
$d\gamma_i\skwendwedge d\gamma_i=0$ since these are wedges of (pullbacks of)
$1$-forms on $1$-manifolds so we are left with $-d\gamma_1\skwendwedge
d\gamma_2+d\gamma_2\skwendwedge d\gamma_1$, which vanishes because the wedge
product of $1$-forms and $\skwend$ are both skew.

This class of surfaces is stable under T-transform: for this we need to find
$(\d+r\eta)$-stable subspaces $U_1$ and $U_2=U_1^\perp$ of
$\submfd\times\R^{n+1,1}$ with ranks $k+1$ and $n-k+1$, such that $\sigma$
induces the same splitting of $\submfd=\submfd_1\times\submfd_2$. To do that
we make an Ansatz that
$U_1=\image\d\gam_1\dsum\vspan{\gam_1+a(r)(\gam_1+\gam_2)}$ and hence
$U_2=\image\d\gam_2\dsum\vspan{\gam_2+a(r)(\gam_1+\gam_2)}$ (with $a(r)$ a
constant for fixed $r$). Let $X_1$ and $X_2$ be vector fields on $\submfd_1$
and $\submfd_2$ respectively, pulled back to the product. Then we compute
\begin{equation*}
(\d+r\eta)\low_{X_2}(\gam_1+a(r)(\gam_1+\gam_2))=a(r)\d\low_{X_2}\gam_2-
r\lip{\gam_1,\gam_1}\d\low_{X_2}\gam_2.
\end{equation*}
This lies in $U_1$ iff it is zero, which (assuming $\lip{\gam_1,\gam_1}=-1$)
gives $a(r)=-r$. For the stability of $U_1$, it remains to consider
$v\in\image\d\gam_1$:
\begin{equation*}
(\d+r\eta)v\mod\image\d\gam_1=-\lip{\d v,\gam_1}\gam_1
-r\lip{\d\gam_1,v}(\gam_1+\gam_2)= \lip{v,\d\gam_1}(\gam_1-r(\gam_1+\gam_2)),
\end{equation*}
which lies in $U_1$ as required.  The stability of $U_2$ follows by
orthogonality. Notice that $U_1$ and $U_2$ become degenerate and intersecting
at $r=1/2$, then change signature, so that at $r=1$ we recover the original
product in reverse order.

We remark that Tojeiro's definition~\cite{Toj:ise} of isothermic submanifolds
of higher dimension generalizes these examples: his submanifolds are extrinsic
products of curves in spaceforms.

\subsubsection*{CMC surfaces in spaceforms}

Another well-known (and classical) class of isothermic surfaces are the CMC
surfaces in $\R^3$ (\ie, surfaces of constant mean curvature $H$ with respect
to the induced metric, also called $H$-surfaces) or in $3$-dimensional
spaceforms~\cite{Bia:rsi,Bia:crsi}. The induced metric identifies $\II^0$ with
a quadratic differential called the Hopf differential, and the Codazzi
equation (in the spaceform) shows that this is holomorphic (and of course, it
commutes with itself).

In higher codimension, the mean curvature is a conormal vector, and the most
natural notion of `constant mean curvature' is parallel mean curvature (using
the connection on the conormal bundle induced by the spaceform metric), but
this essentially forces $\submfd$ to have codimension one or two; however, as
shown in~\cite{Bur:is}, a more general class of surfaces, called
\emphdef{generalized $H$-surfaces}, are isothermic. We present a novel
analysis of this case.

Suppose the spaceform geometry in $S^n$ is defined by $v_\infty\in\R^{n+1,1}$
and let $\sigma\in C^\infty(\submfd,\Ln)$ be the spaceform lift of a
submanifold of $S^n$ (so $\lip{\sigma,v_\infty}=-1$), which is the gauge
induced by the spaceform metric $g$. Let $\wnv$ be a parallel unit section of
$\CV_\submfd^\perp$. Now set $\eta=\sigma\skwend \d\wnv$. This clearly takes
values in the nilradical of $\Ln$, so it remains to determine when it is
closed, \ie, when $\d\sigma\skwendwedge \d\wnv=0$.  For this, let
$\CV_g=\exp(H^g)\CV_\submfd$ be the tangent sphere congruence in the
spaceform, so that $\wnv_g=\exp(H^g)\wnv=\wnv+\ip{H^g\ell,\wnv}\sigma$, where
$\ell$ is dual to $\sigma$. Thus $\d\wnv=\d\wnv^g-\d(\ip{H^g\ell,\wnv}\sigma)$
and $\d\wnv^g=\CS^g\wnv=-\Sh^g\wnv=-\Sh^0\wnv-\ip{H^g\ell, \wnv}\d\sigma$,
since $A^g=0$.  Now $\d\sigma\skwendwedge\Sh^0\wnv$ is identically zero, so
$\d\sigma\skwendwedge \d\wnv=0$ if and only if $\d\ip{H^g\ell,\wnv}=0$, \ie,
the $\wnv$ component of the mean curvature covector $H^g$ is constant.

These are the generalized $H$-surfaces. The quadratic differential $q$ is the
corresponding ($\wnv$) component of $\II^0$, although it is not so obvious to
see directly that this commutes with the other components.  Observe that by
construction $\eta\act v_\infty+\d\wnv=0$. It follows that for all $r\in\R$,
$v_\infty+r\wnv$ is parallel with respect to $\d+r\eta$. This conserved
quantity provides another way to characterize generalized $H$-surfaces, and we
shall explore this further in the sequel to this paper.

\subsubsection*{Quadrics}

Quadrics are natural generalizations of spheres in euclidean space. After the
study of surfaces of constant nonzero gaussian curvature ($K$-surfaces), which
are isometric to spheres or their hyperbolic analogues, surfaces isometric to
quadrics were of great interest in classical differential geometry. An
important classical observation is that quadrics are isothermic
(see~\cite{Eis:tos}), and this was generalized to conformal quadrics by
Darboux~\cite[pp. 212--219]{Dar:socc}.  This is often considered to be
surprising~\cite{Bur:is,Her:imdg}, since quadric surfaces live most naturally
in the projective geometry of $\R P^3$, not the conformal geometry of
$S^3$. Darboux's argument fits well into our formalism, so we sketch it here:
in fact Darboux restricts attention to diagonal quadrics, so we complement his
analysis by describing some of the geometry in the general case.

A quadric $Q$ in $S^3$ is determined by a matrix $A\in\sym(\R^{4,1})$, which
is not a multiple of the identity: then $Q=\{\vspan{\sigma}\in \PL:
\lip{\sigma,\sigma}=0=\lip{A\sigma,\sigma}\}$.  Conversely, the quadric $Q$
only determines $A$ up to affine transformations $A\mapsto aA+bI$ ($a,b\in\R$,
$a\neq 0$). The geometric idea behind Darboux's construction is to extend this
affine action to a projective one: the matrices $(aA+bI)(cA+dI)^{-1}$
($a,b,c,d\in\R$, $ad-bc\neq 0$) determine a one parameter family of quadrics.
(In fact these quadrics are confocal, but we do not need this here.)  Given
$A$, we can parameterize the family by the matrices $A(I-uA)^{-1}$.

Now observe that for fixed $\sigma$,
$\detm(I-uA)\lip{A(I-uA)^{-1}\sigma,\sigma}=0$ is a quartic in $u$, whose
leading coefficient vanishes if $\lip{\sigma,\sigma}=0$. The roots
$u_1,u_2,u_3$ of this cubic determine three quadrics in the family passing
through $\vspan{\sigma}\in \PL$. Now observe that the conditions
\begin{equation*}
\lip{A(I-u_1A)^{-1}\sigma,\sigma}=0=\lip{A(I-u_2A)^{-1}\sigma,\sigma}
\end{equation*}
imply that
\begin{align*}
0&=
\Lip{\bigl(A(1-u_2A)-A(1-u_1A)\bigr)(I-u_1A)^{-1}(I-u_2A)^{-1}\sigma,\sigma}\\
&=(u_1-u_2)\lip{A(I-u_1A)^{-1}\sigma, A(I-u_2A)^{-1}\sigma}
\end{align*}
so that distinct quadrics in the family meet each other orthogonally at
$\vspan{\sigma}\in \PL$.  By Dupin's Theorem, they therefore meet in
curvature lines.

To see that this gives rise to conformal curvature line coordinates, we now
follow Darboux and assume $A$ is diagonalized, with distinct eigenvalues
$(1/a_0,1/a_1,1/a_2,1/a_3,1/a_4)$ in coordinates $x=(x_0,x_1,x_2,x_3,x_4)$ on
$\R^{4,1}$ such that $\lip{x,x}=-x_0^2+x_1^2+x_2^2+x_3^2+x_4^2$. Thus
$A(I-uA)^{-1}$ has eigenvalues $1/(a_i-u)$. Then, for $x$ null, we can solve
for $x_i$ in terms of the parameters $(u_1,u_2,u_3)$ of the three quadrics
passing through that point to get $x_0^2=-c(u_1-a_0)(u_2-a_0)(u_3-a_0)/f'(a_0)$
and $x_i^2=c(u_1-a_i)(u_2-a_i)(u_3-a_i)/f'(a_i)$ for $i=1,\ldots 4$, where
$f(u)=(u-a_0)(u-a_1)(u-a_2)(u-a_3)(u-a_4)$ and $c$ is an overall scale.

By differentiating these formulae for $x_i$, it follows that the conformal
metric on $S^3$ has a representative metric
\begin{equation}\label{eq:S3-metric}
\frac{(u_2-u_1)(u_3-u_1)}{f(u_1)}du_1^2+\frac{(u_3-u_2)(u_1-u_2)}{f(u_2)}du_2^2
+\frac{(u_1-u_3)(u_2-u_3)}{f(u_3)}du_3^2
\end{equation}
in these coordinates. Thus $(u_1,u_2)$ are conformal curvature line
coordinates on the quadrics with $u_3$ constant, which are therefore
isothermic.

If $A$ is not diagonalizable as above, the coordinates $(u_1,u_2,u_3)$ are
still well defined, and so a continuity argument gives the same representative
metric~\eqref{eq:S3-metric} for the conformal structure on $S^3$, where $f(u)$
is the characteristic polynomial of $A^{-1}$.

This argument almost provides the conceptual explanation sought
in~\cite{Bur:is}: it shows that conformal geometry equips a quadric with a
natural conjugate net; then, perhaps not so surprisingly, this net is
orthogonal, and gives rise to the conformal curvature line coordinates which
show that the quadric is isothermic.

We end by remarking that the eigenvectors of $A$ can be used to break the
symmetry to that of a spaceform, and then the quadric is realized as a
quadric in a spaceform. In particular, for matrices with a null eigenvector,
we recover the quadrics in euclidean space.

\subsection{When is a surface isothermic?} 
\label{par:when-isothermic}

The definition we have given of an isothermic surface has the disadvantage
that it does not characterize when a surface admits such a quadratic
differential $q$. There is also a subtle global issue here: $q$ might exist
locally (near any point, perhaps only in a dense open set), but not globally
on $\submfd$. However, $q$ is unique up to a constant multiple away from
umbilic points: it must be a pointwise multiple of any nonvanishing component
$\ip{\II^0,\wnv}$ of $\II^0$ and this multiple is determined up to a constant
since $q$ is holomorphic. Hence $q$ is globally defined on the universal cover
of the open subset $U$ where $\II^0$ is non-vanishing (or equivalently it may
be viewed as a holomorphic quadratic differential with values in a flat line
bundle $\cL$ on $U$).

In the case that the central sphere congruence has a second envelope, \ie,
there is a Weyl derivative $D^0$ with $0=\CA^{D^0}=-\divg^{\nabla,D^0}\Sh^0$,
an intrinsic characterization is available. Writing $\Sh^0=S\vtens\wnv$ for a
weightless unit normal vector $\wnv$ we first see that such a second envelope
can only exist if $\wnv$ is parallel, and then $D^0$ is the Weyl derivative of
a second envelope iff $\divg^{D^0} S=0$. If the weightless normal bundle is
flat, it follows that $\CV\dsum\vspan{\wnv}$ is constant, and so can assume
without loss that we are in codimension one, with $S=\Sh^0$.

Since $A^{D+\gam}=A^D+\gam\circ\Sh^0$, it then follows that such a second
envelope exists precisely on the set where either $\CS=0$ or $\Sh^0\neq 0$. In
the latter case, the second envelope is unique, and $D^0$ can be determined
from an arbitrary Weyl structure $D$ via $D+(\Sh^0)^{-1}\divg^D \Sh^0$.

Now any commuting quadratic differential may be written $q=\sigma\Sh^0$, and
it is holomorphic if and only if $D^0\sigma=0$.  Such a (nonzero) $\sigma$
exists if and only if $D$ is exact. This turns out to be the classical
assertion that
\begin{quote}\itshape
a surface is isothermic iff its central sphere congruence is Ribaucour.
\end{quote}
Furthermore, this gives an equation for isothermic surfaces. Let $D$ be the
unique Weyl derivative with $D(|\Sh^0|^2)=0$. Then $D$ is exact and so, since
$(\Sh^0)^{-1}=2\Sh^0/|\Sh^0|^2$, $D^0=D+2\Sh^0\divg^D \Sh^0/|\Sh^0|^2$, which
is locally exact if and only if $\d (\Sh^0\divg^D \Sh^0/|\Sh^0|^2)=0$.  This
may be written in a conformally invariant way as
$\d(*\ip{\cB^\nabla(\II^0)}/|\II^0|^2)=0$ (simply expand the definition of
$\cB$ using the Weyl derivative $D$).

\section{M\"obius-flat submanifolds}
\label{sec:mflat}

\subsection{Definition and spectral deformation}

Let $\submfd$ be a submanifold of $S^n$ with flat weightless normal bundle. We
wish to impose the condition that there is a enveloped sphere congruence $\CV$
inducing a flat conformal \mob/ structure on $\submfd$. For $m\geq3$, all
enveloped sphere congruences induce the same conformal \mob/ structure on
$\submfd$, since they induce the same conformal structure. On the other hand
by equation~\eqref{eq:MS-formula}, when $m=2$, the conformal \mob/ structure
induced by $\CV$ has $\Mh^\CV=\Mh^\submfd-q$ with $q=H^\CV(\II^0)$, hence
it is flat iff $q$ is a \emph{Cotton--York potential}, \ie, $\d
q=C^\submfd$. Since the weightless normal bundle is flat $q$ \emph{commutes}
with the shape operators in the sense that $q\wedge\Sh^0=0$. (Conversely, if
$q$ is nonzero, the flatness of $\nabla$ follows from this.) 

This motivates the following definition.

\begin{defn} An immersion $\Ln$ of $\submfd$ ($\dimn\submfd\geq 2$) in $S^n$ is
called \emphdef{\mob/-flat} if it has flat weightless normal bundle, and
either $m\geq3$ and $\submfd$ is conformally flat, or $m=2$ and $\submfd$
admits a \emphdef{commuting Cotton--York potential} (CCYP), \ie, there exists
$q\in C^\infty(\submfd,S^0T\dual\submfd)$ with $\d q=C^\submfd$ and
$q\wedge\Sh^0=0$.  (For $m\geq3$, we set $q=0$.) 
\end{defn}

Conversely, if $\submfd$ is \mob/-flat, then for $m\geq3$ any enveloped sphere
congruence induces this flat conformal structure; for $m=2$, an enveloped
sphere congruence $\CV$ induces the flat conformal \mob/ structure
$\Mh^\submfd-q$ iff $q=H^\CV(\II^0)$. If $q=0$, we can take $\CV$ to be the
central sphere congruence; otherwise, on the open set where $q$ is nonzero, we
can write $\II^0=q\tens\nmv$ for a normal vector $\nmv$ and set
$H^\CV=\ip{\nmv,\cdot}/\ip{\nmv,\nmv}$. In codimension one the enveloped
sphere congruence inducing $\Mh^\submfd-q$ is clearly unique. (Note that
$H^\CV$ extends by zero to the zeros of $q$ unless $\II^0$ is also zero
there.) 

\mob/-flat submanifolds admit a spectral deformation and hence an
integrable systems interpretation.

\begin{thm}\label{th:mob-flat-def}
Let $\Ln\subset\submfd\times\R^{n+1,1}$ be a submanifold of $S^n$ of dimension
$m\geq2$ with \GCR/ data $(\conf,\Ms^\submfd,\nabla,\II^0)$ and a quadratic
differential $q\in C^\infty(\submfd,S^2_0T\dual\submfd)$ with $q=0$ if
$m\geq3$.  Then the following conditions are equivalent.
\begin{numlist}
\item $\submfd$ is \mob/-flat \textup(with CCYP $q$ if $m=2$\textup).

\item $(\conf,\Ms^\submfd+(t^2{-}1)q,\nabla,t\II^0)$ satisfies the \GCR/
equations for some $t\notin\{0, \pm1\}$.

\item $(\conf,\Ms^\submfd+(t^2{-}1)q,\nabla,t\II^0)$ satisfies the \GCR/
equations and defines a \mob/-flat immersion of $\submfd$ in $S^n$
\textup(with CCYP $t^2 q$ if $m=2$\textup) for all $t\in\R$.
\end{numlist}
\end{thm}
\begin{proof} The \GCR/ equations hold for the data in (ii)--(iii) iff
\begin{align*}
0&=W^\submfd-t^2\abrack{\iden\wedge\QC}-t^2\Sh^0\wedge\II^0&&m\geq4\\
0&=C^\submfd+(t^2-1)\d q+t^2\ip{\QCoda^\nabla(\II^0)}&&m=2,3\\
0&=t\CCoda^\nabla\II^0&& m\geq3\\
0&=t(\Mh^\nabla)^*(J\II^0)+t(t^2-1)\ip{q,J\II^0}&&m=2\\
0&=R^\nabla-t^2\II^0\wedge\Sh^0&&
\end{align*}
(using~\eqref{eq:mh-dual-q}). These hold for all $t$ iff they hold for $t=1$,
$W^\submfd=0$, $C^\submfd=\d q$, $R^\nabla=0$ and $q\wedge\Sh^0=0$, and it
suffices that the equations hold for $t=1$ and some $t\notin\{0,\pm 1\}$.
\end{proof}

As in the isothermic case, the integrable systems viewpoint is made most
transparent in terms of an associated family of flat connections. This also
reveals an alternative definition which will be very useful to us in the
sequel to this paper.

\begin{thm}\label{th:mob-flat-equiv}
Let $\Ln\subset\submfd\times\R^{n+1,1}$ and
$(\conf,\Ms^\submfd,\nabla,\II^0,q)$ be as in the previous Theorem. Then the
following conditions are equivalent.
\begin{numlist}
\item $\submfd$ is \mob/-flat \textup(with CCYP $q$ if $m=2$\textup).

\item There is a sphere congruence $\CV$ enveloped by $\Ln$ and a $1$-form
$\mff^\CV\in\Omega^1(\submfd,\stab(\Ln)^\perp)$ such that $\d \mff^\CV =
\RDVs$.

\item For any sphere congruence $\CV$ enveloped by $\Ln$, there is
$\mff^\CV\in\Omega^1(\submfd,\stab(\Ln)^\perp\cap \so(\CV))$ such that
$\d_t^\CV:=\CDVs+t\CS^\CV+(t^2-1)\mff^\CV$ is flat for all $t\in\R$ \textup(or
equivalently, since $\d_1$ is flat, $\CDVs-\mff^\CV$ is flat and
$\liebrack{\CS^\CV\wedge\mff^\CV} =0$\textup).
\end{numlist}
When $m=2$, $\mff^\CV$ is a lift of $q-H^\CV(\II^0)$ for a CCYP $q$.
\textup(Notice that the condition in \textup{(ii)} only needs to be
verified for one sphere congruence $\CV$\textup; then the---apparently
stronger---condition in \textup{(iii)} holds for all $\CV$.\textup)
\end{thm}
\begin{proof} We prove (i)$\Rightarrow$(iii)$\Rightarrow$(ii)$\Rightarrow$(i).
Clearly (iii)$\Rightarrow$(ii) by taking $t=0$ and observing that
\begin{equation*}
\d \mff^{\CV} = \dDVs \mff^\CV + \liebrack{\CS^\CV\wedge\mff^\CV},
\end{equation*}
whereas the curvature of $\CDVs-\mff^\CV$ is $\RDVs-\dDVs \mff^\CV$.

Suppose $\d \mff^\CV = \RDVs$ with $\mff^\CV$ as in (ii). By considering the
$\m$-component of this equation, we see that
$\liebrack{\beta\wedge(\mff^\CV)^\perp}=0$, hence $\mff^\CV\in\so(\CV)$ and
$\CD^\CV-\mff^\CV$ is a flat connection on $\CV$. Since
$\exp(-H_\CV)\act(\CD^\CV-\mff^\CV)$ is a flat (hence normal) Cartan
connection on $\CV_\submfd$ inducing the same soldering form as $\CD^\submfd$,
it must equal $\CD^\submfd-q$, where $q=0$ if $m\geq3$ and $q$ is a quadratic
differential for $m=2$. Hence $\mff^\CV=\QC^\CV+q$ and $q$ is the given
quadratic differential by assumption when $m=2$.

Conversely, given (i) we can define $\mff^\CV=\QC^\CV+q$ so that
$\exp(H_\CV)\act (\CD^\submfd -q) =\CD^\CV-\mff^\CV$.

To complete the proof of (ii)$\Rightarrow$(i)$\Rightarrow$(iii), we note that
$\QC^\CV\wedge \Sh^0=\QC\wedge\Sh^0=0$ since the weightless normal bundle is
flat (the latter equality follows either from the explicit formula for $\QC$
or by computing that $\Quabla_\h \liebrack{\QC\wedge\CS}=0$ so that
$\liebrack{\QC\wedge\CS}$ is the differential lift of its homology class,
which is zero because $\liebrack{\QC\wedge\CS}\in\image\LH$). Now it is clear
that $\liebrack{\CS^\CV\wedge\mff^\CV}=0$ iff $q\wedge\Sh^0=0$, and
$\CDVs-\mff^\CV$ is flat iff the normal bundle is flat and $\CD^\submfd-q$ is
flat, \ie, $m\geq3$ and $\submfd$ is conformally-flat or $m=2$ and $\d
q=C^\submfd$.  Since
\begin{equation*}
R^{\d_t^\CV} = (1-t^2)\bigl( R^{\CDVs} - \d^{\CD^\CV} \mff^\CV
-t  \liebrack{\CS^\CV\wedge\mff^\CV}\bigr),
\end{equation*}
the equivalence of these conditions with the flatness of $\d_t^\CV$ is
immediate.
\end{proof}

\subsection{General theory and examples}

The following observation makes precise the idea that \mob/-flat submanifolds
are those with a enveloped sphere congruence inducing a flat \mob/ structure,
and leads to many examples of \mob/-flat submanifolds.

\begin{thm}\label{th:mob-flat-via-sphere-cong}
Let $\submfd$ be an $m$-submanifold of $S^n$ with flat weightless normal
bundle.  Then $\submfd$ is \mob/-flat \textup(with CCYP $q$ if
$m=2$\textup) if and only if* there is an enveloped sphere congruence $\CV$
\textup(with $q=H^\CV(\II^0)$ if $m=2$\textup) such that $\Sh^\CV\wedge
\II^\CV+\abrack{\iden\wedge\QC^\CV}=0$ \textup(which is automatic for $m\leq
3$\textup) and for some \textup(hence any\textup) Weyl derivative $D$, we have
\begin{align}
\d^D \QC^\CV +\CA^{D,\CV}\wedge \II^\CV&=0 & (m=3),\\
\tfrac12 DK^\CV\wedge\conf &=\CA^{D,\CV}\wedge \II^\CV & (m=2).
\end{align}
For $m\geq3$, it then follows that any enveloped sphere congruence satisfies
these conditions.  \textup(*The `only if' part holds only away from umbilic
points when $m=2$.\textup)
\end{thm}
\begin{proof} Let $\CV$ be an enveloped sphere congruence. Then $\submfd$ is
\mob/-flat if and only if it has flat weightless normal bundle, $W^\submfd=0$
(this equation is vacuous for $m\leq 3$) and $C^\submfd=\d q$ for $m\leq 3$.
Now $W^\submfd=\Sh^\CV\wedge \II^\CV+\abrack{\iden\wedge\QC^\CV}$ by the
Gau\ss\ equation~\eqref{eq:g1}, so it is zero if and only if $\Sh^\CV\wedge
\II^\CV+\abrack{\iden\wedge\QC^\CV}=0$. When $m\leq 3$, the Gau\ss\
equation~\eqref{eq:g2} instead gives $C^\submfd=
-\CA^{D,\CV}\wedge\II^\CV-\d^D\QC^\CV$ (with respect to any Weyl derivative
$D$). For $m=2$ we have $\QC^\CV=-H^\CV(\II^0)-\tfrac12 K^\CV\conf$. The
result easily follows.
\end{proof}
\begin{cor} Channel submanifolds are \mob/-flat.
\end{cor}
\begin{proof} By definition, a channel submanifold admits an enveloped sphere
congruence $\CV$ such that $\CS^\CV\colon T\submfd\to
\Hom(\CV,\CV^\perp)\dsum\Hom(\CV^\perp,\CV)$ has codimension one kernel. Hence
$\Sh^\CV\wedge\II^\CV=0$ (so the weightless normal bundle is flat),
$\II^\CV\wedge\Sh^\CV=0$ (and so $\QC^\CV=0$ for $m=3$ and $K^\CV=0$ for
$m=2$) and $\CA^{D,\CV}\wedge\II^\CV=0$ for any Weyl derivative $D$. Hence
$\submfd$ is \mob/-flat.
\end{proof}
The \emphdef{gaussian \textup(sectional\textup) curvature} of a $2$-plane
$U\subseteq T\submfd$ is defined to be $K_U=\ip{\detm\II^\CV\restr U}$. It is a
section of $\Ln^2$. We say that a submanifold of $S^n$ has \emphdef{constant
gaussian curvature} with respect to a compatible metric $g$ on $S^n$ if
$K_U=K^g$ for all $U$ and $D^gK^g=0$.
\begin{cor} A submanifold of a spaceform with flat weightless normal bundle
and constant gaussian curvature is \mob/-flat and the induced metric has
constant curvature.
\end{cor}
\begin{proof} Take $\CV=\CV_g$, the tangent congruence, and $D=D^g$, the
induced Weyl derivative, so that $\CA^{\CV_g,D^g}=0$. For $m=2$ we then have
$K^{\CV_g}=K^g$, hence $D^g K^{\CV_g}=0$.  More generally, $\Sh^g\wedge\II^g=
-\frac12K^g\sum_{i,j}\eps_i\wedge \eps_j\vtens\abrack{\eps_i,e_j}$ with
$D^gK^g=0$ where $e_i$ is a local frame of $T\submfd\ltens\Ln$ with dual frame
$\eps_i$. (This can be seen, for instance by viewing both sides as
$\Ln^2$-valued quadratic forms on $\Wedge^2(T\submfd\ltens\Ln)$.) Therefore
for $m\geq3$, $Q^{\CV_g}=-\frac12 K^g\conf$, hence
$\nr^{D^g,\submfd}=\nr^{D,\CV_g}-\QC^{\CV_g}= (\frac1 n \ns^g+ \frac12
K^g)\conf$.  The result easily follows.
\end{proof}
Of course, it is well-known that for $m\geq 3$, $m$-manifolds of constant
sectional curvature are conformally-flat, so the point of the above Corollary
is to take a submanifold point of view on this result, via gaussian
(sectional) curvature in a spaceform, and extend it to the case $m=2$.

It is also well known that for $m\geq 3$, a product of a $p$-manifold and an
$(m-p)$-manifold of constant sectional curvatures $c_p$ and $c_{m-p}$ is
conformally flat provided $c_p+c_{m-p}=0$. Note that the sectional curvature
of a $1$-manifold is undefined, and a product of a $1$-manifold and an
$(m-1)$-manifold of constant sectional curvature is always conformally flat.

Let us consider the case that this product is an extrinsic product of
submanifolds of spaceforms inside $S^n$, as in~\S\ref{par:sym-break-sub}. As
usual we focus on the nondegenerate case $S^n\setdif S^{k-1}= \cH^k\times
\cS^{n-k}$, but the degenerate case is similar. We thus obtain the following
submanifold analogue of the $m\geq 3$ result, together with an extension to
the case $m=2$.
\begin{thm} Let $\submfd_1^p\to\cH^k$ and $\submfd_2^{m-p}\to\cS^{n-k}$ be
immersed submanifolds of spaceforms \textup(with $1\leq p\leq k$, $1\leq
m-p\leq n-k$\textup). Then the product embedding $\submfd_1\times \submfd_2\to
S^n$ is \mob/-flat if and only if $\submfd_1$ and $\submfd_2$ have flat normal
bundles and constant gaussian curvatures $K_1$ and $K_2$ \textup(the latter
condition being automatic on a $1$-dimensional manifold\textup) such that
$K_1+K_2=0$ if $p>1$ and $m-p>1$.
\end{thm}
\begin{proof} Pulling back by the obvious projections we carry out all
computations on $\submfd=\submfd_1\times\submfd_2$, and let $\sigma_1$,
$\sigma_2$ be the induced maps $\submfd\to\cH^k$, $\submfd\to\cS^{n-k}$.  Then
$\submfd\times\R^{n+1,1}=\CV\dsum\CV^\perp$, with $\CV=
\vspan{\sigma_1,\d\sigma_1}\dsum \vspan{\sigma_2,\d\sigma_2}$,
$\CV^\perp=\CV_1^\perp\dsum\CV_2^\perp$ (corresponding to the normal bundles
of $\submfd_1$ and $\submfd_2$), and $\Ln=\vspan{\sigma_1+\sigma_2}$.
$\Lnc:=\vspan{\sigma_1-\sigma_2}$ is the Weyl structure induced by the product
geometry. Clearly $\CD^\CV$, $\CS^\CV$ and $\nabla^\CV$ split across the
decomposition of $\CV$ and $\CV^\perp$.  We deduce two things: first,
$\nabla^\CV$ is flat if and only if $\submfd_1$ and $\submfd_2$ have flat
normal bundle; second, $\CV$ is enveloped by both $\Ln$ and $\Lnc$, so
$A^{D,\CV}=0$ (for the Weyl derivative $D$ induced by $\Lnc$) and
$\CS^\CV=\II^\CV-\Sh^\CV$, with $\II^\CV=\II^1+\II^2$ being the (direct) sum
of the second fundamental forms of the factors (and $\Sh^\CV$ similarly). Now
\begin{equation*}
\Sh^\CV\wedge\II^\CV = \Sh^1\wedge\II^1+\Sh^2\wedge\II^2
\end{equation*}
and the equation $\Sh^\CV\wedge\II^\CV+\abrack{\iden\wedge\QC^\CV}$, which is
necessary for \mob/-flatness and determines $\QC^\CV$ for $m\geq 3$, implies
that $m=2$ or $\QC^\CV=f (\conf_1 - \conf_2)$ for some function $f$: indeed we
first observe that $\QC^\CV$ must split as $Q_1+Q_2$, with no cross term, but
then we obtain
\begin{equation*}
\abrack{\iden\wedge\QC^\CV}=
\abrack{\iden_1\wedge Q_1}+\abrack{\iden_2\wedge Q_2}
+\abrack{\iden_1\wedge Q_2}+\abrack{\iden_2\wedge Q_1}
\end{equation*}
and the last two terms can only cancel with $Q_1=f\conf_1$,
$Q_2=-f\conf_2$. For $m\geq 4$, $f$ is constant, while for $m=3$, $\d^D
\QC^\CV=0$ if and only if $f$ is constant.  By
Theorem~\ref{th:mob-flat-via-sphere-cong}, this completes the proof for $m\geq
3$.

For $m=2$, we write $\II^1=e_1^2\otimes \nu_1$ and $\II^2=e_2^2\otimes\nu_2$
with respect to a weightless orthonormal basis $e_1,e_2$. Then
$K^\CV=\ip{\detm\II^\CV}=\ip{\nu_1,\nu_2}=0$, so our product surface is always
\mob/-flat (provided the normal bundles are flat), with CCYP $q=H^\CV(\II^0)
=\frac14(e_1^2-e_2^2)\tens(|\nu_1|^2-|\nu_2|^2)$.
\end{proof}
\begin{rem} There is a fast way to see that these submanifolds are \mob/-flat,
using the equation $\d \mff^\CV=\RDVs$: simply take $\mff^\CV=
c(\sigma_1+\sigma_2)\skwend(\d\sigma_1-\d\sigma_2)$. When $m=2$,
$\d\mff^\CV=0$ (as we have seen, these product manifolds are isothermic) and
$\RDVs=0$. In higher dimensions $\d\mff^\CV=c\d\sigma_1\skwendwedge
\d\sigma_1-c\d\sigma_2\skwendwedge\d\sigma_2$, and so we take $c$ proportional
to the gaussian curvatures of the factors to get $\d \mff^\CV=\RDVs$.
\end{rem}

We now give a more explicit description of \mob/-flat submanifolds. For this,
recall that the flatness of the weightless normal bundle means that the
components of $\Sh^0$ commute, and so they are simultaneously
diagonalizable. This means we can locally write
\begin{equation*}\notag
\Sh^0 =\sum_{i=1}^m \eps_i\vtens e_i\vtens \con_i
\end{equation*}
where $\con_i$ are conormal vectors (whose components are the eigenvalues of
the components of $\Sh^0$) with $\sum_i \con_i=0$, and $e_i,\eps_i$ are dual
orthonormal frames of $T\submfd\ltens\Ln$ and $T\dual\submfd\ltens L$ (the
$e_i$ being the eigenvectors of $\Sh^0$). (Thus $\II^0=\sum_{i=1}^m \eps_i^2
\vtens\con_i^\sharp$, using the identification $S^2_0T\dual\submfd\tens
N\submfd\cong S^2_0(T\dual\submfd\vtens L)\tens (\Ln^2 \ltens N\submfd)$,
where $\con(\con_i^\sharp)=\ip{\con,\con_i}$.)  When $m=2$, $q$ is also
diagonalized in this frame, and we write $q=\mu^2 (\eps_1^2-\eps_2^2)$ with
$\mu$ a section of $\Ln$. (Up to reordering we can assume $\mu$ is real.) 

We can now extend results of Cartan and Hertrich-Jeromin to arbitrary
dimension and codimension for $m\geq3$, and a result of \cite{BPP:sdfs} to a
much broader context for $m=2$.

\begin{thm}\label{th:mob-flat-explicit}
Let $\submfd$ be an $m$-dimensional submanifold of $S^n$ $(m\geq 2)$ with flat
normal bundle, and write $\Sh^0 =\sum_{i=1}^m \eps_i\vtens e_i\vtens \con_i$
and \textup(for $m=2$\textup) $q=\mu^2 (\eps_1^2-\eps_2^2)$.  Then $\submfd$
is \mob/-flat if and only if\textup:
\begin{bulletlist}
\item $m=2$ and the \textup(complex\textup) $1$-forms
$\alpha_1=\sqrt{\ip{\con,\con}+\mu^2}\,\eps_1$ and
$\alpha_2=\sqrt{\ip{\con,\con}-\mu^2}\,\eps_2$, where $\con=\con_1=-\con_2$,
are closed\textup;
\item $m=3$ and the \textup(complex\textup) $1$-forms $\alpha_i$ are closed,
where
\begin{equation*}\begin{split}
\alpha_1&=\sqrt{\ip{\con_2-\con_1,\con_3-\con_1}}\,\eps_1,\\
\alpha_2&=\sqrt{\ip{\con_3-\con_2,\con_1-\con_2}}\,\eps_2,\\
\alpha_3&=\sqrt{\ip{\con_1-\con_3,\con_2-\con_3}}\,\eps_3;
\end{split}\end{equation*}
\item $m\geq 4$ and $\cip{\con_i-\con_j,\con_k-\con_\ell}=0$ for all pairwise
distinct $i,j,k,\ell\in\{1,\ldots m\}$.
\end{bulletlist}
\end{thm}
\begin{proof}
For $m\geq 4$ \mob/-flatness is equivalent to the vanishing of the Weyl
curvature $W^\submfd$, \ie, using~\eqref{eq:hG1}, to the equation
$\Sh^0\wedge\II^0+\abrack{\iden\wedge\QC}=0$. We now compute that
$\Sh^0\wedge\II^0 = \sum_{i,j} \ip{\con_i,\con_j} \eps_i\wedge \eps_j\vtens
\abrack{\eps_i, e_j}$ and also $(2-m)\abrack{\iden\wedge\QC}=
\sum_{i,j}\bigl(\ip{\con_i,\con_i}+\ip{\con_j,\con_j}-\frac1{m-1}
\sum_\ell\ip{\con_\ell,\con_\ell}\bigr)\eps_i\wedge \eps_j\vtens
\abrack{\eps_i, e_j}$, so that $W^\submfd=0$ is equivalent to
\begin{equation*}\notag
(m-1)(m-2)\cip{\con_i,\con_j}+(m-1)\cip{\con_i,\con_i}+(m-1)\cip{\con_j,\con_j}
-\tsum_\ell\cip{\con_\ell,\con_\ell}=0
\end{equation*}
for all $i\neq j$. (These conditions are vacuous for $m\leq3$.) The sum over
$i$ ($i\neq j$) of the left hand sides is zero, so it suffices to consider the
differences between pairs of equations, \ie,
\begin{equation*}\notag
\cip{\con_j-\con_k,(m-2)\con_i+\con_j+\con_k}=0
\end{equation*}
for all $i\neq j,k$. Again summing the left hand sides over $i$ (for $i\neq
j,k$) gives zero, so again the equations are equivalent to their
differences, which proves the theorem for $m\geq 4$.

For $m=2,3$, \mob/-flatness is instead expressed by the equation $C^\submfd=\d
q$. By~\eqref{eq:hG2}, the Cotton--York curvature
$C^\submfd=-\ip{\QCoda^\nabla(\II^0)}$ is given by a first order quadratic
differential operator applied to $\II^0$: $\ip{\QCoda^\nabla(\II^0)}=\d^D
\QC+\ip{\II^0,\d^{\nabla,D}\II^0}$ and $\QC$ is Cartan's tensor
$\ip{\II^0_\cdot,\II^0_\cdot}-\frac14 |\II^0|^2\conf$, which has eigenvalues
$\cip{\con_i,\con_i}-\frac14\sum_j \cip{\con_j,\con_j}$ for $1\leq i\leq m$.
We set $\mff=Q+q$ (which is $\mff^{\CV_\submfd}$) and observe that, since
$\sum_k \con_k=0$, $\mff=\sum_i \eps_i^2\vtens\mu_i$, where $\mu_i +\mu_j
=-\ip{\con_i,\con_j}$ for $i\neq j$ (these being the eigenvalues of operator
induced by $\mff$ on $\Wedge^2 T\submfd$).

We now compute
\begin{equation}\label{eq:dchi}
\ip{\d^D \mff,e_i} = \mu_i\d^D \eps_i+D\mu_i\wedge \eps_i+ \tsum_j\mu_j
\ip{De_j,e_i}\wedge \eps_j.
\end{equation}
On the other hand
\begin{equation}\label{eq:2d2}
\ip{\d^{\nabla,D}\II^0,\II^0_{e_i}} = \ip{\con_i,\con_i}\d^D \eps_i
+ \tfrac12 D\ip{\con_i,\con_i}\wedge \eps_i
+ \tsum_j \ip{\con_i,\con_j} \ip{De_j,e_i}\wedge \eps_j
\end{equation}
Now $\ip{De_j,e_i}=-\ip{De_i,e_j}$ (which is zero for $i=j$). Now $\submfd$ is
\mob/-flat iff $\ip{\QCoda^\nabla(\II^0)}+\d q=0$, and by
adding~\eqref{eq:dchi} and~\eqref{eq:2d2}, we find that this holds iff for all
$i=1,\ldots m$,
\begin{align*}
0 &= \ip{\d^D \mff,e_i}+\ip{\d^{\nabla,D}\II^0,\II^0_{e_i}}\\
&= (\mu_i+\ip{\con_i,\con_i})\d^D \eps_i+
D(\mu_i+\half\ip{\con_i,\con_i})\wedge \eps_i
+\tsum_j\mu_i\ip{De_i,e_j}\wedge \eps_j\\
&= (2\mu_i+\ip{\con_i,\con_i}) \d^D \eps_i + 
\half D(2\mu_i+\ip{\con_i,\con_i})\wedge \eps_i
=\tfrac12\lam_i \d(\lam_i \eps_i),
\end{align*}
where $\lam_i^2=2\mu_i+\ip{\con_i,\con_i}$. When $m=2$ and
$\con_1=-\con_2=\con$, $\lam_1^2=2(\ip{\con,\con}+\mu^2)$ and
$\lam_2^2=2(\ip{\con,\con}-\mu^2)$.  When $m=3$, $\lam_1^2 =
\ip{\con_2-\con_1,\con_3-\con_1}$, $\lam_2^2$ and $\lam_3^2$ being obtained by
cyclic permutations. This completes the proof.
\end{proof}
\begin{rem} For $m=2,3$ the proof shows a little more: if we do not assume
$C^\submfd$ is zero, then it is a universal nonzero constant multiple of
$\sum_j\d\alpha_j\tens\alpha_j$.
\end{rem}

We refer to the closed $1$-forms $\alpha_i$ (determined up to signs in the
cases $m=2$ and $m=3$) as the \emphdef{conformal fundamental $1$-forms} or the
\emphdef{conformal principal curvature forms}.  Locally we can write
$\alpha_i=\d x_i$ for coordinates $x_i$ which form a principal curvature net
wherever they are functionally independent, since their differentials are
orthogonal with respect to both $\conf$ and $\II^0$:
\begin{equation*}
\conf = \tsum_{i=1}^m\lam^{-2}_i{\d x_i}^2,\qquad\qquad
\II^0 = \tsum_{i=1}^m\lam_i^{-2}{\d x_i}^2\vtens\con_i.
\end{equation*}
When $m=2$, $\lam_1^2=\ip{\con,\con}+\mu^2$ and
$\lam_2^2=\ip{\con,\con}-\mu^2$, with $\nu=\nu_1=-\nu_2$, whereas for $m=3$,
$\lam_1^2=\ip{\con_{12},\con_{13}}$, $\lam_2^2=\ip{\con_{23},\con_{21}}$,
$\lam_3^2=\ip{\con_{31},\con_{32}}$, with $\con_{ij}=\con_i-\con_j$.

We end this paragraph by relating the case $m=2$ to \emph{Guichard surfaces}.
Suppose that $\ip{\con,\con}\neq\mu^2$, and write $\con^\sharp=\lam\wnv$ for a
weightless unit normal $\wnv$. Then there are locally functions $(x_1, x_2)$
on $\submfd$ (with $x_2$ real or pure imaginary) such that
\begin{equation}\label{eq:Mflat-c-II-q}
\conf = \frac{{\d x_1}^2}{\lam^2+\mu^2}+ \frac{{\d x_2}^2}{\lam^2-\mu^2},
\quad \II^0 = \lam S\vtens\wnv, \quad q = \mu^2 S\quad\text{with}\quad
S:= \frac{{\d x_1}^2}{\lam^2+\mu^2}- \frac{{\d x_2}^2}{\lam^2-\mu^2}.
\end{equation}
We now compare this (in codimension one) with the characterization by
Calapso~\cite{Cal:asg} of Guichard surfaces~\cite{Gui:ssi} in $\R^3$: he shows
that these surfaces have induced metric $g=E_1{\d x_1}^2+E_2{\d x_2}^2$ with
$\sqrt{E_1E_2}(\kappa_1-\kappa_2)=\sqrt{E_2\pm E_1}$, where $\kappa_1$ and
$\kappa_2$ are the principal curvatures. By allowing $x_2$ to be real or pure
imaginary, we can reduce to the condition $\sqrt{\pm
E_1E_2}(\kappa_1-\kappa_2)=\sqrt{\pm(E_1+ E_2)}$. Now, up to reordering of
$(x_1,x_2)$, this is equivalent to its square, which yields
$(\kappa_1-\kappa_2)^2=E_1^{-1}+E_2^{-1}$. This is exactly the condition
satisfied by metrics of the form $\ell^{-2}\conf$ in~\eqref{eq:Mflat-c-II-q}.
The sign of $x_2^2$ distinguishes between cases known by Calapso as Guichard
surfaces of the first and second kind.

\subsection{Conformally-flat hypersurfaces}
\label{par:cfh}

Conformally-flat hypersurfaces in $S^{m+1}$ have been classified (for $m\geq
3$ in~\cite{Car:dhc,Her:cfg,Her:cfh}. From the
Theorem~\ref{th:mob-flat-explicit} we can easily reobtain this classification
entirely within the realm of (\mob/) conformal geometry as follows.

For $m\geq 4$, the equations $(\con_i-\con_j)(\con_k-\con_\ell)=0$ force $m-1$
of the eigenvalues $\con_i$ of $\II^0$ to coincide (and conversely, these
equations are then satisfied). Therefore $\submfd$ has a curvature sphere of
multiplicity $m-1\geq3$ and hence is a channel submanifold by
Proposition~\ref{p:dupin}.

For $m=3$ there are two cases. First, if one of the conformal fundamental
$1$-forms vanishes, then two of the $\con_i$ coincide and $\submfd$ is again a
channel submanifold. Otherwise, two of the conformal fundamental $1$-forms are
nonzero and real, the third is nonzero and imaginary (so that
$x=(x_1,x_2,x_3)$ are really coordinates on $\R^{2,1}$).  Furthermore, $\sum_i
\lam_i^{-2}= (\con_{12}\con_{13})^{-1} +(\con_{23}\con_{21})^{-1}+
(\con_{31}\con_{32})^{-1}=0$ (multiply by $\con_{12}\con_{23}\con_{31}$), so
that the principal curvature net satisfies a condition studied by Guichard:
the sum of the squares of its Lam\'e functions is zero. Conversely, if the
principal curvature net is Guichard, with the $\nu_i$ given by $3\nu_1 =
\lam_1\lam_2/\lam_3-\lam_1\lam_3/\lam_2$ and cyclic
permutations thereof, then the Guichard condition ensures that the conformal
fundamental forms are the $\d x_i$, so that $\submfd$ is conformally-flat.

For $m=2$, there are again two cases: if $\mu^2=\cip{\con,\con}$ then
$q=\con(\II^0)$ and $\CV=\CV_\submfd+\con$ is a sphere congruence with
$K^\CV=0$ which induces the given flat \mob/ structure, so that $\II^\CV$
is degenerate. It follows from Theorem~\ref{th:mob-flat-via-sphere-cong} that
$A^{D,\CV}\wedge\II^\CV=0$ and hence $\submfd$ is a channel submanifold
enveloping the sphere-curve defined by $\CV$. Otherwise, as we saw in the
previous paragraph, $\submfd$ is a Guichard surface of the first or second
kind (given by~\eqref{eq:Mflat-c-II-q}) according to whether
$\mu^2>\cip{\con,\con}$ or $\mu^2<\cip{\con,\con}$.
\begin{bulletlist}
\item When $m\geq4$, a conformally-flat hypersurface in $S^{m+1}$ is a channel
submanifold, \ie, $m-1$ of the eigenvalues of $\II^0$ coincide.
\item When $m=3$, a conformally-flat hypersurface in $S^4$ is a either a
channel submanifold or is \emphdef{Guichard isothermic} in the sense that the
principal curvature net is a Guichard net.
\item When $m=2$, a \mob/-flat surface in $S^3$ is either a channel surface or
is a \emphdef{Guichard surface} as described by Calapso~\cite{Cal:asg}.
\end{bulletlist}
We have a little more to say on the cases $m=3$ and $m=2$.

\subsubsection*{Guichard nets and the \GCR/ equations}

In~\cite{Her:imdg}, Hertrich-Jeromin observes that any Guichard net arises
from a conformally flat hypersurface in $S^4$. We have already seen that
Guichard nets provide solutions to the Gau\ss\ equation for such
conformally-flat hypersurfaces, so it remains to verify the Codazzi equation.
We then apply the conformal Bonnet theorem, thus bypassing some of the
computations in~\cite{Her:imdg}.

Consider a Guichard net with coordinates $(x_1,x_2,x_3)$ and Lam\'e functions
$1/\lam_i$, as in the previous paragraph. Define $(\nu_1,\nu_2,\nu_3)$ by
$3\nu_1 = \lam_1\lam_2/\lam_3-\lam_1\lam_3/\lam_2$ and cyclic permutations and
set $\II^0=\sum_i (\d x_i/\lam_i)\tens\nu_i(\d x_i/\lam_i)$. To verify the
Codazzi equation, we need to show that $\d^D\II^0=\alpha\wedge g$ for some
$1$-form $\alpha$, where $D$ is the Levi-Civita connection of the metric
$g=\sum_i (\d x_i/\lam_i)^2$.

We compute $D$ via Cartan's method using the orthonormal frame $\omega_i=\d
x_i/\lam_i$ and find that $D\omega_i = \sum_j A_{ij}\tens\omega_j$ with
$A_{ij}=-(\del_i\lam_j)\omega_j/\lam_j+(\del_j\lam_i)\omega_i/\lam_i$ (and
$\del_i$ is shorthand for $\del/\del x_i$): we just need to check that
$A_{ij}=-A_{ji}$ and $\d\omega_i =\sum_j A_{ij}\wedge\omega_j$. Then
\begin{align*}
\d^D\II^0&=\sum_i D\omega_i\wedge(\nu_i\omega_i)
+\omega_i\tens\d\nu_i\wedge\omega_i+\omega_i\tens(\nu_i\d\omega_i)\\
&=\sum_{i,j} \omega_j\tens A_{ij}\wedge(\nu_i\omega_i)
+\omega_i\tens(\del_j\nu_i)\omega_j\wedge\omega_i+\omega_i\tens\nu_i
A_{ij}\wedge\omega_j\\
&=\sum_i \omega_i\tens \bigl((\del_j\nu_i)\omega_j
-(\lam_i^{-1}\del_j\lam_i) (\nu_i-\nu_j)\omega_j\bigr)\wedge\omega_i,
\end{align*}
where the last line follows by relabelling the indices in the first term and
substituting for $A_{ij}$. We thus need the term in brackets to be independent
of $i$ (modulo $\omega_i$). It is straightforward to verify the three
conditions this entails, once one appreciates the following identity for the
quantities entering into the definition of the $\nu$'s:
\begin{equation*}
\d\Bigl(\frac{\lam_j\lam_k}{\lam_i}\Bigr)=-\lam_i\Bigl(\frac{d\lam_j}{\lam_k}
+\frac{d\lam_k}{\lam_j}\Bigr)
\end{equation*}
for $i,j,k$ distinct.  This in turn is an easy consequence of the Guichard
condition $\sum_i \lam_i^{-2}=0$. Hence the Codazzi equation holds, and the
data ($[g],\II^0$) define a local conformally-flat immersion into $S^4$.

\subsubsection*{Strictly M\"obius-flat surfaces and Dupin cyclides}

We turn now to M\"obius-flat surfaces with CCYP $q=0$: in this case, or more
generally when $q$ is divergence-free (\ie, $q^{2,0}$ is holomorphic), the
conformal \mob/ structure induced by the central sphere congruence is flat,
and we shall say $\submfd$ is \emphdef{strictly M\"obius-flat}. We obtain the
M\"obius structure by setting $\mu=0$ in~\eqref{eq:Mflat-c-II-q} (we may use
this formula away from umbilic points). However, it follows from these
formulae that any strictly M\"obius-flat surface carries a \emph{nonzero}
holomorphic quadratic differential commuting with $\II^0$, namely $\lam^2 S$,
and hence is also an isothermic surface. Also $g=\lam^2\conf$ is flat, so
these surfaces are very special. We now show that they are in fact Dupin
cyclides.  Although this result is presented in~\cite{BPP:sdfs}, we discuss it
in more detail here to show how the properties of Dupin cyclides can be seen
explicitly in this setting.

\begin{prop} Away from umbilic points, a strictly \mob/-flat surface $\submfd$
in the conformal $3$-sphere has conformal metric and tracefree second
fundamental form given by equation~\eqref{eq:Mflat-c-II-q} with $\mu=0$
\textup(for some coordinates $x_1,x_2$, a weightless unit normal $\wnv$ and a
gauge $\lam$\textup), and the M\"obius structure given by
\begin{equation*}
\qquad\Mh=\sym_0 (D^g)^2+\tfrac t2 ({\d x_1}^2-{\d x_2}^2)
\end{equation*}
for a constant $t\in\R$, $D^g$ being the Levi-Civita connection of the flat
metric $g=\lam^2\conf$. Furthermore, $\submfd$ is isothermic with holomorphic
quadratic differential $q_0={\d x_1}^2-{\d x_2}^2$, and is Willmore iff $t=0$.
\end{prop}
\begin{proof} We have already obtained the formulae for $\conf$ and $\II^0$
and the conformal \mob/ structure is determined by $\Mh=\sym_0 (D^g)^2+\tilde
q$ for some quadratic differential $\tilde q$. However, both $\Mh$ and $\sym_0
(D^g)^2$ are flat, so $\tilde q^{2,0}$ must be holomorphic, and the Codazzi
equation gives $0=(\Mh^\nabla)^*(J\II^0)= \cip{\tilde q,J\II^0}$, so that
$\tilde q$ commutes with $\II^0$, and is therefore a constant multiple of
$q_0={\d x_1}^2-{\d x_2}^2$. Since $q_0$ is holomorphic and commutes with
$\II^0$, the surface is isothermic. It is Willmore iff $0=(\Mh^\nabla)^*\II^0
=\cip{\frac t 2 \tilde q,\II^0}$, \ie, $t=0$.
\end{proof}
Since the vector fields $\dbyd{x_1}$ and $\dbyd{x_2}$ leave the \GCR/ data
invariant, they induce symmetries of the surface. More precisely, with respect
to the decomposition
$\submfd\times\R^{4,1}=\Ln\dsum\US_g\dsum\Lnc_g\dsum\vspan{\wnv}$, where
$\US_g$ is orthogonal to $\Ln\dsum\Lnc_g\dsum\vspan{\wnv}$, we have
\begin{equation}\label{eq:Mflat-d}
\d = \iden + D^g+\nabla + \II^0-\Sh^0 + \tfrac{1+t}2 {\d x_1}^2 +\tfrac{1-t}2
{\d x_2}^2,
\end{equation}
(which, incidently, shows that $\Lnc_g$ is immersed unless $|t|=1$;
furthermore, it is also strictly \mob/-flat). Here we have use the fact that
the Weyl structure $\Lnc_g$ induced by $g$ is the second envelope of the
central sphere congruence (since $\CA^g=
-\divg^{\nabla,D^g}\II^0=0$). From~\eqref{eq:Mflat-c-II-q} (with $\mu=0$) and
\eqref{eq:Mflat-d}, we now easily check that
\begin{equation*}
\theta_1 = \bigl( \dbyd{x_1}, \II^0_{\dbyd{x_1}}, \tfrac{1+t}2 \d x_1 \bigr),
\qquad
\theta_2 = \bigl( \dbyd{x_2}, \II^0_{\dbyd{x_2}}, \tfrac{1-t}2 \d x_2 \bigr)
\end{equation*}
are constant sections of $\submfd\times\so(4,1)\cong (T\submfd\dsum
N\submfd)\dsum \co(T\submfd\dsum N\submfd) \dsum (T\dual\submfd\dsum
N\dual\submfd)$ and so are the differential lifts of $\dbyd{x_1}$ and
$\dbyd{x_2}$. Thus we have two elements of $\so(4,1)$ which commute and induce
vector fields on $S^3$ (their homology classes) which are tangent to
$\submfd$. It follows that $\submfd$ is an open subset of a Dupin cyclide,
cf.~\S\ref{par:sym-break-sub}. Furthermore $\theta_1$ and $\theta_2$ are the
decomposable vectors in their span: identifying $\so(n+1,1)$ with
$\Wedge^2\R^{n+1,1}$, we have
\begin{equation*}
\theta_1 = \bigl(-\lam^{-1}, \wnv, \tfrac{t+1}2\lam\bigr)
\wedge (0, -\lam^{-1}dx_1, 0),\qquad
\theta_2 = \hphantom{-}\bigl(\lam^{-1}, \wnv, \tfrac{t-1}2\lam\bigr)
\wedge (0, \lam^{-1}dx_2, 0).
\end{equation*}
The two planes defined by these decomposables are both spacelike for $|t|<2$,
while for $|t|=2$ one is spacelike, the other degenerate, and for $|t|>2$ one
is spacelike, the other has signature $(1,1)$. Accordingly, $\submfd$ is
(\mob/ equivalent to) an open subset of a circular torus of revolution, a
cylinder of revolution, or a cone of revolution.

Thus strictly \mob/-flat surfaces in $S^3$ form a very restricted class, in
contrast to the $3$-dimensional case, where there is a rich supply of
conformally flat hypersurface in $S^4$ from cones, cylinders or revolutes over
surfaces of constant gaussian curvature in $\cS^3$, $\R^3$ or $\cH^3$
respectively.  This was one of our motivations for broadening the notion of
\mob/-flat surface: the broader class includes arbitrary cones, cylinders, or
revolutes.

\section*{Preview of Part V}

The machinery of this paper reduces the theory of conformal immersions to its
homological essence, and we have seen the efficiency of this in some
applications.  Nevertheless, it is often more expedient to work with the
bundle formalism of sphere congruences, which we develop further in Part IV.
One reason for this is {\it simplicity}: the calculus of bundles and
connections is more straightforward and familiar than that of BGG operators.
Another reason is that the homological approach emphasises the central sphere
congruence, whereas other enveloped sphere congruences my be better adapted to
the problem at hand.

In Part V of our conformal submanifold geometry project, we return to
(constrained) Willmore surfaces, isothermic surfaces and M\"obius-flat
submanifolds to study their transformation theory. For this, sphere
congruences take centre stage, with the homological viewpoint as a
backdrop. In particular, we relate Ribaucour sphere congruences and B\"acklund
transformations of curved flats to Darboux and Eisenhart transformations of
isothermic surfaces and M\"obius-flat submanifolds respectively.  In the
isothermic case, these links are well-known, but they are less well understood
for M\"obius-flat submanifolds.

M\"obius-flat submanifolds arise as orthogonal submanifolds to flat spherical
systems and have a rich transformation theory, generalizing the
transformations of Guichard surfaces, developed by C.~Guichard~\cite{Gui:ssi},
P.~Calapso~\cite{Cal:asg} and L.~Eisenhart~\cite{Eis:tsg} in dimension two and
codimension one.  In dimension $3$ or more and codimension one, conformally
flat hypersurfaces were known to be related to flat cyclic
systems~\cite{Her:cfh}.  We provide a uniform theory in arbitrary dimension
and codimension, both of transformations of M\"obius-flat submanifolds, and of
the relation with curved flat spherical systems, providing further
justification for the definition given in this part.

We also apply the notion of polynomial conserved quantities to isothermic
surfaces and M\"obius-flat submanifolds. The former application has been
developed extensively in~\cite{BuSa:sis,San:sis}---in particular affine
conserved quantities provide a conformal approach to constant mean curvature
surfaces in arbitrary spaceforms (by regarding them as isothermic
surfaces). Similarly, affine conserved quantities for M\"obius-flat
submanifolds provide a conformal approach to constant gaussian curvature
submanifolds in arbitrary spaceforms (cf.~\cite{Eis:tsg}).
%
%
\newcommand{\noopsort}[1]{}
\newcommand{\bauth}[1]{\mbox{#1}}
\newcommand{\bart}[1]{\textit{#1}}
\newcommand{\bjourn}[4]{#1\ifx{}{#2}\else{ \textbf{#2}}\fi{ (#4)}}
\newcommand{\bbook}[1]{\textsl{#1}}
\newcommand{\bseries}[2]{#1\ifx{}{#2}\else{ \textbf{#2}}\fi}
\newcommand{\bpp}[1]{#1}
\newcommand{\bdate}[1]{ (#1)}
\newcommand{\bprint}[1]{#1}
\def\band/{and}
\newif\ifbibtex
\ifbibtex
%
%
\nocite{Aki:cdg}  
\nocite{AkGo:cdg} 
\nocite{Bas:ahs}  
\nocite{BEG:tcp}  
\nocite{BGG:dob}  
\nocite{Bla:vud}  
\nocite{Bry:dtw}  
\nocite{Bry:scg}  
\nocite{BCGGG:eds}
\nocite{Bur:is}   
\nocite{BuCa:sgfv}
\nocite{BFLPP:cgs}
\nocite{BuHJ:rls} 
\nocite{BPP:sdfs} 
\nocite{BuSa:sis} 
\nocite{Bia:rsi}  
\nocite{Bia:crsi} 
\nocite{Cal:asg}  
\nocite{Cal:mew}  
\nocite{CDS:rcd}  
\nocite{CaDi:di}  
\nocite{CaPe:ewg} 
\nocite{Car:dhc}  
\nocite{Car:ecc}  
\nocite{CpGo:tbi} 
\nocite{CpGo:tpg} 
\nocite{CpGo:stc} 
\nocite{CpSc:pgc} 
\nocite{CpSl:pg}  
\nocite{CpSl:wpg} 
\nocite{CSS:ahs2}
\nocite{CSS:bgg}  
\nocite{Dar:tgs}  
\nocite{Dar:socc} 
\nocite{Die:PhD}  
\nocite{EaRi:cio} 
\nocite{Feg:cio}  
\nocite{FePe:cfss}
\nocite{Fia:ctc}  
\nocite{Fia:cdg}  
\nocite{Eis:tos}  
\nocite{Eis:tsg}  
\nocite{Gui:ssi}  
\nocite{Gui:sto}  
\nocite{Gau:ccw}  
\nocite{Haa:cg1}  
\nocite{Haa:cg2}
\nocite{Haa:cg3}
\nocite{Her:cfg}  
\nocite{Her:cfh}  
\nocite{Her:imdg} 
\nocite{Jen:hoc}  
\nocite{KMS:nodg} 
\nocite{Kos:lac}  
\nocite{Lap:dgim} %
\nocite{Lep:bgg}  
\nocite{Mae:dmg}  
\nocite{Mus:caf}  
\nocite{Och:gfh}  
\nocite{OrRo:cgr} 
\nocite{OsSt:sd}  
\nocite{San:sis}  
\nocite{ScSu:mg1} 
\nocite{Sem:ckf}  
\nocite{Sul:mg2}
\nocite{Sul:mg3}
\nocite{Sha:dg}   
\nocite{TeUh:btlg}
\nocite{Tan:epl}  
\nocite{Toj:ise}  
\nocite{Tho:dig}  
\nocite{Tho:ukg}  
\nocite{WaCP:mgs} 
\nocite{WaCP:smg} 
\nocite{Wey:stm}  
\nocite{Wil:tcr}  
\nocite{Yan:cg1-4}
\nocite{Yan:ftc}
\nocite{Yan:ecc}
\nocite{Yan:gcs}
\nocite{Yan:fc12}
\nocite{Yan:fc34}
\nocite{YaMu:tfc}
\bibliographystyle{genbib}
\bibliography{papers}

\begin{thebibliography}{10}

\bibitem{Aki:cdg}
\bauth{M.~A. Akivis}, \bart{On the conformal differential geometry of
  multidimensional surfaces}, \bjourn{Mat. Sb.}{53}{}{1961} \bpp{53--72}.

\bibitem{AkGo:cdg}
\bauth{M.~A. Akivis \band/ V.~V. Goldberg}, \bbook{Conformal Differential
  Geometry and its Generalizations}, John Wiley and Sons, New York\bdate{1996}.

\bibitem{BEG:tcp}
\bauth{T.~N. Bailey, M.~G. Eastwood \band/ A.~R. Gover}, \bart{Thomas's
  structure bundle for conformal, projective and related structures},
  \bjourn{Rocky Mountain J.~Math.}{24}{}{1994} \bpp{1191--1217}.

\bibitem{Bas:ahs}
\bauth{R.~J. Baston}, \bart{Almost hermitian symmetric manifolds, \textup{I}
  {L}ocal twistor theory, \textup{II} {D}ifferential invariants}, \bjourn{Duke
  Math. J.}{63}{}{1991} \bpp{81--138}.

\bibitem{BGG:dob}
\bauth{I.~N. Bernstein, I.~M. Gel'fand \band/ S.~I. Gel'fand},
  \bart{Differential operators on the base affine space and a study of
  {$\mathfrak g$}-modules}, in \bbook{Lie Groups and their Representations},
  Adam Hilger, London\bdate{1975}.

\bibitem{Bia:crsi}
\bauth{L.~Bianchi}, \bart{Complementi alle ricerche sulle superficie isoterme},
  \bjourn{Annali di Mat.}{12}{}{1905} \bpp{19--54}.

\bibitem{Bia:rsi}
\bauth{L.~Bianchi}, \bart{Ricerche sulle superficie isoterme e sulla
  deformazione delle quadriche}, \bjourn{Annali di Mat.}{11}{}{1905}
  \bpp{93--157}.

\bibitem{Bla:vud}
\bauth{W.~Blaschke}, \bbook{Vorlesungen {\"u}ber {D}ifferentialgeometrie},
  Springer, Berlin\bdate{1929}.

\bibitem{Boh:cws}
\bauth{C.~Bohle}, \bart{Constrained {W}illmore surfaces}, in \bbook{Progress in
  Surface Theory} (Oberwolfach, April 29-May 5, 2007), vol.~4, Oberwolfach
  Rep.\bdate{2007}, pp.~1299--1376.

\bibitem{Bry:dtw}
\bauth{R.~L. Bryant}, \bart{A duality theorem for {W}illmore surfaces},
  \bjourn{J.~Diff. Geom.}{20}{}{1984} \bpp{23--53}.

\bibitem{Bry:scg}
\bauth{R.~L. Bryant}, \bart{Surfaces in conformal geometry}, in \bbook{The
  Mathematical Heritage of Hermann Weyl \textup(Durham, NC, 1987\textup)},
  Amer. Math. Soc., Providence, RI\bdate{1988}.

\bibitem{BCGGG:eds}
\bauth{R.~L. Bryant, S.~S. Chern, R.~B. Gardner, H.~L. Goldschmidt \band/ P.~A.
  Griffiths}, \bbook{Exterior Differential Systems}, \bseries{MSRI
  Publications}{18}, Springer-Verlag, New York\bdate{1991}.

\bibitem{Bur:is}
\bauth{F.~E. Burstall}, \bart{Isothermic surfaces: conformal geometry,
  {C}lifford algebras and integrable systems}, in \bbook{Integrable systems,
  Geometry and Topology} (C.-L. Terng, ed.), vol.~36, AMS/IP Studies in
  Advanced Math., Amer. Math. Soc., Providence\bdate{2006}.

\bibitem{BuCa:45}
\bauth{F.~E. Burstall \band/ D.~M.~J. Calderbank}, \bart{Conformal submanifold
  geometry, {IV--V}}, in preparation.

\bibitem{BuCa:psg}
\bauth{F.~E. Burstall \band/ D.~M.~J. Calderbank}, \bart{Parabolic
  subgeometries}, in preparation.

\bibitem{BuCa:sgfv}
\bauth{F.~E. Burstall \band/ D.~M.~J. Calderbank}, \bart{Submanifold goemetry
  in generalized flag varieties}, in \bbook{Winter School in Geometry and
  Physics} (Srni, 2003), \bseries{Rend. del Circ. mat. di
  Palermo}{72}\bdate{2004}, pp.~13--41.

\bibitem{BDPP:issr}
\bauth{F.~E. Burstall, N.~M. Donaldson, F.~Pedit \band/ U.~Pinkall},
  \bart{Isothermic submanifolds of symmetric {R}-spaces}, \bjourn{J.~reine
  angew. Math.}{}{}{to appear}, \bprint{arXiv:0906.1692 [math.DG]}.

\bibitem{BFLPP:cgs}
\bauth{F.~E. Burstall, D.~Ferus, K.~Leschke, F.~Pedit \band/ U.~Pinkall},
  \bbook{Conformal geometry of surfaces in {${\it S}\sp 4$} and quaternions},
  \bseries{Lecture Notes in Mathematics}{1772}, Springer-Verlag,
  Berlin\bdate{2002}.

\bibitem{BuHJ:rls}
\bauth{F.~E. Burstall \band/ U.~Hertrich-Jeromin}, \bart{The {R}ibaucour
  transformation in {L}ie sphere geometry}, \bjourn{Differential Geom.
  Appl.}{24}{}{2006} \bpp{503--520}.

\bibitem{BHPP:cfis}
\bauth{F.~E. Burstall, U.~Hertrich-Jeromin, F.~Pedit \band/ U.~Pinkall},
  \bart{Curved flats and isothermic surfaces}, \bjourn{Math. Z.}{225}{}{1997}
  \bpp{199--209}.

\bibitem{BPP:sdfs}
\bauth{F.~E. Burstall, F.~Pedit \band/ U.~Pinkall}, \bart{Schwarzian
  derivatives and flows of surfaces}, in \bbook{Differential geometry and
  integrable systems (Tokyo, 2000)}, \bseries{Contemp. Math.}{308}, Amer. Math.
  Soc., Providence, RI\bdate{2002}.

\bibitem{BuSa:sis}
\bauth{F.~E. Burstall \band/ S.~Santos}, \bart{Special isothermic surfaces of
  type d}, Preprint\bdate{2010}, \bprint{arXiv:1006.3175v1}.

\bibitem{Cal:asg}
\bauth{P.~Calapso}, \bart{Alcune superficie di {G}uichard e le relative
  trasformazioni}, \bjourn{Annali di Mat.}{11}{}{1905} \bpp{201--251}.

\bibitem{Cal:mew}
\bauth{D.~M.~J. Calderbank}, \bart{M{\"o}bius structures and two dimensional
  {E}instein--{W}eyl geometry}, \bjourn{J.~reine angew. Math.}{504}{}{1998}
  \bpp{37--53}.

\bibitem{CaDi:di}
\bauth{D.~M.~J. Calderbank \band/ T.~Diemer}, \bart{Differential invariants and
  curved {B}ernstein--{G}elfand--{G}elfand sequences}, \bjourn{J.~reine angew.
  Math.}{537}{}{2001} \bpp{67--103}, \bprint{math.DG/0001158}.

\bibitem{CDS:rcd}
\bauth{D.~M.~J. Calderbank, T.~Diemer \band/ V.~Sou{\v c}ek}, \bart{Ricci
  corrected derivatives and invariant differential operators}, \bjourn{Diff.
  Geom. Appl.}{23}{}{2005} \bpp{149--175}, \bprint{math.DG/0310311}.

\bibitem{CaPe:ewg}
\bauth{D.~M.~J. Calderbank \band/ H.~Pedersen}, \bart{{E}instein--{W}eyl
  geometry}, in \bbook{Essays on {E}instein Manifolds} (C.~R. LeBrun \band/
  M.~Wang, eds.), \bseries{Surveys in Differential Geometry}{VI}, International
  Press, Cambridge\bdate{1999}.

\bibitem{CpGo:tbi}
\bauth{A.~{\v C}ap \band/ A.~R. Gover}, \bart{Tractor bundles for irreducible
  parabolic geometries}, in \bbook{Global analysis and harmonic analysis}
  (Luminy, 1999), \bseries{S\'emin. Congr.}{4}, Soc. Math. France,
  Paris\bdate{2000}, pp.~129--154.

\bibitem{CpGo:tpg}
\bauth{A.~{\v C}ap \band/ A.~R. Gover}, \bart{Tractor calculi for parabolic
  geometries}, \bjourn{Trans. Amer. Math. Soc.}{354}{}{2002} \bpp{1511--1548}.

\bibitem{CpGo:stc}
\bauth{A.~{\v C}ap \band/ A.~R. Gover}, \bart{Standard tractors and the
  conformal ambient metric construction}, \bjourn{Ann. Global Anal.
  Geom.}{24}{}{2003} \bpp{231--259}.

\bibitem{CpSc:pgc}
\bauth{A.~{\v C}ap \band/ H.~Schichl}, \bart{Parabolic geometries and canonical
  {C}artan connections}, \bjourn{Hokkaido Math. J.}{29}{}{2000} \bpp{453--505}.

\bibitem{CpSl:wpg}
\bauth{A.~{\v C}ap \band/ J.~Slov{\'a}k}, \bart{{W}eyl structures for parabolic
  geometries}, \bjourn{Math. Scand.}{93}{}{2003} \bpp{53--90},
  \bprint{math.DG/0001166}.

\bibitem{CpSl:pg}
\bauth{A.~{\v C}ap \band/ J.~Slov{\'a}k}, \bbook{Parabolic Geometries I:
  Background and General Theory}, vol. 154, Mathematical Surveys and
  Monographs, Amer. Math. Soc., Providence\bdate{2009}.

\bibitem{CSS:ahs2}
\bauth{A.~{\v C}ap, J.~Slov{\'a}k \band/ V.~Sou{\v c}ek}, \bart{Invariant
  operators on manifolds with almost hermitian symmetric structures,
  \textup{II} {N}ormal {C}artan connections}, \bjourn{Acta Math. Univ.
  Comenianae}{66}{}{1997{\noopsort{b}}} \bpp{203--220}.

\bibitem{CSS:bgg}
\bauth{A.~{\v C}ap, J.~Slov{\'a}k \band/ V.~Sou{\v c}ek},
  \bart{{B}ernstein--{G}elfand--{G}elfand sequences}, \bjourn{Ann.
  Math.}{154}{}{2001} \bpp{97--113}, \bprint{math.DG/0001164}.

\bibitem{Car:dhc}
\bauth{{\'E}.~Cartan}, \bart{La d{\'e}formation des hypersurfaces dans
  l'{\'e}space conforme r{\'e}ell {\`a} {$n\geq 5$} dimensions}, \bjourn{Bull.
  Soc. Math. France}{45}{}{1917} \bpp{57--121}.

\bibitem{Car:ecc}
\bauth{{\'E}.~Cartan}, \bart{Les espaces \`a connexion conforme}, \bjourn{Ann.
  Soc. Pol. Math.}{2}{}{1923} \bpp{171--221}.

\bibitem{Dar:tgs}
\bauth{G.~Darboux}, \bbook{Le{\c c}ons sur la Th{\'e}orie G{\'e}n{\'e}rale des
  Surfaces et les Applications G{\'e}om{\'e}triques du Calcul
  Infinit{\'e}simal}, Gauthier-Villars, Paris\bdate{1887}.

\bibitem{Dar:socc}
\bauth{G.~Darboux}, \bbook{Le{\c c}ons sur les Syst{\`e}mes Orthogonaux et les
  Coordonn{\'e}es Curvilignes}, Gauthier-Villars, Paris\bdate{1910}.

\bibitem{Die:PhD}
\bauth{T.~Diemer}, \bbook{Conformal Geometry, Representation Theory and Linear
  Fields}, PhD. thesis, Universit{\"a}t Bonn\bdate{1999}.

\bibitem{EaRi:cio}
\bauth{M.~G. Eastwood \band/ J.~W. Rice}, \bart{Conformally invariant operators
  on {M}inkowski space and their curved analogues}, \bjourn{Comm. Math.
  Phys.}{109}{}{1987} \bpp{207--228}.

\bibitem{Eis:tsg}
\bauth{L.~P. Eisenhart}, \bart{Transformations of surfaces of {G}uichard and
  surfaces applicable to quadrics}, \bjourn{Annali di Mat.}{22}{}{1914}
  \bpp{191--247}.

\bibitem{Eis:tos}
\bauth{L.~P. Eisenhart}, \bbook{Transformations of surfaces}, Second edition,
  Chelsea Publishing Co., New York\bdate{1962}.

\bibitem{Feg:cio}
\bauth{H.~D. Fegan}, \bart{Conformally invariant first order differential
  operators}, \bjourn{Quart. J.~Math.}{27}{}{1976} \bpp{371--378}.

\bibitem{FePe:cfss}
\bauth{D.~Ferus \band/ F.~Pedit}, \bart{Curved flats in symmetric spaces},
  \bjourn{Manuscripta Math.}{91}{}{1996} \bpp{445--454}.

\bibitem{Fia:ctc}
\bauth{A.~Fialkow}, \bart{The conformal theory of curves}, \bjourn{Trans. Amer.
  Math. Soc.}{51}{}{1942} \bpp{435--501}.

\bibitem{Fia:cdg}
\bauth{A.~Fialkow}, \bart{Conformal differential geometry of a subspace},
  \bjourn{Trans. Amer. Math. Soc.}{56}{}{1944} \bpp{309--433}.

\bibitem{Gau:ccw}
\bauth{P.~Gauduchon}, \bart{Connexion canonique et structures de {W}eyl en
  g{\'e}om{\'e}trie conforme}, Preprint, Ecole Polytechnique\bdate{1990}.

\bibitem{Gui:ssi}
\bauth{C.~Guichard}, \bart{Sur les surfaces isothermiques}, \bjourn{Compt.
  Rend. Acad. Sci. Paris}{130}{}{1900} \bpp{159--162}.

\bibitem{Gui:sto}
\bauth{C.~Guichard}, \bbook{Sur les syst{\`e}mes triplement
  ind{\'e}termin{\'e}s et sur les syst{\`e}mes triple-orthogonaux},
  \bseries{Scienta}{25}, Gauthier-Villars, Paris\bdate{1905}.

\bibitem{Haa:cg1}
\bauth{J.~Haantjes}, \bart{Conformal differential geometry. {C}urves in
  conformal euclidean spaces}, \bjourn{Proc. Nederl. Akad.
  Wetensch.}{44}{}{1941} \bpp{814--824}.

\bibitem{Haa:cg2}
\bauth{J.~Haantjes}, \bart{Conformal differential geometry, \textup{II}
  {C}urves in conformal two-dimensional spaces}, \bjourn{Proc. Nederl. Akad.
  Wetensch.}{45}{}{1942} \bpp{249--255}.

\bibitem{Haa:cg3}
\bauth{J.~Haantjes}, \bart{Conformal differential geometry, \textup{III}
  {S}urfaces in three-dimensional space}, \bjourn{Proc. Nederl. Akad.
  Wetensch.}{45}{}{1942} \bpp{836--841}.

\bibitem{Her:cfg}
\bauth{U.~Hertrich-Jeromin}, \bart{On conformally flat hypersurfaces and
  {G}uichard's nets}, \bjourn{Beitr\"age Algebra Geom.}{35}{}{1994}
  \bpp{315--331}.

\bibitem{Her:cfh}
\bauth{U.~Hertrich-Jeromin}, \bart{On conformally flat hypersurfaces, curved
  flats and cyclic systems}, \bjourn{Manuscripta Math.}{91}{}{1996}
  \bpp{455--466}.

\bibitem{Her:imdg}
\bauth{U.~Hertrich-Jeromin}, \bbook{Introduction to {M}\"obius differential
  geometry}, \bseries{London Mathematical Society Lecture Note Series}{300},
  Cambridge University Press, Cambridge\bdate{2003}.

\bibitem{Jen:hoc}
\bauth{G.~Jensen}, \bbook{Higher order contact of submanifolds of homogeneous
  spaces}, \bseries{Lecture Notes in Mathematics}{610}, Springer,
  Berlin--Heidelberg\bdate{1977}.

\bibitem{KMS:nodg}
\bauth{I.~Kol{\'a\v r}, P.~W. Michor \band/ J.~Slov{\'a}k}, \bbook{Natural
  operations in differential geometry}, Springer-Verlag, Berlin\bdate{1993}.

\bibitem{Kos:lac}
\bauth{B.~Kostant}, \bart{{L}ie algebra cohomology and the generalized
  {B}orel--{W}eil theorem}, \bjourn{Ann. Math.}{74}{}{1961} \bpp{329--387}.

\bibitem{Lap:dgim}
\bauth{G.~F. Laptev}, \bart{Differential geometry of imbedded manifolds.
  {G}roup theoretical method of differential geometric investigations},
  \bjourn{Trudy Moskov. Mat. Ob\v s\v c.}{2}{}{1953} \bpp{275--382}.

\bibitem{Lep:bgg}
\bauth{J.~Lepowsky}, \bart{A generalization of the
  {B}ernstein--{G}elfand--{G}elfand resolution},
  \bjourn{J.~Algebra}{49}{}{1977} \bpp{496--511}.

\bibitem{Mae:dmg}
\bauth{J.~Maeda}, \bart{Differential {M}\"obius geometry of plane curves},
  \bjourn{Jap. J. Math.}{18}{}{1942} \bpp{67--260}.

\bibitem{Mus:caf}
\bauth{E.~Musso}, \bart{The conformal arclength functional}, \bjourn{Math.
  Nachr.}{165}{}{1994} \bpp{107--131}.

\bibitem{Och:gfh}
\bauth{T.~Ochiai}, \bart{Geometry associated with semisimple flat homogeneous
  spaces}, \bjourn{Trans. Amer. Math. Soc.}{152}{}{1970} \bpp{159--193}.

\bibitem{OrRo:cgr}
\bauth{L.~Ornea \band/ G.~Romani}, \bart{Conformal geometry of {R}iemannian
  submanifolds: {G}auss, {C}odazzi and {R}icci equations}, \bjourn{Rend. Mat.
  Appl. (7)}{15}{}{1995} \bpp{233--249}.

\bibitem{OsSt:sd}
\bauth{B.~Osgood \band/ D.~Stowe}, \bart{The {S}chwarzian derivative and
  conformal mapping of {R}iemannian manifolds}, \bjourn{Duke Math.
  J.}{67}{}{1992} \bpp{57--99}.

\bibitem{San:sis}
\bauth{S.~Santos}, \bbook{Special isothermic surfaces}, PhD. thesis, University
  of Bath\bdate{2009}.

\bibitem{ScSu:mg1}
\bauth{C.~Schiemangk \band/ R.~Sulanke}, \bart{Submanifolds of the {M}\"obius
  space}, \bjourn{Math. Nachr.}{96}{}{1980} \bpp{165--183}.

\bibitem{Sem:ckf}
\bauth{U.~Semmelmann}, \bart{Conformal {K}illing forms on {R}iemannian
  manifolds}, \bjourn{Math. Z.}{245}{}{2003} \bpp{503--527}.

\bibitem{Sha:dg}
\bauth{R.~W. Sharpe}, \bbook{Differential Geometry}, \bseries{Graduate Texts in
  Mathematics}{166}, Springer-Verlag, New York\bdate{1997}.

\bibitem{Sul:mg2}
\bauth{R.~Sulanke}, \bart{Submanifolds of the {M}\"obius space, \textup{II}
  {F}renet formulas and curves of constant curvatures}, \bjourn{Math.
  Nachr.}{100}{}{1981} \bpp{235--247}.

\bibitem{Sul:mg3}
\bauth{R.~Sulanke}, \bart{Submanifolds of the {M}\"obius space, \textup{III}
  {T}he analogue of {O}. {B}onnet's theorem for hypersurfaces}, \bjourn{Tensor
  (N.S.)}{38}{}{1982} \bpp{311--317}.

\bibitem{Tan:epl}
\bauth{N.~Tanaka}, \bart{On the equivalence problem associated with simple
  graded {L}ie algebras}, \bjourn{Hokkaido Math. J.}{8}{}{1979} \bpp{23--84}.

\bibitem{TeUh:btlg}
\bauth{C.-L. Terng \band/ K.~Uhlenbeck}, \bart{B\"acklund transformations and
  loop group actions}, \bjourn{Comm. Pure Appl. Math.}{53}{}{2000} \bpp{1--75}.

\bibitem{Tho:dig}
\bauth{T.~Y. Thomas}, \bbook{The Differential Invariants of Generalized
  Spaces}, Cambridge Univ. Press, Cambridge\bdate{1934}.

\bibitem{Tho:ukg}
\bauth{G.~Thomsen}, \bart{{\"U}ber konforme {G}eometrie, \textup{I}
  {G}rundlagen der konformen {F}l{\"a}chentheorie}, \bjourn{Hamb. Math.
  Abh.}{3}{}{1923} \bpp{31--56}.

\bibitem{Toj:ise}
\bauth{R.~Tojeiro}, \bart{Isothermic submanifolds of {E}uclidean space},
  \bjourn{J.~reine angew. Math.}{598}{}{2006} \bpp{1--24}.

\bibitem{WaCP:smg}
\bauth{C.~P. Wang}, \bart{Surfaces in {M}\"obius geometry}, \bjourn{Nagoya
  Math. J.}{125}{}{1992} \bpp{53--72}.

\bibitem{WaCP:mgs}
\bauth{C.~P. Wang}, \bart{M{\"o}bius geometry of submanifolds in ${S}\sp n$},
  \bjourn{Manuscripta Math.}{96}{}{1998} \bpp{517--534}.

\bibitem{Wey:stm}
\bauth{H.~Weyl}, \bbook{Space, Time, Matter}, Dover, New York\bdate{1952},
  Translation of the fourth edition of ``Raum, Zeit, Materie'', the first
  edition of which was published in 1918 by Springer, Berlin.

\bibitem{Wil:tcr}
\bauth{T.~J. Willmore}, \bbook{Total curvature in {R}iemannian geometry}, Ellis
  Horwood Ltd., Chichester\bdate{1982}.

\bibitem{Yan:ecc}
\bauth{K.~Yano}, \bart{Sur la th\'eorie des espaces \`a connexion conforme},
  \bjourn{J. Fac. Sci. Imp. Univ. Tokyo. Sect. 1.}{4}{}{1939} \bpp{1--59}.

\bibitem{Yan:cg1-4}
\bauth{K.~Yano}, \bart{Concircular geometry, \textup{I--IV}}, \bjourn{Proc.
  Imp. Acad. Tokyo}{16}{}{1940} \bpp{195--200, 354--360, 442--448, 505--511}.

\bibitem{Yan:ftc}
\bauth{K.~Yano}, \bart{On the fundamental theorem of conformal geometry},
  \bjourn{Tensor}{5}{}{1942} \bpp{51--59}.

\bibitem{Yan:gcs}
\bauth{K.~Yano}, \bart{Sur les \'equations fondamentales dans la g\'eom\'etrie
  conforme des sous-espaces}, \bjourn{Proc. Imp. Acad. Tokyo}{19}{}{1943}
  \bpp{326--334}.

\bibitem{Yan:fc12}
\bauth{K.~Yano}, \bart{On the flat conformal differential geometry,
  \textup{I--II}}, \bjourn{Proc. Japan Acad.}{21}{}{1945}
  \bpp{419--429,454--465}.

\bibitem{Yan:fc34}
\bauth{K.~Yano}, \bart{On the flat conformal differential geometry,
  \textup{III--IV}}, \bjourn{Proc. Japan Acad.}{22}{}{1946} \bpp{9--31}.

\bibitem{YaMu:tfc}
\bauth{K.~Yano \band/ Y.~Mut\^o}, \bart{Sur le th{\'e}or{\`e}me fondamental
  dans la g{\'e}om{\'e}trie conforme des sous-espaces riemanniens},
  \bjourn{Proc. Phys.-Math. Soc. Japan (3)}{24}{}{1942} \bpp{437--449}.

\end{thebibliography}
\else

\fi
\end{document}